\documentclass[10pt,reqno]{amsart}
\usepackage{amsmath,amssymb,mathrsfs,graphicx}
\usepackage{ifthen}
\usepackage[margin=0.8in]{geometry}
\usepackage{caption}
\usepackage{sidecap}
\usepackage{rotating}
\usepackage{colortbl}

\newcommand{\bbr}{\mathbb{R}}

\provideboolean{shownotes} 
\setboolean{shownotes}{true} 
\newcommand{\margnote}[1]{
\ifthenelse{\boolean{shownotes}}%
{\marginpar{\raggedright\tiny\texttt{#1}}}%
{}%
}
\newcommand{\hole}[1]{
\ifthenelse{\boolean{shownotes}}%
{\begin{center} \fbox{ \rule {.25cm}{0cm} \rule[-.1cm]{0cm}{.4cm}
\parbox{.85\textwidth}{\begin{center} \texttt{#1}\end{center}} \rule
{.25cm}{0cm}}\end{center}} {} }

\graphicspath{{pics/}}

\title[Global solution to full N-S equation]{Compressible Navier-Stokes equation with BV initial data \\ \small{PART II. Global stability}}

\author[Wang]{Haitao Wang}
\address{Haitao Wang, Institute of Natural Sciences and School of Mathematical Sciences, LSC-MOE,  Shanghai Jiao
	Tong University, Shanghai, China}
\email{haitallica@sjtu.edu.cn}

\author[Yu]{Shih-Hsien Yu}
\address{Shih-Hsien Yu, Institute of Mathematics, Academia Sinica,
	Taipei, Taiwan}
\email{matysh@gate.sinica.edu.tw}

\author[Zhang]{Xiongtao Zhang}
\address{Xiongtao Zhang, Center for Mathematical Sciences\newline
 Huazhong University of Science and Technology, Luoyu Road 1037, Wuhan 430074, China}
\email{xtzhang@hust.edu.cn}

\numberwithin{equation}{section}

\newtheorem{theorem}{Theorem}[section]
\newtheorem{lemma}{Lemma}[section]
\newtheorem{corollary}{Corollary}[section]
\newtheorem{proposition}{Proposition}[section]
\newtheorem{remark}{Remark}[section]
\newtheorem{definition}{Definition}[section]

\newcommand{\R}{\mathbb R}
\newcommand{\bbm}{\mathbb M}

\newcommand{\bbn}{\mathbb N}
\newcommand{\bbg}{\mathbb G}

\newcommand{\bbz} {\mathbb Z}

\newcommand{\bq}{\begin{equation}}
\newcommand{\eq}{\end{equation}}

\newcommand{\scrD}{\mathscr D}

\newcommand{\vertiii}[1]{{\left\vert\kern-0.25ex\left\vert\kern-0.25ex\left\vert #1 
		\right\vert\kern-0.25ex\right\vert\kern-0.25ex\right\vert}}

\let\oldtocsection=\tocsection
 \let\oldtocsubsection=\tocsubsection
 \let\oldtocsubsubsection=\tocsubsubsection
 \renewcommand{\tocsection}[2]{\hspace{0em}\oldtocsection{#1}{#2}}
\renewcommand{\tocsubsection}[2]{\hspace{1em}\oldtocsubsection{#1}{#2}}
\renewcommand{\tocsubsubsection}[2]{\hspace{2em}\oldtocsubsubsection{#1}{#2}}

\begin{document}
\allowdisplaybreaks

\date{\today}

\subjclass[]{}

\begin{abstract} 
	In previous work \cite{W-Y-Z-local}, we studied the local well-posedness of weak solution to the 1-D full compressible Navier-Stokes equation with initial data of small total variation. Specifically, the local existence, the regularity, and the uniqueness in certain function space of the weak solution have been established. The basis for previous study is the precise construction of fundamental solution for heat equation with BV conductivity. In this paper, we continue to investigate the global stability and the time asymptotic behavior of the weak solution. The main step is to construct the ``effective Green's function'', which is the combination of the heat kernel with BV coefficient in short time and the Green's function around  constant state in long time. The former one captures the quasi-linear nature of the system, while the latter one respects the dissipative structure. Then the weak solution is written into an integral system in terms of this ``effective Green's function'', and the time asymptotic behavior is established based on a priori estimate. 
\end{abstract}

\maketitle \centerline{\date}

\tableofcontents

%
%
%
%
\section{Introduction}\label{sec:1}
\setcounter{equation}{0}

Consider 1-D full compressible Navier-Stokes equation (NS for short) in Lagrangian coordinate.  More precisely, let $v$ be the specific volume, $u$ be the velocity, $p$ be the pressure, $e$ be the specific internal energy, $\theta$ be the absolute temperature, $\mu$ be the viscosity of $u$, and $\kappa$ be the heat conductivity of $\theta$, the full NS equation in Lagrangian coordinates reads,
\bq\label{NS}
\left\{\begin{aligned}
	&v_t-u_x=0\\
	&u_t+p_x=\left(\frac{\mu u_x}{v}\right)_x\\
	&(e+\frac{1}{2}u^2)_t+(pu)_x=\left(\frac{\kappa}{v}\theta_x+\frac{\mu}{v}uu_x\right)_x.
\end{aligned}
\right.
\eq
The above equation \eqref{NS} can be derived from the original compressible NS equation in Euler coordinates, we refer the readers to \cite{Sm} for details. In order to close the equation \eqref{NS}, we will consider in this paper the polytropic ideal gas which has the constitutive relations below,
\bq\label{constitution}
p(v,\theta)= \frac{K\theta}{v},\quad e=c_v \theta,
\eq
where $K$ is a positive constant and $c_v$ is the heat capacity which is also a positive constant. Note the second equality states that the internal energy  depends only on the temperature for ideal gas. To solve Cauchy problem, the system \eqref{NS} is supplemented with prescribed initial data $(v_0,u_0,\theta_0)$.

If $(v,u,e)$ is a classical solution to \eqref{NS}, then \eqref{NS} is equivalent to the following non-conservative form in terms of function $(v,u,\theta)$,
\bq\label{NS2}
\left\{\begin{aligned}
	&v_t-u_x=0\\
	&u_t+p_x=\left(\frac{\mu u_x}{v}\right)_x\\
	&\theta_t+\frac{p}{c_v}u_x-\frac{\mu}{c_vv}\left(u_x\right)^2=\left(\frac{\kappa}{c_vv}\theta_x\right)_x.
\end{aligned}
\right.
\eq
However, for weak solution, due to the nonlinear term $u^2$,  the equivalence between \eqref{NS} and \eqref{NS2} is not obvious.

In previous work \cite{W-Y-Z-local}, we consider \eqref{NS2} with initial data being a perturbation around constant state $(1,0,1)$ and satisfying the following condition.
\bq\label{D1}
\|v^*_0\|_{BV}+\|u^*_0\|_{BV}+\|\theta^*_0\|_{BV}+\|v^*_0\|_{L^1}+\|u^*_0\|_{L^1}+\|\theta^*_0\|_{L^1}< \delta\ll 1.
\eq
Here
\[
(v,u,\theta)=(v^*+1,u^*,\theta^*+1).
\] 

The local solution is constructed through the following iteration scheme,
\begin{equation}
	\label{NS-lin-2}
	\left\{
	\begin{aligned}
		& v^{n+1}_t - u^{n+1}_x=0,\\
		& u^{n+1}_t + p(v^n,\theta^n)_x = \left(\frac{\mu u^{n+1}_x}{v^n}\right)_x,\\
		& \theta^{n+1}_t  + \frac{p(v^n,\theta^n)}{c_v} u^n_x - \frac{\mu}{c_v v^n}(u^n_x)^2 = \left(\frac{\kappa}{c_v v^n} \theta^{n+1}_x\right)_x.
	\end{aligned}
	\right.
\end{equation}
It is noted the second and third equalities in \eqref{NS-lin-2} are indeed linear inhomogeneous heat equation with BV variable conductivity. The fundamental solution  $H(x,t;y,s;\rho)$ for heat equation with BV coefficient $\rho(x,t)$, i.e.,
\begin{equation}\label{eq:fun}
	\partial_t H(x,t;y,s;\rho)=\partial_x \bigl( \rho(x,t) \partial_x H(x,t;y,s;\rho)		\bigr),\quad H(x,s;y,s;\rho)=\delta(x-y), \quad t>s,
\end{equation}
is explicitly constructed in \cite{L-Y-2020}. Relying on  fundamental solution $H(x,t;y,s;\rho)$, the system \eqref{NS-lin-2} is solved by the following integral representation:
\begin{equation}
	\label{NS-lin-2-sol}
	\left\{
	\begin{aligned}
		v^{n+1}(x,t)& = v_0(x)+ \int_{0}^{t} u_x^{n+1}(x,s) ds,\\
		u^{n+1}(x,t)& = \int_{\bbr} H(x,t;y,0;\frac{\mu}{v^n}) u_0(y) dy + \int_{0}^{t}\int_{\bbr } H_y(x,t;y,s;\frac{\mu}{v^n}) p(v^n,\theta^n)(y,s) dy ds,\\
		\theta^{n+1}(x,t) & = \int_{\bbr} H(x,t;y,0;\frac{\kappa }{c_v v^n}) \theta_0(y) dy + \int_{0}^{t} \int_{\bbr} H(x,t;y,s;\frac{\kappa}{c_v v^n}) \Bigl[-\frac{p(v^n,\theta^n)}{c_v} u^n_y + \frac{\mu}{c_v} \frac{(u^n_y)^2}{v^n}		\Bigr](y,s) dy ds.
	\end{aligned}
	\right.
\end{equation}
Thanks to the pointwise estimates of heat kernel, the comparison estimates when varying the conductivity coefficient constructed in \cite{L-Y-2020}, as well as some new estimates concerning the H\"{o}lder continuity of heat kernel, in \cite{W-Y-Z-local}, we proved the convergence of iteration scheme \eqref{NS-lin-2-sol} in short time, thus obtained the weak solution to system \eqref{NS2}. As the convergence obtained is in strong sense, passing limit  $n\to \infty$ in \eqref{NS-lin-2-sol}, one has the integral representation of the solution $(v,u,\theta)$ to \eqref{NS2}. By exploiting the refined estimates of heat kernel, we further get more descriptions of regularity of the weak solution. For instance, we showed that  both $u_t$ and $\theta_t$ are in $L^1\cap L^\infty$ for $t>0$, and the BV norm of $v(\cdot,t)$ is H\"{o}lder continuous in time. These properties specify the function space which the weak solution belongs to. More importantly, we are able to prove the continuous dependence on initial data of weak solution in that function space, and thus the uniqueness of weak solution. The main results in \cite{W-Y-Z-local} is summarized as follows:

\begin{proposition}[Local Existence, \cite{W-Y-Z-local}]\label{prop:local}
	Suppose initial data $(v_0,u_0,\theta_0)$ satisfy
	\begin{equation}\label{eq:initial}
		\left\|v_0-1\right\|_{L^1_x}+\left\|v_0\right\|_{BV}+\left\|u_0\right\|_{L^1_x}+\left\|u_0\right\|_{BV}+\left\|\theta_0-1\right\|_{L^1_x}+\left\|\theta_0\right\|_{BV}\leq \delta,
	\end{equation}
	for $\delta$ sufficiently small. Then there exist positive constants $t_\sharp$ and $C_\sharp$ such that, the system \eqref{NS2} admits a  weak solution $(v,u,\theta)$ for $t\in(0,t_\sharp)$ satisfying the following properties:
	\begin{enumerate}
		\item The solution has the following regularity estimates: 
		\begin{equation}\label{local-small-strong}
			\left\{\begin{aligned}
				&\delta>0,\quad 0<t<t_\sharp\ll 1,\\
				&\max\left\{\begin{aligned}&\left\|u(\cdot,t)\right\|_{L^1_x},\ \left\|u(\cdot,t)\right\|_{L^\infty_x},\ \left\|u_x(\cdot,t)\right\|_{L^1_x},\ \sqrt{t}\left\|u_x(\cdot,t)\right\|_{L^\infty_x},\\ &\sqrt{t}\left\|u_t(\cdot,t)\right\|_{L^1_x},\ t\left\|u_t(\cdot,t)\right\|_{L^\infty_x}\end{aligned}\right\}\leq 2C_\sharp\delta,\\
				&\max\left\{\begin{aligned}&\left\|\theta(\cdot,t)-1\right\|_{L^1_x},\ \left\|\theta(\cdot,t)-1\right\|_{L^\infty_x},\ \left\|\theta_x(\cdot,t)\right\|_{L^1_x},\ \sqrt{t}\left\|\theta_x(\cdot,t)\right\|_{L^\infty_x},\\ &\sqrt{t}\left\|\theta_t(\cdot,t)\right\|_{L^1_x},\ t\left\|\theta_t(\cdot,t)\right\|_{L^\infty_x}\end{aligned}\right\}\leq 2C_\sharp\delta,\\
				&\max\left\{\int_{\bbr\setminus\mathscr{D}}|v_x(x,t)|dx,\ \left\|v(\cdot,t)-1\right\|_{L^1_x},\ \left\|v(\cdot,t)-1\right\|_{L^\infty_x},\ \sqrt{t}\left\|v_t(\cdot,t)\right\|_{L^\infty_x}\right\}\leq 2C_\sharp\delta,\\
				&v^*=v^*_c+v^*_d,\quad  v^*_d(x,t)=\sum_{z<x,z\in\mathscr{D}}v^*\Big|_{z^-}^{z^+}h(x-z),\quad \text{$v_c^*$ is AC},\\
				&\left|v(\cdot,t)\Big|_{x=z^-}^{x=z^+}\right|\leq 2\left|v^*_0(\cdot)\Big|_{x=z^-}^{x=z^+}\right|,\quad z\in\mathscr{D},
			\end{aligned}
			\right.
		\end{equation}
		where $h(x)$ is the Heaviside step function, and $AC$ means absolutely continuous. 
		\item The fluxes of $u$ and $\theta$ are both global Lipschitz continuous with respect to $x$ for any $t>0$, where the flux of $u$ and $\theta$ in \eqref{NS2} are defined as follows, 
		\[ \bigl[\mbox{flux of } u \bigr]=\frac{\mu u_x}{v}-p,\qquad \bigl[\mbox{flux of } \theta\bigr]=\frac{\kappa}{c_vv}\theta_x-\int_{-\infty}^x\left(\frac{p}{c_v}u_z-\frac{\mu}{c_vv}\left(u_z\right)^2\right)dz.\] 
		\item The specific volume $v(x,t)$  has the following H\"older continuity in time properties for $ 0\leq s<t$,
		\begin{equation}\label{eq:v-conti}
			\left\{\begin{aligned}
				&\int_{\bbr\setminus\mathscr{D}}|v_x(x,t)-v_x(x,s)|dx\leq 2C_\sharp\delta\frac{(t-s)|\log(t-s)|}{\sqrt{t}},\\
				&\sum_{z\in\mathscr{D}}\left|v_x(\cdot,t)\Big|_{z^-}^{z^+}-v_x(\cdot,s)\Big|_{z^-}^{z^+}\right|\leq  2C_\sharp\delta\frac{t-s}{\sqrt{t}},\\
				&\|v(\cdot,t)-v(\cdot,s)\|_{L^\infty}\leq 2C_\sharp\delta\frac{t-s}{\sqrt{t}},\\
				&\|v(\cdot,t)-v(\cdot,s)\|_{L^1}\leq 2C_\sharp\delta(t-s).
			\end{aligned}
			\right.
		\end{equation}
		Here $\int_{\bbr\setminus \scrD}$ is defined in
		\eqref{eq:int-bv}, stands for integration of absolutely
		continuous part of  a BV function.
	\end{enumerate}
\end{proposition}
\begin{remark}\label{rmk:prop-1}
	The regularity of $u_t$ and $\theta_t$ indeed provide the regularity of the fluxes in \eqref{NS2}. In other words, although $u_x$, $\theta_x$ and $v$ might be discontinuous with respect to $x$ individually, the fluxes as combinations of these discontinuous quantities are Lipschitz continuous and differentiable a.e. Therefore, all the derivatives in the equation \eqref{NS2} are in strong sense a.e., 
	which also implies $(v,u, e + u^2/2)$ is a weak solution to the original system \eqref{NS}, where $e$ is defined through \eqref{constitution}. This shows for the solution constructed in Proposition \ref{prop:local}, systems \eqref{NS} and \eqref{NS2} are equivalent.
\end{remark}

It immediately follows that the solution $(v,u,\theta)$ constructed in Proposition \ref{prop:local} belongs to the following function space:
\begin{equation}\label{K-108}
\left\{\begin{aligned}
	&v(x,t)-1\in C\left([0,t_\sharp);L^1(\bbr)\cap L^\infty(\bbr)\cap BV\right),\\
	&u(x,t)\in L^\infty\left(0,t_\sharp;W^{1,1}(\bbr)\cap L^{\infty}(\bbr)\right),\quad \sqrt{t}u_x(x,t)\in L^\infty\left(0,t_\sharp;L^\infty(\bbr)\right),\\
	&\sqrt{t}u_t(x,t)\in L^\infty\left(0,t_\sharp;L^1(\bbr)\right),\quad tu_t(x,t)\in L^\infty\left(0,t_\sharp;L^\infty(\bbr)\right),\\
	&\theta(x,t)\in L^\infty\left(0,t_\sharp;W^{1,1}(\bbr)\cap L^{\infty}(\bbr)\right),\quad \sqrt{t}\theta_x(x,t)\in L^\infty\left(0,t_\sharp;L^\infty(\bbr)\right),\\
	&\sqrt{t}\theta_t(x,t)\in L^\infty\left(0,t_\sharp;L^1(\bbr)\right),\quad t\theta_t(x,t)\in L^\infty\left(0,t_\sharp;L^\infty(\bbr)\right).
\end{aligned}
\right.
\end{equation}
It is worth mentioning that, in the above space, $v$ is continuous in $L^1$, $L^\infty$ and $BV$ sense around $t=0$, which exactly holds for the solution constructed in Proposition \ref{prop:local} due to \eqref{eq:v-conti}.

Moreover, based on the above regularity estimates,  the uniqueness of the constructed weak solution in Proposition \ref{prop:local} is also obtained. Actually, owning to the continuity of $v$ in $L^1\cap L^\infty \cap BV$ in short time, one can construct heat kernel $H(x,t;y,s;\frac{\mu}{v})$ and $H(x,t;y,s;\frac{\kappa}{c_v v})$, which yields an integral representation of $(v,u,\theta)$, then the uniqueness essentially follows from comparison estimates. Precisely, we have

\begin{proposition}[Stability and uniqueness, \cite{W-Y-Z-local}]\label{T3.3}
	Suppose there are two solutions $(v^a,u^a,\theta^a)$ and $(v^b,u^b,\theta^b)$ to the Navier-Stokes equation \eqref{NS2} with the regularity properties stated in \eqref{K-108},  and their initial data  both satisfy for small $\delta_*$,
	\[\|v_0\|_{BV}+\|u_0\|_{BV}+\|\theta_0\|_{BV}+\|v_0-1\|_{L^1}+\|u_0\|_{L^1}+\|\theta_0-1\|_{L^1}<\delta_*\ll 1,\]
	Then, there exists  positive constants $t_*$ and  $C_\flat$ such that for $0<t<t_*$,
	\begin{align*}
		& \vertiii{v^a-v^b}_{\infty} + \vertiii{v^a-v^b}_{1} +\vertiii{v^a-v^b}_{BV}	\notag \\
		&+ \vertiii{u^a-u^b}_{\infty} + \vertiii{u^a-u^b}_{1} + \vertiii{\frac{\sqrt{\tau}}{|\log \tau|}\left(u_x^a-u_x^b\right)}_{\infty} +  \vertiii{\frac{u^a_x-u^b_x}{|\log \tau|}}_{1} \notag \\
		&+ \vertiii{\frac{\theta^a-\theta^b}{|\log \tau|}}_{\infty} + \vertiii{\theta^a-\theta^b}_{1} + \vertiii{\frac{\sqrt{\tau}}{|\log \tau|}\left(\theta^a_x-\theta^b_x\right)}_{\infty} +  \vertiii{\frac{\theta^a_x-\theta^b_x}{|\log \tau|}}_{1} \notag \\
		&\leq C_\flat\Big(\|\theta^a_0-\theta^b_0\|_{L^\infty_x}+\|\theta^a_0-\theta^b_0\|_{L^1_x}\\
		&\hspace{0.4cm}+\|u^a_0-u^b_0\|_{L^\infty_x}+\|u^a_0-u^b_0\|_{L^1_x}+\|v^a_0-v^b_0\|_{L^1_x}+\|v^a_0-v^b_0\|_{L^\infty_x}+\|v^a_0-v^b_0\|_{BV}\Big),
	\end{align*}
	where 
	\[
	\vertiii{v^a-v^b}_{\infty}\equiv \sup_{t\in\left(0,t_*\right)} \lVert v^a(\cdot,t)-v^b(\cdot,t) \rVert_{L^\infty_x}
	\]
	and similarly for other $\vertiii{\cdot}$ norms. This immediately implies the uniqueness of the  weak solution.
\end{proposition}

In this paper, we are curious about whether the local-in-time solution constructed in Propositions \ref{prop:local} and \ref{T3.3} can be extended to a global one and what the time asymptotic behavior it has. The heat kernel with BV conductivity is effective for local well-posedness since it accurately captures the quasi-linear nature of the system \eqref{NS}. On the other hand, concerning large time behavior, the hyperbolic-parabolic structure plays a key role, which is, however, not reflected in the heat kernel. This motivates us to introduce a new ingredient, Green's function for linearized system around equilibrium state.  

In order to preserve the conservative form of equation \eqref{NS}, we define the following state variables,
\bq\label{var rela}
E=e+\frac{1}{2}u^2,\quad U=(v,u,E),\quad p(v,e(E,u))\simeq\frac{E-\frac{1}{2}u^2}{v},\quad e_u=-u,\quad e_E=1.
\eq
Then, the equation \eqref{NS} can be rewritten as following conservation form equation with variables in \eqref{var rela},
\bq\label{main}
\left\{\begin{aligned}
	&v_t-u_x=0\\
	&u_t+p_vv_x+p_ee_uu_x+p_ee_EE_x=\left(\frac{\mu u_x}{v}\right)_x\\
	&E_t+up_vv_x+(p+up_ee_u)u_x+up_ee_EE_x=\left(\frac{\kappa\theta_ee_u+\mu u}{v}u_x+\frac{\kappa\theta_ee_E}{v}E_x\right)_x.
\end{aligned}
\right.
\eq
We can also write the system into a more compact form as follows,
\[U_t+F(U)_x=(B(U)U_x)_x \Longleftrightarrow U_t+F'(U)U_x=(B(U)U_x)_x.\]
where $U$, $F$, $F'(U)$ and $B$ are defined as below,
\[U=\begin{pmatrix}v\\u\\E\end{pmatrix},\quad F(U)=\begin{pmatrix}-u\\p\\pu\end{pmatrix},\quad F'(U)=\begin{pmatrix}0 &-1 &0\\p_v& -p_eu& p_e\\p_vu & p-p_eu^2& p_eu\end{pmatrix},\quad B(U)=\begin{pmatrix}0 &0 &0\\0& \frac{\mu}{v} & 0\\0 & \left(\frac{\mu }{v}-\frac{\kappa\theta_e}{v}\right)u & \frac{\kappa\theta_e}{v}\end{pmatrix}.\]
Now we consider the linearization of equation \eqref{main} around a constant state $\bar{U}$. More precisely, let $U=\bar{U}+V$, then we have 
\bq\label{A2}
V_t+F'(\bar{U})V_x-B(\bar{U})V_{xx}=[N_1(V;\bar{U})+N_2(V;\bar{U})]_x,
\eq
where $N_1$ and $N_2$ are nonlinear terms from the hyperbolic and parabolic parts respectively,
\bq\label{A3}
N_1(V;\bar{U})=-[F(\bar{U}+V)-F(\bar{U})-F'(\bar{U})V],\quad N_2(V;\bar{U})=B(U)-B(\bar{U}).
\eq
The linearization of the nonlinear system \eqref{main} is obtained if we get rid of the right hand side nonlinear terms in \eqref{A2}. For the linearized equation, we follow the methodology in \cite{L-Y-2004,L-Y-2020} to obtain the pointwise estimates of Green's function. More precisely, we  use the spectral analysis to extract the singularities from the high frequency modes (see Theorem \ref{T4.1} and Theorem \ref{T4.2} for details). Then, for the regular part, we apply the long-short wave decomposition and weighted energy estimate to construct the desired results (see Theorem \ref{T4.3} for details).

With  heat kernel with BV coefficient and Green's function in hand, we combine them together to take care both local and global structure of system \eqref{NS} simultaneously. We interpolate the  heat kernel and  Green's function together to construct the ``effective Green's function'', which yields the representation of the solutions to \eqref{main}. As we already proved the regularity of the weak solution to \eqref{NS}, thus the equation \eqref{main} is equivalent to \eqref{NS} and \eqref{NS2}, which means we actually obtain another representation of the solutions to \eqref{NS} and \eqref{NS2}. The effective Green's function can be viewed as ``global nonlinear Green's function'' in the sense that it respects both local and large time structure of the system. The representations via effective Green's function have two benefits: in short time, the heat kernels for variable conductivity coefficients lead to small errors and avoid the non-integrable singularities; while for large time, the dissipation structure of the Green's function guarantees the global estimates. Eventually, we apply the continuity argument to obtain the global well-posedness theory (see Theorem \ref{T5.1} for details).  The large time behavior of the solution is given as below:

\begin{theorem}[Main Theorem: Large Time Behavior]
Suppose initial data $(v_0,u_0,\theta_0)$ of Navier-Stokes equation \eqref{NS2} satisfy
\begin{equation*}
	\left\|v_0-1\right\|_{L^1_x}+\left\|v_0\right\|_{BV}+\left\|u_0\right\|_{L^1_x}+\left\|u_0\right\|_{BV}+\left\|\theta_0-1\right\|_{L^1_x}+\left\|\theta_0\right\|_{BV}\leq \delta,
\end{equation*}
for $\delta$ sufficiently small. Then the solution constructed in Propositions \ref{prop:local} and \ref{T3.3} exists globally in time, and there exists positive constant $C_{\sharp}$ such that, 
\begin{align}
	&\left\| \sqrt{t+1} (v(\cdot,t)-1)\right\|_{L^\infty_x} + \left\|\sqrt{t+1} u(\cdot,t) \right\|_{L^\infty_x} + \left\| \sqrt{t+1} (\theta(\cdot,t)-1)\right\|_{L^\infty_x} \notag\\
	&\hspace{5cm}+ \left\|\sqrt{t} u_x(\cdot,t) \right\|_{L^\infty_x}+ \left\|\sqrt{t} \theta_x(\cdot,t) \right\|_{L^\infty_x} \nonumber \leq C_{\sharp} \delta~ \mbox{ for } t\in(0,+\infty). \nonumber
\end{align}
\end{theorem}
\begin{remark}
	It is even possible to establish the space-time pointwise estimate of the global solution for BV  data, provided the initial data satisfy certain space decay assumption. This will be pursued in  the future work.
\end{remark}

The rest of this paper is organized as follows. In Section \ref{sec:2}, we introduce some preliminary concepts and various estimates of the heat kernel with variable coefficient. The essential ideas arise from \cite{L-Y-2020}, and thus we  only list the estimates for variable coefficient heat kernel that will be used in later sections. For the detailed proof, the reader is referred to \cite{L-Y-2020,W-Y-Z-local}. In Section \ref{sec:4}, we construct the Green's function for the linearzied system \eqref{A2} around constant state. In Section \ref{sec:5}, we first interpolate heat kernel and Green's function to have the ``effective Green's function'', which yields an integral representation of the solution. The time asymptotic behavior then follow from precise estimate of effective Green's function and a priori estimate.


%

%
%

\section{Preliminary}\label{sec:2}
\setcounter{equation}{0}

In this section, we will provide some preliminary concepts and results that will be used in the later sections. We  first give definitions of the weak and a regularity property, then we introduce the pointwise  and comparison estimates of the heat kernel for BV coefficients.  

\subsection{Weak solutions}

\begin{definition}\label{Def2.1}
	A tuple $(v,u,\theta)$ is  a weak solution to the equation \eqref{NS2} in the distribution sense if 	
	for any test function $\varphi(x,t)\in C^1_0(\bbr\times[0,+\infty))$, 
	\begin{equation}\label{eq:def2.1}
		\left\{\begin{aligned}
			&\int_0^{+\infty}\int_\bbr \left[\varphi_x u-\varphi_t v	\right]dxdt= \int_\bbr \varphi(x,0) v(x,0)dx,\\
			&\int_0^{+\infty}\int_\bbr \left[\varphi_x\left(\frac{\mu u_x}{v}-p\right)-\varphi_tu\right]dxdt=\int_\bbr\varphi(x,0)u(x,0)dx,\\
			&\int_0^{+\infty}\int_\bbr \left[\varphi_x\left(\frac{\kappa}{c_vv}\theta_x\right)+\varphi\left(\frac{p}{c_v}u_x-\frac{\mu}{c_vv}\left(u_x\right)^2\right)-\varphi_t\theta\right]dxdt=\int_\bbr\varphi(x,0)\theta(x,0)dx.
		\end{aligned}
		\right.
	\end{equation}
\end{definition}

The weak solutions to the linearized equation \eqref{NS-lin-2} and \eqref{eq:fun}  can be
defined similarly.    

\begin{proposition}

\label{Def2.2}
For the weak solution $(v,u,\theta)$ of \eqref{NS2},
the function $(v,u,c_v \theta+u^2/2)$ is a weak  solution  of  \eqref{NS} if
it is satisfies:
\begin{equation*}
	\left\{\begin{aligned}
		&v(x,t)-1\in C\left([0,t_\sharp);L^1(\bbr)\cap L^\infty(\bbr)\cap BV\right),\\
		&u(x,t)\in L^\infty\left(0,t_\sharp;W^{1,1}(\bbr)\cap L^{\infty}(\bbr)\right),\quad \sqrt{t}u_x(x,t)\in L^\infty\left(0,t_\sharp;L^\infty(\bbr)\right),\\
		&\sqrt{t}u_t(x,t)\in L^\infty\left(0,t_\sharp;L^1(\bbr)\right),\quad tu_t(x,t)\in L^\infty\left(0,t_\sharp;L^\infty(\bbr)\right),\\
		&\theta(x,t)\in L^\infty\left(0,t_\sharp;W^{1,1}(\bbr)\cap L^{\infty}(\bbr)\right),\quad \sqrt{t}\theta_x(x,t)\in L^\infty\left(0,t_\sharp;L^\infty(\bbr)\right),\\
		&\sqrt{t}\theta_t(x,t)\in L^\infty\left(0,t_\sharp;L^1(\bbr)\right),\quad t\theta_t(x,t)\in L^\infty\left(0,t_\sharp;L^\infty(\bbr)\right).
	\end{aligned}
	\right.
\end{equation*}
\end{proposition}


\subsection{Heat kernel estiamtes}

In this subsection, one provides   estimates for the fundamental
solution $H(x,t;y,t_0;\rho)$ (heat kernel) of
a heat equation with  a BV function coefficient $\rho$ in $x$,
\[\partial_t + \partial_x \rho(x,t) \partial_x)u=0.\]   
A BV function $f$ can be decomposed as the sum of an absolutely continuous part and a discrete part,
\begin{align*}
	f(x)=f_c(x)+f_d(x),\quad \text{$f_c$ is absolutely continuous},\quad f_d(x)=\sum_{\alpha\in\mathscr{D},		 \alpha<x}f(y)\Bigr|_{y=\alpha-}^{y=\alpha+} h(x-\alpha),
\end{align*}
where  $\mathscr{D}$ is the discontinuity set of $f(x)$ and  $h(\cdot)$ is the Heaviside function. The total variation can also be represented as
\begin{equation}
	\label{eq:bv-2}
	\lVert f \rvert _{BV} = \int_{\bbr}\lvert \partial_x u_c \rvert  dx + \left.\left. \sum_{\alpha\in \scrD} \right| f_d(x) \big|_{x=\alpha^-}^{x=\alpha^+} \right| .
\end{equation}
Hereafter, we also adopt the notation when integrating the derivative of a BV function
\begin{equation}
	\label{eq:int-bv}	
	\int_{\R\setminus \mathscr{D}} \lvert \partial_x f\rvert  dx\equiv\int_{\bbr}\lvert\partial_x f_c\rvert  dx.
\end{equation}
The fundamental solution $H(x,t;y,t_0;\rho)$ is defined as a weak
solution of the initial value problem:
\begin{equation}\label{heat-x}
	\left\{\begin{aligned}
		&\left(\partial_t-\partial_x\rho(x,t)\partial_x\right)H(x,t;y,t_0;\rho)=0,\quad t>t_0,\\
		&H(x,t_0;y,t_0;\rho)=\delta(x-y),
	\end{aligned}
	\right. 
\end{equation} 
i.e. \[ \displaystyle  \int_{t_0}^\infty \int_{\mathbb R} \left(-\phi_t (x,t)
H(x,t;y,t_0;\rho) +\phi_x(x,t) \rho(x,t) H_x(x,t;y,t_0;\rho)\right)dxdt =
-\phi(y,t_0)\] for all test functions $\phi$; 
and $\rho(x,t)$ is assumed to satisfy  for some positive constants $\bar{\rho}$ and $\delta^*$,
\begin{equation}\label{2-2}
	\left\{\begin{aligned}
		&\|\rho(\cdot)-\bar{\rho}\|_{L^1}\leq \delta^*,\quad \|\rho(\cdot,t)\|_{BV}\leq \delta^*,\quad \|\rho_t(\cdot,t)\|_{\infty}\leq \delta^*\max\left(\frac{1}{\sqrt{t}},1\right),\quad 0<\delta^*\ll 1,\\
		&\mathscr{D}\equiv\left\{z\  |\ \text{$\rho(z,t)$ is not continuous at
			$z$} \right\} \text{ is invariant in }t.
	\end{aligned}
	\right.
\end{equation}

The construction of heat kernel for variable coefficient was introduced in \cite{L-Y-2020}. Roughly speaking, the construction consists of three steps. The first step is to construct the heat kernel for coefficient being step function in space and independent of time. This is achieved by introducing the Laplace wave train and write the solution in transformed variable by stochastic formulation. Then various detailed derivatives estimate can be obtained through this representation. In the second step, one uses a sequence of step functions to approximate general BV function, and prove the convergence of heat kernel sequence. Lastly, one uses frozen coefficient technique to represent the time-dependent heat kernel by integral equation in terms of time-independent heat kernel. Although the general framework is given in \cite{L-Y-2020}, in the current work, we need more detailed estimates than before. 

The results are summarized in the following several Lemmas. 

\begin{lemma}[Liu-Yu, \cite{L-Y-2020}]\label{L2.1}
	Suppose the conditions of $\rho$ in \eqref{2-2} hold. Then, there
	exist positive constants $C_*$ and $t_\sharp\ll 1$ such that the weak
	solution of  \eqref{heat-x}  exists and satisfies 
	for  $t \in (t_0,t_0+t_\sharp)$
	\begin{align}
		&\left|H(x,t;y,t_0;\rho)\right|\leq
		C_*\frac{e^{-\frac{\left(x-y\right)^2}{C_*(t-t_0)}}}{\sqrt{ t-t_0
		}}, 	\label{eq:H-1}\\
		&\left|H_x(x,t;y,t_0;\rho)\right|+\left|H_y(x,t;y,\mu)\right|\leq
		C_*\frac{e^{-\frac{\left(x-y\right)^2}{C_*(t-t_0)}}}{ t-t_0
		}, \label{eq:H-2} \\
		& \left|\int_{t_0}^tH_x(x,\tau;y,t_0;\rho)d\tau\right|, \left|\int_{t_0}^t H_x(x,t;y,s;\rho)ds\right|\leq C_*e^{-\frac{(x-y)^2}{C_*(t-t_0)}}. \label{eq:H-5}
	\end{align}
\end{lemma}

	\begin{lemma}[\cite{WYZ-wp}]\label{lem:heat-pw} Under the same consideration as in
		Lemma \ref{L2.1},  for  $t \in (t_0,t_0+t_\sharp)$
		the weak
		solution of  \eqref{heat-x}   satisfies 
		\begin{align}
			&\left|H_{xy}(x,t;y,t_0;\rho)\right|+\left|H_{t}(x,t;y,t_0;\rho)\right|\leq C_*\frac{e^{-\frac{\left(x-y\right)^2}{C_*(t-t_0)}}}{(t-t_0)^{\frac{3}{2}}},\label{eq:H-3}\\
			&  \left| H_{ty}(x,t;y,t_0;\rho)	\right| \leq
			C_{*} \frac{e^{-	\frac{(x-y)^2}{C_{*}(t-t_0)}
			}}{(t-t_0)^2}, &\quad 
			\label{eq:H-4}
			\\ 
			&  \left|\int_{t_0}^tH_{xy}(x,\tau;y,t_0;\rho)d\tau	-\frac{\delta(x-y)}{\rho(x,t_0)} -\int_{t_0}^{t} \frac{\rho(x,t_0)-\rho(x,\tau)}{\rho(x,t_0)} H_{xy}(x,\tau;y,t_0;\rho) d\tau	\right| \leq C_{*}  \frac{e^{-\frac{(x-y)^2}{C_*(t-t_0)}}}{\sqrt{t-t_0}}, 	\label{eq:H-6}\\
			& \left|\int_{t_0}^t H_{xy}(x,t;y,s;\rho)ds	+\frac{\delta(x-y)}{\rho(y,t)} -\int_{t_0}^{t} \frac{\rho(y,t)-\rho(y,s)}{\rho(y,t)} H_{xy}(x,t;y,s;\rho) ds	\right| \leq C_{*}  \frac{e^{-\frac{(x-y)^2}{C_*(t-t_0)}}}{\sqrt{t-t_0}}, 	\label{eq:H-7}\\
			&\int_{t_0}^{t} H_{xx}(x,\tau;y,t_0;\rho) d\tau = -\frac{\delta(x-y)}{\rho(x,t_0)}-\frac{1}{\rho(x,t_0)} \partial_{x} \left[	\int_{t_0}^{t} \left(\rho(x,\tau)-\rho(x,t_0)\right) H_x(x,\tau;y,t_0;\rho) d\tau	\right]	\notag	\\
			& \hspace{11em} +O(1) \left(|\partial_{x} \rho(x,t_0)| e^{-	\frac{(x-y)^2}{C_{*}(t-t_0)}	} + \frac{e^{-	\frac{(x-y)^2}{C_{*}(t-t_0)}	}}{\sqrt{t-t_0}}\right), \qquad \mbox{ for } x\notin \scrD, \, \label{eq:H-8}\\
			& \int_{t_0}^{t} H_{xxy}(x,\tau;y,t_0;\rho) d\tau  = \frac{1}{\rho(x,t_0)} \left[	\delta^{\prime}(x-y) - \int_{t_0}^{t} \partial_{x}\Big[	\big(\rho(x,\tau)-\rho(x,t_0) \big)  H_{xy}(x,\tau;y,t_0;\rho)		\Big]	d\tau \right] \notag \\
			&\hspace{11em} -\frac{\partial_{x}\rho(x,t_0)}{\rho^2(x,t_0)} \left[\delta(x-y) -  \int_{t_0}^{t} \big(\rho(x,\tau)-\rho(x,t_0) \big)  H_{xy}(x,\tau;y,t_0;\rho) d\tau \right] \notag \\
			& \hspace{11em}+ O(1) \left( |\partial_{x} \rho(x,t_0)| \frac{	e^{	-\frac{(x-y)^2}{C_{*}(t-t_0)}	}	}{\sqrt{t-t_0}}+ \frac{e^{	-\frac{(x-y)^2}{C_{*}(t-t_0)}	}}{t-t_0}    \right), \qquad  \mbox{ for } x\notin \scrD, \label{eq:H-9}\\
			&\int_{t_0}^tH_{t}(x,t;y,s;\rho)ds = H(x,t-t_0;y;\mu^t)-\delta(x-y) + O(1) \delta_{*} e^{-\frac{(x-y)^2}{C_*(t-t_0)}},\qquad \mbox{where } \mu^t(\cdot)\equiv \rho(\cdot,t). \label{eq:H-10}
		\end{align}
	\end{lemma}
	
Notice that the  estimates for the terms involving twice $x-$derivatives  do not hold when $x\in\mathscr{D}$, which is due to the presence of Dirac-delta functions in $H_{xx}$ if $x\in\mathscr{D}$. Moreover, the zeroth order estimate actually can be extended to be global in time, while the higher order estimates are obtained only for local time so far.
 
\begin{lemma}[\cite{WYZ-wp}, H\"{o}lder continuity in time] \label{L2.2-1}   Suppose the conditions in \eqref{2-2} hold for $\rho$. Then the following estimates hold when $t_0<s<t\ll 1$,
	\begin{align*}
		& \left\|	H_x(\cdot,t;y,t_0;\rho)-H_x(\cdot,s;y,t_0;\rho)	\right\|_{\infty} \leq C_{*} \frac{(t-s)\left|\log(t-s)\right|}{(s-t_0)(t-t_0)},\\
		& \left\|	H_x(\cdot,t;y,t_0;\rho)-H_x(\cdot,s;y,t_0;\rho)	\right\|_{1} \leq C_{*} \frac{(t-s)\left|\log(t-s)\right|}{\sqrt{s-t_0}(t-t_0)},\\
		& \left\|	H_{xy}(\cdot,t;y,t_0;\rho)-H_{xy}(\cdot,s;y,t_0;\rho)	\right\|_{\infty} \leq C_{*} \frac{(t-s)\left|\log(t-s)\right|}{(s-t_0)^{3/2}(t-t_0)},\\
		& \left\|	H_{xy}(\cdot,t;y,t_0;\rho)-H_{xy}(\cdot,s;y,t_0;\rho)	\right\|_{1} \leq C_{*} \frac{(t-s)\left|\log(t-s)\right|}{(s-t_0)(t-t_0)}.
	\end{align*}
\end{lemma}

We need the comparison estimate between two heat kernels with different conductivity coefficients. In the following we denote that 
\begin{equation}\label{t-norm}
	\begin{aligned}
		&\vertiii{f}_\infty\equiv\sup_{\sigma\in(0,t_\sharp)}\|f(\cdot,\sigma)\|_{L^\infty_x},\quad \vertiii{f}_1\equiv\sup_{\sigma\in(0,t_\sharp)}\|f(\cdot,\sigma)\|_{L^1_x},\quad \vertiii{f}_{BV}\equiv\sup_{\sigma\in(0,t_\sharp)}\|f(\cdot,\sigma)\|_{BV}.
	\end{aligned}
\end{equation}

\begin{lemma}[Comparison estimates, Corollaries 4.4 and 4.5 in \cite{L-Y-2020}]\label{L2.2}
	Suppose that the conditions in \eqref{2-2} hold for $\rho^a$ and
	$\rho^b$. Then, there exist positive constants $t_\sharp\ll 1$ and
	$C_*$ such that for $t \in (t_0,t_0+t_\sharp)$
	\begin{align*}
		& \left|	H(x,t;y,t_0;\rho^b) - H(x,t;y,t_0;\rho^a)	\right| \leq C_{*}\frac{e^{-	\frac{(x-y)^2}{C_{*}(t-t_0)}	}}{\sqrt{t-t_0}} \vertiii{\rho^a-\rho^b}_{\infty}, \\
		& \left|H_x(x,t;y,t_0;\rho^a)	-	H_x(x,t;y,t_0;\rho^b)	\right|,\; \left|H_y(x,t;y,t_0;\rho^a)	-	H_y(x,t;y,t_0;\rho^b)	\right|	\notag \\
		&\quad \leq C_{*}  \frac{e^{-	\frac{(x-y)^2}{C_{*}(t-t_0)}	}}{t-t_0} \left[	\left|\log(t-t_0)\right|  \vertiii{\rho^a-\rho^b}_{\infty}+ \vertiii{\rho^a-\rho^b}_{BV} +\sqrt{t-t_0} \Big(\vertiii{\rho^a-\rho^b}_{1}+\left|\log t\right|\vertiii{\frac{\sqrt{\tau}}{\left|\log \tau\right|} \partial_{\tau}\left[\rho^a-\rho^b\right]}_{\infty}	\Big) 	\right],	\\
		&\left|	\int_{t_0}^{t} \left[ H_{x} (x,\tau;y,t_0;\rho^a)- H_{x} (x,\tau;y,t_0;\rho^b)	\right] d\tau \right|\\
		&\quad \leq  C_{*} e^{-	\frac{(x-y)^2}{C_{*}(t-t_0)}	} \left[	\vertiii{\rho^a-\rho^b}_{\infty}+ \vertiii{\rho^a-\rho^b}_{BV} +   \vertiii{\rho^a-\rho^b}_{1}	+\vertiii{\frac{\sqrt{\tau}}{\lvert \log \tau \rvert }\partial_{\tau}(\rho^a-\rho^b)}_{\infty} \right].
	\end{align*}
\end{lemma}

\begin{lemma}[\cite{WYZ-wp}]\label{lem:heat-comp} Under the considerations in Lemma
	\ref{L2.2},  there exist positive constants $t_\sharp\ll 1$ and
	$C_*$ such that for $t \in (t_0,t_0+t_\sharp)$
	\begin{align}
		&\left|H_{xy}(x,t;y,t_0;\rho^a)	-	H_{xy}(x,t;y,t_0;\rho^b)	\right|,\; \left|H_{t}(x,t;y,t_0;\rho^a)	-	H_{t}(x,t;y,t_0;\rho^b)	\right| \notag \\
		&\leq C_{*}  \frac{e^{-	\frac{(x-y)^2}{C_{*}(t-t_0)}	}}{(t-t_0)^{3/2}} \left[	\left|\log(t-t_0)\right|  \vertiii{\rho^a-\rho^b}_{\infty}+ \vertiii{\rho^a-\rho^b}_{BV} +\sqrt{t-t_0} \Big(\vertiii{\rho^a-\rho^b}_{1}+\left|\log t\right|\vertiii{\frac{\sqrt{\tau}}{\left|\log \tau\right|} \partial_{\tau}\left[\rho^a-\rho^b\right]}_{\infty}	\Big) 	\right],\notag 
		\\
		&\left|	\int_{t_0}^{t} \left[ H_{y} (x,t;y,s;\rho^a)-
		H_{y} (x,t;y,s;\rho^b)	\right] ds
		\right|  \nonumber \\
		&\qquad \qquad\leq  C_{*} e^{-
			\frac{(x-y)^2}{C_{*}(t-t_0)}	} \left[
		\vertiii{\rho^a-\rho^b}_{\infty}+
		\vertiii{\rho^a-\rho^b}_{BV} +
		\vertiii{\rho^a-\rho^b}_{1}
		+\vertiii{\frac{\sqrt{\tau}}{|\log
				\tau|}\partial_{\tau}(\rho^a-\rho^b)}_{\infty}
		\right], \nonumber \\
		&\left|	\int_{t_0}^{t} \left[ H_{x} (x,t;y,s;\rho^a)-
		H_{x} (x,t;y,s;\rho^b)	\right] ds
		\right|,\; \left|	\int_{t_0}^{t} \left[ H_{y}
		(x,\tau;y,t_0;\rho^a)- H_{y} (x,\tau;y,t_0;\rho^b)
		\right] d\tau	\right|  \nonumber \\
		&\qquad \qquad \leq  C_{*} e^{-
			\frac{(x-y)^2}{C_{*}(t-t_0)}	} \left[
		\vertiii{\rho^a-\rho^b}_{\infty}+
		\vertiii{\rho^a-\rho^b}_{BV} +
		\vertiii{\rho^a-\rho^b}_{1}
		+\vertiii{\frac{\sqrt{\tau}}{|\log
				\tau|}\partial_{\tau}(\rho^a-\rho^b)}_{\infty}
		\right], \nonumber \\
		& \nonumber \int_{t_0}^{t} \left[ H_{xy} (x,\tau;y,t_0;\rho^a)- H_{xy} (x,\tau;y,t_0;\rho^b)	\right] d\tau\\
		&=
		\left[\frac{1}{\rho^a(x,t_0)}-\frac{1}{\rho^b(x,t_0)}\right]
		\delta(x-y) \nonumber \\
		&\quad -\int_{t_0}^{t}
		\left[\frac{\rho^a(x,\tau)-\rho^a(x,t_0)}{\rho^a(x,t_0)}
		H_{xy}(x,\tau;y,t_0;\rho^a)-
		\frac{\rho^b(x,\tau)-\rho^b(x,t_0)}{\rho^b(x,t_0)}
		H_{xy}(x,\tau;y,t_0;\rho^b) \right]d\tau \nonumber \\
		&\quad + O(1)\frac{e^{-	\frac{(x-y)^2}{C_{*}(t-t_0)}
		}}{\sqrt{t-t_0}} \left[
		\left|\log(t-t_0)\right|
		\vertiii{\rho^a-\rho^b}_{\infty} +
		\vertiii{\rho^a-\rho^b}_{BV}+
		\vertiii{\rho^a-\rho^b}_{1} +
		\vertiii{\frac{\sqrt{\tau}}{|\log
				\tau|}\partial_{\tau}(\rho^a-\rho^b)}_{\infty}
		\right], \nonumber 
	\end{align}
	\begin{equation*}
		\begin{aligned}
			& \int_{t_0}^{t} \left[ H_{xy} (x,t;y,s;\rho^a)- H_{xy} (x,t;y,s;\rho^b)	\right] ds\\
			&= \left[\frac{1}{\rho^a(y,t)}-\frac{1}{\rho^b(y,t)}\right] \delta(x-y)\\
			&\quad +\int_{t_0}^{t} \left[\frac{\rho^a(y,t)-\rho^a(y,s)}{\rho^a(y,t)} H_{xy}(x,t;y,s;\rho^a)- \frac{\rho^b(y,t)-\rho^b(y,s)}{\rho^b(y,t)} H_{xy}(x,t;y,s;\rho^b) \right]ds\\
			&\quad +  O(1)\frac{e^{-	\frac{(x-y)^2}{C_{*}(t-t_0)}	}}{\sqrt{t-t_0}}  \left[	\left|\log(t-t_0)\right| \vertiii{\rho^a-\rho^b}_{\infty} + \vertiii{\rho^a-\rho^b}_{BV}+  	\vertiii{\rho^a-\rho^b}_{1} + \vertiii{\frac{\sqrt{\tau}}{\lvert \log \tau \rvert }\partial_{\tau}(\rho^a-\rho^b)}_{\infty} 	 	\right].
		\end{aligned}
	\end{equation*}
\end{lemma}

Lastly, according to the symmetry of the heat equation, one has some symmetric properties of the heat kernel. As the proof can be constructed directly from proper integrals of the equation \eqref{heat-x}, we omit the details and provide the following results. 
\begin{lemma}\label{L2.3}
	For the heat equation in conservative form \eqref{heat-x}, the solution has the following properties,
	\begin{align*}
		\left\{
		\begin{aligned}
			&\int_\bbr H(x,t;y,\tau;\rho)dx=\int_\bbr H(x,t;y,\tau;\rho)dy=1,\\
			&\int_\bbr H_x(x,t;y,\tau;\rho)dx=\int_\bbr H_x(x,t;y,\tau;\rho)dy=0,\\
			&\int_\bbr H_y(x,t;y,\tau;\rho)dx=\int_\bbr H_y(x,t;y,\tau;\rho)dy=0,\\
			&\int_\bbr H_t(x,t;y,\tau;\rho)dx=\int_\bbr H_t(x,t;y,\tau;\rho)dy=0,\\
			&\int_\bbr H_{\tau}(x,t;y,\tau;\rho)dx=\int_\bbr H_\tau(x,t;y,\tau;\rho)dy=0.
		\end{aligned}
		\right.
	\end{align*}
\end{lemma}

\begin{remark}\label{R2.1}
	According to \cite{L-Y-2020}, two remarks should be mentioned: 
	\begin{enumerate}
		\item In order to balance the equation \eqref{heat-x}, $\rho(x,t) H_x(x,t;y,t_0;\rho)$ is actually continuous with respect to $x$. When consider the backward equation, one also has $\rho(y,t_0) H_y(x,t;y,t_0;\rho)$ is continuous with respect to $y$.
		\item The weak solution of the heat equation \eqref{heat-x} can be defined similarly as Definition \ref{Def2.1}. In fact, if the equation \eqref{heat-x} has a source term in the following conservative form, 
		\[u_t(x,t)=\left(\rho(x,t) u_x(x,t)+g(x,t)\right)_x,\]
		then the mild solution constructed by heat kernel and Duhamel's principle is also a weak solution to the above equation in the distribution sense, provided $g(x,t)$ is a BV function with respect to $x$. Furthermore, the flux term $\left(\rho(x,t) u_x(x,t)+g(x,t)\right)$ is continuous with respect to $x$ if one of the following two properties holds,
		\begin{enumerate}
			\item $g(x,t)$ is Lipschitz continuous with respect to $x$, i.e.,
			\[\|g_x(\cdot, t)\|_\infty<+\infty.\]
			\item $g(x,t)$ is H\"older continuous with respect to $t$ in the sense that 
			\[\lvert g(x,t)-g(x,s) \rvert \leq C \frac{(t-s)^\alpha}{s^\alpha},\quad 0<s<t,\quad 0<\alpha<1.\] 
		\end{enumerate}
		
	\end{enumerate}
\end{remark}

\section{Green's function}\label{sec:4}
\setcounter{equation}{0}
As the local-in-time weak solution is established, the next two sections are devoted to the global-in-time behavior  of the solution. In this section, we will prepare the tools to capture the large time structure of the system \eqref{NS}, that is, the Green's function of the system \eqref{A2}. The Green's function $\bbg(x,t;\bar{U})$ of the linearized equation of \eqref{A2} is the solution of the following system,
\bq\label{B1}
\begin{aligned}
&\partial_t\bbg(x,t;\bar{U})=\left(-F'(\bar{U})\partial_x+B(\bar{U})\partial_{xx}\right)\bbg(x,t;\bar{U}),\\
&\bbg(x,0;\bar{U})=\delta(x)I,
\end{aligned}
\eq
where $I$ is the $3\times 3$ identity matrix and $\delta(x)$ is the Dirac-delta function. Then, with Fourier transformation, we can obtain the following equation in transformed variable,
\bq\label{B2}
\begin{aligned}
&\partial_t\hat{\bbg}(\eta,t;\bar{U})=\left(-i\eta F'(\bar{U})-\eta^2B(\bar{U})\right)\hat{\bbg}(\eta,t;\bar{U}),\\
&\hat{\bbg}(\eta,0;\bar{U})=I.
\end{aligned}
\eq
We shall  first show the eigenvalues of the matrix $\left(-i\eta F'(\bar{U})-\eta^2B(\bar{U})\right)$ are distinct and thus represent the Green's function by the state-transition matrix. 

To obtain the Green's function in terms of physical variables, we basically follow the approach introduced in \cite{Deng-Yu-2017,L-Y-2020} to split the Green's function into singular part and regular part. Roughly speaking, the singular part corresponds  to large $|\eta|$, which represents the singularity in the Green's function; while the regular part corresponds to small $|\eta|$, which dominates the large time behavior of Green's function. We shall construct the point-wise estimates for each of them. 

In the following, all the norms of the matrix can be viewed as the trace norm and we define the $L^2$ norm of a matrix valued function $A(x)$ as follows,
\bq\label{A7}
\|A(x)\|_{L^2(\bbr)}=\sqrt{\int_{\bbr}tr\left(A^\mathsf{T}A\right)dx}.
\eq

\subsection{Analysis on eigenvalues}
In this part, we will give a preliminary estimate on the eigenvalues of the matrix $\left(-i\eta F'(\bar{U})-\eta^2B(\bar{U})\right)$. In \cite{L-Y-2020}, the authors studied p-system and construct the exact formula of the eigenvalues. However, for full Navier-Stokes equations, the exact expression is too complicated to do  further analysis. Therefore, we will mainly consider the expansion of the eigenvalues at infinity and zero respectively. 
\subsubsection{Green's function in frequency space}
We first show the three eigenvalues mentioned before are distinct, and thus provide the spectral representation of the Green's function by the resolution. We first provide the characteristic equation from the determinant of the matrix $I\lambda-\left(-i\eta F'(\bar{U})-\eta^2B(\bar{U})\right)$. Actually, we have 
\bq\label{char}
\begin{aligned}
	\det(\lambda I+i\eta F'(\bar{U})+\eta^2B(\bar{U}))=\eta ^2 \lambda  p p_e+\frac{\left(\lambda  \left(\eta ^2 \mu +\lambda  v\right)-\eta ^2
		v p_v\right) \left(\eta ^2 \kappa  \theta _e+\lambda  v\right)}{v^2}=0,
\end{aligned}
\eq
where all the state variables take values at $\bar{U}$. According to the Fundamental Theorem of Algebra, there are three roots of \eqref{char} and we denote them by $\lambda_j$ with $j=1,2,3$. Moreover, from the Residue Theorem and Rouch\'e's Theorem, the eigenvalues continuously depend on the coefficients. As the coefficients are polynomials of $\eta$, we immediately conclude the eigenvalues continuously depend on $\eta$. In the following Lemma, we will further see these three eigenvalues are distinct when $\eta$ is near real axis.  
\begin{lemma}\label{LB1}
	There exists a sufficiently small positive constant $\sigma_0$ such that, the matrix $\left(-i\eta F'(\bar{U})-\eta^2B(\bar{U})\right)$ has three distinct eigenvalues for $|\eta|>0$ and $|Im(\eta)|<\sigma_0$. 
\end{lemma} 
\begin{proof}
	It is enough to show the equation \eqref{char} admits three distinct solutions for $\eta$ near the real line. As \eqref{char} is of order three, it has one real root and two conjugate complex roots, or it has three real roots. If there are one real root and two conjugate roots, the three roots must be distinct. Thus we only need to consider the case when all the three roots are real. 
	
	Let's check the necessary condition for existence of multiple roots. Suppose the three roots are $\lambda_1$, $\lambda_2$ and $\lambda_3$, and $\lambda_2=\lambda_3$. Then the characteristic equation can be expressed as below,
	\bq\label{B4}
	(\lambda-\lambda_1)(\lambda-\lambda_2)^2=0,\quad \lambda_2\neq 0,
	\eq
	where the multiple root $\lambda_2\neq 0$ since the coefficients in \eqref{char} are nonzero when $\eta\neq 0$. Then, we compare the coefficients between \eqref{B4} and \eqref{char} to obtain that 
	\begin{equation}\label{B5}
		\left\{\begin{aligned}
			&\lambda_1+2\lambda_2=-\left(\frac{\eta ^2 \kappa  \theta _e}{v}+\frac{\eta ^2 \mu }{v}\right):=S_1,\\
			&2\lambda_1\lambda_2+\lambda_2^2=\eta ^2 p p_e+\frac{\eta ^4 \kappa  \mu  \theta _e}{v^2}-\eta ^2 p_v:=S_2,\\
			&\lambda_1\lambda_2^2=\frac{\eta ^4 \kappa  \theta _e p_v}{v}:=S_3.
		\end{aligned}
		\right.
	\end{equation}
	We multiply $2\lambda_2$ to the first equation and minus the second equation to obtain an equation for $\lambda_2$ as below,
	\[\lambda_2^2=\frac{2\lambda_2S_1-S_2}{3}.\]
	The above equality is a reduction from the second order term $\lambda_2^2$ to first order and zero order term. Then we can substitute this reduction to \eqref{char} and $\eqref{B5}$ to solve $\lambda_2$ respectively. Actually, we have the following two representation of $\lambda_2$,
	\begin{equation*}
		\begin{aligned}
			&\lambda_2=\frac{S_1 S_2-9 S_3}{2 S_1^2+3 S_1-9 S_2},&\quad \text{from \eqref{char}},\\
			&\lambda_2=\frac{S_1 S_2-9 S_3}{2 \left(S_1^2-3 S_2\right)},&\quad \text{from $\eqref{B5}_2$ and $\eqref{B5}_3$}.
		\end{aligned}
	\end{equation*}
	Then the two representation should  coincide if $\lambda_2$ is the root of \eqref{char}. As we already know $\lambda_2\neq 0$, we immediately have $S_1=S_2$ if the two representation of $\lambda_2$ are equal. Then we solve the equation $S_1=S_2$ to obtain 
	\[\eta=0,\quad \eta=-\frac{i \sqrt{v} \sqrt{\kappa  \theta _e+p v p_e+\mu -v p_v}}{\sqrt{\kappa } \sqrt{\mu }
		\sqrt{\theta _e}},\quad \eta=\frac{i \sqrt{v} \sqrt{\kappa  \theta _e+p v p_e+\mu -v p_v}}{\sqrt{\kappa } \sqrt{\mu }
		\sqrt{\theta _e}}.\]
	Therefore, the multiple root cannot occur unless  when $\eta$ take above values. This means when $|\eta|>0$ and $|Im(\eta)|<\sigma_0$ for sufficiently small positive constant $\sigma_0$, the matrix  $\left(-i\eta F'(\bar{U})-\eta^2B(\bar{U})\right)$ admits three distinct eigenvalues.
	
\end{proof}
Once we know the eigenvalues of the matrix $\left(-i\eta F'(\bar{U})-\eta^2B(\bar{U})\right)$, we can solve the linear system \eqref{B2} via the resolution. In fact, according to \eqref{B2}, we have 
\bq\label{B10}
\hat{\bbg}(\eta,t;\bar{U})=\mathcal{L}^{-1}[\left(s+i\eta F'(\bar{U})+\eta^2B(\bar{U})\right)^{-1}],
\eq
where $\mathcal{L}^{-1}$ is the inverse Laplace transform with respect to $s$. On the other hand, according to Lemma \ref{LB1}, the eigenvalues are distinct near the real line. Therefore, $\left(s+i\eta F'(\bar{U})+\eta^2B(\bar{U})\right)^{-1}$ has three simple poles when $\eta$ near the real line, which guarantees the inverse Laplace transform \eqref{B10} has the following representation,
\bq\label{B11}
\mathcal{L}^{-1}[\left(s+i\eta F'(\bar{U})+\eta^2B(\bar{U})\right)^{-1}]=\sum_{j=1}^3\text{Res}[e^{st}\left(s+i\eta F'(\bar{U})+\eta^2B(\bar{U})\right)^{-1}, \lambda_j],
\eq
where $\text{Res}[f,z_0]$ denotes the residue of function $f$ at the pole $z_0$. More precisely, we can combine \eqref{B10} and \eqref{B11} to obtain the representation of the Green's function in frequency space as follows,
\bq\label{B12}
\hat{\bbg}(\eta,t;\bar{U})=\sum_{j=1}^3e^{\lambda_jt}\hat{\bbm}_j,\quad \hat{\bbm}_j=\frac{\text{adj}\left(s+i\eta F'(\bar{U})+\eta^2B(\bar{U})\right)|_{s=\lambda_j}}{\Pi_{k\neq j}(\lambda_j-\lambda_k)},
\eq
where $\text{adj}(A)$ is the adjugate matrix of $A$.

\subsubsection{Spectral gap}
Next, we will show there exists a spectral gap of the eigenvalues for $\eta$ on the real axis and away from $0$, which will be very important for time decay estimates later. We are motivated by the idea in \cite{Li-2005} to prove the following Lemma.
\begin{lemma}\label{LB2}
	Let $r$ and $R$ be two positive constants, we can find a positive constant $b$ such that $Re\left(\lambda_j(\eta)\right)\leq -b$  for all real number $\eta$ satisfying $0<r<|\eta|<R$,  $j=1,2,3$.
\end{lemma}
\begin{proof}
	The characteristic polynomial \eqref{char} can be rewritten as follows,
	\bq\label{B9}
	f(\lambda)=\lambda ^3+\lambda ^2 \left(\frac{\eta ^2
		\kappa  \theta _e}{v}+\frac{\eta ^2 \mu }{v}\right)+\lambda  \left(\eta ^2 p p_e+\frac{\eta ^4 \kappa  \mu  \theta _e}{v^2}-\eta ^2
	p_v\right)-\frac{\eta ^4 \kappa  \theta _e p_v}{v}.
	\eq
	Then we let $b_1$ be a positive constant such that $0<b_1<\min\left\{1, \frac{\kappa p_v \theta_e r^4v}{v^2(pp_e-p_v)R^2+R^4\kappa\mu\theta_e+v^2}\right\}$, and prove all real roots of \eqref{char} are less than $-b_1$. In fact, for $\lambda\geq b_1$, we have 
	\begin{align*}
		f(\lambda)&\geq -b_1^2-b_1 \left(\eta ^2 p p_e+\frac{\eta ^4 \kappa  \mu  \theta _e}{v^2}-\eta ^2
		p_v\right)-\frac{\eta ^4 \kappa  \theta _e p_v}{v}\\
		&\geq -b_1\left(1+\eta ^2 p p_e+\frac{\eta ^4 \kappa  \mu  \theta _e}{v^2}-\eta ^2
		p_v\right) -\frac{\eta ^4 \kappa  \theta _e p_v}{v}\\
		&\geq \frac{r ^4 \kappa  \theta _e p_v}{v}-\frac{\eta ^4 \kappa  \theta _e p_v}{v}\\
		&>0,
	\end{align*}
	where the last inequality holds due to $r<|\eta|$ and $p_v<0$, which is according to \eqref{constitution}. Now, if \eqref{char} admits three real roots, then all the real roots must be less than $-b_1$. Otherwise, \eqref{char} will admit one real root and two congugate complex roots. The real root will be less than $-b_1$, and thus we only need to deal with the two conjugate roots $\lambda_2$ and $\lambda_3$. In fact, we have the following relations,
	\[Re(\lambda_2)=Re(\lambda_3), \quad 2Re(\lambda_2)=\lambda_2+\lambda_3=-\left(\frac{\eta ^2
		\kappa  \theta _e}{v}+\frac{\eta ^2 \mu }{v}\right)-\lambda_1.\]
	Thus, in order to prove $Re(\lambda_2)$ is negative, we need to show the negative real root $\lambda_1$ is close to zero. To this end, we use Taylor expansion to rewrite \eqref{B9} as follows,
	\begin{align*}
		f(\lambda)&=f\left(-\left(\frac{\eta ^2
			\kappa  \theta _e}{v}+\frac{\eta ^2 \mu }{v}\right)\right)+f'\left(-\left(\frac{\eta ^2
			\kappa  \theta _e}{v}+\frac{\eta ^2 \mu }{v}\right)\right)\left(\lambda+\left(\frac{\eta ^2
			\kappa  \theta _e}{v}+\frac{\eta ^2 \mu }{v}\right)\right)\\
		&+\frac{f''\left(-\left(\frac{\eta ^2
				\kappa  \theta _e}{v}+\frac{\eta ^2 \mu }{v}\right)\right)}{2}\left(\lambda+\left(\frac{\eta ^2
			\kappa  \theta _e}{v}+\frac{\eta ^2 \mu }{v}\right)\right)^2+\left(\lambda+\left(\frac{\eta ^2
			\kappa  \theta _e}{v}+\frac{\eta ^2 \mu }{v}\right)\right)^3,
	\end{align*}
	where the expansion is finite because $f'''=6$ and the higher order derivative vanishes. By direct calculation, we have the following estimates,
	\begin{align*}
		&f\left(-\left(\frac{\eta ^2
			\kappa  \theta _e}{v}+\frac{\eta ^2 \mu }{v}\right)\right)=-\frac{\eta ^4 \left(\eta ^2 \kappa  \mu  \theta _e \left(\kappa  \theta _e+\mu \right)+p
			v^2 p_e \left(\kappa  \theta _e+\mu \right)-\mu  v^2 p_v\right)}{v^3}\leq -\frac{r ^6 \kappa  \mu  \theta _e \left(\kappa  \theta _e+\mu \right)}{v^3},\\
		&f'\left(-\left(\frac{\eta ^2
			\kappa  \theta _e}{v}+\frac{\eta ^2 \mu }{v}\right)\right)=\eta ^2 \left(p p_e+\frac{\eta ^2 \left(\kappa ^2 \theta _e^2+3 \kappa  \mu  \theta
			_e+\mu ^2\right)}{v^2}-p_v\right)\leq R^2 \left(p p_e+\frac{R ^2 \left(\kappa ^2 \theta _e^2+3 \kappa  \mu  \theta
			_e+\mu ^2\right)}{v^2}-p_v\right),\\
		&f''\left(-\left(\frac{\eta ^2
			\kappa  \theta _e}{v}+\frac{\eta ^2 \mu }{v}\right)\right)=-\frac{4 \eta ^2 \left(\kappa  \theta _e+\mu \right)}{v}< 0.
	\end{align*}
	Now we let $b_2$ be a positive constant such that $0<b_2<\min\left\{1,\frac{r^6 \kappa  \mu  \theta _e \left(\kappa  \theta _e+\mu \right)}{v^3\left((p p_e-p_v)R^2+\frac{R ^4 \left(\kappa ^2 \theta _e^2+3 \kappa  \mu  \theta
			_e+\mu ^2\right)}{v^2}+1\right)}\right\}$. Then, for $-\left(\frac{\eta ^2
		\kappa  \theta _e}{v}+\frac{\eta ^2 \mu }{v}\right)\leq x\leq -\left(\frac{\eta ^2
		\kappa  \theta _e}{v}+\frac{\eta ^2 \mu }{v}\right)+b_2$, we can substitute above estimates of $f$, $f'$ and $f''$ into the Taylor expansion and obtain the following estimates, 
	\begin{align*}
		f(x)&\leq -\frac{r^6 \kappa  \mu  \theta _e \left(\kappa  \theta _e+\mu \right)}{v^3}+R^2 \left(p p_e+\frac{R ^2 \left(\kappa ^2 \theta _e^2+3 \kappa  \mu  \theta
			_e+\mu ^2\right)}{v^2}-p_v\right)b_2+b_2^3\\
		&\leq-\frac{r^6 \kappa  \mu  \theta _e \left(\kappa  \theta _e+\mu \right)}{v^3}+ \left((p p_e-p_v)R^2+\frac{R ^4 \left(\kappa ^2 \theta _e^2+3 \kappa  \mu  \theta
			_e+\mu ^2\right)}{v^2}+1\right)b_2 <0.
	\end{align*}
	Thus, the real root $\lambda_1$ must be less than $-\left(\frac{\eta ^2
		\kappa  \theta _e}{v}+\frac{\eta ^2 \mu }{v}\right)$ or larger than $-\left(\frac{\eta ^2
		\kappa  \theta _e}{v}+\frac{\eta ^2 \mu }{v}\right)+b_2$. On the other hand, we already know that $f(\lambda)>0$ for all real $\lambda>-b_1$, therefore, we have 
	\[f\left(-\left(\frac{\eta ^2
		\kappa  \theta _e}{v}+\frac{\eta ^2 \mu }{v}\right)+b_2\right)<0, \quad f(b_1+b_2)>0,\quad -\left(\frac{\eta ^2
		\kappa  \theta _e}{v}+\frac{\eta ^2 \mu }{v}\right)+b_2<b_1+b_2.\]
	Thus, we apply the continuity of $f(\lambda)$ to conclude there is a real root between $\left(-\left(\frac{\eta ^2
		\kappa  \theta _e}{v}+\frac{\eta ^2 \mu }{v}\right)+b_2\right)$  and $b_1+b_2$. Moreover, as we assume there are two conjugate roots $\lambda_2$ and $\lambda_3$, we have $\lambda_1>\left(-\left(\frac{\eta ^2
		\kappa  \theta _e}{v}+\frac{\eta ^2 \mu }{v}\right)+b_2\right)$, which immediately implies that 
	\[Re(\lambda_2)=Re(\lambda_3)=\frac{-\left(\frac{\eta ^2
			\kappa  \theta _e}{v}+\frac{\eta ^2 \mu }{v}\right)-\lambda_1}{2}<-b_2.\]
	Finally, we let $b=\min\left\{b_1,b_2\right\}$ and finish the proof.
	
\end{proof}

\subsubsection{Eigenvalue at infinity}
We assume $\lambda$ to be a Laurent polynomial of $\eta$ with coefficients undetermined, and substitute the polynomial in \eqref{char} to solve all the coefficients and obtain the following three asymptotic expansion when $\eta$ tends to $+\infty$,
\bq\label{evalue1}
\begin{aligned}
	\lambda_1=&\frac{v p_v}{\mu }-\frac{v^3 \left(\kappa  \theta _e p_v^2+\mu  p p_e p_v\right)}{ \kappa 
		\mu ^3 \theta _e}\eta ^{-2}+\frac{v^5 p_v \left(\mu ^2 p^2 p_e^2+2 \kappa ^2 \theta _e^2 p_v^2+\mu  p p_e p_v \left(3
		\kappa  \theta _e+\mu \right)\right)}{\kappa ^2 \mu ^5 \theta _e^2}\eta ^{-4} \\
	& -\frac{v^7
		p_v \left(\mu ^3 p^3 p_e^3+3 \mu ^2 p^2 p_e^2 p_v \left(2 \kappa  \theta _e+\mu
		\right)+5 \kappa ^3 \theta _e^3 p_v^3+\mu  p p_e p_v^2 \left(10 \kappa ^2 \theta
		_e^2+4 \kappa  \mu  \theta _e+\mu ^2\right)\right)}{ \kappa ^3 \mu ^7 \theta
		_e^3}\eta ^{-6}+O(1)\eta^{-8},
\end{aligned}
\eq
\bq\label{evalue2}
\begin{aligned}
	\lambda_2=&-\frac{\eta ^2 \mu }{v} +\frac{v
		\left(\mu  p p_e+\kappa  \theta _e p_v-\mu  p_v\right)}{\mu  \left(\mu -\kappa  \theta
		_e\right)} +\eta ^{-2}\left(\frac{v^3 \left(\mu ^2 p p_e-p_v \left(\mu -\kappa  \theta
		_e\right){}^2\right) \left(\mu  p p_e+p_v \left(\kappa  \theta _e-\mu
		\right)\right)}{ \mu ^3 \left(\mu -\kappa  \theta _e\right){}^3}\right)\\
	&+\eta ^{-4}\left(\frac{v^5 \left(2 \mu ^5 p^3 p_e^3+\mu ^2 p^2 p_e^2 p_v \left(\kappa ^3 \theta _e^2-5
		\kappa ^2 \mu  \theta _e^2+10 \kappa  \mu ^2 \theta _e-6 \mu ^3\right)\right)}{ \mu ^5 \left(\mu -\kappa  \theta
		_e\right){}^5}\right.\\
	&\left.+\frac{v^5\left(\mu  p p_e
		p_v^2 \left(\mu -\kappa  \theta _e\right){}^2 \left(3 \kappa ^2 \theta _e^2-8 \kappa 
		\mu  \theta _e+6 \mu ^2\right)-2 p_v^3 \left(\mu -\kappa  \theta
		_e\right){}^5\right)}{\mu ^5 \left(\mu -\kappa  \theta
		_e\right)^5}\right)\\
	&+\eta ^{-6}\left(\frac{v^7 \left(5 \mu ^7 p^4 p_e^3-\mu ^3 p^3 p_e^3 p_v \left(\kappa ^4
		\theta _e^3-7 \kappa ^3 \mu  \theta _e^3+21 \kappa ^2 \mu ^2 \theta _e^2-35 \kappa 
		\mu ^3 \theta _e+20 \mu ^4\right)\right)}{ \mu ^7 \left(\mu -\kappa  \theta
		_e\right){}^7}\right.\\
	& \left.+\frac{v^7\left(3 \mu ^2 p^2 p_e^2 p_v^2 \left(\mu -\kappa  \theta
		_e\right){}^2 \left(-2 \kappa ^3 \theta _e^3+9 \kappa ^2 \mu  \theta _e^2-15 \kappa 
		\mu ^2 \theta _e+10 \mu ^3\right)\right)}{ \mu ^7 \left(\mu -\kappa  \theta
		_e\right){}^7}\right.\\
	&\left.-\frac{v^7\left(\mu  p p_e p_v^3 \left(\mu -\kappa  \theta
		_e\right){}^3 \left(-10 \kappa ^3 \theta _e^3+36 \kappa ^2 \mu  \theta _e^2-45 \kappa 
		\mu ^2 \theta _e+20 \mu ^3\right)+5 p_v^4 \left(\mu -\kappa  \theta
		_e\right){}^7\right)}{ \mu ^7 \left(\mu -\kappa  \theta
		_e\right){}^7}\right)+O(1)\eta^{-8},\\
\end{aligned}
\eq
\bq\label{evalue3}
\begin{aligned}
	\lambda_3=&-\frac{\eta ^2 \kappa  \theta _e}{v}+\frac{p v p_e}{\kappa  \theta _e-\mu
	}+\eta ^{-2}\left(\frac{p v^3 p_e
		\left(\kappa  \theta _e \left(p p_e+p_v\right)-\mu  p_v\right)}{ \kappa  \theta
		_e \left(\kappa  \theta _e-\mu \right){}^3}\right)\\
	&+\eta ^{-4}\left(\frac{p v^5 p_e \left(2 \kappa ^2 p^2 \theta _e^2 p_e^2+p p_e p_v \left(4 \kappa ^2
		\theta _e^2-5 \kappa  \mu  \theta _e+\mu ^2\right)+p_v^2 \left(\mu -\kappa  \theta
		_e\right){}^2\right)}{ \kappa ^2 \theta _e^2 \left(\kappa  \theta _e-\mu
		\right){}^5}\right)\\
	&+\eta ^{-6}\left(\frac{p v^7 p_e \left(5 \kappa ^3 p^3 \theta _e^3 p_e^2-p^2 p_e^2 p_v
		\left(-15 \kappa ^3 \theta _e^3+21 \kappa ^2 \mu  \theta _e^2-7 \kappa  \mu ^2 \theta
		_e+\mu ^3\right)-3 p p_e p_v^2 \left(\mu
		-3 \kappa  \theta _e\right) \left(\mu -\kappa  \theta _e\right){}^2\right)}{
		\kappa ^3 \theta _e^3 \left(\kappa  \theta _e-\mu \right){}^7}\right.\\
	&\left.-\frac{p v^7 p_e\left(p_v^3 \left(\mu -\kappa  \theta _e\right){}^3\right)}{\eta ^6
		\kappa ^3 \theta _e^3 \left(\kappa  \theta _e-\mu \right){}^7}\right)+O(1)\eta^{-8}.\\
\end{aligned}
\eq
Morevover, as we consider the ideal gas, the Prandtl number is usually smaller than one. Then we combine the constitutive relation \eqref{constitution} and the definition of Prandtl number Pr to obtain
\[\frac{\mu}{\kappa\theta_e}=\frac{\mu c_v}{\kappa}=\text{Pr}\frac{c_v}{c_p}<1.\]
Therefore, the denominators in the expansions above are nonzero and thus the formulae are well-defined. Moreover, when $\eta$ tends to infinity, the quadratic term would be dominant in the expansion of eigenvalues, and we can use above inequality to conclude 
\bq\label{B7}
0>\text{Re}(\lambda_1)>\text{Re}(\lambda_2)>\text{Re}(\lambda_3),\quad \eta\rightarrow\infty.
\eq

\subsubsection{Eigenvalue at zero} 
Now we switch to study the asymptotic behavior of the three eigenvalues when $\eta$ near zero. Similar as the expansion of eigenvalues at infinity, we can assume the eigenvalues to be polynomials of $\eta$ and substitute them into \eqref{char} to obtain the following expansions when $\eta$ tens to zero,
\bq\label{evalue-zero-1}
\begin{aligned}
	\lambda_1=&\frac{\eta ^2 \kappa  \theta _e p_v}{v \left(p
		p_e-p_v\right)}-\frac{\eta ^4 \kappa ^2 p \theta
		_e^2 p_e p_v \left(\mu  p p_e+p_v \left(\kappa  \theta _e-\mu \right)\right)}{v^3
		\left(p_v-p p_e\right){}^4}\\
	&-\frac{\eta ^6 \kappa ^3 p \theta _e^3 p_e p_v \left(\mu ^2 \left(-p^3\right) p_e^3+\mu 
		p^2 p_e^2 p_v \left(\mu -3 \kappa  \theta _e\right)+p p_e p_v^2 \left(-2 \kappa ^2
		\theta _e^2+\kappa  \mu  \theta _e+\mu ^2\right)-p_v^3 \left(\mu -\kappa  \theta
		_e\right){}^2\right)}{v^5 \left(p p_e-p_v\right){}^7}+O(1)\eta^6,
\end{aligned}
\eq
\bq\label{evalue-zero-2}
\begin{aligned}
	\lambda_2=&-i \eta  \sqrt{p p_e-p_v}-\frac{\eta ^2 \left(\kappa  p \theta _e p_e+\mu  p p_e-\mu
		p_v\right)}{2 v \left(p p_e-p_v\right)}-\frac{i \eta ^3 \left(-p^2 p_e^2 \left(\mu -\kappa  \theta _e\right){}^2+2 p p_e p_v
		\left(2 \kappa ^2 \theta _e^2-\kappa  \mu  \theta _e+\mu ^2\right)-\mu ^2
		p_v^2\right)}{8 v^2 \left(p p_e-p_v\right){}^{5/2}}\\
	&+\frac{\eta ^4 \kappa ^2 p \theta
		_e^2 p_e p_v \left(\mu  p p_e+p_v \left(\kappa  \theta _e-\mu \right)\right)}{2 v^3
		\left(p_v-p p_e\right){}^4}+O(1)\eta^4,
\end{aligned}
\eq
\bq\label{evalue-zero-3}
\begin{aligned}
	\lambda_3=&i \eta  \sqrt{p p_e-p_v}-\frac{\eta ^2 \left(\kappa  p \theta _e p_e+\mu  p p_e-\mu
		p_v\right)}{2 v \left(p p_e-p_v\right)}-\frac{i \eta ^3 \left(p^2 p_e^2 \left(\mu -\kappa  \theta _e\right){}^2-2 p p_e p_v
		\left(2 \kappa ^2 \theta _e^2-\kappa  \mu  \theta _e+\mu ^2\right)+\mu ^2
		p_v^2\right)}{8 v^2 \left(p p_e-p_v\right){}^{5/2}}\\
	&+\frac{\eta ^4 \kappa ^2 p \theta
		_e^2 p_e p_v \left(\mu  p p_e+p_v \left(\kappa  \theta _e-\mu \right)\right)}{2 v^3
		\left(p_v-p p_e\right){}^4}+O(1)\eta^4.
\end{aligned}
\eq
Note the coefficients of $i\eta$ in $\lambda_i$ are $ \pm\sqrt{p p_e-p_v}$ and $0$ respectively, which are actually the eigenvalues of the matrix $F'(\bar{U})$ and corresponding to the sound speed. On the other hand, according to \eqref{constitution}, $p_e>0$ and $p_v<0$, therefore the coefficients of $\eta^2$ in all $\lambda_i$ are negative, which correspond to the combined effect of viscosity and heat conductivity and represent the dissipative structure of the system.

\subsection{Singular parts of Green's function}
In this part, we will treat the singular part of Green's function. More precisely, in order to obtain the space time structure of the Green's function, we need inverse Fourier transform of $\hat{\bbg}(\eta,t;\bar{U})$. Following the methodology in \cite{L-Y-2020}, the first step is to study the high frequency part of $\hat{\bbg}(\eta,t;\bar{U})$ and capture the singularities. 

\subsubsection{Approximated eigenvalue} For high frequency, we note the constant terms in the expansion of the eigenvalues at infinity. This relates to the appearance of Dirac-delta function. However, we can not directly use the formula of $\lambda_j$ and $\hat{\bbm}_j$ to construct the space time Green's function. This is because the inverse power of $\eta$ in \eqref{evalue1}, \eqref{evalue2} and \eqref{evalue3} makes the expansion not analytic at zero, which is difficult to apply the complex contour integral to analyze the inverse Fourier transform. Therefore, for convenience in later analysis, we will use another approximate eigenvalues as below, 
\bq\label{lambda^*}
\begin{aligned}
	&\lambda_1^*=\beta_1^*+\sum_{k=1}^3\frac{A_{1,k}}{(1+\eta^2)^k}-\frac{K_1}{\left(\eta ^2+1\right)^4},\\
	&\lambda_2^*=-\alpha^*_2\eta^2+\beta^*_2+\sum_{k=1}^3\frac{A_{2,k}}{(1+\eta^2)^k}-\frac{K_2}{\left(\eta ^2+1\right)^4},\\
	&\lambda_3^*=-\alpha^*_3\eta^2+\beta^*_3+\sum_{k=1}^3\frac{A_{3,k}}{(1+\eta^2)^k}-\frac{K_3}{\left(\eta ^2+1\right)^4},
\end{aligned}
\eq
where the coefficients $\alpha^*_j$, $\beta^*_j$ and $A_{j,k}$ are listed in Appendix \ref{sec:A}, and $K_1$, $K_2$ and $K_3$ are large enough positive constants to guarantee the negative sign of the real part of the approximate eigenvalues. Now we have the following lemma.
\begin{lemma}\label{L4.1}
	For a fixed state $\bar{U}$, we can find positive constants $K_1$, $K_2$, $K_3$, $\sigma_0$, $\sigma_0^*$ and $\sigma_1^*$ such that the following asserts hold:
	\begin{enumerate}
		\item The approximate eigenvalues $\lambda_j^*$ are analytic around the real axis. More precisely, $\lambda_j^*$ is analytic for $j=1,2,3$ in the following region,
		\[\{\eta:\ |Im(\eta)|<\sigma_0\}.\]
		\item $\lambda_j^*$ is an approximation of $\lambda_j$ which is accurate up to the power $\left(\eta ^2+1\right)^{-3}$, and thus we have 
		\bq\label{B18}
		\lambda_j^*=\lambda_j+O(1)|\eta|^{-8}.
		\eq
		\item In the region $\{\eta:\ |Im(\eta)|<\sigma_0\}$, all the three approximated eigenvalues $\lambda_j^*$ have negative real part and are distinct. Actually, we have 
		\[\sup_{|Im(\eta)|<\sigma_0}Re(\lambda_1^*)<-\sigma_0^*,\qquad\qquad \sup_{|Im(\eta)|<\sigma_0}Re(\lambda_2^*+\frac{\mu }{2v}(\eta^2+1))<-\sigma_0^*,\]
		\[\sup_{|Im(\eta)|<\sigma_0}Re(\lambda_3^*+\frac{\kappa\theta_e }{2v}(\eta^2+1))<-\sigma_0^*,\qquad\quad \min_{j,k}\inf_{|Im(\eta)|<\sigma_0}|Re(\lambda_k^*-\lambda_j^*)|=\sigma_1^*.\]
	\end{enumerate}
\end{lemma}
\begin{proof}
	At the fixed state $\bar{U}$, all the coefficients in $\lambda_j^*$ are constants. 
	\begin{enumerate}
		\item For the first assertion, we can let $\sigma_0$ sufficiently small so that the imaginary part of $\eta$ is very small and $1+\eta^2$ has a positive lower bound. Therefore, we immediately have $\lambda_j^*$ are analytic for $j=1,2,3$. 
		
		\item The first four terms of $\lambda_j^*$ is from the expansion of $\lambda_j$ at infinity, therefore we have 
		\[|\lambda_j^*+\frac{K_0}{\left(\eta ^2+1\right)^4}-\lambda_j|=O(1)|\eta|^{-8}\Longrightarrow \lambda_j^*=\lambda_j+O(1)|\eta|^{-8}.\]
		
		\item When $\eta$ tends to infinity, the leading term would be dominant. Then we can apply \eqref{B7} to find a positive constant $\sigma_0^*$ and a sufficiently large constant $C_1$ such that 
		\[\sup_{\substack{|Im(\eta)|<\sigma_0\\ |\eta|>C_1}}Re(\lambda_1^*)<-\sigma_0^*.\] 
		For $|Im(\eta)|<\sigma_0$ and $|\eta|\leq C_1$, we know the first four terms of $\lambda_1^*$ are bounded. Then, we can make $K_1$ sufficiently large to guarantee the negative sign of $\lambda_1^*$. Therefore, we have for sufficiently large $K_1$ that 
		\[\sup_{|Im(\eta)|<\sigma_0}Re(\lambda_1^*)<-\sigma_0^*.\]
		Similar arguments can be applied to prove the negative sign of the real part of $\lambda_2^*$ and $\lambda_3^*$. Moreover, for $|\eta|\leq C_1$ we can choose $K_2,K_3$ be sufficiently large and satisfy $K_1<K_2<K_3$. Then we have 
		\[0>Re(\lambda_1^*)>Re(\lambda_2^*)>Re(\lambda_3^*),\quad |\eta|\leq C_1.\]
		On the other hand, when $|\eta|>C_1$, we apply \eqref{B7} and the fact $K_1<K_2<K_3$ to have 
		\[0>Re(\lambda_1^*)>Re(\lambda_2^*)>Re(\lambda_3^*),\quad |\eta|> C_1.\]
		Thus, we can conclude that $\lambda_j^*$ are distinct for all $\eta$ and we can find a positive constant $\sigma_1^*$ such that 
		\[\min_{j,k}\inf_{|Im(\eta)|<\sigma_0}|Re(\lambda_k^*-\lambda_j^*)|=\sigma_1^*.\]
		This finish the proof of the proposition.\newline
	\end{enumerate}
\end{proof}

\subsubsection{Representation of singular part}
Now we are going to construct the singular part of the Green's function. According to \eqref{B12}, we can express the Green's function in terms of eigenvalues $\lambda_j$ for sufficiently large $\eta$. Actually we have the expression of $\hat{\bbm}_j$ as follows when $\eta$ near infinity,
\bq\label{B13}
\frac{\left(
	\begin{array}{ccc}
		p p_e \eta ^2+\frac{\left(\mu  \eta ^2+v \lambda_j \right) \left(\kappa  \theta _e \eta
			^2+v \lambda_j \right)}{v^2} & \frac{i \eta  \left(\kappa  \theta _e \eta ^2+i u v p_e
			\eta +v \lambda_j \right)}{v} & \eta ^2 p_e \\
		-\frac{i \eta  p_v \left(\kappa  \theta _e \eta ^2+v \lambda_j \right)}{v} & \lambda_j 
		\left(\frac{\kappa  \theta _e \eta ^2}{v}+i u p_e \eta +\lambda_j \right) & -i \eta 
		\lambda_j  p_e \\
		-\frac{\eta  p_v \left(i u \kappa  \theta _e \eta ^2+p v \eta +i u v \lambda_j \right)}{v}
		& \frac{\eta  \left(i v \lambda_j  p_e u^2+v \eta  p_v u+\lambda_j  \left(-i p v-u \eta 
			\mu +u \eta  \kappa  \theta _e\right)\right)}{v} & -p_v \eta ^2-i u \lambda_j  p_e \eta
		+\lambda_j  \left(\frac{\mu  \eta ^2}{v}+\lambda_j \right) \\
	\end{array}
	\right)}{\Pi_{k\neq j}(\lambda_j-\lambda_k)}.
\eq
Note $\hat{\bbm}_j$ can be viewed as a matrix valued function $\hat{\bbm}_j(\eta;\lambda)$, which may not be analytic. We replace $\lambda_j$ by the approximate eigenvalues $\lambda_j^*$ to approximate $\hat{\bbm}_j(\eta;\lambda)$ for high frequency, then it turns out that $\hat{\bbm}_j(\eta;\lambda_j^*)$ is analytic around the real axis in the following lemma. In the following, we will denote $\hat{\bbm}_j(\eta;\lambda_j^*)$ by $\hat{\bbm}_j^*$.
\begin{lemma}\label{L4.2}
	The matrix $\hat{\bbm}_j^*$ is analytic in the region $|Im(\eta)|<\sigma_0$ around the real axis, and it has the following expansion at infinity,
	\begin{equation}\label{M expansion}
		\begin{aligned}
			&\hat{\bbm}_j^{*}=M_j^{*,0}+i\eta^{-1}M_j^{*,1}+\eta^{-2}M_j^{*,2}+i\eta^{-3}M_j^{*,3}+\eta^{-4}M_j^{*,4}+O(1)\eta^{-5},\quad j=1,2,3,\quad |\eta|\rightarrow\infty.
		\end{aligned}
	\end{equation}
	where $M_j^{*,k}$ are listed in \eqref{M_1^*}, \eqref{M_2^*} and \eqref{M_3^*} in Appendix \ref{sec:B} . 
\end{lemma}
\begin{proof}
	According to Lemma \ref{L4.1}, $\lambda_j^*$ are distinct and there is an uniform gap between two approximate eigenvalues. Therefore, according to \eqref{B13}, the denominator of $\bbm_j^*$ are nonzero so that $\bbm_j^*$ are well-defined for $j=1,2,3$. Then the analyticity  of $\lambda_j^*$ immediately implies the analyticity of $\bbm_j^*$ in the region $|Im(\eta)|<\sigma_0$. Finally, the expansion of $\bbm_j^*$ can be obtained by direct calculations. 
\end{proof}
Now we substitute $\lambda_j^*$ and $\bbm_j^*$ into the representation of Green's function \eqref{B12} to obtain the high frequency approximation,
\bq\label{singular-G}
\bbg^*(x,t;\bar{U})=\mathcal{F}^{-1}\left[\hat{\bbg}^*(\eta,t;\bar{U})\right]=\sum_{j=1}^3\mathcal{F}^{-1}\left[\hat{\bbg}^{*,j}\right]=\sum_{j=1}^3G^{*,j}(x,t;\bar{U}),\quad \hat{\bbg}^{*,j}=e^{\lambda_j^*t}\hat{\bbm}_j^*,\quad j=1,2,3,
\eq
which will be shown as the singular part of the Green's function. In the rest of this subsections, we will do inverse Fourier transform to \eqref{singular-G} to find the space-time structure of singular part $\bbg^*(x,t;\bar{U})$.
\begin{remark}\label{R4.1}
	Note in \eqref{B12}, we have the representation of the Green's function in Fourier space as follows,
	\[\hat{\bbg}(\eta,t;\bar{U})=\sum_{j=1}^3e^{\lambda_jt}\hat{\bbm}_j,\quad \hat{\bbm}_j=\frac{\text{adj}\left(s+i\eta F'(\bar{U})+\eta^2B(\bar{U})\right)|_{s=\lambda_j}}{\Pi_{k\neq j}(\lambda_j-\lambda_k)}.\]
	Moreover, from \eqref{B2}, we know the initial data $\hat{\bbg}(\eta,0;\bar{U})=I$, which is independent of $\eta$ and $\lambda_j$ due to the cancelations in $\sum_{j=1}^3\hat{\bbm}_j(\eta;\lambda_j)$. On the other hand, according to \eqref{singular-G}, the construction of the singular part of Green's function is just replacing $\lambda_j$ by $\lambda_j^*$ in \eqref{B12}. Therefore, when we take $t=0$, we have 
	\[\hat{\bbg}^*(\eta,0;\bar{U})=\sum_{j=1}^3\hat{\bbg}^{*,j}(\eta,0;\bar{U})=\sum_{j=1}^3\hat{\bbm}_j^*=\sum_{j=1}^3\hat{\bbm}_j(\eta;\lambda_j^*)=I.\]
	The above formula shows the singular part $\bbg^*$ absorbs all initial data of the Green's function $\bbg$.
\end{remark}

\subsubsection{Pointwise estimate of $\mathcal{F}^{-1}[\hat{\bbm}_j^*]$ and $\mathcal{F}^{-1}[e^{\lambda_j^*t}]$}
In this part, we will calculate the inverse Fourier transform of $\hat{\bbm}_j^*$ and $e^{\lambda_j^*t}$, which will lead to the space-time pointwise estimates of $\bbg^*(x,t;\bar{U})$. We start with the following preliminary lemma.
\begin{lemma}\label{L4.3}
	\cite{L-Y-2020} Let $\hat{f}(\eta)$ be the Fourier transform of $f(x)$. Suppose, in the region $|Im(\eta)|<\sigma_0$, we have $\hat{f}(\eta)$ is analytic and $\hat{f}(\eta)$ has the following asymptotic property,
	\bq\label{B19}
	\hat{f}(\eta)=O(1)\frac{\eta}{(1+|\eta|)^{(2m+2)}}+O(1)\frac{1}{(1+|\eta|)^{(2m+2)}},\quad  m\geq0,\quad \eta\rightarrow\infty.
	\eq
	Then  $f(x)\in H^{2m}(\bbr)$ and 
	\[\sum_{j=0}^{2m}|\frac{d^j}{dx^j}f(x)|=O(1)e^{-\sigma_0|x|},\quad x\in \bbr.\]
\end{lemma}
\begin{proof}
	By Plancherel's Theorem of Fourier transform, the assertion $f(x)\in H^{2m}(\bbr)$ is equivalent to $\eta^{j}\hat{f}(\eta)\in L^2(\bbr)$ for $0\leq j\leq 2m$. As $\hat{f}$ is analytic in the region $Im(\eta)<\sigma_0$ and thus has no singularity on the real axis, we only need check for the highest order $2m$. In fact, we have 
	\[\int_{\bbr}|\frac{d^{2m}}{dx^{2m}}f(x)|^2dx=\int_{\bbr}|(i\eta)^{2m}\hat{f}(\eta)|^2d\eta\leq O(1)\int_{\bbr}\frac{|\eta|^{4m+2}}{(1+|\eta|)^{4m+4}}d\eta\leq O(1),\]
	which shows $f(x)\in H^{2m}(\bbr)$. For the second assertion, we only consider the case $x>0$, and the other case can be similarly treated. As $\hat{f}(\eta)$ is analytic in the region $|Im(\eta)|<\sigma_0$, we can apply Cauchy contour Theorem to shift the integral path of the inverse Fourier transform, and apply the decay property of $\hat{f}(\eta)$ to have the following estimate for $0\leq j\leq 2m$,
	\begin{align*}
		|\frac{d^j}{dx^j}f(x)|&=\left|\frac{1}{2\pi}\int_{\bbr}e^{i\eta x}(i\eta)^j\hat{f}(\eta)d\eta\right|\\
		&=\left|\frac{1}{2\pi}\int_{\bbr}e^{i(\eta+i\sigma_0) x}(i\eta+i\sigma_0)^j\hat{f}(\eta+i\sigma_0)d\eta\right|\\
		&=e^{-\sigma_0 x}\left|\frac{1}{2\pi}\int_{\bbr}e^{i\eta x}(i\eta+i\sigma_0)^j\hat{f}(\eta+i\sigma_0)d\eta\right|.
	\end{align*}
	Now, in the region $|Im(\eta)|<\sigma_0$, if $|\eta|>C$ for some large positive constant $C$, we have the asymptotic property \eqref{B19}. Moreover, for other $\eta$ in the region $|Im(\eta)|<\sigma_0$ such that $|\eta|\leq C$, it is obvious that $(i\eta+i\sigma_0)^j\hat{f}(\eta+i\sigma_0)$ will be uniformly bounded due to the analyticity of $\hat{f}$. Therefore, we have the following estimate for the inverse Fourier transform,
	\begin{align*}
		|\frac{d^j}{dx^j}f(x)|f&=e^{-\sigma_0 x}\left|\frac{1}{2\pi}\int_{|\eta|\leq C}e^{i\eta x}(i\eta+i\sigma_0)^j\hat{f}(\eta+i\sigma_0)d\eta+\frac{1}{2\pi}\int_{|\eta|> C}e^{i\eta x}(i\eta+i\sigma_0)^j\hat{f}(\eta+i\sigma_0)d\eta\right|\\
		&= e^{-\sigma_0 x}\left|O(1)+\frac{O(1)}{2\pi}\int_{|\eta|> C}\frac{e^{i\eta x}(i\eta+i\sigma_0)^j(\eta+i\sigma_0)}{(1+|\eta+i\sigma_0|)^{2m+2}}d\eta+\frac{O(1)}{2\pi}\int_{|\eta|> C}\frac{e^{i\eta x}(i\eta+i\sigma_0)^j}{(1+|\eta+i\sigma_0|)^{2m+2}}d\eta\right|\\
		&=e^{-\sigma_0 x}\left|\frac{O(1)}{2\pi}\int_{|\eta|> C}\frac{e^{i\eta x}}{(1+|\eta|)}d\eta+\frac{O(1)}{2\pi}\int_{|\eta|> C}\frac{e^{i\eta x}}{(1+|\eta|)^2}d\eta\right|\\
		&=O(1)e^{-\sigma_0 x},
	\end{align*}
	where the last equality comes from the integrability of the function $\frac{sinx}{x}$. Then, we add up above estimates for all $j$ to verify the second assertion, and finish the proof.
\end{proof}
Now, we are ready to provide the pointwise space-time estimate of $\bbm_j^*(x)$. Actually, in the following lemma, we can combine Lemma \ref{L4.3} and the expansion of $\hat{\bbm}_j^*$ to obtain the estimates up to fourth order derivative of $\bbm_j^*(x)$.
\begin{lemma}\label{L4.4}
	For the positive constant $\sigma_0$ defined in Lemma \ref{L4.1} and Lemma \ref{L4.2}, we have the following pointwise estimates for $\bbm_j^*$,
	\begin{equation*}
		\begin{aligned}
			&\bbm_j^*(x)=\delta(x)M_j^{*,0}+O(1)e^{-\sigma_0|x|},\\
			&\frac{d}{dx}\bbm_j^*(x)=\frac{d}{dx}\delta(x)M_j^{*,0}-\delta(x)M_j^{*,1}+O(1)e^{-\sigma_0|x|},\\
			&\frac{d^2}{dx^2}\bbm_j^*(x)=\frac{d^2}{dx^2}\delta(x)M_j^{*,0}-\frac{d}{dx}\delta(x)M_j^{*,1}+\delta(x)M_j^{*,2}+O(1)e^{-\sigma_0|x|},\\
			&\frac{d^3}{dx^3}\bbm_j^*(x)=\frac{d^3}{dx^3}\delta(x)M_j^{*,0}-\frac{d^2}{dx^2}\delta(x)M_j^{*,1}+\frac{d}{dx}\delta(x)M_j^{*,2}+\delta(x)M_j^{*,3}+O(1)e^{-\sigma_0|x|},\\
			&\frac{d^4}{dx^4}\bbm_j^*(x)=\frac{d^4}{dx^4}\delta(x)M_j^{*,0}-\frac{d^3}{dx^3}\delta(x)M_j^{*,1}+\frac{d^2}{dx^2}\delta(x)M_j^{*,2}+\frac{d}{dx}\delta(x)M_j^{*,3}+\delta(x)M_j^{*,4}+O(1)e^{-\sigma_0|x|}.
		\end{aligned}
	\end{equation*}
\end{lemma}
\begin{proof}
	We will only prove the zeroth order estimate of $\bbm_1^*(x)$, and the other estimates can be similarly obtained. According to Lemma \ref{L4.2}, $\hat{\bbm}_1^*$ is analytic in the region $|Im(\eta)|<\sigma_0$ and has the expansion \eqref{M expansion}, therefore we have $\hat{\bbm}_1^*-M_1^{*,0}$ is also analytic in the region $|Im(\eta)|<\sigma_0$. Moreover, we have 
	\[\hat{\bbm}_1^*-M_1^{*,0}=O(1)i\eta^{-1}+O(1)\eta^{-2},\quad \eta\rightarrow\infty,\]
	where $O(1)$ denotes an order one matrix in above formula. Thus we have $\hat{\bbm}_1^*-M_1^{*,0}$ satisfies the requirements in Lemma \ref{L4.3} with $m=0$. Therefore, we apply Lemma \ref{L4.3} to obtain  
	\[|\bbm_1^*(x)-\delta(x)M_1^{*,0}|=O(1)e^{-\sigma_0|x|},\]
	which finish the proof of the zeroth order estimate of $\bbm_1^*(x)$. The others can be obtained in the same way.
\end{proof}
Next, we will study the inverse Fourier transform of $e^{\lambda_j^*t}$. It is a little bit different from the inverse of $\hat{\bbm}_j^*$, since the leading order of $\lambda_j^*$ are different. Let  $\Lambda_j^*(x,t)=\mathcal{F}^{-1}\left[e^{\lambda_j^*t}\right]$, we have the following estimates. 
\begin{lemma}\label{L4.5}
	Let $\sigma_0$ and $\sigma_0^*$ be positive constants defined in Lemma \ref{L4.1} and Lemma \ref{L4.2}, then we have the following point-wise estimates for $\Lambda_j^*(x,t)$ up to the third order,
	\bq\label{B20}
	\left\{\begin{aligned}
		&\Lambda_1^*(x,t)=e^{\frac{vp_v}{\mu}t}\delta(x)
		+\left\{\begin{aligned}&O(1)te^{-\sigma_0^*t-\sigma_0|x|},\quad 0<t<1,\\
			&O(1)e^{-\sigma_0^*t-\sigma_0|x|},\quad t\geq 1,
		\end{aligned}\right.\\
		&\partial_x\Lambda_1^*(x,t)=e^{\frac{vp_v}{\mu}t}\frac{d}{dx}\delta(x)
		+\left\{\begin{aligned}&O(1)te^{-\sigma_0^*t-\sigma_0|x|},\quad 0<t<1,\\
			&O(1)e^{-\sigma_0^*t-\sigma_0|x|},\quad t\geq 1,
		\end{aligned}\right.\\
		&\partial_x^2\Lambda_1^*(x,t)=e^{\frac{vp_v}{\mu}t}\frac{d^2}{dx^2}\delta(x)-A_{1,1}te^{\frac{vp_v}{\mu}t}\delta(x)
		+\left\{\begin{aligned}&O(1)te^{-\sigma_0^*t-\sigma_0|x|},\quad 0<t<1,\\
			&O(1)e^{-\sigma_0^*t-\sigma_0|x|},\quad t\geq 1,
		\end{aligned}\right.\\
		&\partial_x^3\Lambda_1^*(x,t)=e^{\frac{vp_v}{\mu}t}\frac{d^3}{dx^3}\delta(x)-A_{1,1}te^{\frac{vp_v}{\mu}t}\frac{d}{dx}\delta(x)
		+\left\{\begin{aligned}&O(1)te^{-\sigma_0^*t-\sigma_0|x|},\quad 0<t<1,\\
			&O(1)e^{-\sigma_0^*t-\sigma_0|x|},\quad t\geq 1.
		\end{aligned}\right.\\
	\end{aligned}
	\right.
	\eq
	\bq\label{B21}
	\left\{\begin{aligned}
		&\partial^k_x\Lambda_2^*(x,t)=\left\{\begin{aligned}&\partial^k_x\left[\frac{e^{\beta^*_2t}}{\sqrt{4\pi \alpha^*_2t}}e^{-\frac{x^2}{4\alpha^*_2t}}\right]+O(1)te^{-\sigma_0^*t-\sigma_0|x|}, &\quad 0<t<1,\\
			&O(1)e^{-\sigma_0^*t-\sigma_0|x|},& t\geq 1,
		\end{aligned}
		\right.\\
		&\partial^{k+2}_x\Lambda_2^*(x,t)=\left\{\begin{aligned}&\partial^{k+2}_x\left[\frac{e^{\beta^*_2t}}{\sqrt{4\pi \alpha^*_2t}}e^{-\frac{x^2}{4\alpha^*_2t}}\right]+tA_{2,1}\partial^k_x\left[\frac{e^{\beta^*_2t}}{\sqrt{4\pi \alpha^*_2t}}e^{-\frac{x^2}{4\alpha^*_2t}}\right]+O(1)te^{-\sigma_0^*t-\sigma_0|x|}, &\quad 0<t<1,\\
			&O(1)e^{-\sigma_0^*t-\sigma_0|x|},& t\geq 1,
		\end{aligned}
		\right.\\
		&\alpha^*_2=\frac{\mu}{v},\quad \beta^*_2=\frac{v \left(\mu  p p_e+\kappa  \theta _e p_v-\mu 
			p_v\right)}{\mu  \left(\mu -\kappa  \theta _e\right)},\quad k=0,1.
	\end{aligned}
	\right.
	\eq
	\bq\label{B22}
	\left\{\begin{aligned}
		&\partial^k_x\Lambda_3^*(x,t)=\left\{\begin{aligned}&\partial^k_x\left[\frac{e^{\beta^*_3t}}{\sqrt{4\pi \alpha^*_3t}}e^{-\frac{x^2}{4\alpha^*_3t}}\right]+O(1)te^{-\sigma_0^*t-\sigma_0|x|}, &\quad 0<t<1,\\
			&O(1)e^{-\sigma_0^*t-\sigma_0|x|},& t\geq 1,
		\end{aligned}
		\right.\\
		&\partial^{k+2}_x\Lambda_3^*(x,t)=\left\{\begin{aligned}&\partial^{k+2}_x\left[\frac{e^{\beta^*_3t}}{\sqrt{4\pi \alpha^*_3t}}e^{-\frac{x^2}{4\alpha^*_3t}}\right]+tA_{3,1}\partial^k_x\left[\frac{e^{\beta^*_3t}}{\sqrt{4\pi \alpha^*_3t}}e^{-\frac{x^2}{4\alpha^*_3t}}\right]+O(1)te^{-\sigma_0^*t-\sigma_0|x|}, &\quad 0<t<1,\\
			&O(1)e^{-\sigma_0^*t-\sigma_0|x|},& t\geq 1,
		\end{aligned}
		\right.\\
		&\alpha^*_3=\frac{\kappa\theta_e}{v},\quad \beta^*_3=\frac{p v p_e}{\kappa  \theta _e-\mu },\quad k=0,1.
	\end{aligned}
	\right.
	\eq
	
\end{lemma}
\begin{proof}
	The proof can be similarly constructed as in \cite{L-Y-2020}. Here we  use two steps to briefly sketch the proof.\newline 
	
	\noindent $\bullet$ (Step 1: $\Lambda_1^*$) We first give the estimate of $\Lambda_1^*$. According to \eqref{lambda^*}, we have 
	\[e^{\lambda_1^*t}=e^{\left(\frac{v p_v}{\mu }+\sum_{k=1}^3\frac{A_{1,k}}{(1+\eta^2)^k}-\frac{K_1}{\left(\eta ^2+1\right)^4}\right)t}=e^{\frac{v p_v}{\mu }t}e^{\left(\sum_{k=1}^3\frac{A_{1,k}}{(1+\eta^2)^k}-\frac{K_1}{\left(\eta ^2+1\right)^4}\right)t}.\]
	Lemma \ref{L4.1} immediately implies that $e^{\lambda_1^*t}$ is analytic in the region $|Im(\eta)|<\sigma_0$ around the real axis. Moreover, we can view $e^{\lambda_1^*t}$ as a function of $(1+\eta^2)^{-1}$ and expand $e^{\lambda_1^*t}$ at $\eta=\infty$ to obtain 
	\begin{align*}
		e^{\lambda_1^*t}=&e^{\frac{v p_v}{\mu }t}\left(1+\frac{A_{1,1} t}{(1+\eta ^2)}+\frac{\frac{A_{1,1}^2 t^2}{2}-A_{1,1} t+A_{1,2}
			t}{(1+\eta ^2)^2}\right)\\
		&+e^{\frac{v p_v}{\mu }t}\left(C(\frac{1}{1+\eta^2})(t^3+t^2+t)\frac{e^{\left(\sum_{k=1}^3\frac{\gamma^kA_{1,k}}{(1+\eta^2)^k}-\frac{\gamma^4K_1}{\left(\eta^2+1\right)^4}\right)t}}{(1+\eta^2)^3}\right),\quad \eta\rightarrow\infty,
	\end{align*}
	where we use the Lagrangian remainder with $\gamma\in(0,1)$, and the coefficient $C(\frac{1}{1+\eta^2})$ is a function of $\frac{1}{1+\eta^2}$. By direct calculation, we obtain that $C(\frac{1}{1+\eta^2})$ is analytic and uniformly bounded. On the other hand, from Lemma \ref{L4.1}, we know $\frac{v p_v}{\mu }<-\sigma_0^*<0$. Thus we can use the factor $e^{\left(\frac{v p_v}{\mu }+\sigma_0^*\right)t}$ to absorb $t$ and $t^2$. Moreover, as $|\eta|$ can be chosen to be sufficiently large, the factor $e^{\left(\frac{v p_v}{\mu }+\sigma_0^*\right)t}$ can also absorb the exponential part in the remainder. More precisely, we have
	\bq\label{B24}
	\begin{aligned}
		&e^{\lambda_1^*t}-e^{\frac{v p_v}{\mu }t}\\
		&=e^{-\sigma_0^*t}\left(\frac{A_{1,1} t e^{\left(\frac{v p_v}{\mu }+\sigma_0^*\right)t}}{(1+\eta ^2)}+\frac{\frac{A_{1,1}^2 t^2e^{\left(\frac{v p_v}{\mu }+\sigma_0^*\right)t}}{2}-A_{1,1} te^{\left(\frac{v p_v}{\mu }+\sigma_0^*\right)t}+A_{1,2}
			te^{\left(\frac{v p_v}{\mu }+\sigma_0^*\right)t}}{(1+\eta ^2)^2}\right)\\
		&\hspace{0.4cm}+e^{-\sigma_0^*t}C(\frac{1}{1+\eta^2})(t^3+t^2+t)\frac{e^{\left(\frac{v p_v}{\mu }+\sigma_0^*\right)t+\left(\sum_{k=1}^3\frac{\gamma^kA_{1,k}}{(1+\eta^2)^k}-\frac{\gamma^4K_1}{\left(\eta^2+1\right)^4}\right)t}}{(1+\eta^2)^3}\\
		&=O(1)te^{-\sigma_0^*t}\left(\frac{1}{(1+\eta^2)}+\frac{1}{(1+\eta^2)^2}+\frac{1}{(1+\eta^2)^3}\right),\quad \eta\rightarrow\infty.
	\end{aligned}
	\eq
	From above analysis, $\frac{1}{t}\left(e^{\lambda_1^*t}-e^{\frac{v p_v}{\mu }t}\right)e^{\sigma_0^*t}$ is analytic in the region $|Im(\eta)|<\sigma_0$ and, according to \eqref{B24}, satisfies the requirement on asymptotic property in Lemma \ref{L4.3}. Therefore, we apply Lemma \ref{L4.3} to conclude that 
	\[\Lambda_1^*(x,t)=e^{\frac{vp_v}{\mu}t}\delta(x)+O(1)te^{-\sigma_0^*t-\sigma_0|x|}.\]
	When $t\geq1$, we can further use the exponential term to absorb $t$ in \eqref{B24}. Therefore, when $t\geq 1$, we actually have 
	\[\Lambda_1^*(x,t)=e^{\frac{vp_v}{\mu}t}\delta(x)+O(1)e^{-\sigma_0^*t-\sigma_0|x|}, \quad t\geq 1.\]
	Then, the higher oder estimates of $\Lambda_1^*(x,t)$ can be obtained via similar arguments, and we omit the details.\newline 
	
	\noindent $\bullet$ (Step 2: $\Lambda_2^*$ and $\Lambda_3^*$) For simplicity, we only do estimates on $\Lambda_3^*$, and the same argument can be applied to $\Lambda_2^*$ to obtain the desired results. According to \eqref{lambda^*}, we have 
	\[e^{\lambda_3^*t}=e^{-\frac{\eta ^2 \kappa
			\theta _e}{v}t+\frac{p v p_e}{\kappa  \theta _e-\mu }t+\left(\sum_{k=1}^3\frac{A_{3,k}}{(1+\eta^2)^k}-\frac{K_3}{\left(\eta ^2+1\right)^4}\right)t}.\]
	
	\noindent $\diamond$ (Case 1) We first consider the case $t\geq 1$. According to Lemma \ref{L4.1}, the real part of $\lambda_3^*$ has a negative upper bound. Actually, we have
	\bq\label{B25}
	e^{\lambda_3^*t}e^{\sigma_0^*t}=e^{\left(\lambda_3^*+\frac{\kappa\theta_e }{2v}(\eta^2+1)\right)t+\sigma_0^*t}e^{-\frac{\kappa\theta_e }{2v}(\eta^2+1)t}=O(1)e^{-\frac{\kappa\theta_e }{2v}(\eta^2+1)t}.
	\eq
	It is obvious that $e^{\lambda_3^*t}e^{\sigma_0^*t}$ is analytic in the region $|Im(\eta)|<\sigma_0$. On the other hand, as $t\geq 1$, the term $e^{-\frac{\kappa\theta_e }{2v}(\eta^2+1)t}$ is analytic, uniformly bounded and has the following estimate for any positive integer $m$ when $\eta$ tends to infinity,
	\bq\label{B26}
	e^{-\frac{\kappa\theta_e }{2v}(\eta^2+1)t}=O(1)\frac{1}{(1+\eta^2)^m},\quad t\geq 1,\quad m\in \bbz^+.
	\eq
	Therefore, we combine \eqref{B25}, \eqref{B26} and apply Lemma \ref{L4.3} to obtain that 
	\[e^{\lambda_3^*t}=O(1)e^{-\sigma_0^*t-\sigma_0|x|}.\]
	
	\noindent $\diamond$ (Case 2) Then, we consider the case $t< 1$. Similar as in estimates of $\Lambda_1^*$, we know $e^{\lambda_3^*t}$ is analytic in the region $|Im(\eta)|<\sigma_0$ around the real axis. Moreover, we can view $e^{\lambda_3^*t}$ as a function of $(1+\eta^2)^{-1}$ and expand $e^{\lambda_3^*t}$ at $\eta=\infty$ to obtain 
	\bq\label{B27}
	\begin{aligned}
		e^{\lambda_3^*t}&=e^{-\frac{\eta ^2 \kappa
				\theta _e}{v}t+\frac{p v p_e}{\kappa  \theta _e-\mu }t}e^{\left(\sum_{k=1}^3\frac{A_{3,k}}{(1+\eta^2)^k}-\frac{K_3}{\left(\eta ^2+1\right)^4}\right)t}\\
		&=e^{-\frac{\eta ^2 \kappa
				\theta _e}{v}t+\frac{p v p_e}{\kappa  \theta _e-\mu }t}\left(1+\frac{A_{3,1} t}{(1+\eta ^2)}+\frac{\frac{A_{3,1}^2 t^2}{2}-A_{3,1} t+A_{3,2}
			t}{(1+\eta ^2)^2}\right)\\
		&\hspace{0.4cm}+e^{-\frac{\eta ^2 \kappa
				\theta _e}{v}t+\frac{p v p_e}{\kappa  \theta _e-\mu }t}\left(C(\frac{1}{1+\eta^2})(t^3+t^2+t)\frac{e^{\left(\sum_{k=1}^3\frac{\gamma^kA_{3,k}}{(1+\eta^2)^k}-\frac{\gamma^4K_3}{\left(\eta^2+1\right)^4}\right)t}}{(1+\eta^2)^3}\right),
	\end{aligned}
	\eq
	where $\gamma$ is the constant chosen in the Lagrangian remainder as before. The leading term of \eqref{B27} is a heat kernel with an exponential coefficient $e^{\frac{p v p_e}{\kappa  \theta _e-\mu }t}$. As $t<1$,  we know the exponential coefficient is uniformly bounded, therefore we can apply the same argument as \eqref{B24} to conclude that 
	\[e^{\lambda_3^*t}-e^{-\frac{\eta ^2 \kappa
			\theta _e}{v}t+\frac{p v p_e}{\kappa  \theta _e-\mu }t}=O(1)\frac{t}{1+\eta^2}.\] 
	Then we apply Lemma \ref{L4.3} to have the zeroth order estimate of $\Lambda_3^*$ as follows,
	\[\Lambda_3^*=\frac{e^{\beta^*_3t}}{\sqrt{4\pi \alpha^*_3t}}e^{-\frac{x^2}{4\alpha^*_3t}}+O(1)te^{-\sigma_0|x|}=\frac{e^{\beta^*_3t}}{\sqrt{4\pi \alpha^*_3t}}e^{-\frac{x^2}{4\alpha^*_3t}}+O(1)te^{-\sigma_0^*t-\sigma_0|x|},\quad 0<t<1,\]
	where the last equality is due to the fact $t<1$. The first order derivative $\partial_x\Lambda_3^*$ can be obtained via similar argument. For second order derivative $\partial^2_x\Lambda_3^*$, the estimate would be a little bit different since the appearance of a new heat kernel. More precisely, we can apply similar argument as above to have 
	\[\eta^2e^{\lambda_3^*t}-\eta^2e^{-\frac{\eta ^2 \kappa
			\theta _e}{v}t+\frac{p v p_e}{\kappa  \theta _e-\mu }t}-A_{3,1}te^{-\frac{\eta ^2 \kappa
			\theta _e}{v}t+\frac{p v p_e}{\kappa  \theta _e-\mu }t}=O(1)\frac{t}{1+\eta^2}.\] 
	Then we can apply Lemma \ref{L4.3} and use the fact $t<1$ again to obtain that 
	\begin{align*}
		\partial^2_x\Lambda_3^*&=\partial^2_x\left[\frac{e^{\beta^*_3t}}{\sqrt{4\pi \alpha^*_3t}}e^{-\frac{x^2}{4\alpha^*_3t}}\right]+A_{3,1}\frac{e^{\beta^*_3t}\sqrt{t}}{\sqrt{4\pi \alpha^*_3}}e^{-\frac{x^2}{4\alpha^*_3t}}+O(1)te^{-\sigma_0|x|}\\
		&=\partial^2_x\left[\frac{e^{\beta^*_3t}}{\sqrt{4\pi \alpha^*_3t}}e^{-\frac{x^2}{4\alpha^*_3t}}\right]+tA_{3,1}\frac{e^{\beta^*_3t}}{\sqrt{4\pi \alpha^*_3t}}e^{-\frac{x^2}{4\alpha^*_3t}}+O(1)te^{-\sigma_0^*t-\sigma_0|x|},\quad 0<t<1.
	\end{align*}
	Finally, the estimate of the third order derivative $\partial^3_x\Lambda_3^*$ can be obtained similarly as $\partial^2_x\Lambda_3^*$.
\end{proof}

\subsubsection{Pointwise estimate of $\bbg^*(x,t;\bar{U})$}
We already define the Fourier transform of the singular part $\bbg^*(x,t;\bar{U})$ in \eqref{singular-G}. In previous discussion, we have obtained the point-wise estimates of the inverse Fourier transform of each component in $\hat{\bbg}^*(\eta,t;\bar{U})$, and now we will combine these estimates together to construct the estimates of $\bbg^*(x,t;\bar{U})$. 
\begin{theorem}\label{T4.1}
	Let $\star_x$ denotes the convolution with respect to $x$. The singular part of the Green's function defined in \eqref{singular-G} has three parts. For the first part $\bbg^{*,1}(x,t)$, we have the following pointwise estimates,
	\bq\label{G^{*,1}}
	\left\{\begin{aligned}
		&\bbg^{*,1}(x,t)=e^{\frac{vp_v}{\mu}t}\delta(x)M_1^{*,0}+O(1)e^{-\sigma_0^*t-\sigma_0|x|},\\
		&\partial_x\bbg^{*,1}(x,t)=e^{\frac{vp_v}{\mu}t}\left(\frac{d}{dx}\delta(x)M_1^{*,0}-\delta(x)M_1^{*,1}\right)+O(1)e^{-\sigma_0^*t-\sigma_0|x|},\\
		&\partial^2_x\bbg^{*,1}(x,t)=e^{\frac{vp_v}{\mu}t}\left(\frac{d^2}{dx^2}\delta(x)M_1^{*,0}-\frac{d}{dx}\delta(x)M_1^{*,1}+\delta(x)\left(M_1^{*,2}-A_{1,1}tM_1^{*,0}\right)\right)+O(1)e^{-\sigma_0^*t-\sigma_0|x|}.
	\end{aligned}
	\right.
	\eq
	For the other two parts $\bbg^{*,2}(x,t)$ and $\bbg^{*,3}(x,t)$, when $t\geq 1$, we have the estimate below,
	\bq\label{B29}
	\partial_x^k\bbg^{*,j}(x,t)=O(1)e^{-\sigma_0^*t-\sigma_0|x|},\quad k=0,1,2,\quad j=2,3, \quad t\geq 1.
	\eq
	While when $0<t<1$, we have the following estimates for $\bbg^{*,2}(x,t)$ and $\bbg^{*,3}(x,t)$,
	\bq\label{B30}
	\left\{\begin{aligned}
		&\bbg^{*,j}(x,t)&&=&&O(1)e^{-\sigma_0^*t-\sigma_0|x|}
		+\frac{e^{\beta^*_jt}}{\sqrt{4\pi \alpha^*_jt}}e^{-\frac{x^2}{4\alpha^*_jt}}M_j^{*,0},\\
		&\partial_x\bbg^{*,j}(x,t)&&=&&O(1)e^{-\sigma_0^*t-\sigma_0|x|}
		+\partial_x\left[\frac{e^{\beta^*_jt}}{\sqrt{4\pi \alpha^*_jt}}e^{-\frac{x^2}{4\alpha^*_jt}}\right]M_j^{*,0}-\frac{e^{\beta^*_jt}}{\sqrt{4\pi \alpha^*_jt}}e^{-\frac{x^2}{4\alpha^*_jt}}M_j^{*,1},\\
		&\partial^2_x\bbg^{*,j}(x,t)&&=&&O(1)e^{-\sigma_0^*t-\sigma_0|x|}
		+\partial^2_x\left[\frac{e^{\beta^*_jt}}{\sqrt{4\pi \alpha^*_jt}}e^{-\frac{x^2}{4\alpha^*_jt}}\right]M_j^{*,0}-\partial_x\left[\frac{e^{\beta^*_jt}}{\sqrt{4\pi \alpha^*_jt}}e^{-\frac{x^2}{4\alpha^*_jt}}\right]M_j^{*,1}\\
		&&&&&+\frac{e^{\beta^*_jt}}{\sqrt{4\pi \alpha^*_jt}}e^{-\frac{x^2}{4\alpha^*_jt}}M_j^{*,2},\\
		&\partial^3_x\bbg^{*,j}(x,t)&&=&&O(1)e^{-\sigma_0^*t-\sigma_0|x|}+\partial^3_x\left[\frac{e^{\beta^*_jt}}{\sqrt{4\pi \alpha^*_jt}}e^{-\frac{x^2}{4\alpha^*_jt}}\right]M_j^{*,0}-\partial^2_x\left[\frac{e^{\beta^*_jt}}{\sqrt{4\pi \alpha^*_jt}}e^{-\frac{x^2}{4\alpha^*_jt}}\right]M_j^{*,1}\\
		&&&&& +\partial_x\left[\frac{e^{\beta^*_jt}}{\sqrt{4\pi \alpha^*_jt}}e^{-\frac{x^2}{4\alpha^*_jt}}\right]M_j^{*,2}+\frac{e^{\beta^*_jt}}{\sqrt{4\pi \alpha^*_jt}}e^{-\frac{x^2}{4\alpha^*_jt}}M_j^{*,3}
	\end{aligned}
	\right.
	\eq

\end{theorem}

\begin{proof}
	We will prove the estimates in two steps.
	
	\noindent $\bullet$ (Step 1: $\bbg^{*,1}(x,t)$) For the part $\bbg^{*,1}(x,t)$, we first apply Lemma \ref{L4.4} and Lemma \ref{L4.5} to provide the zero order estimate as follows, 
	\begin{align*} 
		&\bbg^{*,1}(x,t)=\mathcal{F}^{-1}\left[e^{\lambda_1^*t}\hat{\bbm}_1^*\right]=\Lambda_1^*(x,t)\star_x\bbm_1^*(x)\\
		&=\left(e^{\frac{vp_v}{\mu}t}\delta(x)+O(1)e^{-\sigma_0^*t-\sigma_0|x|}\right)\star_x\left(\delta(x)M_1^{*,0}+O(1)e^{-\sigma_0|x|}\right)\\
		&=e^{\frac{vp_v}{\mu}t}\delta(x)M_1^{*,0}+O(1)e^{-\sigma_0^*t-\sigma_0|x|},
	\end{align*}
	Then for the first order estimate, we apply Lemma \ref{L4.4} and Lemma \ref{L4.5} again to have 
	\begin{align*}
		&\partial_x\bbg^{*,1}(x,t)\\
		&=\partial_x\Lambda_1^*(x,t)\star_x\bbm_1^*(x)\\
		&=\left(e^{\frac{vp_v}{\mu}t}\frac{d}{dx}\delta(x)+O(1)e^{-\sigma_0^*t-\sigma_0|x|}\right)\star_x\bbm_1^*(x)\\
		&=e^{\frac{vp_v}{\mu}t}\frac{d}{dx}\bbm_1^*(x)+O(1)e^{-\sigma_0^*t-\sigma_0|x|}\\
		&=e^{\frac{vp_v}{\mu}t}\left(\frac{d}{dx}\delta(x)M_1^{*,0}-\delta(x)M_1^{*,1}+O(1)e^{-\sigma_0|x|}\right)+O(1)e^{-\sigma_0^*t-\sigma_0|x|}\\
		&=e^{\frac{vp_v}{\mu}t}\left(\frac{d}{dx}\delta(x)M_1^{*,0}-\delta(x)M_1^{*,1}\right)+O(1)e^{-\sigma_0^*t-\sigma_0|x|}.
	\end{align*}
	The second order estimate can be similarly obtained.\newline
	
	\noindent $\bullet$ (Step 2: $\bbg^{*,2}(x,t)$ and $\bbg^{*,3}(x,t)$) For the other two terms, we will construct the estimates in time intervals $t\geq 1$ and $0<t<1$ respectively. When $t\geq 1$, according to Lemma \ref{L4.4} and Lemma \ref{L4.5}, the following estimate holds for $k=0,1,2$ and $j=2,3$,
	\begin{align*}
		&\partial^k_x\bbg^{*,j}(x,t)\\
		&=\partial^k_x\Lambda_j^*(x,t)\star_x\bbm_j^*(x)\\
		&=\left(O(1)e^{-\sigma_0^*t-\sigma_0|x|}\right)\star_x\left(\delta(x)M_j^{*,0}+O(1)e^{-\sigma_0|x|}\right)\\
		&=O(1)e^{-\sigma_0^*t-\sigma_0|x|}.
	\end{align*}
	Then, when $0<t<1$, for the zeroth order estimates, we apply Lemma \ref{L4.4} and Lemma \ref{L4.5} to obtain for $j=2,3$ that  
	\begin{align*}
		&\bbg^{*,j}(x,t)\\
		&=\Lambda_j^*(x,t)\star_x\bbm_j^*(x)\\
		&=\left(O(1)e^{-\sigma_0^*t-\sigma_0|x|}
		+\frac{e^{\beta^*_jt}}{\sqrt{4\pi \alpha^*_jt}}e^{-\frac{x^2}{4\alpha^*_jt}}\right)\star_x\left(\delta(x)M_j^{*,0}+O(1)e^{-\sigma_0|x|}\right)\\
		&=\frac{e^{\beta^*_jt}}{\sqrt{4\pi \alpha^*_jt}}e^{-\frac{x^2}{4\alpha^*_jt}}M_j^{*,0}+O(1)e^{-\sigma_0^*t-\sigma_0|x|}.
	\end{align*}
	Similarly, when $0<t<1$ and $j=2,3$, we have the first order estimates as follows,
	\begin{align*}
		&\partial_x\bbg^{*,j}(x,t)\\
		&=\partial_x\Lambda_j^*(x,t)\star_x\bbm_j^*(x)\\
		&=\left(O(1)e^{-\sigma_0^*t-\sigma_0|x|}
		+\partial_x\left[\frac{e^{\beta^*_jt}}{\sqrt{4\pi \alpha^*_jt}}e^{-\frac{x^2}{4\alpha^*_jt}}\right]\right)\star_x\bbm_j^*(x)\\
		&=\frac{e^{\beta^*_jt}}{\sqrt{4\pi \alpha^*_jt}}e^{-\frac{x^2}{4\alpha^*_jt}}\star_x\left(\frac{d}{dx}\bbm_j^{*}\right)+O(1)e^{-\sigma_0^*t-\sigma_0|x|}\\
		&=\frac{e^{\beta^*_jt}}{\sqrt{4\pi \alpha^*_jt}}e^{-\frac{x^2}{4\alpha^*_jt}}\star_x\left(\frac{d}{dx}\delta(x)M_j^{*,0}-\delta(x)M_j^{*,1}+O(1)e^{-\sigma_0|x|}\right)+O(1)e^{-\sigma_0^*t-\sigma_0|x|}\\
		&=\partial_x\left[\frac{e^{\beta^*_jt}}{\sqrt{4\pi \alpha^*_jt}}e^{-\frac{x^2}{4\alpha^*_jt}}\right]M_j^{*,0}-\frac{e^{\beta^*_jt}}{\sqrt{4\pi \alpha^*_jt}}e^{-\frac{x^2}{4\alpha^*_jt}}M_j^{*,1}+O(1)e^{-\sigma_0^*t-\sigma_0|x|}.
	\end{align*}
	The higher order estimate can be obtained in a similar way, thus we omit the details and finish the proof.
\end{proof}
Then, we will estimate the time derivative of the singular part of Green's function. As there is a diffusion term in the linearized equation \eqref{A2}, the first order time derivative will generate a second order space derivative. The details will be shown in the following lemma.
\begin{theorem}\label{T4.2}
	Let $j=2,3$, the time derivatives and the mixed derivatives of the singular part have the following estimates,
	\begin{align*}
		&\partial_t\bbg^{*,1}(x,t)=\frac{v p_v}{\mu }e^{\frac{vp_v}{\mu}t}\delta(x)M_1^{*,0}+O(1)e^{-\sigma_0^*t-\sigma_0|x|},\\
		&\partial_{xt}\bbg^{*,1}(x,t)=\frac{v p_v}{\mu }e^{\frac{vp_v}{\mu}t}\frac{d}{dx}\delta(x)M_1^{*,0}-\frac{v p_v}{\mu }e^{\frac{vp_v}{\mu}t}\delta(x)M_1^{*,1}+O(1)e^{-\sigma_0^*t-\sigma_0|x|},\\
		&\partial_{t}\bbg^{*,j}(x,t)=O(1)e^{-\sigma_0^*t-\sigma_0|x|}
		+\left\{\begin{aligned}
			&\alpha^*_j\partial^2_x\left[\frac{e^{\beta^*_jt}}{\sqrt{4\pi \alpha^*_jt}}e^{-\frac{x^2}{4\alpha^*_jt}}\right]M_j^{*,0}-\alpha^*_j\partial_x\left[\frac{e^{\beta^*_jt}}{\sqrt{4\pi \alpha^*_jt}}e^{-\frac{x^2}{4\alpha^*_jt}}\right]M_j^{*,1}\\
			&+\frac{e^{\beta^*_jt}}{\sqrt{4\pi \alpha^*_jt}}e^{-\frac{x^2}{4\alpha^*_jt}}\left(\alpha^*_jM_j^{*,2}+\beta^*_jM_j^{*,0}\right),&\quad 0<t<1,\\
			&0,&\quad t\geq 1,
		\end{aligned}\right.\\
		&\partial_{xt}\bbg^{*,j}(x,t)=O(1)e^{-\sigma_0^*t-\sigma_0|x|}
		+\left\{\begin{aligned}
			&\alpha^*_j\partial^3_x\left[\frac{e^{\beta^*_jt}}{\sqrt{4\pi \alpha^*_jt}}e^{-\frac{x^2}{4\alpha^*_jt}}\right]M_j^{*,0}-\alpha^*_j\partial^2_x\left[\frac{e^{\beta^*_jt}}{\sqrt{4\pi \alpha^*_jt}}e^{-\frac{x^2}{4\alpha^*_jt}}\right]M_j^{*,1}\\
			& +\partial_x\left[\frac{e^{\beta^*_jt}}{\sqrt{4\pi \alpha^*_jt}}e^{-\frac{x^2}{4\alpha^*_jt}}\right]\left(\alpha^*_jM_j^{*,2}+\beta^*_jM_j^{*,0}\right)\\
			&+\frac{e^{\beta^*_jt}}{\sqrt{4\pi \alpha^*_jt}}e^{-\frac{x^2}{4\alpha^*_jt}}\left(\alpha^*_jM_j^{*,3}-\beta^*_jM_j^{*,1}\right),&\quad 0<t<1,\\
			&0,&\quad t\geq 1.
		\end{aligned}\right.
	\end{align*}
\end{theorem}

\begin{proof}
	For time derivative, we will provide the proof for $\partial_t\bbg^{*,1}$; while for mixed derivative, we will give the proof for $\partial_{xt}\bbg^{*,j}$ with $j=2,3$. The other estimates can be constructed similarly. First, for $\partial_t\bbg^{*,1}$, the time derivative will be preserved under Fourier transform, and thus we have the following formula,
	\begin{equation*}
		\partial_t\hat{\bbg}^{*,1}(\eta,t)=\lambda_1^*e^{\lambda_1^*t}\hat{\bbm}_1^*=\left(\frac{v p_v}{\mu }+\sum_{k=1}^3\frac{A_{1,k}}{(1+\eta^2)^k}-\frac{K_1}{\left(\eta ^2+1\right)^4}\right)e^{\lambda_1^*t}\hat{\bbm}_1^*.
	\end{equation*}
	Then we apply Lemma \ref{L4.3} and the structure of singular part in Theorem \ref{T4.1} to obtain,
	\begin{equation}
		\begin{aligned}
			\partial_t\bbg^{*,1}(x,t)&=\mathcal{F}^{-1}\left[\left(\frac{v p_v}{\mu }+\sum_{k=1}^3\frac{A_{1,k}}{(1+\eta^2)^k}-\frac{K_1}{\left(\eta ^2+1\right)^4}\right)e^{\lambda_1^*t}\hat{\bbm}_1^*\right]\\
			&=\left(\frac{v p_v}{\mu }\delta(x)+O(1)e^{-\sigma_0^*t-\sigma_0|x|}\right)\star_x\left(e^{\frac{vp_v}{\mu}t}\delta(x)M_1^{*,0}+O(1)e^{-\sigma_0^*t-\sigma_0|x|}\right)\\
			&=\frac{v p_v}{\mu }e^{\frac{vp_v}{\mu}t}\delta(x)M_1^{*,0}+O(1)e^{-\sigma_0^*t-\sigma_0|x|}.
		\end{aligned}
	\end{equation}
	
	%
	
	\noindent Next, for $\partial_{xt}\bbg^{*,j}$ with $j=2,3$, we apply similar method as before to obtain the following exact formula under Fourier transform,
	%
	\begin{equation*}
		\mathcal{F}\left[\partial_{xt}\bbg^{*,j}(x,t)\right]=i\eta\lambda_j^*e^{\lambda_j^*t}\hat{\bbm}_j^*=\left(-i\alpha^*_j\eta^3+i\beta^*_j\eta+\sum_{k=1}^3\frac{A_{j,k}i\eta}{(1+\eta^2)^k}-\frac{K_ji\eta}{\left(\eta ^2+1\right)^4}\right)e^{\lambda_j^*t}\hat{\bbm}_j^*.
	\end{equation*}
	Then we apply inverse Fourier transform, use Lemma \ref{L4.3} and Theorem \ref{T4.1} to obtain the following estimates,
	\begin{align*}
		\partial_{xt}\bbg^{*,j}(x,t)&=\alpha^*_j\partial^3_x\bbg^{*,j}+\beta^*_j\partial_x\bbg^{*,j}+O(1)e^{-\sigma_0^*t-\sigma_0|x|}\\
		&=O(1)e^{-\sigma_0^*t-\sigma_0|x|}
		+\left\{\begin{aligned}
			&\alpha^*_j\partial^3_x\left[\frac{e^{\beta^*_jt}}{\sqrt{4\pi \alpha^*_jt}}e^{-\frac{x^2}{4\alpha^*_jt}}\right]M_j^{*,0}-\alpha^*_j\partial^2_x\left[\frac{e^{\beta^*_jt}}{\sqrt{4\pi \alpha^*_jt}}e^{-\frac{x^2}{4\alpha^*_jt}}\right]M_j^{*,1}\\
			& +\partial_x\left[\frac{e^{\beta^*_jt}}{\sqrt{4\pi \alpha^*_jt}}e^{-\frac{x^2}{4\alpha^*_jt}}\right]\left(\alpha^*_jM_j^{*,2}+\beta^*_jM_j^{*,0}\right)\\
			&+\frac{e^{\beta^*_jt}}{\sqrt{4\pi \alpha^*_jt}}e^{-\frac{x^2}{4\alpha^*_jt}}\left(\alpha^*_jM_j^{*,3}-\beta^*_jM_j^{*,1}\right),&\quad 0<t<1,\\
			&0,&\quad t\geq 1.
		\end{aligned}\right.
	\end{align*}
\end{proof}

\subsection{Regular parts of Green's function}

We next study the point-wise estimate of the regular parts, we define the regular part $\bbg^\dag$ to be the difference between the full Green's function $\bbg$ and the singular part $\bbg^*$, i.e.,
\[\bbg^\dag=\bbg-\bbg^*.\]
According to \eqref{B12} and \eqref{singular-G}, the Fourier transform of the regular part $\bbg^\dag$ has a natural representation as follows,
\bq\label{B34}
\hat{\bbg}^\dag(\eta,t;\bar{U})=\sum_{j=1}^3e^{\lambda_jt}\hat{\bbm}_j-\sum_{j=1}^3e^{\lambda^*_jt}\hat{\bbm}^*_j.
\eq
From \eqref{B34} and Remark \ref{R4.1}, we immediately see the vanishment of $\bbg^\dag$ at initial time, i.e., $\bbg^\dag(x,0)=\hat{\bbg}^\dag(\eta,0)=0$. However, it is not convenient to invert \eqref{B34} to construct the point-wise estimate of $\bbg^\dag$ since we do not have the full information of $\lambda_j$ and $\bbm_j$. Instead, we have the linear differential equations in \eqref{B1}, which provides the equation of $\bbg^\dag$ as below,
\bq\label{B35}
\left\{\begin{aligned}
	&\bbg^\dag_t+F'(\bar{U})\bbg^\dag_x-B(\bar{U})\bbg^\dag_{xx}=-\left(\bbg^*_t+F'(\bar{U})\bbg^*_x-B(\bar{U})\bbg^*_{xx}\right)=:-\Psi(x,t),\\
	&\bbg^\dag(x,0)=0,\quad \Psi(x,t)=\sum_{j=1}^3\Psi_j(x,t),\\
	&\Psi_j(x,t)=\bbg^{*,j}_t+F'(\bar{U})\bbg^{*,j}_x-B(\bar{U})\bbg^{*,j}_{xx}.
\end{aligned}
\right.
\eq
We will see in the next lemma that, the truncation error $\Psi_j$ owns high order regularity and fast space-time decay structure.

\begin{lemma}\label{LB8}
	Let  $\sigma_0$ and $\sigma_0^*$ be positive constants defined in Lemma \ref{L4.1}. Then the Fourier transform $\hat{\Psi}$ is analytic in the region $|Im(\eta)|<\sigma_0$ for $j=1,2,3$. Moreover, when $\eta$ tends to infinity, $\hat{\Psi}$ has the following asymptotic property,
	\[|\hat{\Psi}(\eta,t)|=O(1)e^{-\sigma^*_0t}(1+|\eta|)^{-6},\quad |\eta|\rightarrow+\infty,\quad |Im(\eta)|<\sigma_0.\]
	Consequently, let $\|\cdot\|_{H^4(\bbr)}$ be the fourth order Sobolev norm, then we have
	\[\|\Psi(x,t)\|_{H^4(\bbr)}=O(1)e^{-\sigma^*_0t},\quad \sum_{k=0}^4|\partial_x^k\Psi(x,t)|=O(1)e^{-\sigma_0^*t-\sigma_0|x|}.\]
\end{lemma}
\begin{proof}
	We start with the Fourier transform $\hat{\Psi}_j$.  According to \eqref{singular-G}  and \eqref{B35}, $\hat{\Psi}_j$ satisfies the following equation,
	\bq\label{B36}
	\hat{\Psi}_j(\eta,t)=\hat{\bbg}^{*,j}_t+i\eta F'(\bar{U})\hat{\bbg}^{*,j}+\eta^2B(\bar{U})\hat{\bbg}^{*,j}=\left(\lambda_j^*+i\eta F'(\bar{U})+\eta^2B(\bar{U})\right)e^{\lambda_j^*t}\hat{\bbm}_j^*.
	\eq
	From Lemma \ref{L4.1}, we know $\lambda_j^*$ are analytic in the region $|Im(\eta)|<\sigma_0$ for $j=1,2,3$. Then the above formula immediately implies that $\hat{\Psi}_j$, and thus $\hat{\Psi}$, are analytic in the region $|Im(\eta)|<\sigma_0$. Moreover, if we further have $|\eta|<C$ for some positive constant $C$, then we can use the spectral gap of $\lambda_j^*$ in Lemma \ref{L4.1} to obtain that 
	\bq\label{B37}
	|\hat{\Psi}(\eta,t)|=O(1)e^{-\sigma^*_0t}.
	\eq
	Next, we will show the estimates of $\hat{\Psi}(\eta,t)$ when $\eta$ tends to infinity. In fact, from the equation of full Green's function $\bbg$ in \eqref{B2} and the representation of $\bbg$ in terms of $\bbg_j$ in \eqref{B12}, we have 
	\bq\label{B38}
	0=\sum_{j=1}^3\left(\lambda_j+i\eta F'(\bar{U})+\eta^2B(\bar{U})\right)e^{\lambda_jt}\hat{\bbm}_j.
	\eq
	Combining \eqref{B36} and \eqref{B38}, we obtain that 
	\bq\label{B39}
	\begin{aligned}
		\hat{\Psi}(\eta,t)=&\sum_{j=1}^3\left(\left(\lambda_j^*+i\eta F'(\bar{U})+\eta^2B(\bar{U})\right)e^{\lambda_j^*t}\hat{\bbm}_j^*-\left(\lambda_j+i\eta F'(\bar{U})+\eta^2B(\bar{U})\right)e^{\lambda_jt}\hat{\bbm}_j\right)\\
		=&\sum_{j=1}^3\left(\lambda_j^*-\lambda_j\right)e^{\lambda_j^*t}\hat{\bbm}_j^*+\sum_{j=1}^3\left(\lambda_j+i\eta F'(\bar{U})+\eta^2B(\bar{U})\right)e^{\lambda^*_jt}\left(1-e^{(\lambda_j-\lambda_j^*)t}\right)\hat{\bbm}^*_j\\
		&+\sum_{j=1}^3\left(\lambda_j+i\eta F'(\bar{U})+\eta^2B(\bar{U})\right)e^{\lambda_jt}\left(\hat{\bbm}^*_j-\hat{\bbm}_j\right).
	\end{aligned}
	\eq
	According to \eqref{B18} in Lemma \ref{L4.1}, we know the approximation from $\lambda_j^*$ to $\lambda_j$ is accurate up to eighth order, therefore we have the following asymptotic estimates when $\eta\rightarrow\infty$ in the region $|Im(\eta)|<\sigma_0$, 
	\bq\label{B40}
	\lambda_j^*=\lambda_j+O(1)\frac{1}{|\eta|^8},\quad e^{\lambda_j^*t}=e^{\lambda_jt}+O(1)\frac{1}{|\eta|^8},\quad \hat{\bbm}_j^*=\hat{\bbm}_j+O(1)\frac{1}{|\eta|^8},\quad Re(\lambda_j)<-\sigma^*_0.
	\eq
	Note in the matrix $\left(\lambda_j+i\eta F'(\bar{U})+\eta^2B(\bar{U})\right)$, the highest order of $\eta$ is two. Then we combine \eqref{B37}, \eqref{B39} and \eqref{B40} to obtain that 
	\begin{align*}
		&|\hat{\Psi}(\eta,t)|=O(1)e^{-\sigma_0^*t},\qquad\quad |\eta|\leq C,\quad |Im(\eta)|<\sigma_0,\\
		&|\hat{\Psi}(\eta,t)|=O(1)e^{-\sigma_0^*t}\frac{1}{|\eta|^6},\quad |\eta|\rightarrow+\infty,\quad |Im(\eta)|<\sigma_0.
	\end{align*}
	Finally, we combine above estimates with Lemma \ref{L4.3} to conclude the estimates in space time domain,
	\[\|\Psi(x,t)\|_{H^4(\bbr)}=O(1)e^{-\sigma^*_0t},\quad \sum_{k=0}^4|\partial_x^k\Psi(x,t)|=O(1)e^{-\sigma_0^*t-\sigma_0|x|}.\]
\end{proof}
According to Lemma \ref{LB8}, the source term in \eqref{B35} is regular, moreover the initial data is zero in \eqref{B35}, therefore there is no singularity in the regular part $\bbg^{\dag}$. In the following, we will split the estimates of $\bbg^\dag$ into two parts. We first use energy estimate to treat  $\bbg^\dag$ in initial layer and the region outside wave cone, then we use the spectral information, the long wave and short wave decomposition to construct  $\bbg^\dag$ inside the finite Mach region. 

\subsubsection{Weighted energy estimates}
In this part, we will introduce weighted energy estimate, which will give us the structure of $\bbg^\dag(x,t;\bar{U})$ in the region $0<t<1$ and $|x|>2Mt$ for some positive constants $M$ related to the Mach number. First of all, in order to make sure the matrix $B$ in \eqref{A2} is positive semi-definite, we need the following small perturbation assumption,
\bq\label{assumption on fix state}
\bar{U}=\left(\bar{v},\bar{u},\bar{E}\right),\quad \|\bar{U}-(1,0,1)\|<\varepsilon,\quad 0<\varepsilon\ll 1.
\eq 
\begin{lemma}\label{LB9}
	Under the assumption \eqref{assumption on fix state}, there exists a positive constant $\sigma$ such that the following estimates hold when $0<t<1$,
	\[\|e^{\sigma |x|}\bbg^\dag(x,t)\|_{H^4(\bbr)}\leq O(1)t,\quad \sum_{k=0}^3|\partial_x^k\bbg^\dag(x,t)|\leq O(1)te^{-\sigma |x|}.\]
\end{lemma}
\begin{proof}
	According to constitutive relation \eqref{constitution} and the assumption \eqref{assumption on fix state}, we immediately obtain that $p-p_eu^2>0$ at the fixed state $\bar{U}$. Therefore, we have the following matrix $Q$ is positive definite,
	\bq\label{Q}
	Q:=\begin{pmatrix}
		-p_vp_e^{-1}\left(p-p_eu^2\right) & 0 & 0\\ 0& p_e^{-1}\left(p-p_eu^2\right) &0\\0&0&1
	\end{pmatrix}.
	\eq
	Then we multiply $Q$ to \eqref{B35} to obtain an equivalent equation as follows,
	\bq\label{B43}
	Q\bbg^\dag_t+QF'(\bar{U})\bbg^\dag_x-QB(\bar{U})\bbg^\dag_{xx}=-Q\Psi(x,t),
	\eq
	where the coefficients have the following decompositions,
	\begin{align*}
		&QF'(\bar{U})=S_1+P_1,\quad S_1=\begin{pmatrix}0 & -p_vp_e^{-1}\left(p-p_eu^2\right) & 0\\  -p_vp_e^{-1}\left(p-p_eu^2\right)& -u\left(p-p_eu^2\right)&p-p_eu^2\\ 0&p-p_eu^2&p_eu\end{pmatrix},\quad P_1=\begin{pmatrix}0&0&0\\0&0&0\\p_vu&0&0\end{pmatrix},\\
		&QB(\bar{U})=S_2+P_2,,\quad S_2=\begin{pmatrix}0 &0 &0\\0& \frac{\mu\left(p-p_eu^2\right)}{vp_e} & 0\\0 & 0 & \frac{\kappa\theta_e}{v}\end{pmatrix},\quad P_2=\begin{pmatrix}0&0&0\\0&0&0\\0&\left(\frac{\mu }{v}-\frac{\kappa\theta_e}{v}\right)u&0\end{pmatrix},
	\end{align*}
	where all the state variables take values at $\bar{U}$ in \eqref{B43}. In \eqref{B43}, $S_1$ and $P_1$ ($S_2$ and $P_2$) are the symmetric part and non-symmetric part of the matrix $QF'(\bar{U})$ (matrix $QB(\bar{U})$) respectively. Then we choose a positive constant $\sigma$ such that $0<\sigma<\sigma_0$ (given in Lemma \ref{L4.1}), multiply $e^{2\sigma x}(\bbg^\dag)^\mathsf{T}$ to \eqref{B43}, take the trace norm and integrate with respect to $x$ to have 
	\bq\label{B44}
	\begin{aligned}
		&\int_\bbr tr\left(e^{2\sigma x}(\bbg^\dag)^\mathsf{T} Q\bbg^\dag_t\right)dx+\int_\bbr tr\left(e^{2\sigma x}(\bbg^\dag)^\mathsf{T}(S_1+P_1)\bbg^\dag_x\right)dx-\int_\bbr tr\left(e^{2\sigma x}(\bbg^\dag)^\mathsf{T}(S_2+P_2)\bbg^\dag_{xx}\right)dx\\
		&=\int_\bbr tr\left(-e^{2\sigma x}(\bbg^\dag)^\mathsf{T}Q\Psi(x,t)\right)dx,
	\end{aligned}
	\eq
	where $Q^\mathsf{T}$ is the transpose of $Q$ and $tr(Q)$ is the trace of the matrix $Q$. Now as $S_1$ is symmetric with all elements being constants, we have the following relation 
	\[tr\left((\bbg^\dag)^\mathsf{T}S_1\bbg^\dag_x\right)=\frac{1}{2}tr\left((\bbg^\dag)^\mathsf{T}S_1\bbg^\dag\right)_x,\] 
	and similar identity holds for $S_2$. Therefore, we substitute above relation into \eqref{B44} and apply integration by parts to conclude that 
	\bq\label{B45}
	\begin{aligned}
		&\frac{1}{2}\frac{d}{dt}\int_\bbr e^{2\sigma x}tr\left((\bbg^\dag)^\mathsf{T} Q\bbg^\dag\right)dx\\
		&-\int_\bbr\sigma e^{2\sigma x}tr\left((\bbg^\dag)^\mathsf{T}(S_1)\bbg^\dag\right)dx-\int_\bbr e^{2\sigma x}tr\left((\bbg^\dag)_x^\mathsf{T}(P_1)\bbg^\dag\right)dx-\int_\bbr 2\sigma e^{2\sigma x}tr\left((\bbg^\dag)^\mathsf{T}(P_1)\bbg^\dag\right)dx\\
		&+\int_\bbr e^{2\sigma x}tr\left((\bbg^\dag)_x^\mathsf{T}B(\bar{U})\bbg^\dag_{x}\right)dx-\int_\bbr 2\sigma^2 e^{2\sigma x} tr\left((\bbg^\dag)^\mathsf{T}(S_2)\bbg^\dag\right)dx+\int_\bbr 2\sigma e^{2\sigma x}tr\left((\bbg^\dag)^\mathsf{T}(P_2)\bbg^\dag_{x}\right)dx\\
		&=\int_\bbr tr\left(-e^{2\sigma x}(\bbg^\dag)^\mathsf{T}Q\Psi(x,t)\right)dx.
	\end{aligned}
	\eq
	It is obvious that $Q$ is positive definite according to \eqref{Q}, thus the trace of $(\bbg^\dag)^\mathsf{T} Q\bbg^\dag_t$ is equivalent to the trace of $(\bbg^\dag)^\mathsf{T}\bbg^\dag_t$. Actually, we can find two positive constants $C_1$, $C_2$, $C_3$ and apply Young's inequality and Cauchy inequality to obtain that 
	\bq\label{B46}
	\begin{aligned}
		&0<C_1tr\left((\bbg^\dag)^\mathsf{T} Q\bbg^\dag\right)\leq tr\left((\bbg^\dag)^\mathsf{T}\bbg^\dag\right)\leq C_2tr\left((\bbg^\dag)^\mathsf{T} Q\bbg^\dag\right),\\
		&0<C_1tr\left((\bbg^\dag)_x^\mathsf{T}B(\bar{U})\bbg^\dag_{x}\right)\leq tr\left((\bbg^\dag)_x^\mathsf{T}\begin{pmatrix}0&0&0\\0&1&0\\0&0&1\end{pmatrix}\bbg^\dag_{x}\right)\leq C_2tr\left((\bbg^\dag)_x^\mathsf{T}B(\bar{U})\bbg^\dag_{x}\right),\\
		&\left|tr\left((\bbg^\dag)^\mathsf{T}(S_1)\bbg^\dag\right)\right|+\left|tr\left((\bbg^\dag)^\mathsf{T}(P_1)\bbg^\dag\right)\right|+\left|
		tr\left((\bbg^\dag)^\mathsf{T}(S_2)\bbg^\dag\right)\right|\leq C_2tr\left((\bbg^\dag)^\mathsf{T}\bbg^\dag\right),\\
		&\left|tr\left((\bbg^\dag)_x^\mathsf{T}(P_1)\bbg^\dag\right)\right|+\left|tr\left((\bbg^\dag)^\mathsf{T}(P_2)\bbg^\dag_{x}\right)\right|\leq \frac{C_2}{C_3}tr\left((\bbg^\dag)_x^\mathsf{T}\begin{pmatrix}0&0&0\\0&1&0\\0&0&1\end{pmatrix}\bbg_x^\dag\right)+C_3C_2tr\left((\bbg^\dag)^\mathsf{T}\bbg^\dag\right),\\
		&\left|tr\left((\bbg^\dag)^\mathsf{T}Q\Psi(x,t)\right)\right|\leq C_2\sqrt{tr\left((\bbg^\dag)^\mathsf{T}\bbg^\dag\right)}\sqrt{tr\left((\Psi(x,t))^\mathsf{T}\Psi(x,t)\right)},
	\end{aligned}
	\eq
	where the second inequality holds because $B(\bar{U})$ is positive semi-definite and the first row and first column vanish in $B(\bar{U})$. Similarly, in the fourth inequality, the first row of matrix $P_1$ vanishes, which implies that $tr\left((\bbg^\dag)_x^\mathsf{T}(P_1)\bbg^\dag\right)$ does not contain the differentiation of the elements in the first row of $\bbg^\dag$, and thus the fourth inequality holds. Now we combine \eqref{B45}, \eqref{B46} and H\"older's inequality to obtain that 
	\bq\label{B47}
	\begin{aligned}
		&\frac{1}{2}\frac{d}{dt}\int_\bbr e^{2\sigma x}tr\left((\bbg^\dag)^\mathsf{T} Q\bbg^\dag\right)dx\\
		&\leq \left(3\sigma C_2^2+2C_2^2\sigma^2+\left(1+2\sigma\right)C_3C_2^2\right)\int_\bbr e^{2\sigma x}tr\left((\bbg^\dag)^\mathsf{T}Q\bbg^\dag\right)dx\\
		&+\left(\frac{C_2}{C_3}\left(1+2\sigma\right)-\frac{1}{C_2}\right)\int_\bbr e^{2\sigma x}tr\left((\bbg^\dag)_x^\mathsf{T}\begin{pmatrix}0&0&0\\0&1&0\\0&0&1\end{pmatrix}\bbg_x^\dag\right)dx\\
		&+ (C_2)^{\frac{3}{2}}\sqrt{\int_\bbr e^{2\sigma x}tr\left((\bbg^\dag)^\mathsf{T} Q\bbg^\dag\right)dx}\sqrt{\int_\bbr e^{2\sigma x}tr\left((\Psi(x,t))^\mathsf{T}\Psi(x,t)\right)dx}.
	\end{aligned}
	\eq
	Now, we let $C_3$ be large enough so that the coefficient of the higher order term in \eqref{B47} becomes negative. Then \eqref{B47} implies for a positive constant $C$ that 
	\begin{align*}
		\frac{d}{dt}\sqrt{\int_\bbr e^{2\sigma x}tr\left((\bbg^\dag)^\mathsf{T} Q\bbg^\dag\right)dx}\leq C\sqrt{\int_\bbr e^{2\sigma x}tr\left((\bbg^\dag)^\mathsf{T}Q\bbg^\dag\right)dx}+ (C_2)^{\frac{3}{2}}\|e^{\sigma x}\Psi(x,t)\|_{L^2(\bbr)}.
	\end{align*}
	According to Lemma \ref{LB8} and the assumption that $\sigma<\sigma_0$, we know that $\|e^{\sigma x}\Psi(x,t)\|_{L^2(\bbr)}$ is integrable. Moreover, as the initial data is zero for $\bbg^\dag$ and we only consider the case $0<t<1$, we can apply Gronwall's inequality to obtain that 
	\[\sqrt{\int_\bbr e^{2\sigma x}tr\left((\bbg^\dag)^\mathsf{T} Q\bbg^\dag\right)dx}\leq \int_0^te^{C(t-s)}(C_2)^{\frac{3}{2}}\|e^{\sigma x}\Psi(x,s)\|_{L^2(\bbr)}ds\leq \int_0^tO(1)ds=O(1)t.\]
	Then we apply \eqref{B46} again to have 
	\[\|e^{\sigma x}\bbg^\dag\|_{L^2(\bbr)}=\sqrt{\int_\bbr e^{2\sigma x}tr\left((\bbg^\dag)^\mathsf{T}\bbg^\dag\right)dx}\leq C_2\sqrt{\int_\bbr e^{2\sigma x}tr\left((\bbg^\dag)^\mathsf{T} Q\bbg^\dag\right)dx}\leq O(1)t.\]
	Next, we consider $-\sigma_0<\sigma<0$, and we can obtain the similar estimate as above. Thus we obtain the desired zeroth order estimate as below,
	\[\|e^{\sigma |x|}\bbg^\dag\|_{L^2(\bbr)}\leq O(1)t,\quad 0<t<1.\]
	Since \eqref{B35} is a linear system, the higher order estimates can be similarly constructed. As $\Psi$ has finite Sobolev norm up to fourth order, we have the estimates as follows,
	\[\|e^{\sigma |x|}\bbg^\dag\|_{H^4(\bbr)}\leq O(1)t,\quad 0<t<1.\]
	Finally, we apply Sobolev inequality and above $L^2$ norm of higher order derivatives to obtain that 
	\[\|\partial^k_x\bbg^\dag\|_{L^\infty(\bbr)}\leq O(1)te^{-\sigma |x|},\quad 0\leq k\leq 3,\quad 0<t<1.\]
\end{proof}
Next, we study the pointwise structure of $\bbg^\dag$ for $t>1$ and the region outside the wave cone. More precisely, we will consider the region $1\leq t$ and $|x|>2Mt$, where $M$ is a sufficient large positive constant related to the sound speed. 

\begin{lemma}\label{LB10}
	Under the assumption \eqref{assumption on fix state}, there exist positive constants $\sigma$ and $M$ such that the following estimates hold when $t\geq 1$,
	\[\|e^{\pm\sigma (x-Mt)}\bbg^\dag(x,t)\|_{H^4(\bbr)}\leq O(1)e^{-\sigma_0^*t},\quad \sum_{k=0}^3|\partial_x^k\bbg^\dag(x,t)|\leq O(1)e^{-\frac{\sigma}{2} |x|-\sigma_0^*t},\quad t\geq 1,\quad |x|>2Mt.\]
\end{lemma}
\begin{proof}
	We will continue using the notations in Lemma \ref{LB9}. We first choose a sufficiently large constant $M$ and a positive constant $\sigma$ such that $0<\sigma<\sigma_0$. Then, we multiply \eqref{B43} by $e^{2\sigma (x-Mt)}(\bbg^\dag)^\mathsf{T}$, take the trace norm and integrate with respect to $x$ to have 
	\bq\label{B48}
	\begin{aligned}
		&\frac{1}{2}\frac{d}{dt}\int_\bbr tr\left(e^{2\sigma (x-Mt)}\left((\bbg^\dag)^\mathsf{T} Q\bbg^\dag\right)\right)dx+2\sigma M\int_\bbr tr\left(e^{2\sigma (x-Mt)}\left((\bbg^\dag)^\mathsf{T} Q\bbg^\dag\right)\right)dx\\
		&-\int_\bbr\sigma e^{2\sigma (x-Mt)}tr\left((\bbg^\dag)^\mathsf{T}(S_1)\bbg^\dag\right)dx-\int_\bbr e^{2\sigma (x-Mt)}tr\left((\bbg^\dag)_x^\mathsf{T}(P_1)\bbg^\dag\right)dx-\int_\bbr 2\sigma e^{2\sigma (x-Mt)}tr\left((\bbg^\dag)^\mathsf{T}(P_1)\bbg^\dag\right)dx\\
		&+\int_\bbr e^{2\sigma (x-Mt)}tr\left((\bbg^\dag)_x^\mathsf{T}B(\bar{U})\bbg^\dag_{x}\right)dx-\int_\bbr 2\sigma^2 e^{2\sigma (x-Mt)} tr\left((\bbg^\dag)^\mathsf{T}(S_2)\bbg^\dag\right)dx+\int_\bbr 2\sigma e^{2\sigma (x-Mt)}tr\left((\bbg^\dag)^\mathsf{T}(P_2)\bbg^\dag_{x}\right)dx\\
		&=\int_\bbr tr\left(-e^{2\sigma (x-Mt)}(\bbg^\dag)^\mathsf{T}Q\Psi(x,t)\right)dx.
	\end{aligned}
	\eq
	Comparing to \eqref{B45}, the only difference is that \eqref{B48} contains one more term with coefficient $M$. Then we apply \eqref{B46} to have the following estimates,
	\bq\label{B49}
	\begin{aligned}
		&\frac{1}{2}\frac{d}{dt}\int_\bbr e^{2\sigma x}tr\left((\bbg^\dag)^\mathsf{T} Q\bbg^\dag\right)dx\\
		&\leq \left(3\sigma C_2^2+2C_2^2\sigma^2+\left(1+2\sigma\right)C_3C_2^2-2\sigma M\right)\int_\bbr e^{2\sigma (x-Mt)}tr\left((\bbg^\dag)^\mathsf{T}Q\bbg^\dag\right)dx\\
		&+\left(\frac{C_2}{C_3}\left(1+2\sigma\right)-\frac{1}{C_2}\right)\int_\bbr e^{2\sigma  (x-Mt)}tr\left((\bbg^\dag)_x^\mathsf{T}\begin{pmatrix}0&0&0\\0&1&0\\0&0&1\end{pmatrix}\bbg_x^\dag\right)dx\\
		&+ (C_2)^{\frac{3}{2}}\sqrt{\int_\bbr e^{2\sigma  (x-Mt)}tr\left((\bbg^\dag)^\mathsf{T} Q\bbg^\dag\right)dx}\sqrt{\int_\bbr e^{2\sigma (x-Mt)}tr\left((\Psi(x,t))^\mathsf{T}\Psi(x,t)\right)dx}.
	\end{aligned}
	\eq
	Therefore, we can let $C_2\ll C_3\ll M$ so that the coefficients of zeroth order term and first order term are both negative. Then we have 
	\[\frac{d}{dt}\sqrt{\int_\bbr e^{2\sigma (x-Mt)}tr\left((\bbg^\dag)^\mathsf{T} Q\bbg^\dag\right)dx}\leq -\sigma_0^*\sqrt{\int_\bbr e^{2\sigma (x-Mt)}tr\left((\bbg^\dag)^\mathsf{T}Q\bbg^\dag\right)dx}+ (C_2)^{\frac{3}{2}}\|e^{\sigma (x-Mt)}\Psi(x,t)\|_{L^2(\bbr)}.
	\] 
	According to Lemma \ref{LB8} and the assumption that $\sigma<\sigma_0$, we know that $\|e^{\sigma x}\Psi(x,t)\|_{L^2(\bbr)}$ is integrable and has time exponential decay. Moreover, the initial data is zero for $\bbg^\dag$, therefore, we apply Gronwall's inequality to obtain that 
	\begin{align*}
		&\|e^{\sigma (x-Mt)}\bbg^\dag\|_{L^2(\bbr)}\\
		&\leq C_2\sqrt{\int_\bbr e^{2\sigma (x-Mt)}tr\left((\bbg^\dag)^\mathsf{T} Q\bbg^\dag\right)dx}\\
		&\leq C_2\int_0^te^{-\sigma_0^*(t-s)}(C_2)^{\frac{3}{2}}\|e^{\sigma (x-Ms)}\Psi(x,s)\|_{L^2(\bbr)}ds\\
		&\leq O(1)\int_0^te^{-\sigma_0^*(t-s)}\|e^{\sigma (x-Ms)}e^{-\sigma_0x-\sigma_0^*s}\|_{L^2(\bbr)}ds\\
		&\leq O(1)\int_0^te^{-\sigma_0^*(t-s)}e^{-\sigma Ms}e^{-\sigma_0^*s}ds\\
		&\leq O(1)e^{-\sigma_0^*t}.
	\end{align*}
	Similarly, we can also get
	\[
	\|e^{\sigma (-x-Mt)}\bbg^\dag\|_{L^2(\bbr)}\leq O(1) e^{-\sigma_0^*t}.
	\]
	Then as in Lemma \ref{LB9}, the higher order estimates can be obtained in a similar way for the linear system \eqref{B43}. 
	Finally for $|x|>2Mt$, we apply Sobolev inequality to obtain the $L^\infty$ estimates as below,
	\[\|\partial_x^k\bbg^\dag(x,t)\|_{L^\infty(\bbr)}\leq O(1)e^{-\sigma(|x|-Mt)-\sigma_0^*t}\leq O(1)e^{-\frac{\sigma |x|}{2}-\sigma_0^*t},\quad t\geq 1,\quad |x|>2Mt,\quad k=1,2,3.\]
\end{proof}

\subsubsection{Long wave and short wave decomposition}
In this part, we will construct the estimates of $\bbg^\dag$ in the region $|x|\leq 2Mt$ when $t\geq 1$. First, for full Green's function $\bbg$, we can decompose it into long wave part $\bbg_L$ and short wave part $\bbg_S$ as follows,
\bq\label{B50}
\bbg=\bbg_L+\bbg_S,\qquad\hat{\bbg}_L(\eta,t):=\left\{\begin{aligned}&\hat{\bbg}(\eta,t),\quad |\eta|<\delta,\\ &0,\quad |\eta|\geq\delta, \end{aligned}\right.\qquad \hat{\bbg}_S(\eta,t):=\left\{\begin{aligned}&0,\quad |\eta|<\delta,\\ &\hat{\bbg}(\eta,t),\quad |\eta|\geq\delta, \end{aligned}\right.
\eq
where $\delta$ is a sufficient small positive constant. Then, we can similarly define the long wave part $\bbg^*_L$ ($\bbg^\dag_L$) and short wave part $\bbg^*_S$ ($\bbg^\dag_S$) for $\bbg^*$ ($\bbg^\dag$). From \eqref{B50}, we immediately have 
\bq\label{B51}
\bbg_L=\bbg^*_L+\bbg^\dag_L,\qquad \bbg_S=\bbg^*_S+\bbg^\dag_S.
\eq
In the following, we will study the long wave part and short wave part respectively.\newline

\noindent $\bullet$ (Case 1: Long wave part $\bbg^\dag_L$) We start from the long wave part of singular part $\bbg^*_L$. Actually, we will provide the estimates for $\bbg$ and $\bbg^*$, and then the desired results follows from \eqref{B51}.
\begin{lemma}\label{LB11}
	The long wave of the singular part $\bbg^*$ has the following estimates inside the wave cone when $t\geq 1$,
	\[\left|\partial^l_t\partial^k_x\bbg_L^*(x,t;\bar{U})\right|\leq O(1)e^{-\sigma_0^*t}\leq O(1)e^{-\frac{\sigma_0^*}{2}\left(t+\frac{|x|}{2M}\right)},\quad t\geq 1,\quad |x|\leq2Mt,\quad l, k\in\bbn.\]
\end{lemma}
\begin{proof}
	the Fourier transform of $\bbg^*_L$ satisfies the formula in \eqref{singular-G}, i.e.,
	\[\hat{\bbg}_L^*(\eta,t;\bar{U})=\sum_{j=1}^3\hat{G}^{*,j}(\eta,t;\bar{U})=\sum_{j=1}^3e^{\lambda_j^*t}\hat{\bbm}_j^*,\quad |\eta|<\delta.\]  
	According to Lemma \ref{L4.1}, $\lambda^*_j$ has spectral gap such that $Re(\lambda^*_j)<-\sigma_0^*$ for $j=1,2,3$. Moreover, $\hat{\bbg}^*_L(\eta,t)$ vanishes when $|\eta|\geq\delta$, and $\hat{\bbm}_j^*=O(1)$ when $|\eta|<\delta$. Therefore, we take inverse Fourier transform of $\hat{\bbg}_L^*(\eta,t;\bar{U})$ to obtain 
	\[|\bbg_L^*(x,t;\bar{U})|\leq O(1)e^{-\sigma_0^*t}\leq O(1)e^{-\frac{\sigma_0^*}{2}\left(t+\frac{|x|}{2M}\right)},\quad t\geq 1,\quad |x|\leq2Mt.\]
	where the last inequality comes from the fact $|x|\leq2Mt$. For higher order estimates $\partial^l_t\partial^k_x\bbg_L^*(x,t;\bar{U})$, we need to multiply $(i\eta)^k$ or $(\lambda_j^*)^l$ to $\hat{\bbg}^*$ and take inverse Fourier transform. Since $|\eta|\leq \delta$ for long wave, the desired higher order estimates follow similar arguments in zeroth order estimates, and we omit the details.
\end{proof}
Next, we need the estimates of the long wave part of the full Green's function $\bbg$. Actually, we have the asymptotic expansion of the eigenvalues $\lambda_j$ near zero in \eqref{evalue-zero-1}, \eqref{evalue-zero-2} and \eqref{evalue-zero-3}. Moreover, if we expand up to second order, then we have a simpler formula as follows,
\bq\label{B52}
\begin{aligned}
	&\lambda_1=i\eta \beta_1-\alpha_1\eta^2+O(1)\eta^3,\quad \lambda_2=i \eta \beta_2-\alpha_2\eta^2+O(1)\eta^3,\quad\lambda_3=i \eta \beta_3-\alpha_3\eta^2+O(1)\eta^3.\\
	& \alpha_1=   \frac{-\kappa  \theta _e p_v}{v \left(pp_e-p_v\right)},\quad \alpha_2=\alpha_3=\frac{ \left(\kappa  p \theta _e p_e+\mu  p p_e-\mu
		p_v\right)}{2 v \left(p p_e-p_v\right)},\quad \beta_1=0,\quad \beta_2=  -\sqrt{p p_e-p_v},\quad \beta_3=  \sqrt{p p_e-p_v},
\end{aligned}
\eq
where $\alpha_j>0$ and $\beta_j$ are real numbers as we discussed before. Then, similar as in Lemma \ref{L4.2}, the expansion of $\lambda_j$ yields the asymptotical expansion of $\hat{M}_j$ at zero. In fact, we have 
\bq\label{B53}
\hat{\bbm}_j=M_j^0+i\eta M_j^1+O(1)\eta^2, \quad j=1,2,3,
\eq
where $M_j^k$ are listed in Appendix \ref{sec:C}. Now, we combine the explicit expression of full Green's function in \eqref{B12}, the expansion of eigenvalue in \eqref{B52} and the expansion of $\hat{\bbm}_j$ in \eqref{B53} to have for $|\eta|<\delta$ that,
\bq\label{B54}
\begin{aligned}
	\hat{\bbg}_L(\eta,t)&=\hat{\bbg}_L^1+\hat{\bbg}_L^2+\hat{\bbg}_L^3\\
	&=e^{-\alpha_1\eta^2t+O(1)\eta^3t}\left(M_1^0+i\eta M_1^1+O(1)\eta^2\right)\\
	&+e^{i \eta  \beta_2 t-\alpha_2 \eta^2t+O(1)\eta^3t}\left(M_2^0+i\eta M_2^1+O(1)\eta^2\right)\\
	&+e^{i \eta \beta_3t-\alpha_3\eta^2t+O(1)\eta^3t}\left(M_3^0+i\eta M_3^1+O(1)\eta^2\right).    
\end{aligned}
\eq

\begin{lemma}\label{LB12}
	Suppose $t\geq 1$ and $|x|\leq 2Mt$. Then there exists a positive constant $C$ such that, the long wave part of the full Green's function has the following estimate for $j=1,2,3$ and $k\geq 0$,
	\begin{align*}
		&\left|\partial_x^k\bbg_L^j(x,t)-\partial_x^k\left(\frac{e^{-\frac{(x+\beta_jt)^2}{4\alpha_jt}}}{2\sqrt{\pi\alpha_jt}}\right)M_j^0-\partial^{k+1}_x\left(\frac{e^{-\frac{(x+\beta_jt)^2}{4\alpha_jt}}}{2\sqrt{\pi\alpha_jt}}\right)M_j^1\right|\\
		&\hspace{6cm}\leq \frac{O(1)e^{-\frac{(x+\beta_jt)^2}{4Ct}}}{t^{\frac{k+2}{2}}}M_j^0+\frac{O(1)e^{-\frac{(x+\beta_jt)^2}{4Ct}}}{t^{\frac{k+3}{2}}}+O(1)e^{-\alpha_j\delta^2(\frac{t}{2}+\frac{|x|}{4M})}.
	\end{align*}
\end{lemma}
\begin{proof}
	According to the discussion at the beginning of Section \ref{sec:4}, we already knew $\lambda_j$ are analytic functions of  $\eta$ when $|\eta|\ll1$. Therefore, in \eqref{B54}, we can expand the exponential part of $\hat{\bbg}_L^1$ and apply the mean value theorem to find $\bar{\eta}$ between $0$ and $\eta$ such that
	\bq\label{B55}
	e^{\lambda_1t}=e^{-\alpha_1\eta^2t}\left(1+e^{O(1)\bar{\eta}^3t}O(1)\eta^3t\right)=e^{-\alpha_1\eta^2t}+O(1)\eta e^{-\bar{\alpha}_1\eta^2t},
	\eq
	where $\bar{\alpha}_1$ is some positive constant such that $0<\bar{\alpha}_1<\alpha_1$. The last equality holds because $|\bar{\eta}|\leq |\eta|<\delta$, and thus we can choose a sufficiently small $\bar{\alpha}_1$ and use the term $e^{-(\alpha_1-\bar{\alpha}_1)\eta^2t}$ to absorb the part $\eta^2te^{O(1)\bar{\eta}^3t}$. Without loss of generality, we may let $\delta$ be sufficiently small and choose $\bar{\alpha}_1=\frac{\alpha_1}{2}$. Now we combine \eqref{B54} and \eqref{B55} to obtain that 
	\bq\label{B56}
	\begin{aligned}
		\hat{\bbg}_L^1(\eta,t)&=\left(e^{-\alpha_1\eta^2t}+O(1)\eta e^{-\frac{\alpha_1}{2}\eta^2t}\right)\left(M_1^0+i\eta M_1^1+O(1)\eta^2\right)\\
		&=\left(e^{-\alpha_1\eta^2t}+O(1)\eta e^{-\frac{\alpha_1}{2}\eta^2t}\right)M_1^0+i\eta e^{-\alpha_1\eta^2t}M_1^1+O(1)\eta^2 \left(e^{-\alpha_1\eta^2t}+e^{-\frac{\alpha_1}{2}\eta^2t}\right),\quad |\eta|<\delta.
	\end{aligned}
	\eq
	We will take inverse Fourier transform of above formula term by term.\newline 
	
	\noindent  $\diamond$ (Step 1: Leading term $e^{-\alpha_1\eta^2t}$) : First, when $|\eta|<\delta$, the leading term $e^{-\alpha_1\eta^2t}$ is exactly the heat kernel. While when $|\eta|>\delta$, the long wave part vanishes due to \eqref{B50}. Therefore, the leading term of $\hat{\bbg}_L^1(\eta,t)$ has inverse Fourier transform as follows, 
	\bq\label{B57}
	\frac{1}{2\pi}\int_{-\delta}^{\delta}e^{i\eta x-\alpha_1\eta^2t}d\eta=\frac{1}{2\pi}e^{-\frac{x^2}{4\alpha_1t}}\int_{-\delta}^{\delta}e^{-\alpha_1t\left(\eta-i\frac{x}{2\alpha_1t}\right)^2}d\eta.
	\eq
	Now, as the term $e^{-\alpha_1t\left(\eta-i\frac{x}{2\alpha_1t}\right)^2}$ is analytic as a function of $\eta$, we can use  contour integral and the change of variable to rewrite the above integration \eqref{B57} as follows,
	\bq\label{B58}
	\begin{aligned}
		&\int_{-\delta}^{\delta}e^{-\alpha_1t\left(\eta-i\frac{x}{2\alpha_1t}\right)^2}d\eta\\
		&=\left(\int_{\delta}^{\delta+i\frac{x}{2\alpha_1t}}+\int_{\delta+i\frac{x}{2\alpha_1t}}^{-\delta+i\frac{x}{2\alpha_1t}}+\int_{-\delta+i\frac{x}{2\alpha_1t}}^{-\delta}\right)e^{-\alpha_1t\left(\eta-i\frac{x}{2\alpha_1t}\right)^2}d\eta\\
		&=\left(-\int_{\delta}^{\delta+i\frac{x}{2\alpha_1t}}+\int^{+\infty+i\frac{x}{2\alpha_1t}}_{-\infty+i\frac{x}{2\alpha_1t}}-\int_{\delta+i\frac{x}{2\alpha_1t}}^{+\infty+i\frac{x}{2\alpha_1t}}-\int^{-\delta+i\frac{x}{2\alpha_1t}}_{-\infty+i\frac{x}{2\alpha_1t}}+\int^{-\delta+i\frac{x}{2\alpha_1t}}_{-\delta}\right)e^{-\alpha_1t\left(\eta-i\frac{x}{2\alpha_1t}\right)^2}d\eta\\
		&=\int_{-\infty}^{+\infty}e^{-\alpha_1t\nu^2}d\nu-\int_{|\nu|>\delta}e^{-\alpha_1t\nu^2}d\nu+i\int^{\frac{x}{2\alpha_1t}}_{0}e^{-\alpha_1t\left(\delta+i\nu\right)^2}d\nu-i\int^{\frac{x}{2\alpha_1t}}_{0}e^{-\alpha_1t\left(\delta-i\nu\right)^2}d\nu.
	\end{aligned}
	\eq
	We can directly calculate the first term, and for the second term, we have a spectral gap with order $\delta$ which provides a exponential decay. More precisely, we have 
	\begin{align*}
		&\int_{-\infty}^{+\infty}e^{-\alpha_1t\nu^2}d\nu=\sqrt{\frac{\pi}{\alpha_1t}},\\
		&\int_{|\nu|>\delta}e^{-\alpha_1t\nu^2}d\nu=e^{-\alpha_1\delta^2t}\int_{|\nu|>\delta}e^{-\alpha_1t(|\nu|-\delta)(|\nu|+\delta)}d\nu\leq e^{-\alpha_1\delta^2t}\int_{|\nu|>\delta}e^{-\alpha_1\delta t(|\nu|-\delta)}d\nu=\frac{2e^{-\alpha_1\delta^2t}}{\alpha_1\delta t}.
	\end{align*}
	For the third and fourth terms in \eqref{B58}, it is obviously they are conjugate and thus we can combine them together to obtain the following estimate,
	\begin{align*}
		&\left|i\int^{\frac{x}{2\alpha_1t}}_{0}e^{-\alpha_1t\left(\delta+i\nu\right)^2}d\nu-i\int^{\frac{x}{2\alpha_1t}}_{0}e^{-\alpha_1t\left(\delta-i\nu\right)^2}d\nu\right|\\
		&=\left|i\int^{\frac{x}{2\alpha_1t}}_{0}e^{-\alpha_1\delta^2t+\alpha_1\nu^2t}\left(e^{-2i\alpha_1\delta \nu t}-e^{2i\alpha_1\delta \nu t}\right)d\nu\right|\leq 2\left|\int^{\frac{x}{2\alpha_1t}}_{0}e^{-\alpha_1\delta^2t+\alpha_1\nu^2t}d\nu\right|\leq \frac{|x|}{\alpha_1t}e^{-\alpha_1\delta^2t}e^{\frac{x^2}{4\alpha_1t}}.
	\end{align*}
	Combining \eqref{B57}, \eqref{B58} and above estimates of \eqref{B58}, we have 
	\begin{align*}
		\left|\frac{1}{2\pi}\int_{-\delta}^{\delta}e^{i\eta x-\alpha_1\eta^2t}d\eta-\frac{1}{2\sqrt{\pi\alpha_1t}}e^{-\frac{x^2}{4\alpha_1t}}\right|\leq \frac{1}{2\pi}e^{-\frac{x^2}{4\alpha_1t}}\left(\frac{2e^{-\alpha_1\delta^2t}}{\alpha_1\delta t}\right)+\frac{|x|}{2\pi\alpha_1t}e^{-\alpha_1\delta^2t}.
	\end{align*}
	As we consider the region $t\geq 1$ and $|x|\leq 2Mt$, we can further use the exponential decay $e^{-\alpha_1\delta^2t}$ to absorb the $|x|$ and $\frac{1}{t}$. Then we obtain 
	\bq\label{B59}
	\left|\frac{1}{2\pi}\int_{-\delta}^{\delta}e^{i\eta x-\alpha_1\eta^2t}d\eta-\frac{1}{2\sqrt{\pi\alpha_1t}}e^{-\frac{x^2}{4\alpha_1t}}\right|\leq  O(1)e^{-\alpha_1\delta^2(\frac{t}{2}+\frac{|x|}{4M})}.
	\eq
	
	\noindent $\diamond$ (Step 2: Higher order terms) The higher order terms are similar as derivatives of the heat kernel. Similarly as before, as the long wave only contains low frequency part, we have to estimate the remainders. Actually, we have 
	\bq\label{B60}
	\begin{aligned}
		&\frac{1}{2\pi}\int_{-\delta}^{\delta}i\eta e^{i\eta x-\alpha_1\eta^2t}d\eta\\
		&=-\frac{x}{4\pi\alpha_1t}e^{-\frac{x^2}{4\alpha_1t}}\int_{-\delta}^{\delta}e^{-\alpha_1t\left(\eta-i\frac{x}{2\alpha_1t}\right)^2}d\eta+\frac{1}{2\pi}e^{-\frac{x^2}{4\alpha_1t}}\int_{-\delta}^{\delta}i\left(\eta-i\frac{x}{2\alpha_1t}\right) e^{-\alpha_1t\left(\eta-i\frac{x}{2\alpha_1t}\right)^2}d\eta.
	\end{aligned}
	\eq
	The estimate of the first term in \eqref{B60} directly follows from the estimates of the leading term in Step 1. For the second term, as it is analytic in $\eta$, we can apply contour integral similarly as before to obtain that 
	\begin{align*}
		&\int_{-\delta}^{\delta}i\left(\eta-i\frac{x}{2\alpha_1t}\right)e^{-\alpha_1t\left(\eta-i\frac{x}{2\alpha_1t}\right)^2}d\eta\\
		&=\left(-\int_{\delta}^{\delta+i\frac{x}{2\alpha_1t}}+\int^{+\infty+i\frac{x}{2\alpha_1t}}_{-\infty+i\frac{x}{2\alpha_1t}}-\int_{\delta+i\frac{x}{2\alpha_1t}}^{+\infty+i\frac{x}{2\alpha_1t}}-\int^{-\delta+i\frac{x}{2\alpha_1t}}_{-\infty+i\frac{x}{2\alpha_1t}}+\int^{-\delta+i\frac{x}{2\alpha_1t}}_{-\delta}\right)i\left(\eta-i\frac{x}{2\alpha_1t}\right)e^{-\alpha_1t\left(\eta-i\frac{x}{2\alpha_1t}\right)^2}d\eta\\
		&=i\int_{-\infty}^{+\infty}\nu e^{-\alpha_1t\nu^2}d\nu-i\int_{|\nu|>\delta}\nu e^{-\alpha_1t\nu^2}d\nu+\int^{\frac{x}{2\alpha_1t}}_{0}\left(i\nu+\delta\right)e^{-\alpha_1t\left(\delta+i\nu\right)^2}d\nu-\int^{\frac{x}{2\alpha_1t}}_{0}\left(i\nu-\delta\right)e^{-\alpha_1t\left(\delta-i\nu\right)^2}d\nu.
	\end{align*}
	Then, the first two terms are zero due the symmetric property, the last two terms can be treated similar as before. In fact, we have 
	\begin{align*}
		&\left|\int^{\frac{x}{2\alpha_1t}}_{0}\left(i\nu+\delta\right)e^{-\alpha_1t\left(\delta+i\nu\right)^2}d\nu-\int^{\frac{x}{2\alpha_1t}}_{0}\left(i\nu-\delta\right)e^{-\alpha_1t\left(\delta-i\nu\right)^2}d\nu\right|\\
		&\leq \left|\int^{\frac{x}{2\alpha_1t}}_{0}i\nu e^{-\alpha_1\delta^2t+\alpha_1\nu^2t}\left(e^{-2i\alpha_1\delta \nu t}-e^{2i\alpha_1\delta \nu t}\right)d\nu\right|+ \left|\int^{\frac{x}{2\alpha_1t}}_{0}\delta e^{-\alpha_1\delta^2t+\alpha_1\nu^2t}\left(e^{-2i\alpha_1\delta \nu t}+e^{2i\alpha_1\delta \nu t}\right)d\nu\right|\\
		&\leq 2\left|\int^{\frac{x}{2\alpha_1t}}_{0}i\nu e^{-\alpha_1\delta^2t+\alpha_1\nu^2t}d\nu\right|+2 \left|\int^{\frac{x}{2\alpha_1t}}_{0}\delta e^{-\alpha_1\delta^2t+\alpha_1\nu^2t}d\nu\right|\\
		&\leq \frac{|x|^2}{(2\alpha_1t)^2}e^{-\alpha_1\delta^2t}e^{\frac{x^2}{4\alpha_1t}}+\frac{\delta |x|}{\alpha_1t}e^{-\alpha_1\delta^2t}e^{\frac{x^2}{4\alpha_1t}}.
	\end{align*}
	Combing \eqref{B59}, \eqref{B60} and above estimates, and similarly as in the end of Step 1, we apply the property $t\geq 1$ and $|x|\leq 2Mt$ to have 
	\bq\label{B61}
	\begin{aligned}
		&\left|\frac{1}{2\pi}\int_{-\delta}^{\delta}i\eta e^{i\eta x-\alpha_1\eta^2t}d\eta-\partial_x\left(\frac{1}{2\sqrt{\pi\alpha_1t}}e^{-\frac{x^2}{4\alpha_1t}}\right)\right|\\
		&=\left|\frac{1}{2\pi}\int_{-\delta}^{\delta}i\eta e^{i\eta x-\alpha_1\eta^2t}d\eta+\frac{x}{2\alpha_1t}\left(\frac{1}{2\sqrt{\pi\alpha_1t}}e^{-\frac{x^2}{4\alpha_1t}}\right)\right|\\
		&\leq\frac{x}{2\alpha_1t}\left(O(1)e^{-\alpha_1\delta^2(\frac{t}{2}+\frac{|x|}{4M})}\right)+\frac{1}{2\pi}e^{-\frac{x^2}{4\alpha_1t}}\left(\frac{|x|^2}{(2\alpha_1t)^2}e^{-\alpha_1\delta^2t}e^{\frac{x^2}{4\alpha_1t}}+\frac{\delta |x|}{\alpha_1t}e^{-\alpha_1\delta^2t}e^{\frac{x^2}{4\alpha_1t}}\right)\\
		&=O(1)e^{-\alpha_1\delta^2(\frac{t}{2}+\frac{|x|}{4M})}.
	\end{aligned}
	\eq
	Moreover, if  only concerning about the decay estimate, we can rewrite \eqref{B61} as follows,
	\[\left|\frac{1}{2\pi}\int_{-\delta}^{\delta}i\eta e^{i\eta x-\alpha_1\eta^2t}d\eta\right|\leq \frac{|x|}{2\alpha_1t}\left(\frac{1}{2\sqrt{\pi\alpha_1t}}e^{-\frac{x^2}{4\alpha_1t}}\right)+O(1)e^{-\alpha_1\delta^2(\frac{t}{2}+\frac{|x|}{4M})}.\]
	Then, as $\frac{|x|}{\sqrt{t}}$ can be absorbed by a small part from $e^{-\frac{x^2}{4\alpha_1t}}$ to bounded, we can find a proper positive constant $C>\alpha_1$ such that
	\bq\label{B62}
	\left|\frac{1}{2\pi}\int_{-\delta}^{\delta}i\eta e^{i\eta x-\alpha_1\eta^2t}d\eta\right|\leq \frac{O(1)}{t}e^{-\frac{x^2}{4Ct}}+O(1)e^{-\alpha_1\delta^2(\frac{t}{2}+\frac{|x|}{4M})}.
	\eq
	Then, the higher order estimates are followed by similar manner.  We omit the details and obtain that 
	\bq\label{B63}
	\left\{\begin{aligned}
		&\left|\frac{1}{2\pi}\int_{-\delta}^{\delta}(i\eta)^k e^{i\eta x-\alpha_1\eta^2t}d\eta-\partial^k_x\left(\frac{1}{2\sqrt{\pi\alpha_1t}}e^{-\frac{x^2}{4\alpha_1t}}\right)\right|\leq O(1)e^{-\alpha_1\delta^2(\frac{t}{2}+\frac{|x|}{4M})},\\
		&\left|\frac{1}{2\pi}\int_{-\delta}^{\delta}(i\eta)^k e^{i\eta x-\alpha_1\eta^2t}d\eta|\right|\leq \frac{O(1)}{t^{\frac{k+1}{2}}}e^{-\frac{x^2}{4Ct}}+O(1)e^{-\alpha_1\delta^2(\frac{t}{2}+\frac{|x|}{4M})}.
	\end{aligned}
	\right.
	\eq
	Finally, we combine \eqref{B56} and \eqref{B63} to obtain the desired estimates. 
\end{proof}
Now we combine the estimates for $\bbg^*_L$ in Lemma \ref{LB11} and the estimates of $\bbg_L$ in Lemma \ref{LB12} to construct the estimates of $\bbg^\dag_L$.
\begin{lemma}\label{LB13}
	When $t\geq 1$ and $|x|\leq 2Mt$, there exist sufficiently small positive constant $\sigma_0^*$ and large constant $C$ such that, the long wave of the regular part $\bbg^\dag_L(x,t)$ has the following estimates
	\begin{align*}
		&\left|\partial_x^k\bbg_L^\dag(x,t)-\sum_{j=1}^3\partial_x^k\left(\frac{e^{-\frac{(x+\beta_jt)^2}{4\alpha_jt}}}{2\sqrt{\pi\alpha_jt}}\right)M_j^0-\sum_{j=1}^3\partial^{k+1}_x\left(\frac{e^{-\frac{(x+\beta_jt)^2}{4\alpha_jt}}}{2\sqrt{\pi\alpha_jt}}\right)M_j^1\right|\\
		&\hspace{6cm}\leq \sum_{j=1}^3\frac{O(1)e^{-\frac{(x+\beta_jt)^2}{4Ct}}}{t^{\frac{k+2}{2}}}M_j^0+\sum_{j=1}^3\frac{O(1)e^{-\frac{(x+\beta_jt)^2}{4Ct}}}{t^{\frac{k+3}{2}}}+O(1)e^{-\sigma_0^*(\frac{t}{2}+\frac{|x|}{4M})}.
	\end{align*}
\end{lemma}

\noindent $\bullet$ (Case 2: Short wave part $\bbg^\dag_S$)
In this part, we will study the short wave part $\bbg^\dag_S$. From the equation \eqref{B35}, we have the following equation of $\hat{\bbg}^\dag$,
\[\hat{\bbg}^\dag_t+i\eta F'(\bar{U})\hat{\bbg}^\dag+\eta^2B(\bar{U})\hat{\bbg}^\dag=-\hat{\Psi}(\eta,t),\]
with zero initial data. For the homogeneous equation, we already have the solution operator in \eqref{B12}, therefore we apply Duhamel's principle to obtain the solution to the inhomogeneous equation, and express $\hat{\bbg}^\dag_S$ as follows,
\[\hat{\bbg}^\dag_S(\eta,t)=-\int_0^{t}\sum_{j=1}^3e^{\lambda_j(t-\tau)}\hat{\bbm}_j\hat{\Psi}(\eta,\tau)d\tau,\quad |\eta|\geq\delta.\]

\begin{lemma}\label{LB14}
	There exists a sufficiently small positive constant $\sigma_0^*$ such that, when $t\geq 1$ and $|x|\leq 2Mt$, the short wave part $\bbg^\dag_S$ has the following estimates,
	\begin{align*}
		&|\hat{\bbg}^\dag_S(\eta,t)|=O(1)\frac{e^{-\sigma_0^*t}}{(1+|\eta|)^6},\\
		&\|\bbg^\dag_S(x,t)\|_{H^4(\bbr)}\leq O(1)e^{-\sigma_0^*t},\quad \sum_{k=0}^3\left|\partial^k_x\bbg^\dag_S(x,t)\right|\leq O(1)e^{-\sigma_0^*t}\leq O(1)e^{-\sigma_0^*\left(\frac{t}{2}+\frac{|x|}{4M}\right)},\quad t\geq 1,\quad |x|\leq 2Mt.
	\end{align*}
\end{lemma}
\begin{proof}
	From Lemma \ref{LB8}, we know the source term $\hat{\Psi}$  exponentially decays with respect to $t$, and has a small tail for large $\eta$. On the other hand, according to Lemma \ref{LB2}, there is a spectral gap for $r<|\eta|<R$, where $r$ and $R$ are two arbitrary positive constants. Moreover, according to Lemma \ref{L4.1}, there is also a spectral gap for sufficiently large $\eta$. Therefore, without loss of generality, we may let $\sigma_0^*$ be sufficiently small so that $Re(\lambda_j)<-\sigma_0^*$ for $|\eta|\geq\delta$. Finally, $\hat{\bbm}_j$ is bounded when $|\eta|$ near infinity from the expansion, and it is also bounded for $\eta$ in any finite region because $\lambda_j(\eta)$ is continuous as we discussed in Section \ref{sec:4}. Therefore, $\hat{\bbm}_j$ is simply an order one term. Then, we combine the spectral gap, $O(1)$ estimates of $\hat{\bbm}_j$ and Lemma \ref{LB8} to have 
	\begin{align*}
		|\hat{\bbg}^\dag_S(\eta,t)|=|\int_0^{t}\sum_{j=1}^3e^{\lambda_j(t-\tau)}\hat{\bbm}_j\hat{\Psi}(\eta,\tau)d\tau|\leq O(1)\int_0^{t}\sum_{j=1}^3e^{-\sigma_0^*(t-\tau)}e^{-\sigma_0^*\tau}(1+|\eta|)^{-6}d\tau=O(1)\frac{e^{-\sigma_0^*t}}{(1+|\eta|)^6}.
	\end{align*}
	As we have the decay estimate of $\eta$ up to order six when $|\eta|$ tends to infinity, the Plancherel's theorem immediately yields that 
	\begin{align*}
		\|\bbg^\dag_S(x,t)\|_{H^4(\bbr)}\leq O(1)e^{-\sigma_0^*t},\quad \sum_{k=0}^3\left|\partial^k_x\bbg^\dag_S(x,t)\right|\leq O(1)e^{-\sigma_0^*t}\leq O(1)e^{-\left(\frac{\sigma_0^*t}{2}+\frac{|x|}{4M}\right)},\quad t\geq 1,\quad |x|\leq 2Mt.
	\end{align*}
\end{proof}

\subsubsection{Pointwise estimate of $\bbg^\dag(x,t,\bar{U})$}
In the final part of this subsection, we will combine Lemma \ref{LB9} ($t< 1$), Lemma \ref{LB10} ($t\geq 1$ and $|x|>2Mt$), Lemma \ref{LB13} ($t\geq 1$, $|x|\leq 2Mt$ and $|\eta|<\delta$) and Lemma \ref{LB14} ($t\geq 1$, $|x|\leq 2Mt$ and $|\eta|\geq\delta$) to construct the full estimates of the regular part $\bbg^\dag(x,t)$. 
\begin{theorem}\label{T4.3}
	Suppose the assumption \eqref{assumption on fix state} holds. Then we can find a sufficiently large constant $C$, sufficiently small positive constant $\sigma_0$ and $\sigma^*_0$ such that   
	\begin{align*}
		&\sum_{k=0}^3|\partial_x^k\bbg^\dag(x,t)|\leq O(1)te^{-\sigma_0 |x|},&\quad 0<t<1,\\
		&\left|\partial_x^k\bbg^\dag(x,t)-\sum_{j=1}^3\partial_x^k\left(\frac{e^{-\frac{(x+\beta_jt)^2}{4\alpha_jt}}}{2\sqrt{\pi\alpha_jt}}\right)M_j^0-\sum_{j=1}^3\partial^{k+1}_x\left(\frac{e^{-\frac{(x+\beta_jt)^2}{4\alpha_jt}}}{2\sqrt{\pi\alpha_jt}}\right)M_j^1\right|\\
		&\hspace{3cm}\leq \sum_{j=1}^3\frac{O(1)e^{-\frac{(x+\beta_jt)^2}{4Ct}}}{t^{\frac{k+2}{2}}}M_j^0+\sum_{j=1}^3\frac{O(1)e^{-\frac{(x+\beta_jt)^2}{4Ct}}}{t^{\frac{k+3}{2}}}+O(1)e^{-\sigma_0^*t-\sigma_0|x|},&\quad t\geq 1,\quad 0\leq k\leq 3,
	\end{align*}
	where $\alpha_j$ and $\beta_j$ are given below,
	\[ \alpha_1=   \frac{-\kappa  \theta _e p_v}{v \left(pp_e-p_v\right)},\quad \alpha_2=\alpha_3=\frac{ \left(\kappa  p \theta _e p_e+\mu  p p_e-\mu
		p_v\right)}{2 v \left(p p_e-p_v\right)},\quad \beta_1=0,\quad \beta_2=  -\sqrt{p p_e-p_v},\quad \beta_3=  \sqrt{p p_e-p_v}.\]
	
\end{theorem}
\begin{proof}
	The proof directly follows from Lemma \ref{LB9}, Lemma \ref{LB10}, Lemma \ref{LB13} and Lemma \ref{LB14}. We emphasize that, $\sigma_0$ and $\sigma_0^*$ are first introduced in Lemma \ref{L4.1}, but we can choose them further smaller so that all the previous estimates hold. 
\end{proof}
$\ $

%
%
%
%

\section{Global well-posedness}\label{sec:5}

\setcounter{equation}{0}
In this section, we will show the local-in-time solution constructed in \cite{W-Y-Z-local} (see Propositions \ref{prop:local} and \ref{T3.3}) can be extended to exist globally in time. To this end, we will take use of the Green's function constructed in Section \ref{sec:4}. For simplicity, we assume the initial data to be a small perturbation of the constant state $(\bar{v},\bar{u},\bar{\theta})=(1,0,1)$. Then, if we use total energy $E$ as a variable, we actually have 
\[\bar{U}=(\bar{v},\bar{u},\bar{E})=(1,0,c_v),\]
where $c_v$ is the heat capacity defined in \eqref{constitution}, and $\bar{U}$ is the equilibrium state in \eqref{A2}. In this setting, the expression of the Green's function $\bbg$ would be much simplified. Actually, recall the singular part $\bbg^*$ and the regular part $\bbg^\dag$ in Section \ref{sec:4}, the matrix $\bbg$ can be represented as follows, 
\bq\label{5-1}
\bbg=\bbg^*+\bbg^\dag,\quad
\bbg^*=\left(
\begin{array}{ccc}
	\bbg_{11}^* & \bbg_{12}^* & \bbg_{13}^*\\
	\bbg_{21}^* & \bbg_{22}^* & \bbg_{23}^*\\
	\bbg_{31}^* & \bbg_{32}^* & \bbg_{33}^*
\end{array}
\right),\quad 
\bbg^\dag=\left(
\begin{array}{ccc}
	\bbg_{11}^\dag & \bbg_{12}^\dag & \bbg_{13}^\dag\\
	\bbg_{21}^\dag & \bbg_{22}^\dag & \bbg_{23}^\dag\\
	\bbg_{31}^\dag & \bbg_{32}^\dag & \bbg_{33}^\dag
\end{array}
\right).
\eq
Since $\bbg(x,t;\bar{U})$ satisfies the forward equation \eqref{B1}, it follows that $\bbg(x-y,t-\tau;\bar{U})$ also satisfies the following backward equation,
\bq\label{backward-G}
\partial_\tau\bbg(x-y,t-\tau;\bar{U})+\partial_y\bbg(x-y,t-\tau;\bar{U})F'(\bar{U})+\partial_y^2\bbg(x-y,t-\tau;\bar{U})B(\bar{U})=0.
\eq
In the following, we will briefly denote the Green's function by $\bbg(x-y,t-\tau)$. As we already chose $\bar{U}=(\bar{v},\bar{u},\bar{E})=(1,0,c_v)$, the matrices $F'(\bar{U})$ and $B(\bar{U})$ have the following form,
\[F'(\bar{U})=\begin{pmatrix} 0& -1 & 0\\ -K&0&\frac{K}{c_v}\\0&K&0\end{pmatrix},\quad B(\bar{U})=\begin{pmatrix}0&0&0\\0&\mu&0\\0&0&\frac{\kappa}{c_v}\end{pmatrix}.\]
Now, we substitute above $F'(\bar{U})$ and $B(\bar{U})$ into the backward equation \eqref{backward-G} to obtain the following equations for each component,
\bq\label{5-3}
\left(\begin{array}{ccc}
	\partial_\tau\bbg_{11}-K\partial_y\bbg_{12}\ \ \ &\partial_\tau\bbg_{12}-\partial_y\bbg_{11}+K\partial_y\bbg_{13}+\mu\partial_y^2\bbg_{12}\ \ \ &\partial_\tau\bbg_{13}+\frac{K}{c_v}\partial_y\bbg_{12}+\frac{\kappa}{c_v}\partial_y^2\bbg_{13}\\
	\ &\ &\ \\
	\partial_\tau\bbg_{21}-K\partial_y\bbg_{22}\ \ \ &\partial_\tau\bbg_{22}-\partial_y\bbg_{21}+K\partial_y\bbg_{23}+\mu\partial_y^2\bbg_{22}\ \ \ &\partial_\tau\bbg_{23}+\frac{K}{c_v}\partial_y\bbg_{22}+\frac{\kappa}{c_v}\partial_y^2\bbg_{23}\\
	\ &\ &\ \\
	\partial_\tau\bbg_{31}-K\partial_y\bbg_{32}\ \ \ &\partial_\tau\bbg_{32}-\partial_y\bbg_{31}+K\partial_y\bbg_{33}+\mu\partial_y^2\bbg_{32}\ \ \ &\partial_\tau\bbg_{33}+\frac{K}{c_v}\partial_y\bbg_{32}+\frac{\kappa}{c_v}\partial_y^2\bbg_{33}\\
\end{array}\right)=0.
\eq

\subsection{Representation by Green's function}
The representation of the solution in terms of variable coefficient heat kernel in Section \ref{sec:2} is mainly used to construct the local solution, but not convenient to get  large time structure. In this part, we also include the Green's function into the representations. Now, in order to construct the new representation of $v$, $u$ and $\theta$ by Green's function, we first introduce an effective Green's function $G$ as in \cite{L-Y-2020}. We define a smooth non-increasing cutoff function as follows,
\bq\label{5-4}
\mathcal{X}(t)\in C^\infty\left(\bbr_+\right),\quad \mathcal{X}'(t)\leq 0,\quad \|\mathcal{X}'\|_{L^\infty(\bbr_+)}\leq 2,\quad \mathcal{X}(t)=\left\{\begin{aligned}&1,\quad \text{for $t\in(0,1]$,}\\
	&0,\quad \text{for $t>2$.}\end{aligned}\right.
\eq
Then, we choose a  small positive constant $\nu_0$ such that, the heat kernel $H(x,t;y,\tau;\frac{1}{v})$ and the local weak solution $(v(x,\tau),u(x,\tau),E(x,\tau))$ for \eqref{NS} both exist when $\tau\in(t-2\nu_0,t]$. We interpolate the heat kernel for short time and Green's function for large time via cutoff function \eqref{5-4}, thus introduce the effective Green's functions as follows,
\bq\label{5-5}
\left\{\begin{aligned}
	&G_{22}(x,t;y,\tau)=\mathcal{X}\left(\frac{t-\tau}{\nu_0}\right)H\left(x,t;y,\tau;\frac{\mu}{v}\right)+\left(1-\mathcal{X}\left(\frac{t-\tau}{\nu_0}\right)\right)\bbg_{22}(x-y;t-\tau),\\
	&G_{33}(x,t;y,\tau)=\mathcal{X}\left(\frac{t-\tau}{\nu_0}\right)H\left(x,t;y,\tau;\frac{\kappa}{c_vv}\right)+\left(1-\mathcal{X}\left(\frac{t-\tau}{\nu_0}\right)\right)\bbg_{33}(x-y;t-\tau).
\end{aligned}
\right.
\eq
Now we can represent the solution $(v,u,E)$ in terms of the effective Green's function, which captures both the local-in-time regularity and global-in-time space-time structure of the solution.

\begin{lemma}\label{L5.1}
	Suppose the weak solution $(v(x,\tau),u(x,\tau),E(x,\tau))$ for \eqref{NS} exists for $\tau\in[0,t]$, and the heat kernel $H(x,t;y,\tau;\frac{\mu}{v})$ exists for $\tau\in(t-2\nu_0,t)$ for a sufficiently small positive constant $\nu_0$ such that $2\nu_0<t$. Then we have the representation of $u(x,t)$ as follows,
	\begin{align}
		&u(x,t)\label{u-repre}\\
		&=\int_\bbr\bbg_{21}(x-y,t)(v(y,0)-1)dy+\int_\bbr G_{22}(x,t;y,0)u(y,0)dy+\int_\bbr \bbg_{23}(x-y,t)\left(E(y,0)-c_v\right)dy+\sum_{i=1}^3\mathcal{R}^u_i,\notag
	\end{align}
	where the inhomogeneous remainders $\mathcal{R}_i$ are listed as below,
	\begin{align}
		\mathcal{R}^u_1
		&=\int_0^{t-2\nu_0}\int_\bbr \partial_y\bbg_{22}(x-y,t-\tau)\left(\frac{K(v-1)^2}{v}+\frac{K(\theta-1)(1-v)}{v}-\frac{Ku^2}{2c_v}+\frac{\mu u_y(v-1)}{v}\right)dyd\tau\notag\\
		&\hspace{0.4cm}+\int_0^{t-2\nu_0}\int_\bbr \partial_y\bbg_{23}(x-y,t-\tau)\left(\left(\frac{K(\theta-1)+K(1-v)}{v}\right)u+\frac{\kappa\theta_y(v-1)}{v}+\left(\frac{\kappa}{c_v}-\frac{\mu}{v}\right)uu_y\right)dyd\tau,\notag\\
		\mathcal{R}^u_2
		&=\int_{t-2\nu_0}^{t-\nu_0}\int_\bbr \partial_y\bbg_{22}(x-y,t-\tau)\left(\frac{K(v-1)^2}{v}+\frac{K(\theta-1)(1-v)}{v}+\frac{\mu u_y(v-1)}{v}\left(1-\mathcal{X}\right)-\frac{Ku^2}{2c_v}\right)dyd\tau\notag\\
		&\hspace{0.4cm}+\int_{t-2\nu_0}^{t-\nu_0}\int_\bbr \frac{1}{\nu_0}\mathcal{X}'\left(\frac{t-\tau}{\nu_0}\right)\left(\bbg_{22}(x-y;t-\tau)-H\left(x,t;y,\tau;\frac{\mu}{v}\right)\right)u(y,\tau)dyd\tau\notag\\
		&\hspace{0.4cm}+\int_{t-2\nu_0}^{t-\nu_0}\int_\bbr \mathcal{X}\left(\frac{t-\tau}{\nu_0}\right)\left(K\partial_y\bbg_{23}(x-y,t-\tau)-\partial_y\bbg_{21}(x-y,t-\tau)\right)u(y,\tau)dyd\tau\notag\\
		&\hspace{0.4cm}+\int_{t-2\nu_0}^{t-\nu_0}\int_\bbr \mathcal{X}\left(\frac{t-\tau}{\nu_0}\right)\left(H_y\left(x,t;y,\tau;\frac{\mu}{v}\right)-\partial_y\bbg_{22}(x-y;t-\tau)\right)\frac{K(\theta-v)}{v}dyd\tau\notag\\
		&\hspace{0.4cm}+\int_{t-2\nu_0}^{t-\nu_0}\int_\bbr \partial_y\bbg_{23}(x-y;t-\tau)\left(\frac{\kappa\theta_y(v-1)}{v}+\frac{Ku(\theta-v)}{v}+\left(\frac{\kappa}{c_v}-\frac{\mu}{v}\right)uu_y\right)dyd\tau,\notag\\
		\mathcal{R}^u_3
		&=\int_{t-\nu_0}^t\int_\bbr K\partial_y\bbg_{22}(x-y,t-\tau)(v(y,\tau)-1)dyd\tau-\int_{t-\nu_0}^t\int_\bbr\partial_y\bbg_{21}(x-y,t-\tau)u(y,\tau)dyd\tau\notag\\
		&\hspace{0.4cm}+\int_{t-\nu_0}^t\int_\bbr H_y\left(x,t;y,\tau;\frac{\mu}{v}\right)\left(\frac{K\theta(y,\tau)-Kv(y,\tau)}{v(y,\tau)}\right)dyd\tau\notag\\
		&\hspace{0.4cm}-\int_{t-\nu_0}^t\int_\bbr \partial_y\bbg_{22}(x-y,t-\tau)\left(K(\theta-1)+\frac{Ku^2}{2c_v}\right)dyd\tau\notag\\
		&\hspace{0.4cm}+\int_{t-\nu_0}^t\int_\bbr \partial_y\bbg_{23}(x-y,t-\tau)\left((pu)+\frac{\kappa(v-1)\theta_y}{v}+\left(\frac{\kappa}{c_v}-\frac{\mu}{v}\right)uu_y\right)dyd\tau.\notag
	\end{align}
\end{lemma}

\begin{proof}
	In fact, we multiply the vector $\left(\bbg_{21}(x-y,t-\tau), G_{22}(x,t;y,\tau), \bbg_{23}(x-y,t-\tau)\right)$ to the system \eqref{NS}, and apply the integration by part to have 
	\begin{align*}
		0&=\int_0^t\int_\bbr\bbg_{21}(x-y,t-\tau)\left(v_\tau(y,\tau)-u_y(y,\tau)\right)dyd\tau\\
		&\hspace{0.4cm}+\int_0^t\int_\bbr G_{22}(x,t;y,\tau)\left(u_\tau+p_y-\left(\frac{\mu u_y}{v}\right)_y\right)dyd\tau\\
		&\hspace{0.4cm}+\int_0^t\int_\bbr \bbg_{23}(x-y,t-\tau)\left(E_\tau+(pu)_y-\left(\frac{\kappa}{v}\theta_y+\frac{\mu}{v}uu_y\right)_y\right)dyd\tau\\
		&=\int_0^t\int_\bbr\bbg_{21}(x-y,t-\tau)\left((v-1)_\tau(y,\tau)-u_y(y,\tau)\right)dyd\tau\\
		&\hspace{0.4cm}+\int_0^t\int_\bbr G_{22}(x,t;y,\tau)\left(u_\tau+(p(v,\theta)-p(1,1))_y-\left(\frac{\mu u_y}{v}\right)_y\right)dyd\tau\\
		&\hspace{0.4cm}+\int_0^t\int_\bbr \bbg_{23}(x-y,t-\tau)\left((E-c_v)_\tau+(pu)_y-\left(\frac{\kappa}{v}\theta_y+\frac{\mu}{v}uu_y\right)_y\right)dyd\tau\\
		&=-\int_\bbr\bbg_{21}(x-y,t)(v_0(y)-1)dy-\int_0^t\int_\bbr\partial_\tau\bbg_{21}(x-y,t-\tau)(v(y,\tau)-1)dyd\tau\\
		&\hspace{0.4cm}+\int_0^t\int_\bbr\partial_y\bbg_{21}(x-y,t-\tau)u(y,\tau)dyd\tau\\
		&\hspace{0.4cm}+u(x,t)-\int_\bbr G_{22}(x,t;y,0)u(y,0)dy-\int_0^t\int_\bbr \partial_\tau G_{22}(x,t;y,\tau)u(y,\tau)dyd\tau\\
		&\hspace{0.4cm}-\int_0^t\int_\bbr \partial_yG_{22}(x,t;y,\tau)\left((p-p(1,1))-\frac{\mu u_y}{v}\right)dyd\tau\\
		&\hspace{0.4cm}-\int_\bbr \bbg_{23}(x-y,t)\left(E(y,0)-c_v\right)dy-\int_0^t\int_\bbr \partial_\tau\bbg_{23}(x-y,t-\tau)\left(E(y,\tau)-c_v\right)dyd\tau\\
		&\hspace{0.4cm}-\int_0^t\int_\bbr \partial_y\bbg_{23}(x-y,t-\tau)\left((pu)-\left(\frac{\kappa}{v}\theta_y+\frac{\mu}{v}uu_y\right)\right)dyd\tau.
	\end{align*}
	In above formula, we obtain $u(x,t)$ since $G_{22}(x-y;0)=\delta(x-y)$, and we do not have other boundary term at $t$ because the Green's function $\bbg(x,0)=I$ and thus $\bbg_{ij}(x,0)=0$ when $i\neq j$. Then, we immediately obtain a representation of $u(x,t)$ as follows,
	\begin{align}
		&u(x,t)\label{5-6}\\
		&=\int_\bbr\bbg_{21}(x-y,t)(v(y,0)-1)dy+\int_\bbr G_{22}(x,t;y,0)u(y,0)dy+\int_\bbr \bbg_{23}(x-y,t)\left(E(y,0)-c_v\right)dy\notag\\
		&\hspace{0.4cm}+\int_0^t\int_\bbr\partial_\tau\bbg_{21}(x-y,t-\tau)(v(y,\tau)-1)dyd\tau-\int_0^t\int_\bbr\partial_y\bbg_{21}(x-y,t-\tau)u(y,\tau)dyd\tau\notag\\
		&\hspace{0.4cm}+\int_0^t\int_\bbr \partial_\tau G_{22}(x,t;y,\tau)u(y,\tau)dyd\tau+\int_0^t\int_\bbr \partial_yG_{22}(x,t;y,\tau)\left(p-p(1,1)-\frac{\mu u_y}{v}\right)dyd\tau\notag\\
		&\hspace{0.4cm}+\int_0^t\int_\bbr \partial_\tau\bbg_{23}(x-y,t-\tau)\left(E(y,\tau)-c_v\right)dyd\tau+\int_0^t\int_\bbr \partial_y\bbg_{23}(x-y,t-\tau)\left((pu)-\left(\frac{\kappa}{v}\theta_y+\frac{\mu}{v}uu_y\right)\right)dyd\tau\notag\\
		&=\sum_{i=1}^9\mathcal{I}_i.\notag
	\end{align}
	Now, we apply \eqref{5-3} to represent $\bbg_\tau$ in terms of derivative in $y$. Actually, for $\mathcal{I}_4$ and $\mathcal{I}_8$ in \eqref{5-6}, we have the following representation,
	\begin{align}
		&\mathcal{I}_4=\int_0^t\int_\bbr\partial_\tau\bbg_{21}(x-y,t-\tau)(v(y,\tau)-1)dyd\tau=\int_0^t\int_\bbr K\partial_y\bbg_{22}(x-y,t-\tau)(v(y,\tau)-1)dyd\tau,\label{5-7}\\
		&\mathcal{I}_8=\int_0^t\int_\bbr \partial_\tau\bbg_{23}(x-y,t-\tau)\left(E(y,\tau)-c_v\right)dyd\tau\notag\\
		&\hspace{0.4cm}=-\int_0^t\int_\bbr \left(\frac{K}{c_v}\partial_y\bbg_{22}(x-y,t-\tau)+\frac{\kappa}{c_v}\partial_y^2\bbg_{23}(x-y,t-\tau)\right)\left(E(y,\tau)-c_v\right)dyd\tau.\label{5-8}
	\end{align}
	Moreover, according to \eqref{5-5}, the effective Green's function is constructed by partition of unit. Therefore, when dealing with $G_{22}$, we need to split the integral into three parts. More precisely, for $\mathcal{I}_6$ in \eqref{5-6}, we have the representation as below,
	\begin{align}
		\mathcal{I}_6&=\int_0^t\int_\bbr \partial_\tau G_{22}(x,t;y,\tau)u(y,\tau)dyd\tau\label{5-9}\\
		&=\int_0^{t-2\nu_0}\int_\bbr\left(\partial_y\bbg_{21}-K\partial_y\bbg_{23}-\mu\partial_y^2\bbg_{22}\right)(x-y,t-\tau)u(y,\tau)dyd\tau\notag\\
		&\hspace{0.4cm}+\int_{t-2\nu_0}^{t-\nu_0}\int_\bbr-\frac{\mathcal{X}'\left(\frac{t-\tau}{\nu_0}\right)}{\nu_0}H\left(x,t;y,\tau;\frac{\mu}{v}\right)u(y,\tau)+\frac{\mathcal{X}'\left(\frac{t-\tau}{\nu_0}\right)}{\nu_0}\bbg_{22}(x-y;t-\tau)u(y,\tau)dyd\tau\notag\\
		&\hspace{0.4cm}+\int_{t-2\nu_0}^{t-\nu_0}\int_\bbr-\mathcal{X}\left(\frac{t-\tau}{\nu_0}\right)\left(\frac{\mu}{v}H_y\left(x,t;y,\tau;\frac{\mu}{v}\right)\right)_yu(y,\tau)dyd\tau\notag\\
		&\hspace{0.4cm}+\int_{t-2\nu_0}^{t-\nu_0}\int_\bbr\left(1-\mathcal{X}\left(\frac{t-\tau}{\nu_0}\right)\right)\left(\partial_y\bbg_{21}-K\partial_y\bbg_{23}-\mu\partial_y^2\bbg_{22}\right)(x-y,t-\tau)u(y,\tau)dyd\tau\notag\\
		&\hspace{0.4cm}+\int_{t-\nu_0}^t\int_\bbr-\left(\frac{\mu}{v}H_y\left(x,t;y,\tau;\frac{\mu}{v}\right)\right)_yu(y,\tau)dyd\tau.\notag
	\end{align}
	Similarly, for $\mathcal{I}_7$ in \eqref{5-6}, we have the following representation, 
	\begin{align}
		\mathcal{I}_7&=\int_0^t\int_\bbr \partial_yG_{22}(x,t;y,\tau)\left(p-p(1,1)-\frac{\mu u_y}{v}\right)dyd\tau\label{5-10}\\
		&=\int_0^{t-2\nu_0}\int_\bbr\partial_y\bbg_{22}(x,t;y,\tau)\left(p-p(1,1)-\frac{\mu u_y}{v}\right)dyd\tau\notag\\
		&\hspace{0.4cm}+\int_{t-2\nu_0}^{t-\nu_0}\int_\bbr\mathcal{X}\left(\frac{t-\tau}{\nu_0}\right)H_y\left(x,t;y,\tau;\frac{\mu}{v}\right)\left(p-p(1,1)-\frac{\mu u_y}{v}\right)\notag\\
		&\hspace{0.4cm}+\int_{t-2\nu_0}^{t-\nu_0}\int_\bbr\left(1-\mathcal{X}\left(\frac{t-\tau}{\nu_0}\right)\right)\partial_y\bbg_{22}(x-y;t-\tau)\left(p-p(1,1)-\frac{\mu u_y}{v}\right)dyd\tau\notag\\
		&\hspace{0.4cm}+\int_{t-\nu_0}^t\int_\bbr H_y\left(x,t;y,\tau;\frac{\mu}{v}\right)\left(p-p(1,1)-\frac{\mu u_y}{v}\right).\notag
	\end{align}
	Now, we combine the representations in \eqref{5-6}, \eqref{5-7}, \eqref{5-8}, \eqref{5-9} and \eqref{5-10}, and apply the backward equation \eqref{5-3} to obtain a new representation of $u(x,t)$ as follows,
	\begin{align}
		&u(x,t)\label{5-11}\\
		&=\int_\bbr\bbg_{21}(x-y,t)(v(y,0)-1)dy+\int_\bbr G_{22}(x,t;y,0)u(y,0)dy+\int_\bbr \bbg_{23}(x-y,t)\left(E(y,0)-c_v\right)dy\notag\\
		&\hspace{0.4cm}+\int_0^t\int_\bbr K\partial_y\bbg_{22}(x-y,t-\tau)(v(y,\tau)-1)dyd\tau-\int_0^t\int_\bbr\partial_y\bbg_{21}(x-y,t-\tau)u(y,\tau)dyd\tau\notag\\
		&\hspace{0.4cm}+\int_0^{t-2\nu_0}\int_\bbr\left(\partial_y\bbg_{21}-K\partial_y\bbg_{23}-\mu\partial_y^2\bbg_{22}\right)(x-y,t-\tau)u(y,\tau)dyd\tau\notag\\
		&\hspace{0.4cm}+\int_{t-2\nu_0}^{t-\nu_0}\int_\bbr-\frac{1}{\nu_0}\mathcal{X}'\left(\frac{t-\tau}{\nu_0}\right)H\left(x,t;y,\tau;\frac{\mu}{v}\right)u(y,\tau)+\frac{1}{\nu_0}\mathcal{X}'\left(\frac{t-\tau}{\nu_0}\right)\bbg_{22}(x-y;t-\tau)u(y,\tau)dyd\tau\notag\\
		&\hspace{0.4cm}+\int_{t-2\nu_0}^{t-\nu_0}\int_\bbr-\mathcal{X}\left(\frac{t-\tau}{\nu_0}\right)\left(\frac{\mu}{v}H_y\left(x,t;y,\tau;\frac{\mu}{v}\right)\right)_yu(y,\tau)dyd\tau\notag\\
		&\hspace{0.4cm}+\int_{t-2\nu_0}^{t-\nu_0}\int_\bbr\left(1-\mathcal{X}\left(\frac{t-\tau}{\nu_0}\right)\right)\left(\partial_y\bbg_{21}-K\partial_y\bbg_{23}-\mu\partial_y^2\bbg_{22}\right)(x-y,t-\tau)u(y,\tau)dyd\tau\notag\\
		&\hspace{0.4cm}+\int_{t-\nu_0}^t\int_\bbr-\left(\frac{\mu}{v}H_y\left(x,t;y,\tau;\frac{\mu}{v}\right)\right)_yu(y,\tau)dyd\tau\notag\\
		&\hspace{0.4cm}+\int_0^{t-2\nu_0}\int_\bbr\partial_y\bbg_{22}(x,t;y,\tau)\left(p-p(1,1)-\frac{\mu u_y}{v}\right)dyd\tau\notag\\
		&\hspace{0.4cm}+\int_{t-2\nu_0}^{t-\nu_0}\int_\bbr\mathcal{X}\left(\frac{t-\tau}{\nu_0}\right)H_y\left(x,t;y,\tau;\frac{\mu}{v}\right)\left(p-p(1,1)-\frac{\mu u_y}{v}\right)dyd\tau\notag\\
		&\hspace{0.4cm}+\int_{t-2\nu_0}^{t-\nu_0}\int_\bbr\left(1-\mathcal{X}\left(\frac{t-\tau}{\nu_0}\right)\right)\partial_y\bbg_{22}(x-y;t-\tau)\left(p-p(1,1)-\frac{\mu u_y}{v}\right)dyd\tau\notag\\
		&\hspace{0.4cm}+\int_{t-\nu_0}^t\int_\bbr H_y\left(x,t;y,\tau;\frac{\mu}{v}\right)\left(p-p(1,1)-\frac{\mu u_y}{v}\right)dyd\tau\notag\\
		&\hspace{0.4cm}-\int_0^t\int_\bbr \left(\frac{K}{c_v}\partial_y\bbg_{22}(x-y,t-\tau)+\frac{\kappa}{c_v}\partial_y^2\bbg_{23}(x-y,t-\tau)\right)\left(E(y,\tau)-c_v\right)dyd\tau\notag\\
		&\hspace{0.4cm}+\int_0^t\int_\bbr \partial_y\bbg_{23}(x-y,t-\tau)\left((pu)-\left(\frac{\kappa}{v}\theta_y+\frac{\mu}{v}uu_y\right)\right)dyd\tau.\notag
	\end{align}
	Finally, we obtain the desired representation by separating the representation \eqref{5-11} with respect to the time intervals $[0,t-2\nu_0]$, $[t-2\nu_0,t-\nu_0]$ and $[t-\nu_0,t]$. 
	
\end{proof}
\begin{lemma}\label{L5.2}
	Suppose the weak solution $(v(x,\tau),u(x,\tau),E(x,\tau))$ for \eqref{NS} exists for $\tau\in[0,t]$, and the heat kernel $H(x,t;y,\tau;\frac{\kappa}{c_vv})$ exists for $\tau\in(t-2\nu_0,t)$ for a sufficiently small positive constant $\nu_0$ such that $2\nu_0<t$. Then we have the representation of $v(x,t)$ and $E(x,t)$ as follows,
	\begin{align}
		&v(x,t)-1\label{v-repre}\\
		&=\int_\bbr\bbg_{11}(x-y,t)(v(y,0)-1)dy+\int_\bbr \bbg_{12}(x,t;y,0)u(y,0)dy+\int_\bbr \bbg_{13}(x-y,t)\left(E(y,0)-c_v\right)dy\notag\\
		&\hspace{0.4cm}+\int_0^{t}\int_\bbr \partial_y\bbg_{12}(x,t;y,\tau)\left(\frac{K(v-1)^2}{v}+\frac{(K\theta-K)(1-v)}{v}+\frac{\mu u_y(v-1)}{v}-\frac{Ku^2}{2c_v}\right)dyd\tau\notag\\
		&\hspace{0.4cm}+\int_0^{t}\int_\bbr \partial_y\bbg_{13}(x-y,t-\tau)\left((\frac{K\theta}{v}-K)u+\left(\kappa-\frac{\kappa}{v}\right)\theta_y+\left(\frac{\kappa}{c_v}-\frac{\mu}{v}\right)uu_y\right)dyd\tau,\notag\\
		&E(x,t)-c_v\label{theta-repre}\\
		&=\int_\bbr\bbg_{31}(x-y,t)(v(y,0)-1)dy+\int_\bbr \bbg_{32}(x,t;y,0)u(y,0)dy+\int_\bbr G_{33}(x-y,t)\left(E(y,0)-c_v\right)dy+\sum_{i=1}^3\mathcal{R}^\theta_i,\notag
	\end{align}
	where the remainders $\mathcal{R}_i^\theta$ are listed as below,
	\begin{align*}
		\mathcal{R}_1^\theta&=\int_0^{t-2\nu_0}\int_\bbr \partial_y\bbg_{32}(x,t;y,\tau)\left(\frac{K(v-1)^2}{v}+\frac{(K\theta-K)(1-v)}{v}+\frac{\mu u_y(v-1)}{v}-\frac{Ku^2}{2c_v}\right)dyd\tau\notag\\
		&\hspace{0.4cm}+\int_0^{t-2\nu_0}\int_\bbr \partial_y\bbg_{33}(x-y,t-\tau)\left((\frac{K\theta}{v}-K)u+\left(\kappa-\frac{\kappa}{v}\right)\theta_y+\left(\frac{\kappa}{c_v}-\frac{\mu}{v}\right)uu_y\right)dyd\tau,\notag\\
		\mathcal{R}_2^\theta&=\int_{t-2\nu_0}^{t-\nu_0}\int_\bbr \partial_y\bbg_{32}(x,t;y,\tau)\left(\frac{K(\theta-1)}{v}+\frac{\mu (v-1)u_y}{v}+\frac{K(v-1)^2}{v}\right)dyd\tau\notag\\
		&\hspace{0.4cm}-\int_{t-2\nu_0}^{t-\nu_0}\int_\bbr  \left(1-\mathcal{X}\left(\frac{t-\tau}{\nu_0}\right)\right)\left(\frac{K}{c_v}\partial_y\bbg_{32}\right)\left(E(y,\tau)-c_v\right)dyd\tau\notag\\
		&\hspace{0.4cm}+\int_{t-2\nu_0}^{t-\nu_0}\int_\bbr \frac{1}{\nu_0}\mathcal{X}'\left(\frac{t-\tau}{\nu_0}\right) \left(-H\left(x,t;y,\tau;\frac{\kappa}{c_vv}\right)+\bbg_{33}(x-y;t-\tau)\right)\left(E(y,\tau)-c_v\right)dyd\tau\notag\\
		&\hspace{0.4cm}+\int_{t-2\nu_0}^{t-\nu_0}\int_\bbr \mathcal{X}\left(\frac{t-\tau}{\nu_0}\right)H_y\left(x,t;y,\tau;\frac{\kappa}{c_vv}\right)\left((pu)+\left(\frac{\kappa}{c_vv}-\frac{\mu}{v}\right)uu_y\right)dyd\tau\notag\\
		&\hspace{0.4cm}+\int_{t-2\nu_0}^{t-\nu_0}\int_\bbr \left(1-\mathcal{X}\left(\frac{t-\tau}{\nu_0}\right)\right)\partial_y\bbg_{33}(x-y;t-\tau)\left((pu)+\frac{\kappa(v-1)}{v}\theta_y+\left(\frac{\kappa}{c_v}-\frac{\mu}{v}\right)uu_y\right)dyd\tau\notag\\
		&\hspace{0.4cm}+\int_{t-2\nu_0}^{t-\nu_0}\int_\bbr \left(-K\partial_y\bbg_{33}\right)u(y,\tau)dyd\tau,\notag\\
		\mathcal{R}_3^\theta&=\int_{t-\nu_0}^t\int_\bbr \partial_y\bbg_{32}(x-y,t-\tau)\left(K(v(y,\tau)-1)+\mu u_y\frac{v-1}{v}+\frac{K\theta-Kv}{v}\right)dyd\tau\notag\\
		&\hspace{0.4cm}+\int_{t-\nu_0}^t\int_\bbr \left(-K\partial_y\bbg_{33}\right)u(y,\tau)dyd\tau\notag\\
		&\hspace{0.4cm}+\int_{t-\nu_0}^t\int_\bbr \partial_yH\left(x,t;y,\tau;\frac{\kappa}{c_vv}\right)\left((pu)+\left(\frac{\kappa}{c_vv}-\frac{\mu}{v}\right)uu_y\right)dyd\tau.\notag
	\end{align*}
	
\end{lemma}
\begin{proof}
	The calculations are similar to those in Lemma \ref{L5.1}. Indeed, multiplying the vector 
	\[
	\begin{pmatrix}
		\bbg_{11}(x-y,t-\tau)& \bbg_{12}(x-y,t-\tau) & \bbg_{13}(x-y,t-\tau)
	\end{pmatrix}
	\]
	on the system \eqref{NS}, integrating with respect to $y$ and $\tau$, using integration by parts and the backward equation, one can get the representation for $v(x,t)$. 
	
	For the representation of $E(x,t)$, we just need to multiply 
	\[
	\begin{pmatrix}
		\bbg_{31}(x-y,t-\tau)& \bbg_{32}(x-y,t-\tau) & G_{33}(x,t;y,\tau)
	\end{pmatrix}
	\]
	on the system \eqref{NS} and follow the similar procedure. As the computations are straightforward though a bit tedious, here we omit the details.

\end{proof}

\begin{remark}\label{R5.1-1}
	Although the computations in Lemmas \ref{L5.1} and \ref{L5.2} are a bit lengthy, the goal is clear, that is, we want to write the solution into the sum of linear  terms and  nonlinear terms. Since we are dealing with perturbation problem, one expects the linear terms are dominant. Notice that when $\tau\in (t-2\nu_0,t)$, there are also some linear terms, it turns out they can be controlled either by comparison estimates of Green's function or by the appropriate small parameter $\nu_0$.
\end{remark}

Note that in the representation of zeroth order solutions, we need the information of first order derivative of $u$ and $\theta$. If we formally differentiate the expressions \eqref{u-repre} and \eqref{theta-repre} with respect to $x$, the derivative will induce additional singularity in time, so that the time integral is not integrable anymore. Therefore, in order to gain the derivative estimates, we provide another representation for the derivatives, which is a direct consequence of the previous representations.


\begin{corollary}\label{Cor5.1}
	For the first order derivative of $\mathcal{R}^u_3$ and $\mathcal{R}^\theta_3$, we have the following estimates,
	\begin{align*}
		\partial_x\mathcal{R}^u_3
		&=\int_{-\infty}^x \int_{-\infty}^y\frac{1}{\mu}\left(\partial_x\bbg_{22}(x-z,\nu_0)-H_x\left(x,t;z,t-\nu_0;\frac{\mu}{v}\right)\right)dz(Kv(y,t-\nu_0)-K\theta(y,t-\nu_0))dy\\
		&\hspace{0.4cm}-\int_{t-\nu_0}^t\int_{-\infty}^x \int_{-\infty}^y-\frac{1}{\mu}\left(\partial_x\bbg_{22}(x-z,t-\tau)-H_x\left(x,t;z,\tau;\frac{\mu}{v}\right)\right)dz(Kv_\tau(y,\tau)-K\theta_\tau(y,\tau))dyd\tau\\
		&\hspace{0.4cm}-\int_x^{+\infty} \int_y^{+\infty}\frac{1}{\mu}\left(\partial_x\bbg_{22}(x-z,\nu_0)-H_x\left(x,t;z,t-\nu_0;\frac{\mu}{v}\right)\right)dz(Kv(y,t-\nu_0)-K\theta(y,t-\nu_0))dy\\
		&\hspace{0.4cm}+\int_{t-\nu_0}^t\int_x^{+\infty} \int_y^{+\infty}-\frac{1}{\mu}\left(\partial_x\bbg_{22}(x-z,t-\tau)-H_x\left(x,t;z,\tau;\frac{\mu}{v}\right)\right)dz(Kv_\tau(y,\tau)-K\theta_\tau(y,\tau))dyd\tau\\
		&\hspace{0.4cm}+\int_{t-\nu_0}^t\int_\bbr -\frac{1}{\mu}\left(-\partial_x\bbg_{21}(x-y,t-\tau)+K\partial_x\bbg_{23}(x-y,t-\tau)\right)(Kv(y,\tau)-K\theta(y,\tau))dyd\tau\\
		&\hspace{0.4cm}+\int_{t-\nu_0}^t\int_\bbr\partial_x\bbg_{21}(x-y,t-\tau)u_y(y,\tau)dyd\tau\notag\\
		&\hspace{0.4cm}+\int_{t-\nu_0}^t\int_\bbr \partial_x\bbg_{22}(x-y,t-\tau)\frac{Kuu_y}{c_v}dyd\tau\notag\\
		&\hspace{0.4cm}+\int_{t-\nu_0}^t\int_\bbr \partial_{xy}\bbg_{23}(x-y,t-\tau)\left((pu)+\frac{\kappa(v-1)\theta_y}{v}+\left(\frac{\kappa}{c_v}-\frac{\mu}{v}\right)uu_y\right)dyd\tau,\notag
	\end{align*}
and	
	\begin{align*}
		\partial_x\mathcal{R}_3^\theta
		&=\int_{t-\nu_0}^t\int_\bbr \partial_{xy}\bbg_{32}(x-y,t-\tau)\left(K(v(y,\tau)-1)+\mu u_y\frac{v-1}{v}+\frac{K\theta-Kv}{v}\right)dyd\tau\notag\\
		&\hspace{0.4cm}+\int_{t-\nu_0}^t\int_\bbr K\left(\partial_x\bbg_{33}(x-y,t-\tau)-H_x\left(x,t;y,\tau;\frac{\kappa}{c_vv}\right)\right)u_y(y,\tau)dyd\tau\notag\\
		&\hspace{0.4cm}+\int_{t-\nu_0}^t\int_\bbr H_x\left(x,t;y,\tau;\frac{\kappa}{c_vv}\right)\left(1-\frac{\kappa}{c_v\mu}\right)uu_\tau dyd\tau\notag\\
		&\hspace{0.4cm}+\int_{t-\nu_0}^t\int_\bbr H_x\left(x,t;y,\tau;\frac{\kappa}{c_vv}\right)\left(1-\frac{\kappa}{c_v\mu}\right)u_y\left(K-p+\frac{\mu}{v}u_y\right)dyd\tau\notag\\
		&\hspace{0.4cm}+\int_{-\infty}^x\int_{-\infty}^y H_x\left(x,t;z,t-\nu_0;\frac{\kappa}{c_vv}\right)dz\left(\frac{\kappa K}{c_v\mu}\right)\frac{(\theta(y,t-\nu_0)-v(y,t-\nu_0))}{v(y,t-\nu_0)}u(y,t-\nu_0)dy\notag\\
		&\hspace{0.4cm}+\left(\frac{\kappa K}{c_v\mu}\right)\frac{(\theta(x,t-\nu_0)-v(x,t-\nu_0))}{v(x,t-\nu_0)}u(x,t-\nu_0)\notag\\
		&\hspace{0.4cm}+\int_{t-\nu_0}^t\int_{-\infty}^x\int_{-\infty}^y H_x\left(x,t;z,\tau;\frac{\kappa}{c_vv}\right)dz\left(\frac{\kappa K}{c_v\mu}\right)\left(\frac{\left(\theta_\tau v-\theta v_\tau\right)}{v^2}u(y,\tau)+\frac{(\theta-v)}{v}u_\tau(y,\tau)\right)dyd\tau\notag\\
		&\hspace{0.4cm}+\int_{t-\nu_0}^t\left(\frac{\kappa K}{c_v\mu}\right)\left(\frac{\left(\theta_\tau(x,\tau) v(x,\tau)-\theta(x,\tau) v_\tau(x,\tau)\right)}{v^2(x,\tau)}u(x,\tau)+\frac{(\theta(x,\tau)-v(x,\tau))}{v(x,\tau)}u_\tau(x,\tau)\right)d\tau\notag\\
		&\hspace{0.4cm}-\int_x^{+\infty}\int_y^{+\infty} H_x\left(x,t;z,t-\nu_0;\frac{\kappa}{c_vv}\right)dz\left(\frac{\kappa K}{c_v\mu}\right)\left(\frac{\theta(y,t-\nu_0)-v(y,t-\nu_0)}{v(y,t-\nu_0)}\right)u(y,t-\nu_0)dy\notag\\
		&\hspace{0.4cm}-\int_{t-\nu_0}^t\int_x^{+\infty}\int_y^{+\infty} H_x\left(x,t;z,\tau;\frac{\kappa}{c_vv}\right)dz\left(\frac{\kappa K}{c_v\mu}\right)\left(\frac{\left(\theta_\tau v-\theta v_\tau\right)}{v^2}u(y,\tau)+\frac{(\theta-v)}{v}u_\tau(y,\tau)\right)dyd\tau.\notag\\
	\end{align*}
	
\end{corollary}

\begin{proof}
	We will only show the details for the proof of $\partial_x\mathcal{R}^u_3$, and the other one can be obtained by similar arguments.  
	\begin{align*}
		\mathcal{R}^u_3 & = \int_{t-\nu_0}^t\int_{-\infty}^x \int_{-\infty}^y \left[\partial_{zz} \bbg_{22} (x-z,t-\tau) - \partial_z\left(\frac{1}{v(z,\tau)}H_z\right)\left(x,t;z,\tau;\frac{\mu}{v}\right) \right]dz K \left(v-\theta\right)(y,\tau) dy d\tau\\
		&\quad -  \int_{t-\nu_0}^t\int_{x}^{+\infty} \int_{y}^{+\infty} \left[\partial_{zz} \bbg_{22} (x-z,t-\tau) - \partial_z\left(\frac{1}{v(z,\tau)}H_z\right)\left(x,t;z,\tau;\frac{\mu}{v}\right) \right]dz K \left(v-\theta\right)(y,\tau) dy d\tau\\
		&\quad+\int_{t-\nu_0}^t\int_\bbr\bbg_{21}(x-y,t-\tau)u_y(y,\tau)dyd\tau +\int_{t-\nu_0}^t\int_\bbr \bbg_{22}(x-y,t-\tau)\frac{Kuu_y}{c_v}dyd\tau\notag\\
		&\quad+\int_{t-\nu_0}^t\int_\bbr \partial_y\bbg_{23}(x-y,t-\tau)\left((pu)+\frac{\kappa(v-1)\theta_y}{v}+\left(\frac{\kappa}{c_v}-\frac{\mu}{v}\right)uu_y\right)dyd\tau.\notag
	\end{align*}
	
	One uses the backward equations for $\bbg_{22}$ in \eqref{5-3} and $H(x,t;z,\tau,\frac{\mu}{v})$ to replace the second order derivatives by
	\begin{align*}
		&\partial_{zz} \bbg_{22} (x-z,t-\tau)  = -\frac{1}{\mu} \left(	\partial_{\tau} \bbg_{22} -\partial_{z}\bbg_{21} + K\partial_{z} \bbg_{23}		\right)(x-z,t-\tau),\\
		&\partial_z\left(\frac{1}{v(z,\tau)}H_z\right)  = -\frac{1}{\mu} \partial_{\tau} H (x,t;z,\tau;\frac{\mu}{v}).
	\end{align*}
	
	For the time derivative, one applies integration by parts with respect to $\tau$. Considering that  
	\[\bbg_{22}(x-z,0)=H(x,t;z,t;\frac{\mu}{v})=\delta(x-z)\] 
	when $\tau=t$, only boundary terms at $\tau=t-\nu_0$ remain. Collecting all the terms, taking derivative with respect to $x$, one obtains the desired results.

	The formula for $\mathcal{R}^\theta_3$ can be similarly obtained.

\end{proof}

\subsection{Global existence, uniquenss and large time behavior}
In this part, we will study the global existence of the solution to the nonlinear Navier-Stokes equation \eqref{NS}. According to Propositions \ref{prop:local} and  \ref{T3.3}, if the initial data is controlled by a sufficiently small constant $\delta$ as in \eqref{D1}, there exists a unique weak solution $(v,u,E)$ to \eqref{NS} for $t<t_\sharp$. Moreover, the solutions are kept small in the sense of \eqref{local-small-strong}. Then, we define a stopping time as below,
\begin{align}
	&T=\sup_{t \geq 0}\left\{t\ \Big|\ \mathcal{G}(\tau)< \delta,\quad \text{for $0< \tau<t$}\right\},	\label{5-15}\\
	&\mathcal{G}(\tau):=\|\sqrt{\tau+1}(v(\cdot,\tau)-1)\|_\infty+\|\sqrt{\tau+1}u(\cdot,\tau)\|_\infty+\|\sqrt{\tau+1}(\theta(\cdot,\tau)-1)\|_\infty\notag\\
	&\hspace{1.2cm}+\|v(\cdot,\tau)-1\|_{L^1}+\|u(\cdot,\tau)\|_{L^1}+\|\theta(\cdot,\tau)-1\|_{L^1}\notag\\
	&\hspace{1.2cm}+\|v(\cdot,\tau)-1\|_{BV}+\|u(\cdot,\tau)\|_{BV}+\|\theta(\cdot,\tau)-1\|_{BV}\notag\\
	&\hspace{1.2cm}+\|\sqrt{\tau}u_x(\cdot,\tau)\|_\infty+\|\sqrt{\tau}\theta_x(\cdot,\tau)\|_\infty.\notag
\end{align}
According to Proposition \ref{prop:local}, the life span of the weak solution is at least $T+t_\sharp$. It is obvious $T\geq 0$, and we will actually show in the next lemma that, $T>0$ for sufficiently small initial data. 
\begin{lemma}\label{L5.3}
	Let $(v,u,E)$, $C_\sharp$, $t_\sharp$ and $\delta$  be the local solution and corresponding parameters constructed in Proposition \ref{prop:local}. If we further let the initial data satisfy the following small condition,
	\bq\label{global-initial} 
	\|v_0-1\|_{BV}+\|u_0\|_{BV}+\|\theta_0-1\|_{BV}+\|v_0-1\|_{L^1}+\|u_0\|_{L^1}+\|\theta_0-1\|_{L^1}<\delta^*<\frac{\delta}{220 C_\sharp},
	\eq 
	then we have $T>t_\sharp$, where $T$ is the stopping time defined in \eqref{5-15}. Moreover, we have the following estimates for all $t<T+t_\sharp$,
	\begin{equation}\label{global-small}
		\left\{\begin{aligned}
			&0<\delta,\\
			&\max\left\{\begin{aligned}&\left\|u(\cdot,t)\right\|_{L^1_x},\ \left\|u(\cdot,t)\right\|_{L^\infty_x},\ \left\|u_x(\cdot,t)\right\|_{L^1_x},\\ &\sqrt{\min(t,t_\sharp)}\left\|u_x(\cdot,t)\right\|_{L^\infty_x},\ \sqrt{\min(t,t_\sharp)}\left\|u_t(\cdot,t)\right\|_{L^1_x},\ \min(t,t_\sharp)\left\|u_t(\cdot,t)\right\|_{L^\infty_x}\end{aligned}\right\}\leq 2C_\sharp\delta,\\
			&\max\left\{\begin{aligned}&\left\|\theta(\cdot,t)-1\right\|_{L^1_x},\ \left\|\theta(\cdot,t)-1\right\|_{L^\infty_x},\ \left\|\theta_x(\cdot,t)\right\|_{L^1_x},\\ &\sqrt{\min(t,t_\sharp)}\left\|\theta_x(\cdot,t)\right\|_{L^\infty_x},\ \sqrt{\min(t,t_\sharp)}\left\|\theta_t(\cdot,t)\right\|_{L^1_x},\ \min(t,t_\sharp)\left\|\theta_t(\cdot,t)\right\|_{L^\infty_x}\end{aligned}\right\}\leq 2C_\sharp\delta,\\
			&\max\left\{\int_{\bbr\setminus\mathscr{D}}|v_x(x,t)|dx,\ \left\|v(\cdot,t)-1\right\|_{L^1_x},\ \left\|v(\cdot,t)-1\right\|_{L^\infty_x},\ \sqrt{\min(t,t_\sharp)}\left\|v_t(\cdot,t)\right\|_{L^\infty_x}\right\}\leq 2C_\sharp\delta,\\
			&v^*=v^*_c+v^*_d,\quad  v^*_d(x,t)=\sum_{z<x,z\in\mathscr{D}}v^*\Big|_{z^-}^{z^+}h(x-z),\quad \text{$v_c^*$ is AC},\\
			&\left|v(\cdot,t)\Big|_{x=z^-}^{x=z^+}\right|\leq 2\left|v^*_0(\cdot)\Big|_{x=z^-}^{x=z^+}\right|,\quad z\in\mathscr{D},
		\end{aligned}
		\right.
	\end{equation}
	where $v^*=v-1$ as before.
\end{lemma}
\begin{proof}
	Notice that the local existence result in Proposition \ref{prop:local} can be applied for the initial data \eqref{global-initial} for same $C_\sharp$ and $t_\sharp$. More precisely, we have for $\tau<t_\sharp$ that,
	\begin{align*}
		\left\{
		\begin{aligned}
			&\max\left\{\|v(\cdot,\tau)-1\|_{BV},\ \|u(\cdot,\tau)\|_{BV},\ \|\theta(\cdot,\tau)-1\|_{BV},\ \|v(\cdot,\tau)-1\|_{L^1},\ \|u(\cdot,\tau)\|_{L^1},\ \|\theta(\cdot,\tau)-1\|_{L^1}\right\}\leq 2C_\sharp\delta^*,\\
			&\max\left\{\|\sqrt{1+\tau}u(\cdot,\tau)\|_{L^\infty},\ \|\sqrt{1+\tau}(\theta(\cdot,\tau)-1)\|_{L^\infty},\ \|\sqrt{1+\tau}(v(\cdot,\tau)-1)\|_{L^\infty}\right\}\leq 2C_\sharp\sqrt{1+t_\sharp}\delta^*,\\
			&\max\left\{\|\sqrt{\tau}u_x(\cdot,\tau)\|_{L^\infty},\ \|\sqrt{\tau}\theta_x(\cdot,\tau)\|_{L^\infty}\right\}\leq 2C_\sharp\delta^*.
		\end{aligned}
		\right.
	\end{align*}
	Therefore, combining the above inequality, the definition of $\delta^*$ in \eqref{global-initial} and $\mathcal{G}(\tau)$ in \eqref{5-15}, we immediately obtain that 
	\bq\label{5-17}
	\mathcal{G}(\tau)\leq 22C_\sharp\sqrt{1+t_\sharp}\delta^*<\frac{\delta}{5},\quad \text{for $\tau<t_\sharp$}.
	\eq 
	Then, \eqref{5-17}  and the continuity in $t$ of the  solution $(v,u,E)$ immediately imply $T>t_\sharp$. Next, we prove the estimates \eqref{global-small} for $t<T+t_\sharp$. In fact, according to \eqref{5-15}, we have 
	\[\max(0,t-t_\sharp)<T,\quad \mathcal{G}\left(\max(0,t-t_\sharp)\right)<\delta.\]
	Therefore, we use $v(x,\max(0,t-t_\sharp))$, $u(x,\max(0,t-t_\sharp))$ and $\theta(x,\max(0,t-t_\sharp))$ as the initial data, and apply Proposition \ref{prop:local} to obtain \eqref{global-small}. Then we finish the proof of the lemma.
	
\end{proof}

In Lemma \ref{L5.3}, we can only show $\mathcal{G}(\tau)<\delta$ for $\tau\leq t_\sharp$ but not for all $\tau>0$, because we use the estimates  of local existence. In the following, we will take advantage of the representations in Lemma \ref{L5.1} and Lemma \ref{L5.2}, and construct sharper estimates of the solution for large time. First, we recall that we linearize \eqref{NS} around $(1,0,c_v)$, and therefore we apply Theorem \ref{T4.1} and Theorem \ref{T4.3} to have simpler estimates for Green's function around $(v, u, \theta)=(1,0,1)$.
\begin{lemma}\label{L5.4}
	Let $\bbg(x,t)$ be the Green's function of the linearized equation of \eqref{NS} around the constant equilibrium state $(v,u,\theta)=(1,0,1)$. Then, the Green's function $\bbg(x,t)$ has the following estimates for $t\leq 1$, 
	\begin{align*}
		&\Bigg|\bbg(x-y,t)
		- e^{-\frac{K}{\mu}t}\delta(x-y)
		\begin{pmatrix}
			1 & 0 & 0 \\
			0 & 0 & 0 \\
			0 & 0 & 0 \\
		\end{pmatrix}
		-\frac{e^{\beta^*_2t}}{\sqrt{4\pi \mu t}}e^{-\frac{(x-y)^2}{4\mu t}}
		\begin{pmatrix}
			0 & 0 & 0 \\
			0 & 1 & 0 \\
			0 & 0 & 0 \\
		\end{pmatrix}\\
		&\hspace{3cm}-\frac{e^{\beta^*_3t}}{\sqrt{4\pi \frac{\kappa}{c_v}t}}e^{-\frac{(x-y)^2}{4\frac{\kappa}{c_v}t}}
		\begin{pmatrix}
			0 & 0 & 0 \\
			0 & 0 & 0 \\
			0 & 0 & 1 \\
		\end{pmatrix}\Bigg|\leq O(1)e^{-\sigma_0^*t-\sigma_0|x-y|}+O(1)te^{-\sigma_0 |x-y|},\qquad t\leq 1,
	\end{align*}

	\begin{align*}
		&\Bigg|\bbg_x(x-y,t)
		- e^{-\frac{K}{\mu}t}\delta'(x-y)
		\begin{pmatrix}
			1 & 0 & 0 \\
			0 & 0 & 0 \\
			0 & 0 & 0 \\
		\end{pmatrix}
		+e^{-\frac{K}{\mu}t}\delta(x-y)
		\begin{pmatrix}
			0 & \frac{1}{\mu} & 0 \\
			\frac{K}{\mu} & 0 & 0 \\
			0 & 0 & 0 \\
		\end{pmatrix}\\
		& -\partial_x\left(\frac{e^{\beta^*_2t}}{\sqrt{4\pi \mu t}}e^{-\frac{(x-y)^2}{4\mu t}}\right)
		\begin{pmatrix}
			0 & 0 & 0 \\
			0 & 1 & 0 \\
			0 & 0 & 0 \\
		\end{pmatrix}
		+\frac{e^{\beta^*_2t}}{\sqrt{4\pi \mu t}}e^{-\frac{(x-y)^2}{4\mu t}}
		\begin{pmatrix}
			0 & -\frac{ 1}{\mu } & 0 \\
			-\frac{ K}{\mu } & 0 & \frac{K}{c_v\mu
				-\kappa} \\
			0 & \frac{c_vK}{c_v\mu -\kappa } &
			0
		\end{pmatrix}\\
		&-\partial_x\left(\frac{e^{\beta^*_3t}}{\sqrt{4\pi \frac{\kappa}{c_v}t}}e^{-\frac{(x-y)^2}{4\frac{\kappa}{c_v}t}}\right)
		\begin{pmatrix}
			0 & 0 & 0 \\
			0 & 0 & 0 \\
			0 & 0 & 1 \\
		\end{pmatrix}
		+\frac{e^{\beta^*_3t}}{\sqrt{4\pi \frac{\kappa}{c_v}t}}e^{-\frac{(x-y)^2}{4\frac{\kappa}{c_v}t}}
		\begin{pmatrix}
			0 & 0 & 0 \\
			0 & 0 & -\frac{K}{c_v\mu -\kappa} \\
			0 & - \frac{c_vK}{c_v\mu -\kappa } & 0 
		\end{pmatrix}\Bigg|\\
		&\hspace{8cm}\leq O(1)e^{-\sigma_0^*t-\sigma_0|x-y|}+O(1)te^{-\sigma_0 |x-y|},\qquad t\leq 1,
	\end{align*}

	\begin{align*}
		&\Bigg|\bbg_{xx}(x-y,t)
		- e^{-\frac{K}{\mu}t}\delta''(x-y)
		\begin{pmatrix}
			1 & 0 & 0 \\
			0 & 0 & 0 \\
			0 & 0 & 0 \\
		\end{pmatrix}
		+e^{-\frac{K}{\mu}t}\delta'(x-y)
		\begin{pmatrix}
			0 & \frac{1}{\mu} & 0 \\
			\frac{K}{\mu} & 0 & 0 \\
			0 & 0 & 0 
		\end{pmatrix}\\
		&-e^{-\frac{K}{\mu}t}\delta(x-y)
		\left( \begin{pmatrix}
			\frac{K}{\mu ^2} & 0 & \frac{K}{\kappa  \mu } \\
			0 & -\frac{K}{\mu ^2} & 0 \\
			\frac{c_vK^2}{\kappa  \mu } & 0 & 0 
		\end{pmatrix} 
		-t\begin{pmatrix}
			-\frac{ \left(\kappa K^2-\mu  K^3 \right)}{ \kappa  \mu ^3 }& 0 & 0 \\
			0 & 0 & 0 \\
			0 & 0 & 0
		\end{pmatrix}\right)\\
		& -\partial_{xx}\left(\frac{e^{\beta^*_2t}}{\sqrt{4\pi \mu t}}e^{-\frac{(x-y)^2}{4\mu t}}\right)
		\begin{pmatrix}
			0 & 0 & 0 \\
			0 & 1 & 0 \\
			0 & 0 & 0 \\
		\end{pmatrix}
		+\partial_x\left(\frac{e^{\beta^*_2t}}{\sqrt{4\pi \mu t}}e^{-\frac{(x-y)^2}{4\mu t}}\right)
		\begin{pmatrix}
			0 & -\frac{ 1}{\mu } & 0 \\
			-\frac{ K}{\mu } & 0 & \frac{K}{c_v\mu
				-\kappa} \\
			0 & \frac{c_vK}{c_v\mu -\kappa } &
			0
		\end{pmatrix}\\
		&-\frac{e^{\beta^*_2t}}{\sqrt{4\pi \mu t}}e^{-\frac{(x-y)^2}{4\mu t}}
		\begin{pmatrix}
			-\frac{K}{\mu ^2} & 0 & \frac{K}{c_v\mu ^2-\kappa  \mu } \\
			0 & \frac{ c_vK^2}{ \left(c_v\mu -\kappa \right)^2}+\frac{K}{\mu
				^2 } & 0 \\
			\frac{c_vK^2}{c_v\mu ^2-\kappa  \mu } & 0 & -\frac{c_vK^2}{\left(c_v\mu -\kappa \right){}^2}
		\end{pmatrix}\\
		&-\partial_{xx}\left(\frac{e^{\beta^*_3t}}{\sqrt{4\pi \frac{\kappa}{c_v}t}}e^{-\frac{(x-y)^2}{4\frac{\kappa}{c_v}t}}\right)
		\begin{pmatrix}
			0 & 0 & 0 \\
			0 & 0 & 0 \\
			0 & 0 & 1 \\
		\end{pmatrix}
		+\partial_x\left(\frac{e^{\beta^*_3t}}{\sqrt{4\pi \frac{\kappa}{c_v}t}}e^{-\frac{(x-y)^2}{4\frac{\kappa}{c_v}t}}\right)
		\begin{pmatrix}
			0 & 0 & 0 \\
			0 & 0 & -\frac{K}{c_v\mu -\kappa} \\
			0 & - \frac{c_vK}{c_v\mu -\kappa } & 0 
		\end{pmatrix}\\
		&+\frac{e^{\beta^*_3t}}{\sqrt{4\pi \frac{\kappa}{c_v}t}}e^{-\frac{(x-y)^2}{4\frac{\kappa}{c_v}t}}
		\begin{pmatrix}
			0 & 0 & \frac{c_vK}{\kappa  \left(c_v\mu -\kappa \right)} \\
			0 & \frac{c_vK^2}{\left(c_v\mu -\kappa \right)^2} & 0 \\
			\frac{c_v^2K^2}{\kappa  \left(c_v\mu -\kappa \right)} & 0 & -\frac{c_vK^2}{\left(c_v\mu -\kappa\right)^2}
		\end{pmatrix}\Bigg|\leq O(1)e^{-\sigma_0^*t-\sigma_0|x-y|}+O(1)te^{-\sigma_0 |x-y|},\qquad t\leq 1.\\
		\ 
	\end{align*}
On the other hand, for large time $t\geq 1$, the Green's function has the following estimates,

\begin{align*}
&\Bigg|\bbg(x-y,t)
- e^{-\frac{K}{\mu}t}\delta(x-y)
\begin{pmatrix}
1 & 0 & 0 \\
0 & 0 & 0 \\
0 & 0 & 0 \\
\end{pmatrix}\\
&-\sum_{j=1}^3\frac{e^{-\frac{(x-y+\beta_jt)^2}{4\alpha_jt}}}{2\sqrt{\pi\alpha_jt}}M_j^0-\sum_{j=1}^3\partial_x\left(\frac{e^{-\frac{(x-y+\beta_jt)^2}{4\alpha_jt}}}{2\sqrt{\pi\alpha_jt}}\right)M_j^1\Bigg|\\
&\hspace{2.5cm}\leq \sum_{j=1}^3\frac{O(1)e^{-\frac{(x-y+\beta_jt)^2}{4Ct}}}{t}M_j^0+\sum_{j=1}^3\frac{O(1)e^{-\frac{(x-y+\beta_jt)^2}{4Ct}}}{t^{\frac{3}{2}}}+O(1)e^{-\sigma_0^*t-\sigma_0|x-y|},\qquad 1\leq t,
\end{align*}

\begin{align*}
&\Bigg|\bbg_x(x-y,t)
- e^{-\frac{K}{\mu}t}\delta'(x-y)
\begin{pmatrix}
1 & 0 & 0 \\
0 & 0 & 0 \\
0 & 0 & 0 \\
\end{pmatrix}
+e^{-\frac{K}{\mu}t}\delta(x-y)
\begin{pmatrix}
0 & \frac{1}{\mu} & 0 \\
\frac{K}{\mu} & 0 & 0 \\
0 & 0 & 0 \\
\end{pmatrix}\\
&-\sum_{j=1}^3\partial_x\left(\frac{e^{-\frac{(x-y+\beta_jt)^2}{4\alpha_jt}}}{2\sqrt{\pi\alpha_jt}}\right)M_j^0-\sum_{j=1}^3\partial^2_x\left(\frac{e^{-\frac{(x-y+\beta_jt)^2}{4\alpha_jt}}}{2\sqrt{\pi\alpha_jt}}\right)M_j^1\Bigg|\\
&\hspace{2.5cm}\leq \sum_{j=1}^3\frac{O(1)e^{-\frac{(x-y+\beta_jt)^2}{4Ct}}}{t^{\frac{3}{2}}}M_j^0+\sum_{j=1}^3\frac{O(1)e^{-\frac{(x-y+\beta_jt)^2}{4Ct}}}{t^2}+O(1)e^{-\sigma_0^*t-\sigma_0|x-y|},\qquad 1\leq t,
\end{align*}

\begin{align*}
&\Bigg|\bbg_{xx}(x-y,t)
- e^{-\frac{K}{\mu}t}\delta''(x-y)
\begin{pmatrix}
1 & 0 & 0 \\
0 & 0 & 0 \\
0 & 0 & 0 \\
\end{pmatrix}
+e^{-\frac{K}{\mu}t}\delta'(x-y)
\begin{pmatrix}
0 & \frac{1}{\mu} & 0 \\
\frac{K}{\mu} & 0 & 0 \\
0 & 0 & 0 
\end{pmatrix}\\
&-e^{-\frac{K}{\mu}t}\delta(x-y)
\left( \begin{pmatrix}
\frac{K}{\mu ^2} & 0 & \frac{K}{\kappa  \mu } \\
0 & -\frac{K}{\mu ^2} & 0 \\
\frac{c_vK^2}{\kappa  \mu } & 0 & 0 
\end{pmatrix} 
-t\begin{pmatrix}
-\frac{ \left(\kappa K^2-\mu  K^3 \right)}{ \kappa  \mu ^3 }& 0 & 0 \\
0 & 0 & 0 \\
0 & 0 & 0
\end{pmatrix}\right)\\
&-\sum_{j=1}^3\partial_x^2\left(\frac{e^{-\frac{(x-y+\beta_jt)^2}{4\alpha_jt}}}{2\sqrt{\pi\alpha_jt}}\right)M_j^0-\sum_{j=1}^3\partial^{3}_x\left(\frac{e^{-\frac{(x-y+\beta_jt)^2}{4\alpha_jt}}}{2\sqrt{\pi\alpha_jt}}\right)M_j^1\Bigg|\\
&\hspace{2.5cm}\leq \sum_{j=1}^3\frac{O(1)e^{-\frac{(x-y+\beta_jt)^2}{4Ct}}}{t^2}M_j^0+\sum_{j=1}^3\frac{O(1)e^{-\frac{(x-y+\beta_jt)^2}{4Ct}}}{t^{\frac{5}{2}}}+O(1)e^{-\sigma_0^*t-\sigma_0|x-y|},\qquad 1\leq t,\\
\ 
\end{align*}
where the parameters in the above formulae are listed as below,
\begin{align*}
&\alpha_1=   \frac{\kappa }{ \left(K+c_v\right)},\quad \alpha_2=\alpha_3=\frac{ \left(\kappa K+c_v\mu  K+c_v^2\mu
\right)}{2 \left(c_vK+c_v^2\right)},\quad \beta_1=0,\quad \beta_2=  -\sqrt{\frac{K^2}{c_v}+K},\quad \beta_3=  \sqrt{\frac{K^2}{c_v}+K},\\
&\beta^*_2=\frac{v \left(\mu  K^2 \frac{1}{c_v}-\kappa \frac{1}{c_v} K+\mu 
K\right)}{\mu  \left(\mu -\frac{\kappa}{c_v}  \right)}=\frac{K}{\mu}+\frac{ \left(K^2 \right)}{ \left(c_v\mu -\kappa  \right)},\quad \beta^*_3=\frac{p v p_e}{\kappa  \theta _e-\mu }.
\end{align*}

\end{lemma}
\begin{proof}
We just need to substitute the constant state $(v,u,\theta)=(1,0,1)$ into Theorem \ref{T4.1}, Theorem \ref{T4.2} and Theorem \ref{T4.3} to obtain the desired estimates. 
\end{proof}
\begin{remark}\label{R5.1}
In the above estimates, we rescale time $t$ by $\frac{t}{\nu_0}$ with $\nu_0$ a positive constant. Then, we can obtain similar estimates for $t\leq \nu_0$ and $t\geq \nu_0$. The only difference is that, all the $O(1)$ terms depend on $\nu_0$ due to the rescaling.
\end{remark}

Now,  we will provide several a priori estimates for the global estimates of the solution constructed in  Proposition \ref{prop:local}. As the proof is rather long, we will split the estimates into several lemmas. 
\begin{lemma}\label{L5.5}
	$(\mathbf{\|u(\cdot,t)\|_{L^1},\ \|\sqrt{1+t}u(\cdot,t)\|_{L^\infty}})$ Let $(v,u,E)$, $C_\sharp$, $t_\sharp$ and $\delta$  be the local solution and corresponding parameters constructed in Proposition \ref{prop:local}. We further suppose the following properties hold for the solution, 
	\bq\label{5-19}
	\left\{
	\begin{aligned}
		&\|v_0-1\|_{BV}+\|u_0\|_{BV}+\|\theta_0-1\|_{BV}+\|v_0-1\|_{L^1}+\|u_0\|_{L^1}+\|\theta_0-1\|_{L^1}<\delta^*,\\
		&\mathcal{G}(\tau)<\delta,\quad \text{for $\forall\tau<t$,}\\
		&t_\sharp\geq 4\nu_0,
	\end{aligned} 
	\right.
	\eq
	where $\mathcal{G}(\tau)$ is provided in \eqref{5-15}, $\delta^*$ is defined in \eqref{global-initial} and $\nu_0$ is given in \eqref{5-5}. Then, $u(x,t)$ has the following zeroth order estimates for $t\geq t_\sharp$,
	\begin{equation*}
		\left\{
		\begin{aligned}
			&\|u(\cdot,t)\|_{L^1}\leq C(\nu_0)\delta^*+O(1)\left(\delta^2+\sqrt{\nu_0}\delta^2+\sqrt{\nu_0}\delta+\nu_0\delta+\delta^2\right),\\
			&\|\sqrt{1+t}u(\cdot,t)\|_{L^\infty}\leq C(\nu_0)\delta^*+O(1)\left(\sqrt{\nu_0}\delta+\delta^2\right).
		\end{aligned}
		\right.
	\end{equation*}
	
\end{lemma}

\begin{remark}\label{R5.2}
	According to Lemma \ref{L5.3} , we know that  $\mathcal{G}(t)<\delta$ for $t\leq t_\sharp$ if $\delta^*$ is sufficient small, in particular, much smaller than $\delta$. Therefore, our assumption $t>t_\sharp$ in the lemma makes sense.
\end{remark}

\begin{proof} 
	According to \eqref{u-repre} in Lemma \ref{L5.1}, we have the following representation of $u(x,t)$, 
	\begin{align*}
		&u(x,t)\\
		&=\int_\bbr\bbg_{21}(x-y,t)(v(y,0)-1)dy+\int_\bbr \bbg_{22}(x-y,t)u(y,0)dy+\int_\bbr \bbg_{23}(x-y,t)\left(E(y,0)-c_v\right)dy+\sum_{i=1}^3\mathcal{R}^u_i(x,t).
	\end{align*}
	Note, as $t>t_\sharp\geq 4\nu_0$, we actually have $G_{22}(x,t)=\bbg_{22}(x,t)$, and thus obtain above representation. Moreover, as the Green's function has short time and large time estimates, we should split the proof into two cases, i.e., $t\leq 1$ and $t\geq 1$. As the estimates of the case $t\geq 1$ also involves the estimates of $t\leq 1$, we will only show the details of the estimates of case $t\geq 1$.\newline
	
	\noindent $\bullet$ ($\mathbf{\|u(\cdot,t)\|_{L^1}}$) We first estimate  the $L^1$ norm of $u$. According to Lemma \ref{L5.4}, the zeroth order Green's function $\bbg_{2j}$  are constructed by heat kernel, its derivative and remainders $\mathcal{R}_j$. Moreover, according to Remark \ref{R5.1}, we can rescale the Green's function, and have the following estimates for $t\geq2\nu_0$, 
	\begin{align}
		&|\bbg_{2k}(x,t)|\leq \sum_{j=1}^3\frac{O(1)e^{-\frac{(x+\beta_jt)^2}{4Ct}}}{\sqrt{t}}+O(1)e^{-\sigma_0^*t-\sigma_0|x|},\quad k=1,2,3,\label{5-20}\\
		&|\partial_x\bbg_{2k}(x,t)|\leq \sum_{j=1}^3\frac{O(1)e^{-\frac{(x+\beta_jt)^2}{4Ct}}}{t}+O(1)e^{-\sigma_0^*t-\sigma_0|x|},\quad k=2,3,\label{5-21}\\
		&|\partial_x\bbg_{21}(x,t)+\frac{K}{\mu}e^{-\frac{K}{\mu}t}\delta(x)| \leq \sum_{j=1}^3\frac{O(1)e^{-\frac{(x+\beta_jt)^2}{4Ct}}}{t}+O(1)e^{-\sigma_0^*t-\sigma_0|x|},\label{5-21-1}
	\end{align}
	where, according to Remark \ref{R5.1}, the coefficients $O(1)$ depend on $\nu_0$, and thus we will denote them by $C(\nu_0)$. Then, as $t\geq t_\sharp\geq 4\nu_0$, we directly substitute \eqref{5-20} into the first three homogeneous terms  in the representation of $u$ to obtain that 
	\begin{equation}
		\label{5-22}
		\left\{
		\begin{aligned}
			&\int_\bbr\left|\int_\bbr\bbg_{21}(x-y,t)(v(y,0)-1)dy\right|dx\leq C(\nu_0) \delta^*,\\
			&\int_\bbr\left|\int_\bbr\bbg_{22}(x-y,t)u(y,0)dy\right|dx\leq C(\nu_0) \delta^*,\\
			&\int_\bbr\left|\int_\bbr \bbg_{23}(x-y,t)\left(E(y,0)-c_v\right)dy\right|dx\leq C(\nu_0) \delta^*.
		\end{aligned}
		\right.
	\end{equation}
	
	Next, for inhomogeneous terms $\mathcal{R}^u_j$, we need to use $\nu_0$ to gain small factor. Thus, we do not rescale the Green's function, and directly apply Lemma \ref{L5.4} to estimate $\mathcal{R}^u_j$. For $\mathcal{R}^u_1$, we split it into integrations from $0$ to $t-1$, and from $t-1$ to $t-2\nu_0$ as follows,
	\begin{align}
		\mathcal{R}^u_1&=\int_0^{t-1}\int_\bbr \partial_y\bbg_{22}(x-y,t-\tau)\left(\frac{K(v-1)^2}{v}+\frac{K(\theta-1)(1-v)}{v}-\frac{Ku^2}{2c_v}+\frac{\mu u_y(v-1)}{v}\right)dyd\tau\label{5-23}\\
		&\hspace{0.4cm}+\int_0^{t-1}\int_\bbr \partial_y\bbg_{23}(x-y,t-\tau)\left(\left(\frac{K(\theta-1)+K(1-v)}{v}\right)u+\frac{\kappa\theta_y(v-1)}{v}+\left(\frac{\kappa}{c_v}-\frac{\mu}{v}\right)uu_y\right)dyd\tau\notag\\
		&\hspace{0.4cm}+\int_{t-1}^{t-2\nu_0}\int_\bbr \partial_y\bbg_{22}(x-y,t-\tau)\left(\frac{K(v-1)^2}{v}+\frac{K(\theta-1)(1-v)}{v}-\frac{Ku^2}{2c_v}+\frac{\mu u_y(v-1)}{v}\right)dyd\tau\notag\\
		&\hspace{0.4cm}+\int_{t-1}^{t-2\nu_0}\int_\bbr \partial_y\bbg_{23}(x-y,t-\tau)\left(\left(\frac{K(\theta-1)+K(1-v)}{v}\right)u+\frac{\kappa\theta_y(v-1)}{v}+\left(\frac{\kappa}{c_v}-\frac{\mu}{v}\right)uu_y\right)dyd\tau\notag\\
		&=\mathcal{I}_1+\mathcal{I}_2.\notag
	\end{align}
	
	When $t\geq 1$, we apply Theorem \ref{T4.1} and Theorem \ref{T4.3} to obtain the estimates of $\bbg$, which are same as \eqref{5-20} and \eqref{5-21} with $O(1)$ coefficients. On the other hand, we use $\|u\|_{L^1}<\delta$ and $\|u\|_{BV}<\delta$ to obtain that $\|u(\cdot,\tau)\|_\infty<\delta$ for $\tau<t$. Similarly, we have $\|v(\cdot,\tau)-1\|_\infty<\delta$ and $\|\theta(\cdot,\tau)-1\|_\infty<\delta$ for $\tau<t$. Then we combine the $L^\infty$ estimates and \eqref{5-21} to have the estimates of $\mathcal{I}_1$ as follows,
	\begin{align}
		&\int_\bbr\left|\mathcal{I}_1\right|dx\label{5-24}\\
		&\leq\int_\bbr\left|\int_0^{t-1}\int_\bbr \partial_y\bbg_{22}(x-y,t-\tau)\left(\frac{K(v-1)^2}{v}+\frac{K(\theta-1)(1-v)}{v}-\frac{Ku^2}{2c_v}+\frac{\mu u_y(v-1)}{v}\right)dyd\tau\right|dx\notag\\
		&\hspace{0.4cm}+\int_\bbr\left|\int_0^{t-1}\int_\bbr \partial_y\bbg_{23}(x-y,t-\tau)\left(\left(\frac{K(\theta-1)+K(1-v)}{v}\right)u+\frac{\kappa\theta_y(v-1)}{v}+\left(\frac{\kappa}{c_v}-\frac{\mu}{v}\right)uu_y\right)dyd\tau\right|dx,\notag\\
		&\leq \int_0^{t-1}\int_\bbr\int_\bbr \left|\partial_y\bbg_{22}(x-y,t-\tau)\right|dx\left(\frac{K|v-1|\delta}{v\sqrt{1+\tau}}+\frac{K|\theta-1|\delta}{v\sqrt{1+\tau}}+\frac{K|u|\delta}{2c_v\sqrt{1+\tau}}+\frac{\mu |u_y|\delta}{v\sqrt{1+\tau}}\right)dyd\tau\notag\\
		&\hspace{0.4cm}+\int_0^{t-1}\int_\bbr\int_\bbr \left|\partial_y\bbg_{23}(x-y,t-\tau)\right|dx\left(\left(\frac{2K\delta}{v\sqrt{1+\tau}}\right)|u|+\frac{\kappa|\theta_y|\delta}{v\sqrt{1+\tau}}+\left(\frac{\kappa}{c_v}-\frac{\mu}{v}\right)\frac{\delta}{\sqrt{1+\tau}}|u_y|\right)dyd\tau,\notag\\
		&\leq O(1)\int_0^{t-1}\left(\frac{1}{\sqrt{t-\tau}}+\frac{1}{t-\tau}+\frac{1}{(t-\tau)^{\frac{3}{2}}}+e^{-\sigma^*_0(t-\tau)} \right)\left(\frac{\delta^2}{\sqrt{1+\tau}}+\frac{\delta^2}{\sqrt{1+\tau}}+\frac{\delta^2}{\sqrt{1+\tau}}+\frac{\delta^2}{\sqrt{1+\tau}}\right)d\tau\notag\\
		&\leq O(1)\int_0^{t-1}\frac{1}{\sqrt{t-\tau}}\frac{\delta^2}{\sqrt{1+\tau}}d\tau\notag\\
		&\leq O(1)\delta^2.\notag
	\end{align}
	In order to estimates $\mathcal{I}_2$ in \eqref{5-23}, we need the estimates of $\bbg$ for $t\leq 1$. Actually, we apply Lemma \ref{L5.4} to obtain the following estimates of the Green's function $\bbg$ when $t\leq 1$,
	\bq\label{5-25}
	\left\{
	\begin{aligned}
		&|\bbg_{21}(x-y,t)|\leq O(1)e^{-\sigma_0^*t-\sigma_0|x-y|}+O(1)te^{-\sigma_0 |x-y|},\\
		&\left|\partial_x\bbg_{21}(x-y,t)-\frac{K}{\mu}e^{-Kt/\mu}\delta(x-y)\right|\leq O(1) \left(\frac{e^{\beta^*_2t}}{\sqrt{4\pi \mu t}}e^{-\frac{(x-y)^2}{4\mu t}}+e^{-\sigma_0^*t-\sigma_0|x-y|}+te^{-\sigma_0 |x-y|}\right),\\
		&|\partial^k_x\bbg_{22}(x-y,t)|\leq \left|\partial^k_x\frac{e^{\beta^*_2t}}{\sqrt{4\pi \mu t}}e^{-\frac{(x-y)^2}{4\mu t}}\right|+O(1)e^{-\sigma_0^*t-\sigma_0|x-y|}+O(1)te^{-\sigma_0 |x-y|},\quad k=0,1,\\
		&|\bbg_{23}(x-y,t)|\leq O(1)e^{-\sigma_0^*t-\sigma_0|x-y|}+O(1)te^{-\sigma_0 |x-y|},\\
		&|\partial_x\bbg_{23}(x-y,t)|\leq O(1) \left(\frac{e^{\beta^*_2t}}{\sqrt{4\pi \mu t}}e^{-\frac{(x-y)^2}{4\mu t}}+\frac{e^{\beta^*_3t}}{\sqrt{4\pi \frac{\kappa}{c_v}t}}e^{-\frac{(x-y)^2}{4\frac{\kappa}{c_v}t}}+e^{-\sigma_0^*t-\sigma_0|x-y|}+te^{-\sigma_0 |x-y|}\right)
	\end{aligned}
	\right.
	\eq
	Then, we substitute \eqref{5-25} into $\mathcal{I}_2$ in \eqref{5-23}, and follow the estimates \eqref{5-24} to obtain that, 
	\begin{align}
		&\int_\bbr\left|\mathcal{I}_2\right|dx\leq O(1)\delta^2. \label{5-26}
	\end{align}
	
	Next, in order to estimate the remainder $\mathcal{R}^u_2$ and $\mathcal{R}^u_3$, we need the estimates of the Green's function $\bbg_{2j}(x,t)$ for $t\leq 2\nu_0<1$, which has been listed in \eqref{5-25}. Moreover, we can have a more accurate estimates of $\bbg_{22}$. Actually, from \eqref{5-3} and the initial data for Green's function $\bbg$, we have the following equation of $\bbg_{22}$,
	\[\partial_\tau\bbg_{22}-\partial_y\bbg_{21}+K\partial_y\bbg_{23}+\mu\partial_y^2\bbg_{22}=0,\]
	with initial data $\bbg_{22}(x,0)=\delta(x)$. Therefore, we can solve the above equation by using heat kernel $H(x,t;y,\tau;\mu)$, and have an exact representation of $\bbg_{22}$ as follows,
	\begin{align}
		&\bbg_{22}(x-y,t-\tau)\label{5-27}\\
		&=H(x,t;y,\tau;\mu)-\int_\tau^t\int_\bbr H(z,s;y,\tau;\mu)\left(\partial_z\bbg_{21}(x-z,t-s)-K\partial_z\bbg_{23}(x-z,t-s)\right)dzds.\notag
	\end{align}
	Now, we substitute \eqref{5-25} and \eqref{5-27} into the representation $\mathcal{R}^u_2$, and apply the estimates of $H$ in  Lemma \ref{L2.1}, Lemma \ref{L2.2} and the assumption of $\mathcal{G}(\tau)$ in \eqref{5-19} to obtain that,
	\begin{align}
		&\int_\bbr\left|\mathcal{R}^u_2\right|dx\label{5-28}\\
		&\leq \int_\bbr\left|\int_{t-2\nu_0}^{t-\nu_0}\int_\bbr \partial_y\bbg_{22}(x-y,t-\tau)\left(\frac{K(v-1)^2}{v}+\frac{K(\theta-1)(1-v)}{v}+\frac{\mu u_y(v-1)}{v}\left(1-\mathcal{X}\right)-\frac{Ku^2}{2c_v}\right)dyd\tau\right|dx\notag\\
		&\hspace{0.4cm}+\int_\bbr\left|\int_{t-2\nu_0}^{t-\nu_0}\int_\bbr \frac{1}{\nu_0}\mathcal{X}'\left(\frac{t-\tau}{\nu_0}\right)\left(\bbg_{22}(x-y;t-\tau)-H\left(x,t;y,\tau;\frac{\mu}{v}\right)\right)u(y,\tau)dyd\tau\right|dx\notag\\
		&\hspace{0.4cm}+\int_\bbr\left|\int_{t-2\nu_0}^{t-\nu_0}\int_\bbr \mathcal{X}\left(\frac{t-\tau}{\nu_0}\right)\left(K\partial_y\bbg_{23}(x-y,t-\tau)-\partial_y\bbg_{21}(x-y,t-\tau)\right)u(y,\tau)dyd\tau\right|dx\notag\\
		&\hspace{0.4cm}+\int_\bbr\left|\int_{t-2\nu_0}^{t-\nu_0}\int_\bbr \mathcal{X}\left(\frac{t-\tau}{\nu_0}\right)\left(H_y\left(x,t;y,\tau;\frac{\mu}{v}\right)-\partial_y\bbg_{22}(x-y;t-\tau)\right)\frac{K(\theta-v)}{v}dyd\tau\right|dx\notag\\
		&\hspace{0.4cm}+\int_\bbr\left|\int_{t-2\nu_0}^{t-\nu_0}\int_\bbr \partial_y\bbg_{23}(x-y;t-\tau)\left(\frac{\kappa\theta_y(v-1)}{v}+\frac{Ku(\theta-v)}{v}+\left(\frac{\kappa}{c_v}-\frac{\mu}{v}\right)uu_y\right)dyd\tau\right|dx\notag\\
		&\leq O(1)\sqrt{\nu_0}\delta^2+O(1)\sqrt{\nu_0}\delta\notag\\
		&\hspace{0.4cm}+O(1)\int_{t-2\nu_0}^{t-\nu_0}\int_\bbr\int_\bbr \frac{1}{\nu_0}\left|H\left(x,t;y,\tau;\mu\right)-H\left(x,t;y,\tau;\frac{\mu}{v}\right)\right| dx|u(y,\tau)|dyd\tau\notag\\
		&\hspace{0.4cm}+O(1)\int_{t-2\nu_0}^{t-\nu_0}\int_\bbr\int_\bbr \frac{1}{\nu_0}\left|\int_\tau^t\int_\bbr H(z,s;y,\tau;\mu)\left(\partial_z\bbg_{21}(x-z,t-s)\right)dzds\right|dx|u(y,\tau)|dyd\tau\notag\\
		&\hspace{0.4cm}+O(1)\int_{t-2\nu_0}^{t-\nu_0}\int_\bbr\int_\bbr \frac{1}{\nu_0}\left|\int_\tau^t\int_\bbr H(z,s;y,\tau;\mu)\left(K\partial_z\bbg_{23}(x-z,t-s)\right)dzds\right|dx|u(y,\tau)|dyd\tau\notag\\
		&\leq O(1)\sqrt{\nu_0}\delta^2+O(1)\sqrt{\nu_0}\delta+O(1)\nu_0\delta\notag\\
		&\hspace{0.4cm}+O(1)\int_{t-2\nu_0}^{t-\nu_0}\int_\bbr\int_\bbr \frac{1}{\nu_0}\|\mu-\frac{\mu}{v}\|_{L^\infty}\frac{1}{\sqrt{t-\tau}}e^{-\frac{(x-y)^2}{C(t-\tau)}} dx|u(y,\tau)|dyd\tau\notag\\
		&\leq O(1)\left(\sqrt{\nu_0}\delta^2+\sqrt{\nu_0}\delta+\nu_0\delta+\delta^2\right).\notag
	\end{align}
	With similar arguments, we have the estimates of $\mathcal{R}^u_3$ as follows,
	\begin{align}
		&\int_\bbr\left|\mathcal{R}^u_3\right|dx\label{5-29}\\
		&\leq\int_\bbr\left|\int_{t-\nu_0}^t\int_\bbr K\partial_y\bbg_{22}(x-y,t-\tau)(v(y,\tau)-1)dyd\tau\right|dx\notag\\
		&\hspace{0.4cm}+\int_\bbr\left|\int_{t-\nu_0}^t\int_\bbr\partial_y\bbg_{21}(x-y,t-\tau)u(y,\tau)dyd\tau\right|dx\notag\\
		&\hspace{0.4cm}+\int_\bbr\left|\int_{t-\nu_0}^t\int_\bbr H_y\left(x,t;y,\tau;\frac{\mu}{v}\right)\left(\frac{K\theta(y,\tau)-Kv(y,\tau)}{v(y,\tau)}\right)dyd\tau\right|dx\notag\\
		&\hspace{0.4cm}+\int_\bbr\left|\int_{t-\nu_0}^t\int_\bbr \partial_y\bbg_{22}(x-y,t-\tau)\left(K(\theta-1)+\frac{Ku^2}{2c_v}\right)dyd\tau\right|dx\notag\\
		&\hspace{0.4cm}+\int_\bbr\left|\int_{t-\nu_0}^t\int_\bbr \partial_y\bbg_{23}(x-y,t-\tau)\left((pu)+\frac{\kappa(v-1)\theta_y}{v}+\left(\frac{\kappa}{c_v}-\frac{\mu}{v}\right)uu_y\right)dyd\tau\right|dx\notag\\
		&\leq O(1)\left(\sqrt{\nu_0}\delta+\nu_0\delta+\sqrt{\nu_0}\delta^2\right).\notag
	\end{align}
	Now, we note the smallness of $\delta$ and $\nu_0\leq \frac{t_\sharp}{4}$, and combine \eqref{5-22}, \eqref{5-23}, \eqref{5-24}, \eqref{5-26}, \eqref{5-28} and \eqref{5-29} to obtain that 
	\bq\label{5-30}
	\|u(\cdot,t)\|_{L^1}\leq O(1)C(\nu_0)\delta^*+O(1)\left(\delta^2+\sqrt{\nu_0}\delta^2+\sqrt{\nu_0}\delta+\nu_0\delta+\delta^2\right),\quad t\geq 1.
	\eq
	
	For $t\leq 1$, the only difference is that we do not have $\mathcal{I}_1$ in \eqref{5-23}, and the integration in $\mathcal{I}_2$ is from $0$ to $t-2\nu_0$. As $t\leq 1$ and all the estimates have no singularity at $\tau=0$, we can use same arguments to obtain that \eqref{5-30} also holds for $t\leq 1$. Thus we conclude that \eqref{5-30} holds for any $t>0$.\newline
	
	\noindent $\bullet$ ($\mathbf{\|u(\cdot,t)\|_{L^\infty}}$) Next, we study the infinity norm of $u$. In fact, we can use almost the same arguments as in the estimates of $\|u(\cdot,t)\|_{L^1}$ to obtain the estimate of $\|u(\cdot,t)\|_{L^\infty}$. Actually, as we only need to consider $t\geq t_\sharp$, for the homogeneous terms, we apply \eqref{5-19} and \eqref{5-20} to have,
	\bq\label{5-31}
	\left\{
	\begin{aligned}
		&\left|\int_\bbr\bbg_{21}(x-y,t)(v(y,0)-1)dy\right|\leq \frac{C(\nu_0)}{t_\sharp}\frac{\sqrt{1+t}}{\sqrt{t}}\frac{\delta^*}{\sqrt{1+t}}\leq \frac{C(\nu_0)\delta^*}{\sqrt{1+t}},\\
		&\left|\int_\bbr\bbg_{22}(x-y,t)u(y,0)dy\right|\leq  \frac{C(\nu_0)}{t_\sharp}\frac{\sqrt{1+t}}{\sqrt{t}}\frac{\delta^*}{\sqrt{1+t}}\leq \frac{C(\nu_0)\delta^*}{\sqrt{1+t}},\\
		&\left|\int_\bbr \bbg_{23}(x-y,t)\left(E(y,0)-c_v\right)dy\right|\leq  \frac{C(\nu_0)}{t_\sharp}\frac{\sqrt{1+t}}{\sqrt{t}}\left(\frac{\delta^*}{\sqrt{1+t}}+\frac{(\delta^*)^2}{\sqrt{1+t}}\right)\leq \frac{C(\nu_0)\delta^*}{\sqrt{1+t}}.
	\end{aligned}
	\right.
	\eq
	Next, for the remainders $\mathcal{R}^u_j$, we follow the arguments as in the estimates of $\|u\|_{L^1}$ and apply \eqref{5-19}, \eqref{5-21} and \eqref{5-25} to obtain the estimates. Actually, for $\mathcal{R}^u_1$, we have for $t\geq 1$ that,
	\begin{align}
		\left|\mathcal{R}^u_1\right|&=\left|\int_0^{\frac{t-1}{2}}\int_\bbr \partial_y\bbg_{22}(x-y,t-\tau)\left(\frac{K(v-1)^2}{v}+\frac{K(\theta-1)(1-v)}{v}-\frac{Ku^2}{2c_v}+\frac{\mu u_y(v-1)}{v}\right)dyd\tau\right|\label{5-32}\\
		&\hspace{0.4cm}+\left|\int_0^{\frac{t-1}{2}}\int_\bbr \partial_y\bbg_{23}(x-y,t-\tau)\left(\left(\frac{K(\theta-1)+K(1-v)}{v}\right)u+\frac{\kappa\theta_y(v-1)}{v}+\left(\frac{\kappa}{c_v}-\frac{\mu}{v}\right)uu_y\right)dyd\tau\right|\notag\\
		&\hspace{0.4cm}+\left|\int_{\frac{t-1}{2}}^{t-1}\int_\bbr \partial_y\bbg_{22}(x-y,t-\tau)\left(\frac{K(v-1)^2}{v}+\frac{K(\theta-1)(1-v)}{v}-\frac{Ku^2}{2c_v}+\frac{\mu u_y(v-1)}{v}\right)dyd\tau\right|\notag\\
		&\hspace{0.4cm}+\left|\int_{\frac{t-1}{2}}^{t-1}\int_\bbr \partial_y\bbg_{23}(x-y,t-\tau)\left(\left(\frac{K(\theta-1)+K(1-v)}{v}\right)u+\frac{\kappa\theta_y(v-1)}{v}+\left(\frac{\kappa}{c_v}-\frac{\mu}{v}\right)uu_y\right)dyd\tau\right|\notag\\
		&\hspace{0.4cm}+\left|\int_{t-1}^{t-2\nu_0}\int_\bbr \partial_y\bbg_{22}(x-y,t-\tau)\left(\frac{K(v-1)^2}{v}+\frac{K(\theta-1)(1-v)}{v}-\frac{Ku^2}{2c_v}+\frac{\mu u_y(v-1)}{v}\right)dyd\tau\right|\notag\\
		&\hspace{0.4cm}+\left|\int_{t-1}^{t-2\nu_0}\int_\bbr \partial_y\bbg_{23}(x-y,t-\tau)\left(\left(\frac{K(\theta-1)+K(1-v)}{v}\right)u+\frac{\kappa\theta_y(v-1)}{v}+\left(\frac{\kappa}{c_v}-\frac{\mu}{v}\right)uu_y\right)dyd\tau\right|\notag\\
		&\leq O(1)\int_0^{\frac{t-1}{2}} \frac{1}{t-\tau}\frac{\delta^2}{\sqrt{\tau+1}}d\tau+O(1)\int_{\frac{t-1}{2}}^{t-1}\frac{1}{\sqrt{t-\tau}}\frac{\delta^2}{\sqrt{1+\tau}{\sqrt{\tau}}}d\tau+O(1)\int_{t-1}^{t-2\nu_0}\frac{1}{\sqrt{t-\tau}}\frac{\delta^2}{\sqrt{1+\tau}\sqrt{\tau}}d\tau\notag\\
		&\leq O(1)\frac{\delta^2}{\sqrt{1+t}}.\notag
	\end{align}
	Next, for $\mathcal{R}^u_2$, we recall the estimates of Green's function in \eqref{5-25} and \eqref{5-27}, and combine the estimates of Heat kernel to obtain for $t\geq 1$ that, 
	\begin{align}
		&\left|\mathcal{R}^u_2(x,t)\right|\label{5-33}\\
		&\leq \left|\int_{t-2\nu_0}^{t-\nu_0}\int_\bbr \partial_y\bbg_{22}(x-y,t-\tau)\left(\frac{K(v-1)^2}{v}+\frac{K(\theta-1)(1-v)}{v}+\frac{\mu u_y(v-1)}{v}\left(1-\mathcal{X}\right)-\frac{Ku^2}{2c_v}\right)dyd\tau\right|\notag\\
		&\hspace{0.4cm}+\left|\int_{t-2\nu_0}^{t-\nu_0}\int_\bbr \frac{1}{\nu_0}\mathcal{X}'\left(\frac{t-\tau}{\nu_0}\right)\left(\bbg_{22}(x-y;t-\tau)-H\left(x,t;y,\tau;\frac{\mu}{v}\right)\right)u(y,\tau)dyd\tau\right|\notag\\
		&\hspace{0.4cm}+\left|\int_{t-2\nu_0}^{t-\nu_0}\int_\bbr \mathcal{X}\left(\frac{t-\tau}{\nu_0}\right)\left(K\partial_y\bbg_{23}(x-y,t-\tau)-\partial_y\bbg_{21}(x-y,t-\tau)\right)u(y,\tau)dyd\tau\right|\notag\\
		&\hspace{0.4cm}+\left|\int_{t-2\nu_0}^{t-\nu_0}\int_\bbr \mathcal{X}\left(\frac{t-\tau}{\nu_0}\right)\left(H_y\left(x,t;y,\tau;\frac{\mu}{v}\right)-\partial_y\bbg_{22}(x-y;t-\tau)\right)\frac{K(\theta-v)}{v}dyd\tau\right|\notag\\
		&\hspace{0.4cm}+\left|\int_{t-2\nu_0}^{t-\nu_0}\int_\bbr \partial_y\bbg_{23}(x-y;t-\tau)\left(\frac{\kappa\theta_y(v-1)}{v}+\frac{Ku(\theta-v)}{v}+\left(\frac{\kappa}{c_v}-\frac{\mu}{v}\right)uu_y\right)dyd\tau\right|\notag\\
		&\leq O(1)\int_{t-2\nu_0}^{t-\nu_0}\frac{1}{\sqrt{t-\tau}}\frac{\delta^2}{\sqrt{1+\tau}\sqrt{\tau}}d\tau+O(1)\int_{t-2\nu_0}^{t-\nu_0}\frac{\delta}{\sqrt{1+\tau}}d\tau+O(1)\int_{t-2\nu_0}^{t-\nu_0}\frac{1}{\sqrt{t-\tau}}\frac{\delta}{\sqrt{1+\tau}}d\tau\notag\\
		&\hspace{0.4cm}+O(1)\int_{t-2\nu_0}^{t-\nu_0}\int_\bbr \frac{1}{\nu_0}\left|H\left(x,t;y,\tau;\mu\right)-H\left(x,t;y,\tau;\frac{\mu}{v}\right)\right| |u(y,\tau)|dyd\tau\notag\\
		&\hspace{0.4cm}+O(1)\int_{t-2\nu_0}^{t-\nu_0}\int_\bbr \frac{1}{\nu_0}\left|\int_\tau^t\int_\bbr H(z,s;y,\tau;\mu)\left(\partial_z\bbg_{21}(x-z,t-s)\right)dzds\right||u(y,\tau)|dyd\tau\notag\\
		&\hspace{0.4cm}+O(1)\int_{t-2\nu_0}^{t-\nu_0}\int_\bbr \frac{1}{\nu_0}\left|\int_\tau^t\int_\bbr H(z,s;y,\tau;\mu)\left(K\partial_z\bbg_{23}(x-z,t-s)\right)dzds\right||u(y,\tau)|dyd\tau\notag\\
		&\leq O(1)\frac{\delta^2}{\sqrt{1+t}}+O(1)(\nu_0+\sqrt{\nu_0})\frac{\delta}{\sqrt{t+1}}\notag\\
		&\hspace{0.4cm}+O(1)\int_{t-2\nu_0}^{t-\nu_0}\int_\bbr \frac{1}{\nu_0}\|\mu-\frac{\mu}{v}\|_{L^\infty}\frac{1}{\sqrt{t-\tau}}e^{-\frac{(x-y)^2}{C(t-\tau)}} |u(y,\tau)|dyd\tau\notag\\
		&\leq O(1)\left(\sqrt{\nu_0}\delta+\delta^2\right)\frac{1}{\sqrt{1+t}}.\notag
	\end{align}
	With similar arguments, we have the estimates of $\mathcal{R}^u_3$ as follows,
	\begin{align}
		\left|\mathcal{R}^u_3\right|
		&\leq\left|\int_{t-\nu_0}^t\int_\bbr K\partial_y\bbg_{22}(x-y,t-\tau)(v(y,\tau)-1)dyd\tau\right|\label{5-34}\\
		&\hspace{0.4cm}+\left|\int_{t-\nu_0}^t\int_\bbr\partial_y\bbg_{21}(x-y,t-\tau)u(y,\tau)dyd\tau\right|\notag\\
		&\hspace{0.4cm}+\left|\int_{t-\nu_0}^t\int_\bbr H_y\left(x,t;y,\tau;\frac{\mu}{v}\right)\left(\frac{K\theta(y,\tau)-Kv(y,\tau)}{v(y,\tau)}\right)dyd\tau\right|\notag\\
		&\hspace{0.4cm}+\left|\int_{t-\nu_0}^t\int_\bbr \partial_y\bbg_{22}(x-y,t-\tau)\left(K(\theta-1)+\frac{Ku^2}{2c_v}\right)dyd\tau\right|\notag\\
		&\hspace{0.4cm}+\left|\int_{t-\nu_0}^t\int_\bbr \partial_y\bbg_{23}(x-y,t-\tau)\left((pu)+\frac{\kappa(v-1)\theta_y}{v}+\left(\frac{\kappa}{c_v}-\frac{\mu}{v}\right)uu_y\right)dyd\tau\right|\notag\\
		&\leq O(1)\left(\sqrt{\nu_0}\delta+\nu_0\delta+\sqrt{\nu_0}\delta^2\right)\frac{1}{\sqrt{1+t}}.\notag
	\end{align}
	Now, 
	we combine \eqref{5-31}, \eqref{5-32}, \eqref{5-33} and \eqref{5-34} to obtain the $L^\infty$ estimate of $u$ as follows,
	\bq\label{5-35}
	\|u(\cdot,t)\|_{L^\infty}\leq \frac{C(\nu_0)\delta^*}{\sqrt{1+t}}+O(1)\left(\frac{\sqrt{\nu_0}\delta+\delta^2}{\sqrt{1+t}}\right),\quad t\geq 1.
	\eq
	Similar as before, we can simply prove \eqref{5-35} also holds for $t\leq 1$, and we omit the details of the proof.
	
\end{proof}

\begin{lemma}\label{L5.6}
	$(\mathbf{\|u_x(\cdot,t)\|_{L^1},\ \|\sqrt{t}u_x(\cdot,t)\|_{L^\infty}})$ Let $(v,u,E)$ be the constructed solution, $C_\sharp$, $t_\sharp$ and $\delta$ be the associated parameters  in Proposition \ref{prop:local}. We further suppose the properties \eqref{5-19} hold for the solution.
	Then, $u(x,t)$ has the following first order estimates for $t\geq t_\sharp$,
	\begin{equation*}
		\left\{
		\begin{aligned}
			&\|u_x(\cdot,t)\|_{L^1}\leq C(\nu_0)\delta^*+O(1)\frac{|\log(\nu_0)|}{\sqrt{\nu_0}}\delta^2+O(1)\sqrt{\nu_0}\delta,\\
			&\|\sqrt{t}u_x(\cdot,t)\|_{L^\infty}\leq C(\nu_0)\delta^*+O(1)\frac{|\log(\nu_0)|}{\sqrt{\nu_0}}\delta^2+O(1)\sqrt{\nu_0}\delta.
		\end{aligned}
		\right.
	\end{equation*}
	
\end{lemma}

\begin{proof}
	\noindent $\bullet$ ($\mathbf{\|u_x(\cdot,t)\|_{L^\infty}}$) For the higher order estimates, we have the following representation,
	\bq\label{5-37}
	\begin{aligned}
		u_x(x,t)=&\int_\bbr\partial_x\bbg_{21}(x-y,t)(v(y,0)-1)dy+\int_\bbr \partial_x\bbg_{22}(x-y,t)u(y,0)dy\\
		&+\int_\bbr \partial_x\bbg_{23}(x-y,t)\left(E(y,0)-c_v\right)dy+\sum_{i=1}^3\partial_x\mathcal{R}^u_i(x,t).
	\end{aligned}
	\eq
	Similar as lower order estimates, we only need to consider the case $t\geq t_\sharp$. Then, for the first three homogeneous terms in \eqref{5-37}, we can directly apply the estimates in \eqref{5-21} and \eqref{5-21-1} to obtain the following estimates,
	\begin{equation}\label{5-38}
		\left\{
		\begin{aligned}
			&\left|\int_\bbr\partial_x\bbg_{21}(x-y,t)(v(y,0)-1)dy\right|\leq C(\nu_0)\left(\sqrt{t}e^{-\frac{Kt}{\mu}}+1\right)\frac{\delta^*}{\sqrt{t}}\leq \frac{C(\nu_0)\delta^*}{\sqrt{t}},\\
			&\left|\int_\bbr\partial_x\bbg_{22}(x-y,t)u(y,0)dy\right|\leq \frac{C(\nu_0)\delta^*}{\sqrt{t}},\\
			&\left|\int_\bbr \partial_x\bbg_{23}(x-y,t)\left(E(y,0)-c_v\right)dy\right|\leq \frac{C(\nu_0)\delta^*}{\sqrt{t}}.
		\end{aligned}
		\right.
	\end{equation}
	Next, for the estimates of $\partial_x\mathcal{R}^u_j$, it's obviously, we need higher order estimates of the Green's function $\bbg$. In fact, according to Lemma \ref{L5.4}, we have the following estimates,
	\begin{align}
		&|\partial_{xx}\bbg_{21}(x,t)+\frac{K}{\mu}e^{-\frac{Kt}{\mu}}\delta'(x)|\label{5-39}\\
		&\leq \sum_{j=1}^3\frac{O(1)e^{-\frac{(x+\beta_jt)^2}{4Ct}}}{t^{\frac{3}{2}}}+\sum_{j=1}^3\frac{O(1)e^{-\frac{(x+\beta_jt)^2}{4Ct}}}{t^2}+\sum_{j=1}^3\frac{O(1)e^{-\frac{(x+\beta_jt)^2}{4Ct}}}{t^{\frac{5}{2}}}+O(1)e^{-\sigma_0^*t-\sigma_0|x|},\quad t\geq 1,\notag\\
		&|\partial_{xx}\bbg_{22}(x,t)+\frac{K}{\mu^2}e^{-\frac{Kt}{\mu}}\delta(x)|\label{5-40}\\
		&\leq \sum_{j=1}^3\frac{O(1)e^{-\frac{(x+\beta_jt)^2}{4Ct}}}{t^{\frac{3}{2}}}+\sum_{j=1}^3\frac{O(1)e^{-\frac{(x+\beta_jt)^2}{4Ct}}}{t^2}+\sum_{j=1}^3\frac{O(1)e^{-\frac{(x+\beta_jt)^2}{4Ct}}}{t^{\frac{5}{2}}}+O(1)e^{-\sigma_0^*t-\sigma_0|x|},\quad t\geq 1,\notag\\
		&|\partial_{xx}\bbg_{23}(x,t)|\label{5-41}\\
		&\leq \sum_{j=1}^3\frac{O(1)e^{-\frac{(x+\beta_jt)^2}{4Ct}}}{t^{\frac{3}{2}}}+\sum_{j=1}^3\frac{O(1)e^{-\frac{(x+\beta_jt)^2}{4Ct}}}{t^2}+\sum_{j=1}^3\frac{O(1)e^{-\frac{(x+\beta_jt)^2}{4Ct}}}{t^{\frac{5}{2}}}+O(1)e^{-\sigma_0^*t-\sigma_0|x|},\quad t\geq 1,\notag\\
		&|\partial_{xx}\bbg_{21}(x,t)+\frac{K}{\mu}e^{-\frac{Kt}{\mu}}\delta'(x)|\leq \frac{O(1)e^{-\frac{x^2}{4Ct}}}{t}+O(1)e^{-\sigma_0^*t-\sigma_0|x|}+O(1)te^{-\sigma_0 |x|},\quad t\leq 1,\label{5-42}\\
		&|\partial_{xx}\bbg_{22}(x,t)+\frac{K}{\mu^2}e^{-\frac{Kt}{\mu}}\delta(x)|\leq \frac{O(1)e^{-\frac{x^2}{4Ct}}}{t^{\frac{3}{2}}}+\frac{O(1)e^{-\frac{x^2}{4Ct}}}{\sqrt{t}}+O(1)e^{-\sigma_0^*t-\sigma_0|x|}+O(1)te^{-\sigma_0 |x|},\quad t\leq 1,\label{5-43}\\
		&|\partial_{xx}\bbg_{23}(x,t)|\leq \frac{O(1)e^{-\frac{x^2}{4Ct}}}{t}+O(1)e^{-\sigma_0^*t-\sigma_0|x|}+O(1)te^{-\sigma_0 |x|},\quad t\leq 1.\label{5-44}
	\end{align}
	
	For the remainder $\mathcal{R}^u_1$, we follow the previous arguments, assume $t\geq 1$, and apply \eqref{5-19}, \eqref{5-40}, \eqref{5-41}, \eqref{5-43} and \eqref{5-44} to obtain the estimates as follows,
	\begin{align}
		&\left|\partial_x\mathcal{R}^u_1\right|\label{5-45}\\
		&=\left|\int_0^{\frac{t-1}{2}}\int_\bbr \partial_{xy}\bbg_{22}(x-y,t-\tau)\left(\frac{K(v-1)^2}{v}+\frac{K(\theta-1)(1-v)}{v}-\frac{Ku^2}{2c_v}+\frac{\mu u_y(v-1)}{v}\right)dyd\tau\right|\notag\\
		&\hspace{0.4cm}+\left|\int_0^{\frac{t-1}{2}}\int_\bbr \partial_{xy}\bbg_{23}(x-y,t-\tau)\left(\left(\frac{K(\theta-1)+K(1-v)}{v}\right)u+\frac{\kappa\theta_y(v-1)}{v}+\left(\frac{\kappa}{c_v}-\frac{\mu}{v}\right)uu_y\right)dyd\tau\right|\notag\\
		&\hspace{0.4cm}+\left|\int_{\frac{t-1}{2}}^{t-1}\int_\bbr \partial_{xy}\bbg_{22}(x-y,t-\tau)\left(\frac{K(v-1)^2}{v}+\frac{K(\theta-1)(1-v)}{v}-\frac{Ku^2}{2c_v}+\frac{\mu u_y(v-1)}{v}\right)dyd\tau\right|\notag\\
		&\hspace{0.4cm}+\left|\int_{\frac{t-1}{2}}^{t-1}\int_\bbr \partial_{xy}\bbg_{23}(x-y,t-\tau)\left(\left(\frac{K(\theta-1)+K(1-v)}{v}\right)u+\frac{\kappa\theta_y(v-1)}{v}+\left(\frac{\kappa}{c_v}-\frac{\mu}{v}\right)uu_y\right)dyd\tau\right|\notag\\
		&\hspace{0.4cm}+\left|\int_{t-1}^{t-2\nu_0}\int_\bbr \partial_{xy}\bbg_{22}(x-y,t-\tau)\left(\frac{K(v-1)^2}{v}+\frac{K(\theta-1)(1-v)}{v}-\frac{Ku^2}{2c_v}+\frac{\mu u_y(v-1)}{v}\right)dyd\tau\right|\notag\\
		&\hspace{0.4cm}+\left|\int_{t-1}^{t-2\nu_0}\int_\bbr \partial_{xy}\bbg_{23}(x-y,t-\tau)\left(\left(\frac{K(\theta-1)+K(1-v)}{v}\right)u+\frac{\kappa\theta_y(v-1)}{v}+\left(\frac{\kappa}{c_v}-\frac{\mu}{v}\right)uu_y\right)dyd\tau\right|\notag\\
		&\leq O(1)\int_0^{\frac{t-1}{2}} \left(\frac{1}{\sqrt{\tau}}e^{-\frac{K}{\mu}(t-\tau)}+\frac{1}{(t-\tau)^{\frac{3}{2}}}+e^{-\sigma^*_0(t-\tau)}\right)\frac{\delta^2}{\sqrt{\tau+1}}d\tau\notag\\
		&\hspace{0.4cm}+O(1)\int_{\frac{t-1}{2}}^{t-1}\left(e^{-\frac{K}{\mu}(t-\tau)}+\frac{1}{(t-\tau)}+e^{-\sigma^*_0(t-\tau)}\right)\frac{\delta^2}{\sqrt{\tau}\sqrt{1+\tau}}d\tau\notag\\
		&\hspace{0.4cm}+O(1)\int_{t-1}^{t-2\nu_0}\left(e^{-\frac{K}{\mu}(t-\tau)}+\frac{1}{(t-\tau)}+1+e^{-\sigma^*_0(t-\tau)}\right)\frac{\delta^2}{\sqrt{\tau}\sqrt{1+\tau}}d\tau\notag\\
		&\leq O(1)\frac{(1+|\log(\nu_0)|)\delta^2}{\sqrt{t}},\notag
	\end{align}
	Note, if $t\leq 1$, we only have the integration is from $0$ to $t-2\nu_0$. Then we have $1+\tau\geq 1\geq t$, and the above estimates still hold. Next, for $\mathcal{R}^u_2$, we recall the estimates of the heat kernel and comparison principle in Lemma \ref{L2.2}, and combine the estimates \eqref{5-27}, \eqref{5-42}, \eqref{5-43} and \eqref{5-44} to obtain that, 
	\begin{align}
		&\left|\partial_x\mathcal{R}^u_2(x,t)\right|\notag\\
		&\leq \left|\int_{t-2\nu_0}^{t-\nu_0}\int_\bbr \partial_{xy}\bbg_{22}(x-y,t-\tau)\left(\frac{K(v-1)^2}{v}+\frac{K(\theta-1)(1-v)}{v}+\frac{\mu u_y(v-1)}{v}\left(1-\mathcal{X}\right)-\frac{Ku^2}{2c_v}\right)dyd\tau\right|\notag\\
		&\hspace{0.4cm}+\left|\int_{t-2\nu_0}^{t-\nu_0}\int_\bbr \frac{1}{\nu_0}\mathcal{X}'\left(\frac{t-\tau}{\nu_0}\right)\left(\partial_x\bbg_{22}(x-y;t-\tau)-H_x\left(x,t;y,\tau;\frac{\mu}{v}\right)\right)u(y,\tau)dyd\tau\right|\notag\\
		&\hspace{0.4cm}+\left|\int_{t-2\nu_0}^{t-\nu_0}\int_\bbr \mathcal{X}\left(\frac{t-\tau}{\nu_0}\right)\left(K\partial_{xy}\bbg_{23}(x-y,t-\tau)-\partial_{xy}\bbg_{21}(x-y,t-\tau)\right)u(y,\tau)dyd\tau\right|\notag\\
		&\hspace{0.4cm}+\left|\int_{t-2\nu_0}^{t-\nu_0}\int_\bbr \mathcal{X}\left(\frac{t-\tau}{\nu_0}\right)\left(H_{xy}\left(x,t;y,\tau;\frac{\mu}{v}\right)-\partial_{xy}\bbg_{22}(x-y;t-\tau)\right)\frac{K(\theta-v)}{v}dyd\tau\right|\notag\\
		&\hspace{0.4cm}+\left|\int_{t-2\nu_0}^{t-\nu_0}\int_\bbr \partial_{xy}\bbg_{23}(x-y;t-\tau)\left(\frac{\kappa\theta_y(v-1)}{v}+\frac{Ku(\theta-v)}{v}+\left(\frac{\kappa}{c_v}-\frac{\mu}{v}\right)uu_y\right)dyd\tau\right|\notag\\
		&\equiv T_1 +T_2 +T_3 +T_4+T_5. \notag
	\end{align}
	One applies \eqref{5-42},\eqref{5-43} and \eqref{5-44} to obtain
	\begin{align*}
		&T_1+T_3+T_5\\
		& \leq O(1)\int_{t-2\nu_0}^{t-\nu_0}e^{-\frac{K}{\mu}(t-\tau)}\frac{\delta^2}{\sqrt{1+\tau}\sqrt{\tau}}d\tau+O(1)\int_{t-2\nu_0}^{t-\nu_0} \frac{1}{t-\tau}\frac{\delta^2}{1+\tau}d\tau+O(1)\int_{t-2\nu_0}^{t-\nu_0}\left(1+\frac{1}{\sqrt{t-\tau}}\right)\frac{\delta}{\sqrt{1+\tau}}d\tau\\
		& \leq O(1) \left(	\frac{\nu_0 \delta^2}{t+1}+ \frac{\delta^2}{t+1}+ \frac{\sqrt{\nu_0}\delta}{\sqrt{t+1}}	\right) \leq O(1) \frac{\sqrt{\nu_0}\delta +\delta^2}{\sqrt{t+1}}.
	\end{align*}
	For $T_2$, thanks to \eqref{5-27}, \eqref{5-42}--\eqref{5-44} and Lemma \ref{L2.2}, one has
	\begin{align*}
		T_2 & \leq  O(1)\int_{t-2\nu_0}^{t-\nu_0}\int_\bbr \frac{1}{\nu_0}\left|H_x\left(x,t;y,\tau;\mu\right)-H_x\left(x,t;y,\tau;\frac{\mu}{v}\right)\right| |u(y,\tau)|dyd\tau\notag\\
		&\hspace{0.4cm}+O(1)\int_{t-2\nu_0}^{t-\nu_0}\int_\bbr \frac{1}{\nu_0}\left|\int_\tau^t\int_\bbr H_z(z,s;y,\tau;\mu)\left(\partial_x\bbg_{21}(x-z,t-s)\right)dzds\right||u(y,\tau)|dyd\tau\notag\\
		&\hspace{0.4cm}+O(1)\int_{t-2\nu_0}^{t-\nu_0}\int_\bbr \frac{1}{\nu_0}\left|\int_\tau^t\int_\bbr H_z(z,s;y,\tau;\mu)\left(K\partial_x\bbg_{23}(x-z,t-s)\right)dzds\right||u(y,\tau)|dyd\tau\notag\\
		&\leq O(1)\int_{t-2\nu_0}^{t-\nu_0}\int_\bbr \frac{1}{\nu_0} \frac{e^{-\frac{(x-y)^2}{C_{*}(t-\tau)}	}}{t-\tau} \left[\left|\log(t-\tau)\right| \sup_{t-2\nu_0 < s<t}\|v(\cdot,s)-1\|_{L^\infty} + \sup_{t-2\nu_0 < s<t}\|v(\cdot,s)-1\|_{BV}\right] \frac{\delta}{\sqrt{1+\tau}}dyd\tau\notag\\
		&\hspace{0.4cm}+O(1)\int_{t-2\nu_0}^{t-\nu_0}\int_\bbr \frac{1}{\nu_0} \frac{e^{-\frac{(x-y)^2}{C_{*}(t-\tau)}	}}{t-\tau}\sqrt{t-\tau} \left[\sup_{t-2\nu_0 < s<t}\|\sqrt{s}u_x(\cdot,s)\|_{L^\infty} + \sup_{t-2\nu_0 < s<t}\|v(\cdot,s)-1\|_{L^1}\right] \frac{\delta}{\sqrt{1+\tau}}dyd\tau\notag\\
		&\hspace{0.4cm}+O(1)\int_{t-2\nu_0}^{t-\nu_0}\frac{1}{\nu_0}\int_{\tau}^t\frac{1}{\sqrt{s-\tau}}ds\frac{\delta}{\sqrt{1+\tau}}d\tau\notag\\
		&\leq O(1) \frac{|\log \nu_0| \delta^2}{\sqrt{\nu_0} \sqrt{t+1}} + O(1) \frac{\sqrt{\nu_0}\delta}{\sqrt{t+1}}.
	\end{align*}
	Similarly by \eqref{5-27},  one gets
	\begin{align*}
		T_4 & \leq O(1)\int_{t-2\nu_0}^{t-\nu_0}\int_\bbr \left|H_{xy}\left(x,t;y,\tau;\mu\right)-H_{xy}\left(x,t;y,\tau;\frac{\mu}{v}\right)\right| |\theta(y,\tau)-v(y,\tau)|dyd\tau\notag\\
		&\hspace{0.4cm}+O(1)\int_{t-2\nu_0}^{t-\nu_0}\int_\bbr \left|\int_\tau^t\int_\bbr H_y(z,s;y,\tau;\mu)\left(\partial_{xz}\bbg_{21}(x-z,t-s)-\frac{K}{\mu}e^{-\frac{K(t-s)}{\mu}}\delta'(x-z)\right)dzds\right||\theta(y,\tau)-v(y,\tau)|dyd\tau\notag\\
		&\hspace{0.4cm}+O(1)\int_{t-2\nu_0}^{t-\nu_0}\left|\int_\bbr \int_\tau^t\int_\bbr H_y(z,s;y,\tau;\mu)\frac{K}{\mu}e^{-\frac{K(t-s)}{\mu}}\delta'(x-z)dzds\left(\theta(y,\tau)-v(y,\tau)\right)dy\right|d\tau\notag\\
		&\hspace{0.4cm}+O(1)\int_{t-2\nu_0}^{t-\nu_0}\int_\bbr \left|\int_\tau^t\int_\bbr H_y(z,s;y,\tau;\mu)\left(K\partial_{xz}\bbg_{23}(x-z,t-s)\right)dzds\right||\theta(y,\tau)-v(y,\tau)|dyd\tau\notag\\
		& \leq O(1)\int_{t-2\nu_0}^{t-\nu_0}\int_\bbr  \frac{e^{-\frac{(x-y)^2}{C_{*}(t-\tau)}	}}{(t-\tau)^{\frac{3}{2}}} \left[\left|\log(t-\tau)\right| \sup_{t-2\nu_0 < s<t}\|v(\cdot,s)-1\|_{L^\infty} + \sup_{t-2\nu_0 < s<t}\|v(\cdot,s)-1\|_{BV}\right] \frac{\delta}{\sqrt{1+\tau}}dyd\tau\notag\\
		&\hspace{0.4cm}+O(1)\int_{t-2\nu_0}^{t-\nu_0}\int_\bbr \frac{e^{-\frac{(x-y)^2}{C_{*}(t-\tau)}	}}{(t-\tau)^{\frac{3}{2}}} \sqrt{t-\tau}\left[\sup_{t-2\nu_0 < s<t}\|\sqrt{s}u_x(\cdot,s)\|_{L^\infty} + \sup_{t-2\nu_0 < s<t}\|v(\cdot,s)-1\|_{L^1}\right] \frac{\delta}{\sqrt{1+\tau}}dyd\tau\notag\\
		&\hspace{0.4cm}+O(1)\int_{t-2\nu_0}^{t-\nu_0}\left|\int_\bbr \int_\tau^t H_{xy}(x,s;y,\tau;\mu)e^{-\frac{K(t-s)}{\mu}}ds\left(\theta(y,\tau)-v(y,\tau)\right)dy\right|d\tau\notag\\
		&\hspace{0.4cm}+O(1)\int_{t-2\nu_0}^{t-\nu_0}\int_{\tau}^t\frac{1}{\sqrt{t-s}\sqrt{s-\tau}}ds\frac{\delta}{\sqrt{1+\tau}}d\tau\notag\\
		&\leq O(1) \frac{|\log \nu_0 |\delta^2}{\sqrt{t+1}}+ O(1) \frac{\nu_0 \delta}{\sqrt{1+t}}	\\
		&\hspace{0.4cm}+O(1)\int_{t-2\nu_0}^{t-\nu_0}\left|\int_\bbr \int_\tau^t H_{s}(x,s;y,\tau;\mu)e^{-\frac{K(t-s)}{\mu}}ds\left(\theta(y,\tau)-v(y,\tau)\right)dy\right|d\tau\notag\\
		&\leq O(1) \frac{|\log \nu_0 |\delta^2}{\sqrt{t+1}}+ O(1) \frac{\nu_0 \delta}{\sqrt{1+t}},
	\end{align*}
	where we used $H_{xy}(x,s;y,\tau;\mu)=-H_s(x,s;y,\tau;\mu)$ and integration by parts with respect to $s$ since $\mu$ is constant heat conductivity. Combing the estimates of $T_1$--$T_5$, one obtains
	\begin{equation} \label{5-46}
		\left|\partial_{x} \mathcal{R}^{u}_{2}\right| \leq O(1) \frac{|\log \nu_0|}{\sqrt{\nu_0}} \frac{\delta^2}{\sqrt{t+1}} + O(1) \frac{\sqrt{\nu_0} \delta}{\sqrt{t+1}}.
	\end{equation}

		Next, in order to deal with $\mathcal{R}^u_3$, we  apply Corollary \ref{Cor5.1} to obtain the following estimates, 
		\begin{align}
			&\left|\partial_x\mathcal{R}^u_3\right|\label{5-47}\\
			&\leq \int_{-\infty}^x \int_{-\infty}^y\frac{K}{\mu}\left|\partial_x\bbg_{22}(x-z,\nu_0)-H_x\left(x,t;z,t-\nu_0;\frac{\mu}{v}\right)\right|dz\left|v(y,t-\nu_0)-\theta(y,t-\nu_0)\right|dy\notag\\
			&\hspace{0.4cm}+\int_{t-\nu_0}^t\int_{-\infty}^x \int_{-\infty}^y\frac{K}{\mu}\left|\partial_x\bbg_{22}(x-z,t-\tau)-H_x\left(x,t;z,\tau;\frac{\mu}{v}\right)\right|dz|v_\tau(y,\tau)|dyd\tau\notag\\
			&\hspace{0.4cm}+\int_{t-\nu_0}^t\int_{-\infty}^x \int_{-\infty}^y\frac{K}{\mu}\left|\partial_x\bbg_{22}(x-z,t-\tau)-H_x\left(x,t;z,\tau;\frac{\mu}{v}\right)\right|dz\left(\left|\frac{p}{c_v}u_y\right|+\left|\frac{\mu}{c_vv}(u_y)^2\right|\right)dyd\tau\notag\\
			&\hspace{0.4cm}+\int_x^{+\infty} \int_y^{+\infty}\frac{K}{\mu}\left|\partial_x\bbg_{22}(x-z,\nu_0)-H_x\left(x,t;z,t-\nu_0;\frac{\mu}{v}\right)\right|dz|v(y,t-\nu_0)-\theta(y,t-\nu_0)|dy\notag\\
			&\hspace{0.4cm}+\int_{t-\nu_0}^t\int_x^{+\infty} \int_y^{+\infty}\frac{K}{\mu}\left|\partial_x\bbg_{22}(x-z,t-\tau)-H_x\left(x,t;z,\tau;\frac{\mu}{v}\right)\right|dz|v_\tau(y,\tau)|dyd\tau\notag\\
			&\hspace{0.4cm}+\int_{t-\nu_0}^t\int_x^{+\infty} \int_y^{+\infty}\frac{K}{\mu}\left|\partial_x\bbg_{22}(x-z,t-\tau)-H_x\left(x,t;z,\tau;\frac{\mu}{v}\right)\right|dz\left(\left|\frac{p}{c_v}u_y\right|+\left|\frac{\mu}{c_vv}(u_y)^2\right|\right)dyd\tau\notag\\
			&\hspace{0.4cm}+\int_{t-\nu_0}^t\int_\bbr \frac{K}{\mu}\left|\partial_x\bbg_{21}(x-y,t-\tau)-\frac{K}{\mu}e^{-K(t-\tau)/\mu}\delta(x-y)\right||v(y,\tau)-\theta(y,\tau)|dyd\tau\notag\\
			&\hspace{0.4cm}+\int_{t-\nu_0}^t\frac{K^2}{\mu^2}e^{-K(t-\tau)/\mu}|v(x,\tau)-\theta(x,\tau)|d\tau\notag\\
			&\hspace{0.4cm}+\int_{t-\nu_0}^t\int_\bbr \frac{K^2}{\mu}\left|\partial_x\bbg_{23}(x-y,t-\tau)\right||v(y,\tau)-\theta(y,\tau)|dyd\tau\notag\\
			&\hspace{0.4cm}+\int_{t-\nu_0}^t\int_\bbr\left|\partial_x\bbg_{21}(x-y,t-\tau)-\frac{K}{\mu}e^{-K(t-\tau)/\mu}\delta(x-y)\right||u_y(y,\tau)|dyd\tau\notag\\
			&\hspace{0.4cm}+\int_{t-\nu_0}^t\frac{K}{\mu}e^{-K(t-\tau)/\mu}|u_x(x,\tau)|d\tau\notag\\
			&\hspace{0.4cm}+\int_{t-\nu_0}^t\int_\bbr \left|\partial_x\bbg_{22}(x-y,t-\tau)\right|\frac{K|u||u_y|}{c_v}dyd\tau\notag\\
			&\hspace{0.4cm}+\int_{t-\nu_0}^t\int_\bbr \left|\partial_{xy}\bbg_{23}(x-y,t-\tau)\right|\left(|pu|+\frac{\kappa|v-1||\theta_y|}{|v|}+\left(\frac{\kappa}{c_v}+\frac{\mu}{|v|}\right)|u||u_y|\right)dyd\tau.\notag
		\end{align}
		Next, for the first six terms in \eqref{5-47}, we should take advantage of \eqref{5-27} and comparison principle Lemma \ref{L2.2}. Actually, we have the following estimates, 
		\begin{align}
			&\int_{-\infty}^x \int_{-\infty}^y\frac{K}{\mu}\left|\partial_x\bbg_{22}(x-z,\nu_0)-H_x\left(x,t;z,t-\nu_0;\frac{\mu}{v}\right)\right|dzdy\label{5-48}\\
			&\leq O(1)\int_{-\infty}^x \int_{-\infty}^y\left|H_x(x,t;z,t-\nu_0;\mu)-H_x\left(x,t;z,t-\nu_0;\frac{\mu}{v}\right)\right|dzdy\notag\\
			&\hspace{0.4cm}+O(1)\int_{-\infty}^x \int_{-\infty}^y\left|\int_{t-\nu_0}^t\int_\bbr H_w(w,s;z,t-\nu_0;\mu)\left(\partial_x\bbg_{21}(x-w,t-s)-\frac{K}{\mu}e^{-K(t-s)/\mu}\delta(x-w)\right)dwds\right|dzdy\notag\\
			&\hspace{0.4cm}+O(1)\int_{-\infty}^x \int_{-\infty}^y\left|\int_{t-\nu_0}^t H_x(x,t;z,s;\mu)ds\right|dzdy\notag\\
			&\hspace{0.4cm}+O(1)\int_{-\infty}^x \int_{-\infty}^y\left|\int_{t-\nu_0}^t\int_\bbr H_w(w,s;z,t-\nu_0;\mu)\left(\partial_x\bbg_{23}(x-w,t-s)\right)dwds\right|dzdy\notag\\
			&\leq O(1)\int_{-\infty}^x \int_{-\infty}^y  \frac{e^{-\frac{(x-z)^2}{C_{*}(\nu_0)}}}{\nu_0}dzdy|\log(\nu_0)|\delta+O(1)\int_{t-\nu_0}^t \int_{-\infty}^x \int_{-\infty}^y\frac{e^{-\frac{(x-z)^2}{C_*(s-t+\nu_0)}}}{s-t+\nu_0}dzdyds\notag\\
			&\hspace{0.4cm}+O(1) \int_{-\infty}^{x} \int_{-\infty}^{y} \int_{t-\nu_0}^{t} \int_{\R}	\frac{e^{-\frac{(w-z)^2}{C_*(s-t+\nu_0)}}}{s-t+\nu_0}\left(		\frac{e^{-\frac{(x-w)^2}{C_*(t-s)}}}{\sqrt{t-s}} + e^{-\sigma_0 |x-w|}	\right)	dw ds dz dy	 \notag	\\
			&\leq O(1)|\log(\nu_0)|\delta+O(1)\nu_0 + O(1)\sqrt{\nu_0}.\notag
		\end{align}
		Similarly as \eqref{5-48}, we have the estimates of the other part of comparison principle as follows,
		\bq\label{5-49}
		\left\{\begin{aligned}
			&\int_x^{+\infty} \int_y^{+\infty}\frac{K}{\mu}\left|\partial_x\bbg_{22}(x-z,\nu_0)-H_x\left(x,t;z,t-\nu_0;\frac{\mu}{v}\right)\right|dzdy\\
			&\hspace{6cm}\leq O(1)|\log(\nu_0)|\delta+O(1)\nu_0+O(1)\sqrt{\nu_0},\\
			&\int_x^{+\infty} \int_y^{+\infty}\frac{K}{\mu}\left|\partial_x\bbg_{22}(x-z,t-\tau)-H_x\left(x,t;z,\tau;\frac{\mu}{v}\right)\right|dzdy\\
			&\hspace{6cm}\leq O(1)|\log(t-\tau)|\delta+O(1)(t-\tau)+O(1)\sqrt{t-\tau},\\
			&\int_{-\infty}^x \int_{-\infty}^y\frac{K}{\mu}\left|\partial_x\bbg_{22}(x-z,t-\tau)-H_x\left(x,t;z,\tau;\frac{\mu}{v}\right)\right|dzdy\\
			&\hspace{6cm}\leq O(1)|\log(t-\tau)|\delta+O(1)(t-\tau)+O(1)\sqrt{t-\tau}.
		\end{aligned}
		\right.
		\eq
		Now, we combine the representation \eqref{5-47}, the comparison principle in \eqref{5-48} and \eqref{5-49}, and apply similar arguments as before to obtain the estimates of $\partial_x\mathcal{R}^u_3$ as follows,
		\begin{align}
			\left|\partial_x\mathcal{R}^u_3\right|
			&\leq O(1)\left(|\log(\nu_0)|\delta+\nu_0+\sqrt{\nu_0}\right)\frac{\delta}{\sqrt{1+t-\nu_0}}\label{5-50}\\
			&\hspace{0.4cm}+O(1)\left(|\log(\nu_0)|\nu_0\delta+\nu_0^2+(\nu_0)^{\frac{3}{2}}\right)\frac{\delta}{\sqrt{t-\nu_0}}\notag\\
			&\hspace{0.4cm}+O(1)\left(|\log(\nu_0)|\nu_0\delta+\nu_0^2+(\nu_0)^{\frac{3}{2}}\right)\left(\frac{\delta}{\sqrt{t-\nu_0}}+\frac{\delta^2}{t-\nu_0}\right)\notag\\
			&\hspace{0.4cm}+O(1)\left(\nu_0\frac{\delta}{\sqrt{t-\nu_0}}+\sqrt{\nu_0}\frac{\delta^2}{\sqrt{1+t-\nu_0}\sqrt{t-\nu_0}}+\sqrt{\nu_0}\frac{\delta}{\sqrt{t-\nu_0}}\right)\notag.
		\end{align} 
		Now, we combine \eqref{5-37}, \eqref{5-38}, \eqref{5-45}, \eqref{5-46} and \eqref{5-50} to obtain the following estimates of $\|u(\cdot,t)\|_{L^\infty}$, 
		\begin{align}
			|u_x(x,t)|
			&\leq \frac{C(\nu_0)\delta^*}{\sqrt{t}}+O(1)\frac{\frac{|\log(\nu_0)|}{\sqrt{\nu_0}}\delta^2}{\sqrt{t-\nu_0}}+O(1)\frac{\sqrt{\nu_0}\delta}{\sqrt{t-\nu_0}}\label{5-51}\\
			&\leq \frac{C(\nu_0)\delta^*}{\sqrt{t}}+O(1)\frac{\frac{|\log(\nu_0)|}{\sqrt{\nu_0}}\delta^2}{\sqrt{t}}+O(1)\frac{\sqrt{\nu_0}\delta}{\sqrt{t}},\notag
		\end{align}
		where the last inequality holds due to the fact that $4\nu_0\leq t_\sharp\leq t$, and thus $\frac{\sqrt{t}}{\sqrt{t-\nu_0}}\leq \frac{2}{\sqrt{3}}$. Moreover, similar as zeroth order estimates, we can follow the same arguments to conclude \eqref{5-51} also holds for $t\leq 1$, and we omit the details.\newline
		
		\noindent $\bullet$ ($\mathbf{\|u_x(\cdot,t)\|_{L^1}}$) For the $L^1$ estimates of $u_x$, we can follow the arguments as for $L^\infty$ estimates. First, the estimates of the homogeneous terms can be simply obtained as follows,
		\bq\label{5-52}
		\left\{
		\begin{aligned}
			&\int_\bbr\left|\int_\bbr\partial_x\bbg_{21}(x-y,t)(v(y,0)-1)dy\right|dx\leq C(\nu_0)\delta^*,\\
			&\int_\bbr\left|\int_\bbr\partial_x\bbg_{22}(x-y,t)u(y,0)dy\right|dx\leq C(\nu_0)\delta^*,\\
			&\int_\bbr\left|\int_\bbr \partial_x\bbg_{23}(x-y,t)\left(E(y,0)-c_v\right)dy\right|dx\leq C(\nu_0)\delta^*.
		\end{aligned}
		\right.
		\eq
		For the estimates of the remainders, we also apply similar estimates as before. Actually, we have for $t\geq 1$ that,
		\begin{align}
			\int_\bbr\left|\partial_x\mathcal{R}^u_1\right|dx
			&\leq O(1)\int_0^{\frac{t-1}{2}} \left(e^{-\frac{K}{2\mu}(t+1)}+\frac{1}{(t+1)}+e^{-\sigma^*_0(t+1)}\right)\frac{\delta^2}{\sqrt{\tau+1}}d\tau\label{5-53}\\
			&\hspace{0.4cm}+O(1)\int_{\frac{t-1}{2}}^{t-1}\left(\frac{\sqrt{t-\tau}e^{-\frac{K}{\mu}(t-\tau)}}{\sqrt{t-\tau}}+\frac{1}{(t-\tau)}+e^{-\sigma^*_0(t-\tau)}\right)\frac{\delta^2}{\sqrt{1+t}}d\tau\notag\\
			&\hspace{0.4cm}+O(1)\int_{t-1}^{t-2\nu_0}\left(\frac{\sqrt{t-\tau}e^{-\frac{K}{\mu}(t-\tau)}}{\sqrt{t-\tau}}+\frac{1}{(t-\tau)}+e^{-\sigma^*_0(t-\tau)}\right)\frac{\delta^2}{\sqrt{1+\tau}}d\tau\notag\\
			&\leq O(1)(1+|\log(\nu_0)|)\delta^2.\notag
		\end{align}
		Similar as before, when $t\leq 1$, we only have the third integration in \eqref{5-53}, which is from $0$ to $t-2\nu_0$. Thus, we immediately conclude \eqref{5-53} also holds for $t\leq 1$. Then, we omit the details and provide the estimates of $\|\mathcal{R}^u_2\|_{L^1}$ as follows,
		\bq\label{5-54}
		\begin{aligned}
			&\int_\bbr\left|\partial_x\mathcal{R}^u_2\right|dx\leq O(1)\frac{|\log(\nu_0)|}{\sqrt{\nu_0}}\delta^2+O(1)\sqrt{\nu_0}\delta.
		\end{aligned}
		\eq
		Next, for the estimates of $\|\partial_{x}\mathcal{R}^u_3\|_{L^1}$, as we have one more integration with respect to $x$, we can directly apply the representation \eqref{5-34} to obtain that,
		\begin{align}
			&\int_\bbr\left|\partial_x\mathcal{R}^u_3\right|dx\label{5-55}\\
			&\leq\int_\bbr\left|\int_{t-\nu_0}^t\int_\bbr K\partial_{xy}\bbg_{22}(x-y,t-\tau)(v(y,\tau)-1)dyd\tau\right|dx\notag\\
			&\hspace{0.4cm}+\int_\bbr\left|\int_{t-\nu_0}^t\int_\bbr\partial_{xy}\bbg_{21}(x-y,t-\tau)u(y,\tau)dyd\tau\right|dx\notag\\
			&\hspace{0.4cm}+\int_\bbr\left|\int_{t-\nu_0}^t\int_\bbr H_{xy}\left(x,t;y,\tau;\frac{\mu}{v}\right)\left(\frac{K\theta(y,\tau)-Kv(y,\tau)}{v(y,\tau)}\right)dyd\tau\right|dx\notag\\
			&\hspace{0.4cm}+\int_\bbr\left|\int_{t-\nu_0}^t\int_\bbr \partial_{xy}\bbg_{22}(x-y,t-\tau)\left(K(\theta-1)+\frac{Ku^2}{2c_v}\right)dyd\tau\right|dx\notag\\
			&\hspace{0.4cm}+\int_\bbr\left|\int_{t-\nu_0}^t\int_\bbr \partial_{xy}\bbg_{23}(x-y,t-\tau)\left((pu)+\frac{\kappa(v-1)\theta_y}{v}+\left(\frac{\kappa}{c_v}-\frac{\mu}{v}\right)uu_y\right)dyd\tau\right|dx\notag\\
			&\leq \int_\bbr\left|\int_{t-\nu_0}^t\int_{\bbr\setminus\mathscr{D}} K\partial_{x}\bbg_{22}(x-y,t-\tau)v_y(y,\tau)dyd\tau\right|dx\notag\\
			&\hspace{0.4cm}+\int_\bbr\left|\int_{t-\nu_0}^t K\sum_{z\in\mathscr{D}}\partial_{x}\bbg_{22}(x-z,t-\tau)v(z,\tau)\Big|_{z^-}^{z^+}d\tau\right|dx\notag\\
			&\hspace{0.4cm}+O(1)\int_\bbr\left|\int_{t-\nu_0}^t\int_\bbr\left(\frac{e^{-\frac{(x-y)^2}{4C(t-\tau)}}}{t-\tau}+e^{-\sigma_0^*(t-\tau)-\sigma_0|x-y|}+(t-\tau)e^{-\sigma_0 |x-y|}\right)u(y,\tau)dyd\tau\right|dx\notag\\
			&\hspace{0.4cm}+O(1)\int_\bbr\int_{t-\nu_0}^t|u_x(x,\tau)|d\tau dx\notag\\
			&\hspace{0.4cm}+\int_\bbr\left|\int_{t-\nu_0}^t\int_{\bbr\setminus\mathscr{D}} H_{x}\left(x,t;y,\tau;\frac{\mu}{v}\right)\left(\frac{K\theta_y(y,\tau)v(y,\tau)-K\theta(y,\tau)v_y(y,\tau)}{v^2(y,\tau)}\right)dyd\tau\right|dx\notag\\
			&\hspace{0.4cm}+\int_\bbr\left|\int_{t-\nu_0}^t\sum_{z\in\mathscr{D}}H_{x}\left(x,t;z,\tau;\frac{\mu}{v}\right)\frac{K\theta(z,\tau)}{v(z,\tau)}\Big|_{z=z^-}^{z=z^+}d\tau\right|dx\notag\\
			&\hspace{0.4cm}+\int_\bbr\left|\int_{t-\nu_0}^t\int_\bbr \partial_{x}\bbg_{22}(x-y,t-\tau)\left(K\theta_y+\frac{Kuu_y}{c_v}\right)dyd\tau\right|dx\notag\\
			&\hspace{0.4cm}+O(1)\int_\bbr\left|\int_{t-\nu_0}^t\int_\bbr \left(\frac{e^{-\frac{(x-y)^2}{4C(t-\tau)}}}{t-\tau}+e^{-\sigma_0 |x-y|}\right)\left((pu)+\frac{\kappa(v-1)\theta_y}{v}+\left(\frac{\kappa}{c_v}-\frac{\mu}{v}\right)uu_y\right)dyd\tau\right|dx\notag\\
			&\leq O(1)\int_{t-\nu_0}^t\int_{\bbr\setminus\mathscr{D}} \int_\bbr \left(\frac{e^{-\frac{(x-y)^2}{4C(t-\tau)}}}{t-\tau}+e^{-\sigma_0 |x-y|}\right)dx\left(|v_y(y,\tau)|+|\theta_y(y,\tau)|\right)dyd\tau \notag\\
			&\hspace{0.4cm}+\int_{t-\nu_0}^t K\sum_{z\in\mathscr{D}}\int_\bbr\left(\frac{e^{-\frac{(x-y)^2}{4C(t-\tau)}}}{t-\tau}+e^{-\sigma_0 |x-y|}\right)dx\left|v(z,\tau)\Big|_{z^-}^{z^+}\right|d\tau \notag\\
			&\hspace{0.4cm}+O(1)\int_{t-\nu_0}^t\int_\bbr\int_\bbr\left(\frac{e^{-\frac{(x-y)^2}{4C(t-\tau)}}}{t-\tau}+e^{-\sigma_0 |x-y|}\right)dx|u(y,\tau)|dyd\tau \notag\\
			&\hspace{0.4cm}+O(1)\int_{t-\nu_0}^t\int_\bbr|u_x(x,\tau)|dx d\tau \notag\\
			&\hspace{0.4cm}+O(1)\int_{t-\nu_0}^t\int_\bbr\int_\bbr\left(\frac{e^{-\frac{(x-y)^2}{4C(t-\tau)}}}{t-\tau}+e^{-\sigma_0 |x-y|}\right)dx\left(|\theta_y(y,\tau)|+|u||u_y|+|p||u|\right)dyd\tau \notag\\
			&\leq O(1)\sqrt{\nu_0}\delta.\notag
		\end{align}
		Now, we combine \eqref{5-52}, \eqref{5-53}, \eqref{5-54} and \eqref{5-55} to obtain that,
		\begin{equation}\label{5-56}
			\|u_x(\cdot,t)\|_{L^1}\leq  C(\nu_0)\delta^*+O(1)\frac{|\log(\nu_0)|}{\sqrt{\nu_0}}\delta^2+O(1)\sqrt{\nu_0}\delta.
		\end{equation}
		Now, we substitute the condition $t\geq t_\sharp\geq 4\nu_0$ into \eqref{5-56} to obtain the desired estimates.
	\end{proof}

\begin{lemma}\label{L5.7}
	$(\mathbf{\|\theta(\cdot,t)-1\|_{L^1},\ \|\sqrt{1+t}\left(\theta(\cdot,t)-1\right)\|_{L^\infty}})$ Let $(v,u,E)$, $C_\sharp$, $t_\sharp$ and $\delta$  be the local solution and corresponding parameters constructed in Proposition \ref{prop:local}. We further suppose the properties \eqref{5-19} hold for the solution.
	Then, $\theta(x,t)$ has the following zeroth order estimates for $t\geq t_\sharp$,
	\begin{equation*}
		\left\{
		\begin{aligned}
			&\|\theta(\cdot,t)-1\|_{L^1}\leq O(1)\left(C(\nu_0)\delta^*+\sqrt{\nu_0}\delta+\delta^2\right)+O(1)\left(C(\nu_0)\delta^*+\sqrt{\nu_0}\delta+\delta^2\right)^2,\\
			&\|\sqrt{1+t}(\theta(\cdot,t)-1)\|_{L^\infty}\leq O(1)\left(C(\nu_0)\delta^*+\sqrt{\nu_0}\delta+\delta^2\right)+O(1)\left(C(\nu_0)\delta^*+\sqrt{\nu_0}\delta+\delta^2\right)^2.
		\end{aligned}
		\right.
	\end{equation*}
	
\end{lemma}

\begin{proof}
	Similar as the proof of $u(x,t)$, we first use the representation of energy $E(x,t)$ \eqref{theta-repre} to write the representation of $\theta(x,t)$ as follows, 
	\begin{align}
		&c_v\left(\theta(x,t)-1\right)\notag\\
		&=\int_\bbr\bbg_{31}(x-y,t)(v(y,0)-1)dy+\int_\bbr \bbg_{32}(x,t;y,0)u(y,0)dy+\int_\bbr G_{33}(x-y,t)\left(E(y,0)-c_v\right)dy\notag\\
		&\hspace{0.4cm}+\sum_{i=1}^3\mathcal{R}^\theta_i-\frac{(u(x,t))^2}{2},\notag
	\end{align}
	where the remainders $\mathcal{R}_i^\theta$ are listed in Lemma \ref{L5.2}. As we already have the estimates of $u(x,t)$ in Lemma \ref{L5.5} and Lemma \ref{L5.6}, we only need to deal with the first six terms in above representations.\newline 
	
	\noindent $\bullet$ ($\mathbf{\|\theta(\cdot,t)-1\|_{L^\infty}}$) We first study the zeroth order $L^\infty$ estimates. Actually, according to Lemma \ref{L5.4}, the Green's function $\bbg_{3j}$ have similar estimates as $\bbg_{2j}$. Therefore, we omit the details and directly obtain the estimates of the homogeneous terms as follows,  
	\bq\label{5-57}
	\left\{
	\begin{aligned}
		&\left|\int_\bbr\bbg_{31}(x-y,t)(v(y,0)-1)dy\right|\leq \frac{C(\nu_0)\delta^*}{\sqrt{1+t}},\\
		&\left|\int_\bbr\bbg_{32}(x-y,t)u(y,0)dy\right|\leq \frac{C(\nu_0)\delta^*}{\sqrt{1+t}},\\
		&\left|\int_\bbr \bbg_{33}(x-y,t)\left(E(y,0)-c_v\right)dy\right|\leq \frac{C(\nu_0)\delta^*}{\sqrt{1+t}}.
	\end{aligned}
	\right.
	\eq
	Next, for the remainders $\mathcal{R}^\theta_j$, we apply a simple recombination to gain additional comparison terms, and thus obtain the following representations,
	\begin{align*}
		\mathcal{R}_1^\theta&=\int_0^{t-2\nu_0}\int_\bbr \partial_y\bbg_{32}(x,t;y,\tau)\left(\frac{K(v-1)^2}{v}+\frac{(K\theta-K)(1-v)}{v}+\frac{\mu u_y(v-1)}{v}-\frac{Ku^2}{2c_v}\right)dyd\tau\notag\\
		&\hspace{0.4cm}+\int_0^{t-2\nu_0}\int_\bbr \partial_y\bbg_{33}(x-y,t-\tau)\left((\frac{K\theta}{v}-K)u+\left(\kappa-\frac{\kappa}{v}\right)\theta_y+\left(\frac{\kappa}{c_v}-\frac{\mu}{v}\right)uu_y\right)dyd\tau,\notag\\
		\mathcal{R}_2^\theta&=\int_{t-2\nu_0}^{t-\nu_0}\int_\bbr \partial_y\bbg_{32}(x,t;y,\tau)\left(\frac{K(\theta-1)}{v}+\frac{\mu (v-1)u_y}{v}+\frac{K(v-1)^2}{v}\right)dyd\tau\notag\\
		&\hspace{0.4cm}-\int_{t-2\nu_0}^{t-\nu_0}\int_\bbr  \left(1-\mathcal{X}\left(\frac{t-\tau}{\nu_0}\right)\right)\left(\frac{K}{c_v}\partial_y\bbg_{32}\right)\left(E(y,\tau)-c_v\right)dyd\tau\notag\\
		&\hspace{0.4cm}+\int_{t-2\nu_0}^{t-\nu_0}\int_\bbr \frac{1}{\nu_0}\mathcal{X}'\left(\frac{t-\tau}{\nu_0}\right) \left(-H\left(x,t;y,\tau;\frac{\kappa}{c_vv}\right)+\bbg_{33}(x-y;t-\tau)\right)\left(E(y,\tau)-c_v\right)dyd\tau\notag\\
		&\hspace{0.4cm}+\int_{t-2\nu_0}^{t-\nu_0}\int_\bbr K\mathcal{X}\left(\frac{t-\tau}{\nu_0}\right)\left(H_y\left(x,t;y,\tau;\frac{\kappa}{c_vv}\right)-\partial_y\bbg_{33}(x-y;t-\tau)\right)u(y,\tau)dyd\tau\notag\\
		&\hspace{0.4cm}+\int_{t-2\nu_0}^{t-\nu_0}\int_\bbr \mathcal{X}\left(\frac{t-\tau}{\nu_0}\right)H_y\left(x,t;y,\tau;\frac{\kappa}{c_vv}\right)\left((p-K)u+\left(\frac{\kappa}{c_vv}-\frac{\mu}{v}\right)uu_y\right)dyd\tau\notag\\
		&\hspace{0.4cm}+\int_{t-2\nu_0}^{t-\nu_0}\int_\bbr \left(1-\mathcal{X}\left(\frac{t-\tau}{\nu_0}\right)\right)\partial_y\bbg_{33}(x-y;t-\tau)\left((p-K)u+\frac{\kappa(v-1)}{v}\theta_y+\left(\frac{\kappa}{c_v}-\frac{\mu}{v}\right)uu_y\right)dyd\tau,\notag\\
		\mathcal{R}_3^\theta&=\int_{t-\nu_0}^t\int_\bbr \partial_y\bbg_{32}(x-y,t-\tau)\left(K(v(y,\tau)-1)+\mu u_y\frac{v-1}{v}+\frac{K\theta-Kv}{v}\right)dyd\tau\notag\\
		&\hspace{0.4cm}+\int_{t-\nu_0}^t\int_\bbr K\left(\partial_yH\left(x,t;y,\tau;\frac{\kappa}{c_vv}\right)-\partial_y\bbg_{33}(x-y,t-\tau)\right)u(y,\tau)dyd\tau\notag\\
		&\hspace{0.4cm}+\int_{t-\nu_0}^t\int_\bbr \partial_yH\left(x,t;y,\tau;\frac{\kappa}{c_vv}\right)\left((p-K)u+\left(\frac{\kappa}{c_vv}-\frac{\mu}{v}\right)uu_y\right)dyd\tau.\notag
	\end{align*}
	For the first term $\mathcal{R}^\theta_1$, note the quadratic terms are same as in $\mathcal{R}^u_1$. Moreover, the Green's function $\partial_y\bbg_{32}$ and $\partial_y\bbg_{33}$ are similar as the estimates of $\partial_y\bbg_{22}$ and $\partial_y\bbg_{23}$. Therefore, we follow the arguments in Lemma \ref{L5.5} to obtain the desired estimates as follows,
	\[\mathcal{R}^\theta_1\leq O(1)\frac{\delta^2}{\sqrt{1+t}}.\]
	It is worth to mention that, when $t\geq 1$, the time decay of $\mathcal{R}^\theta_1$ comes from the Green's function $\bbg_{32}$ and $\bbg_{33}$ for $\tau\leq \frac{t-1}{2}$, and from $\mathcal{G}(\tau)$ when $\tau\geq \frac{t-1}{2}$. 
	
	As for $\mathcal{R}^\theta_2$, similar to \eqref{5-27}, we can represent $\mathbb{G}_{22}$ in terms of $H(x,t;y,\tau;\frac{\kappa}{c_v})$ and the difference between $\bbg_{33}$ and $H$ provides us one more $\delta$. One also observes that all other terms in $\mathcal{R}^\theta_2$  contain quadratic terms except for the first two terms, which contain the linear term $(\theta-1)$. Fortunately, the first two terms only depend on $\bbg_{32}$, which is much regular than $\bbg_{33}$. Thus, we can follow similar arguments as in Lemma \ref{L5.5} to yield that,
	\[\mathcal{R}^\theta_2\leq O(1)\left(\sqrt{\nu_0}\delta+\delta^2\right)\frac{1}{\sqrt{1+t}}.\]
	With similar arguments, we follow \eqref{5-34} to obtain the estimates of the third part $\mathcal{R}^\theta_3$ as below,
	\[\mathcal{R}^\theta_3\leq O(1)\frac{\sqrt{\nu_0}\delta}{\sqrt{1+t}}.\]
	
	Same as before, above estimates hold for $t\geq 1$. When $t\leq 1$, we can follow the arguments in Lemma \ref{L5.5} to find similar estimates as $t\geq 1$. Therefore, we combine all above estimates to conclude that,
	\begin{align*}
		\|\theta(\cdot,t)-1\|_{L^\infty}&\leq \frac{C(\nu_0)\delta^*}{\sqrt{1+t}}+O(1)\frac{\sqrt{\nu_0}\delta+\delta^2}{\sqrt{1+t}}+\frac{\|u(\cdot,t)\|_{L^{\infty}}^2}{2}\\
		&\leq \left(\frac{C(\nu_0)\delta^*}{\sqrt{1+t}}+O(1)\frac{\sqrt{\nu_0}\delta+\delta^2}{\sqrt{1+t}}\right)+\left(\frac{C(\nu_0)\delta^*}{\sqrt{1+t}}+O(1)\frac{\sqrt{\nu_0}\delta+\delta^2}{\sqrt{1+t}}\right)^2.
	\end{align*}
	
	\noindent $\bullet$ ($\mathbf{\|\theta(\cdot,t)-1\|_{L^1}}$) For the $L^1$ estimates of $\theta-1$, we can follow the $L^1$ estimates of $u$ in Lemma \ref{L5.5} to obtain the desired results. As the arguments are very similar as the $L^\infty$ estimates, we omit the details and directly provide the following results,
	\begin{align*}
		\|\theta(\cdot,t)-1\|_{L^1}&\leq O(1)\left(C(\nu_0)\delta^*+\sqrt{\nu_0}\delta+\delta^2\right)+\frac{1}{2}\|u(\cdot,t)\|_{L^\infty}\|u(\cdot,t)\|_{L^1}\\
		&\leq O(1)\left(C(\nu_0)\delta^*+\sqrt{\nu_0}\delta+\delta^2\right)+\frac{O(1)\left(C(\nu_0)\delta^*+\sqrt{\nu_0}\delta+\delta^2\right)^2}{\sqrt{1+t}}\\
		&\leq O(1)\left(C(\nu_0)\delta^*+\sqrt{\nu_0}\delta+\delta^2\right)+O(1)\left(C(\nu_0)\delta^*+\sqrt{\nu_0}\delta+\delta^2\right)^2.
	\end{align*}
	
\end{proof}

The first order estimates of $\theta-1$ can also be obtained in a similar way. The difference appears only when $\tau$ is closed to $t$, i.e., for the remainder $\mathcal{R}^\theta_3$. Actually, we have the following lemma. 
\begin{lemma}\label{L5.8}
	$(\mathbf{\|\theta_x(\cdot,t)\|_{L^1},\ \|\sqrt{t}\theta_x(\cdot,t)\|_{L^\infty}})$ Let $(v,u,E)$, $C_\sharp$, $t_\sharp$ and $\delta$  be the local solution and corresponding parameters constructed in Proposition \ref{prop:local}. We further suppose the properties \eqref{5-19} hold for the solution. Then, $\theta(x,t)$ has the following first order estimates at $t\geq t_\sharp$,
	\begin{equation*}
		\left\{
		\begin{aligned}
			&\|\theta_x(\cdot,t)\|_{L^1}\leq C(\nu_0)\delta^*+O(1)\frac{|\log(\nu_0)|}{\sqrt{\nu_0}}\delta^2+O(1)\sqrt{\nu_0}\delta,\\
			&\|\sqrt{t}\theta_x(\cdot,t)\|_{L^\infty}\leq C(\nu_0)\delta^*+O(1)\frac{|\log(\nu_0)|}{\sqrt{\nu_0}}\delta^2+O(1)\sqrt{\nu_0}\delta.
		\end{aligned}
		\right.
	\end{equation*}
	
\end{lemma}
\begin{proof}
	According to Lemma \ref{L5.2} and Corollary \ref{Cor5.1}, $\theta_x$ has the following representation,
	\begin{align}
		&c_v\theta_x(x,t)\notag\\
		&=\int_\bbr\partial_x\bbg_{31}(x-y,t)(v(y,0)-1)dy+\int_\bbr \partial_x\bbg_{32}(x,t;y,0)u(y,0)dy+\int_\bbr \partial_xG_{33}(x-y,t)\left(E(y,0)-c_v\right)dy\notag\\
		&\hspace{0.4cm}+\sum_{i=1}^3\partial_x\mathcal{R}^\theta_i-uu_x,\notag
	\end{align}
	where the only difference between the representations in Lemma \ref{L5.2} and Corollary \ref{Cor5.1} is $\partial_x\mathcal{R}^\theta_3$. As the Green's function $\bbg_{3j}$ and $\bbg_{2j}$ have similar estimates in Lemma \ref{L5.4}, it is obvious that the estimates of the homogeneous terms, $\partial_x\mathcal{R}^\theta_1$ and $\partial_x\mathcal{R}^\theta_2$ can be obtained in a similar way as Lemma \ref{L5.6}. On the other hand, the way to handle the singularity at $t$ in $\mathcal{R}^\theta_3$ will be a little bit different from the methods we used for $\mathcal{R}^u_3$. Thus it is sufficient to provide the details of the estimates of $\partial_x\mathcal{R}^\theta_3$.\newline
	
	\noindent $\bullet$ ($\mathbf{\|\theta_x(\cdot,t)\|_{L^1}}$) For the $L^1$ estimates of $\theta_x$, we omit the details and provide the estimates of the homogeneous terms, $\partial_x\mathcal{R}^\theta_1$ and $\partial_x\mathcal{R}^\theta_2$ as follows,
	\bq\label{5-58}
	\left\{
	\begin{aligned}
		&\int_\bbr\left|\int_\bbr\partial_x\bbg_{21}(x-y,t)(v(y,0)-1)dy\right|dx\leq C(\nu_0)\delta^*,\\
		&\int_\bbr\left|\int_\bbr\partial_x\bbg_{22}(x-y,t)u(y,0)dy\right|dx\leq C(\nu_0)\delta^*,\\
		&\int_\bbr\left|\int_\bbr \partial_x\bbg_{23}(x-y,t)\left(E(y,0)-c_v\right)dy\right|dx\leq C(\nu_0)\delta^*,\\
		&\int_\bbr\left|\partial_x\mathcal{R}^\theta_1\right|dx\leq O(1)(1+|\log(\nu_0)|)\delta^2,\\
		&\int_\bbr\left|\partial_x\mathcal{R}^\theta_2\right|dx\leq O(1)\frac{|\log(\nu_0)|}{\sqrt{\nu_0}}\delta^2+O(1)\sqrt{\nu_0}\delta.
	\end{aligned}
	\right.
	\eq
	Now, for the third term $\partial_x\mathcal{R}^\theta_3$, as we only need to consider the case $t\geq t_\sharp$ and we have already assumed $\mathcal{G}(\tau)<\delta$ for $\tau<t$, we directly apply Lemma \ref{L5.3} to obtain that the solution can be extended to $t+t_\sharp$ and has the following estimates,
	\[\|u_\tau(\cdot,\tau)\|_{L^1}\leq \frac{2C_\sharp\delta}{\sqrt{t_\sharp}},\quad\|u_\tau(\cdot,\tau)\|_{L^\infty}\leq \frac{2C_\sharp\delta}{t_\sharp},\quad \|\theta_\tau(\cdot,\tau)\|_{L^\infty}\leq \frac{2C_\sharp\delta}{t_\sharp},\quad t_\sharp\leq \tau\leq t.\]
	Then, we combine above estimates and the representation \eqref{theta-repre} to have the following estimates,
	\begin{align}
		&\int_\bbr\left|\partial_x\mathcal{R}_3^\theta\right|dx\label{5-59}\\
		&\leq \int_\bbr\left|\int_{t-\nu_0}^t\int_\bbr \partial_{xy}\bbg_{32}(x-y,t-\tau)\left(K(v(y,\tau)-1)+\mu u_y\frac{v-1}{v}+\frac{K\theta-Kv}{v}\right)dyd\tau\right|dx\notag\\
		&\hspace{0.4cm}+\int_\bbr\left|\int_{t-\nu_0}^t\int_\bbr K\left(\partial_x\bbg_{33}(x-y,t-\tau)\right)u_y(y,\tau)dyd\tau\right|dx\notag\\
		&\hspace{0.4cm}+\int_\bbr\left|\left(1-\frac{\kappa}{c_v\mu}\right)\int_{t-\nu_0}^t\int_\bbr \partial_xH\left(x,t;y,\tau;\frac{\kappa}{c_vv}\right)\left(u_y\left(-p+\frac{\mu u_y}{v}\right)+uu_\tau\right)dyd\tau\right|dx\notag\\
		&\hspace{0.4cm}+\int_\bbr\left|\int_{t-\nu_0}^t\int_\bbr \partial_xH\left(x,t;y,\tau;\frac{\kappa}{c_vv}\right)\left(\frac{\kappa}{c_v\mu}pu_y\right)dyd\tau\right|dx\notag\\
		&\hspace{0.4cm}+\int_\bbr\left|\int_{t-\nu_0}^t\int_{\bbr\setminus\mathscr{D}} \partial_xH\left(x,t;y,\tau;\frac{\kappa}{c_vv}\right)\left(\frac{\kappa K}{c_v\mu}\frac{\theta_y v-\theta v_y}{v^2}u\right)dyd\tau\right|dx\notag\\
		&\hspace{0.4cm}+\int_\bbr\left|\int_{t-\nu_0}^t\sum_{z\in\mathscr{D}}\partial_xH\left(x,t;z,\tau;\frac{\kappa}{c_vv}\right)\left(\frac{\kappa}{c_v\mu}\frac{K\theta(z,\tau)u(z,\tau)}{v(z,\tau)}\right)\Big|_{z=z^-}^{z=z^+}d\tau\right|dx\notag\\
		&\leq O(1)\sqrt{\nu_0}\left(\delta+\delta^2+\delta^3\right)+O(1)\left(\delta^2+\frac{\delta^2}{\sqrt{t-\nu_0}}\right)\notag\\
		&\leq O(1)\left(\sqrt{\nu_0}\delta+\frac{\delta^2}{\sqrt{\nu_0}}\right).\notag
	\end{align}
	Now, we combine \eqref{5-58} and \eqref{5-59} to obtain that,
	\[\|\theta_x(\cdot,t)\|_{L^1}\leq C(\nu_0)\delta^*+O(1)\frac{|\log(\nu_0)|}{\sqrt{\nu_0}}\delta^2+O(1)\sqrt{\nu_0}\delta.\]
	
	\noindent $\bullet$ ($\mathbf{\|\theta_x(\cdot,t)\|_{L^\infty}}$) For the $L^\infty$ estimates, the above representation is not proper, since the last two terms in \eqref{5-59} is not integrable with respect to $\tau$ if we do not involve the integration of $x$. Therefore, we switch to the representation in Corollary \ref{Cor5.1} to obtain that, 
	\begin{align*}
		\partial_x\mathcal{R}_3^\theta
		&=\int_{t-\nu_0}^t\int_\bbr \partial_{xy}\bbg_{32}(x-y,t-\tau)\left(K(v(y,\tau)-1)+\mu u_y\frac{v-1}{v}+\frac{K\theta-Kv}{v}\right)dyd\tau\notag\\
		&\hspace{0.4cm}+\int_{t-\nu_0}^t\int_\bbr K\left(\partial_x\bbg_{33}(x-y,t-\tau)-H_x\left(x,t;y,\tau;\frac{\kappa}{c_vv}\right)\right)u_y(y,\tau)dyd\tau\notag\\
		&\hspace{0.4cm}+\int_{t-\nu_0}^t\int_\bbr H_x\left(x,t;y,\tau;\frac{\kappa}{c_vv}\right)\left(1-\frac{\kappa}{c_v\mu}\right)uu_\tau dyd\tau\notag\\
		&\hspace{0.4cm}+\int_{t-\nu_0}^t\int_\bbr H_x\left(x,t;y,\tau;\frac{\kappa}{c_vv}\right)\left(1-\frac{\kappa}{c_v\mu}\right)u_y\left(K-p+\frac{\mu}{v}u_y\right)dyd\tau\notag\\
		&\hspace{0.4cm}+\int_{-\infty}^x\int_{-\infty}^y H_x\left(x,t;z,t-\nu_0;\frac{\kappa}{c_vv}\right)dz\left(\frac{\kappa K}{c_v\mu}\right)\frac{(\theta(y,t-\nu_0)-v(y,t-\nu_0))}{v(y,t-\nu_0)}u(y,t-\nu_0)dy\notag\\
		&\hspace{0.4cm}+\left(\frac{\kappa K}{c_v\mu}\right)\frac{(\theta(x,t-\nu_0)-v(x,t-\nu_0))}{v(x,t-\nu_0)}u(x,t-\nu_0)\notag\\
		&\hspace{0.4cm}+\int_{t-\nu_0}^t\int_{-\infty}^x\int_{-\infty}^y H_x\left(x,t;z,\tau;\frac{\kappa}{c_vv}\right)dz\left(\frac{\kappa K}{c_v\mu}\right)\left(\frac{\left(\theta_\tau v-\theta v_\tau\right)}{v^2}u(y,\tau)+\frac{(\theta-v)}{v}u_\tau(y,\tau)\right)dyd\tau\notag\\
		&\hspace{0.4cm}+\int_{t-\nu_0}^t\left(\frac{\kappa K}{c_v\mu}\right)\left(\frac{\left(\theta_\tau(x,\tau) v(x,\tau)-\theta(x,\tau) v_\tau(x,\tau)\right)}{v^2(x,\tau)}u(x,\tau)+\frac{(\theta(x,\tau)-v(x,\tau))}{v(x,\tau)}u_\tau(x,\tau)\right)d\tau\notag\\
		&\hspace{0.4cm}-\int_x^{+\infty}\int_y^{+\infty} H_x\left(x,t;z,t-\nu_0;\frac{\kappa}{c_vv}\right)dz\left(\frac{\kappa K}{c_v\mu}\right)\left(\frac{\theta(y,t-\nu_0)-v(y,t-\nu_0)}{v(y,t-\nu_0)}\right)u(y,t-\nu_0)dy\notag\\
		&\hspace{0.4cm}-\int_{t-\nu_0}^t\int_x^{+\infty}\int_y^{+\infty} H_x\left(x,t;z,\tau;\frac{\kappa}{c_vv}\right)dz\left(\frac{\kappa K}{c_v\mu}\right)\left(\frac{\left(\theta_\tau v-\theta v_\tau\right)}{v^2}u(y,\tau)+\frac{(\theta-v)}{v}u_\tau(y,\tau)\right)dyd\tau.\notag
	\end{align*}
	Similar as in previous $L^1$ estimates, we apply Lemma \ref{L5.3} to obtain the following estimates,
	\bq\label{u_t-local}\|u_\tau(\cdot,\tau)\|_{L^1}\leq \frac{O(1)\delta}{\sqrt{t_\sharp}},\quad\|u_\tau(\cdot,\tau)\|_{L^\infty}\leq \frac{O(1)\delta}{t_\sharp},\quad \|\theta_\tau(\cdot,\tau)\|_{L^\infty}\leq \frac{O(1)\delta}{t_\sharp},\quad t_\sharp\leq \tau\leq t.\eq
	Then, we apply above estimates of time derivatives, the comparison principle Lemma \ref{L2.2} in Section \ref{sec:2}, the estimates of the Green's function $\bbg$ to obtain that,
	\begin{align}
		\left|\partial_x\mathcal{R}_3^\theta\right|&\leq O(1)\sqrt{\nu_0}\frac{\delta}{\sqrt{1+t-\nu_0}}+O(1)\frac{\delta^2}{\sqrt{1+t-\nu_0}(\sqrt{\nu_0})}\label{5-60}\\
		&\hspace{0.4cm}+O(1)\sqrt{\nu_0}\frac{\delta^2}{\sqrt{1+t-\nu_0}\sqrt{t-\nu_0}}+O(1)\sqrt{\nu_0}\frac{\delta^2}{(t-\nu_0)}\notag\\
		&\hspace{0.4cm}+O(1)\frac{\delta^2}{1+t-\nu_0}+O(1)\nu_0\left(\frac{\delta^2}{\nu_0\sqrt{1+t-\nu_0}}+\frac{\delta^2}{\sqrt{t-\nu_0}\sqrt{1+t-\nu_0}}\right)\notag\\
		&\hspace{0.4cm}+\int_{t-\nu_0}^t\int_\bbr K\left|H_x(x,t;y,\tau;\frac{\kappa}{c_v})-H_x\left(x,t;y,\tau;\frac{\kappa}{c_vv}\right)\right||u_y(y,\tau)|dyd\tau\notag\\
		&\hspace{0.4cm}+\int_{t-\nu_0}^t\int_\bbr K\int_{\tau}^t \left|H_z(z,s;y,\tau;\frac{\kappa}{c_v})\right|\left|-\frac{K}{c_v}\partial_{x}\bbg_{32}(x-z,t-s)\right|dzds|u_y(y,\tau)|dyd\tau\notag\\
		&\leq O(1)\sqrt{\nu_0}\frac{\delta}{\sqrt{t}}+O(1)\frac{\delta^2}{t}+O(1)\sqrt{\nu_0}|\log(\nu_0)|\frac{\delta^2}{\sqrt{t}}\notag\\
		&\leq O(1)\sqrt{\nu_0}\frac{\delta}{\sqrt{t}}+O(1)\frac{\delta^2}{\sqrt{\nu_0}\sqrt{t}}+O(1)\sqrt{\nu_0}|\log(\nu_0)|\frac{\delta^2}{\sqrt{t}},\notag
	\end{align}
	where the last inequality is due to the fact $t\geq t_\sharp\geq 4\nu_0$. For the homogeneous terms and the other remainder terms, we omit the details and follow previous estimates to obtain that,
	\begin{align}
		&\left|\int_\bbr\partial_x\bbg_{21}(x-y,t)(v(y,0)-1)dy\right|\leq C(\nu_0)\frac{\delta^*}{\sqrt{t}},\notag\\
		&\left|\int_\bbr\partial_x\bbg_{22}(x-y,t)u(y,0)dy\right|\leq C(\nu_0)\frac{\delta^*}{\sqrt{t}},\notag\\
		&\left|\int_\bbr \partial_x\bbg_{23}(x-y,t)\left(E(y,0)-c_v\right)dy\right|\leq C(\nu_0)\frac{\delta^*}{\sqrt{t}},\label{5-61}\\
		&\left|\partial_x\mathcal{R}^\theta_1\right|\leq O(1)(1+|\log(\nu_0)|)\frac{\delta^2}{\sqrt{t}},\notag\\
		&\left|\partial_x\mathcal{R}^\theta_2\right|\leq O(1)\frac{|\log(\nu_0)|}{\sqrt{\nu_0}}\frac{\delta^2}{\sqrt{t}}+O(1)\sqrt{\nu_0}\frac{\delta}{\sqrt{t}}.\notag
	\end{align}
	Now, we combine \eqref{5-60} and \eqref{5-61} to obtain the desired estimates, and finish the proof of the lemma. 
\end{proof}

Finally, we study the estimates of the specific volume $v(x,t)$. We recall the representation of $v(x,t)$ in Lemma \ref{L5.2} as below,
\begin{align}
	&v(x,t)-1\notag\\
	&=\int_\bbr\bbg_{11}(x-y,t)(v(y,0)-1)dy+\int_\bbr \bbg_{12}(x,t;y,0)u(y,0)dy+\int_\bbr \bbg_{13}(x-y,t)\left(E(y,0)-c_v\right)dy\notag\\
	&\hspace{0.4cm}+\int_0^{t}\int_\bbr \partial_y\bbg_{12}(x,t;y,\tau)\left(\frac{K(v-1)^2}{v}+\frac{(K\theta-K)(1-v)}{v}+\frac{\mu u_y(v-1)}{v}-\frac{Ku^2}{2c_v}\right)dyd\tau\notag\\
	&\hspace{0.4cm}+\int_0^{t}\int_\bbr \partial_y\bbg_{13}(x-y,t-\tau)\left((\frac{K\theta}{v}-K)u+\left(\kappa-\frac{\kappa}{v}\right)\theta_y+\left(\frac{\kappa}{c_v}-\frac{\mu}{v}\right)uu_y\right)dyd\tau.\notag
\end{align}
According to the estimates of the Green's function $\bbg$ in Theorem \ref{T4.1}, Theorem \ref{T4.2} and Theorem \ref{T4.3}, we immediately know that $\bbg_{13}$ has similar estimates as $\bbg_{31}$, and $\bbg_{12}$ has similar estimates as $\bbg_{21}$, which are more regular than the diagonal parts. Then, we have the following lemma.

\begin{lemma}\label{L5.9}
	$(\mathbf{\|v(\cdot,t)\|_{BV},\ \sqrt{t+1}\|v(\cdot,t)-1\|_{L^\infty},\ \|v(\cdot,t)-1\|_{L^1}})$ Let $(v,u,E)$, $C_\sharp$, $t_\sharp$ and $\delta$  be the local solution and corresponding parameters constructed in Proposition \ref{prop:local}. We further suppose the properties \eqref{5-19} hold for the solution. Then, $v(x,t)$ has the following estimates at $t\geq t_\sharp$,
	\begin{equation*}
		\left\{
		\begin{aligned}
			&\|v(\cdot,t)-1\|_{L^1}\leq C(\nu_0)\delta^*+O(1)\delta^2,\\
			&\|\sqrt{1+t}(v(\cdot,t)-1)\|_{L^\infty}\leq C(\nu_0)\delta^*+O(1)\delta^2,\\
			&\|v(\cdot,t)\|_{BV}\leq C(\nu_0)\delta^*+O(1)\frac{\delta^2}{\sqrt{\nu_0}}.
		\end{aligned}
		\right.
	\end{equation*}
	
\end{lemma}
\begin{proof}
	First of all, we will apply  Theorem \ref{T4.1} and Theorem \ref{T4.3} and list the estimates of $\bbg_{11}$, $\bbg_{12}$ and $\bbg_{13}$ for the convenience. For $t\geq 1$, we have the following estimates of $\bbg_{1j}$,
	\begin{align}
		&|\bbg_{11}(x,t)-e^{-\frac{Kt}{\mu}}\delta(x)|\leq \sum_{j=1}^3\frac{O(1)e^{-\frac{(x+\beta_jt)^2}{4Ct}}}{\sqrt{t}}+O(1)e^{-\sigma_0^*t-\sigma_0|x|},\quad t\geq 1,\label{5-62}\\
		&|\partial_x\bbg_{11}(x,t)-e^{-\frac{Kt}{\mu}}\delta'(x)|\leq \sum_{j=1}^3\frac{O(1)e^{-\frac{(x+\beta_jt)^2}{4Ct}}}{t}+O(1)e^{-\sigma_0^*t-\sigma_0|x|},\quad t\geq 1,\label{5-63}\\
		&|\bbg_{12}(x,t)|\leq \sum_{j=1}^3\frac{O(1)e^{-\frac{(x+\beta_jt)^2}{4Ct}}}{\sqrt{t}}+O(1)e^{-\sigma_0^*t-\sigma_0|x|},\quad t\geq 1,\label{5-64}\\
		&|\partial_x\bbg_{12}(x,t)+\frac{1}{\mu}e^{-\frac{Kt}{\mu}}\delta(x)|\leq \sum_{j=1}^3\frac{O(1)e^{-\frac{(x+\beta_jt)^2}{4Ct}}}{t}+O(1)e^{-\sigma_0^*t-\sigma_0|x|},\quad t\geq 1,\label{5-65}\\
		&|\partial_{xx}\bbg_{12}(x,t)+\frac{1}{\mu}e^{-\frac{Kt}{\mu}}\delta'(x)|\leq \sum_{j=1}^3\frac{O(1)e^{-\frac{(x+\beta_jt)^2}{4Ct}}}{t^{\frac{3}{2}}}+O(1)e^{-\sigma_0^*t-\sigma_0|x|},\quad t\geq 1,\label{5-66}\\
		&|\bbg_{13}(x,t)|\leq \sum_{j=1}^3\frac{O(1)e^{-\frac{(x+\beta_jt)^2}{4Ct}}}{\sqrt{t}}+O(1)e^{-\sigma_0^*t-\sigma_0|x|},\quad t\geq 1,\label{5-67}\\
		&|\partial_x\bbg_{13}(x,t)|\leq \sum_{j=1}^3\frac{O(1)e^{-\frac{(x+\beta_jt)^2}{4Ct}}}{t}+O(1)e^{-\sigma_0^*t-\sigma_0|x|},\quad t\geq 1,\label{5-68}\\
		&|\partial_{xx}\bbg_{13}(x,t)-\frac{c_vK^2}{\kappa\mu}e^{-\frac{Kt}{\mu}}\delta(x)|\leq \sum_{j=1}^3\frac{O(1)e^{-\frac{(x+\beta_jt)^2}{4Ct}}}{t^{\frac{3}{2}}}+O(1)e^{-\sigma_0^*t-\sigma_0|x|},\quad t\geq 1.\label{5-69}
	\end{align}
	And for $t\leq 1$, we have the following estimates of $\bbg_{12}$ and $\bbg_{13}$, 
	\begin{align}
		&|\bbg_{12}(x,t)|\leq O(1)e^{-\sigma_0^*t-\sigma_0|x|}+O(1)te^{-\sigma_0 |x|},\quad t\leq 1,\label{5-70}\\
		&|\partial_x\bbg_{12}(x,t)+\frac{1}{\mu}e^{-\frac{Kt}{\mu}}\delta(x)|\leq \frac{O(1)e^{-\frac{x^2}{4Ct}}}{\sqrt{t}}+O(1)e^{-\sigma_0^*t-\sigma_0|x|}+O(1)te^{-\sigma_0 |x|},\quad t\leq 1,\label{5-71}\\
		&|\partial_{xx}\bbg_{12}(x,t)+\frac{1}{\mu}e^{-\frac{Kt}{\mu}}\delta'(x)|\leq \frac{O(1)e^{-\frac{x^2}{4Ct}}}{t}+O(1)e^{-\sigma_0^*t-\sigma_0|x|}+O(1)te^{-\sigma_0 |x|},\quad t\leq 1,\label{5-72}\\
		&|\bbg_{13}(x,t)|\leq O(1)e^{-\sigma_0^*t-\sigma_0|x|}+O(1)te^{-\sigma_0 |x|},\quad t\leq 1,\label{5-73}\\
		&|\partial_x\bbg_{13}(x,t)|\leq O(1)e^{-\sigma_0^*t-\sigma_0|x|}+O(1)te^{-\sigma_0 |x|},\quad t\leq 1,\label{5-74}\\
		&|\partial_{xx}\bbg_{13}(x,t)-\frac{K}{\kappa\mu}e^{-\frac{Kt}{\mu}}\delta(x)|\leq \frac{O(1)e^{-\frac{x^2}{4Ct}}}{\sqrt{t}}+O(1)e^{-\sigma_0^*t-\sigma_0|x|}+O(1)te^{-\sigma_0 |x|},\quad t\leq 1.\label{5-75}
	\end{align}
	In the following, we will use two steps to provide the zeroth order and first order estimates of $v(x,t)$ respectively.\newline
	
	\noindent $\bullet$ ($\mathbf{\sqrt{1+t}\|v(\cdot,t)-1\|_{L^\infty},\ \|v(\cdot,t)-1\|_{L^1}}$) 
	We will only show the proof of the infinity norm estimates, and the $L^1$ estimates can be obtained similarly. For the homogeneous terms, as we have $t\geq t_\sharp$, we can follow previous arguments to obtain the  estimates of $\bbg_{1j}$ for any $t\geq t_\sharp$, which are similar as \eqref{5-62}, \eqref{5-64} and \eqref{5-67} with coefficients $C(\nu_0)$. Then, we omit the details and obtain that,
	\bq\label{5-76}
	\left\{
	\begin{aligned}
		&\left|\int_\bbr\bbg_{11}(x-y,t)(v(y,0)-1)dy\right|\leq \frac{C(\nu_0)\delta^*}{\sqrt{1+t}},\\
		&\left|\int_\bbr\bbg_{12}(x-y,t)u(y,0)dy\right|\leq \frac{C(\nu_0)\delta^*}{\sqrt{1+t}},\\
		&\left|\int_\bbr \bbg_{13}(x-y,t)\left(E(y,0)-c_v\right)dy\right|\leq \frac{C(\nu_0)\delta^*}{\sqrt{1+t}}.
	\end{aligned}
	\right.
	\eq
	Next, for the inhomogeneous terms, we split it into three parts. This is because the estimates of the Green's function are different for small and large time, and also the time decay comes from different parts for small time and large time. Actually, we may define for $t\geq 1$ that,
	\begin{align}
		\mathcal{R}^v_1&=\int_0^{\frac{t-1}{2}}\int_\bbr \partial_y\bbg_{12}(x,t;y,\tau)\left(\frac{K(v-1)^2}{v}+\frac{(K\theta-K)(1-v)}{v}+\frac{\mu u_y(v-1)}{v}-\frac{Ku^2}{2c_v}\right)dyd\tau\label{5-77}\\
		&\hspace{0.4cm}+\int_0^{\frac{t-1}{2}}\int_\bbr \partial_y\bbg_{13}(x-y,t-\tau)\left((\frac{K\theta}{v}-K)u+\left(\kappa-\frac{\kappa}{v}\right)\theta_y+\left(\frac{\kappa}{c_v}-\frac{\mu}{v}\right)uu_y\right)dyd\tau\notag\\
		\mathcal{R}^v_2&=\int_{\frac{t-1}{2}}^{t-1}\int_\bbr \partial_y\bbg_{12}(x,t;y,\tau)\left(\frac{K(v-1)^2}{v}+\frac{(K\theta-K)(1-v)}{v}+\frac{\mu u_y(v-1)}{v}-\frac{Ku^2}{2c_v}\right)dyd\tau\label{5-78}\\
		&\hspace{0.4cm}+\int_{\frac{t-1}{2}}^{t-1}\int_\bbr \partial_y\bbg_{13}(x-y,t-\tau)\left((\frac{K\theta}{v}-K)u+\left(\kappa-\frac{\kappa}{v}\right)\theta_y+\left(\frac{\kappa}{c_v}-\frac{\mu}{v}\right)uu_y\right)dyd\tau\notag\\
		\mathcal{R}^v_3&=\int_{t-1}^{t}\int_\bbr \partial_y\bbg_{12}(x,t;y,\tau)\left(\frac{K(v-1)^2}{v}+\frac{(K\theta-K)(1-v)}{v}+\frac{\mu u_y(v-1)}{v}-\frac{Ku^2}{2c_v}\right)dyd\tau\label{5-79}\\
		&\hspace{0.4cm}+\int_{t-1}^{t}\int_\bbr \partial_y\bbg_{13}(x-y,t-\tau)\left((\frac{K\theta}{v}-K)u+\left(\kappa-\frac{\kappa}{v}\right)\theta_y+\left(\frac{\kappa}{c_v}-\frac{\mu}{v}\right)uu_y\right)dyd\tau.\notag
	\end{align}
	Similar as before, we only have $\mathcal{R}^v_3$ when $t\leq 1$, and the integration in $\mathcal{R}^v_3$ will be from $0$ to $t$. Now, for $\mathcal{R}^v_1$, the Green's function $\bbg_{1j}$ will provide the decay, and we apply \eqref{5-65}, \eqref{5-68} to have for $t\geq 1$ that, 
	\begin{align}
		|\mathcal{R}^v_1|
		&=\int_0^{\frac{t-1}{2}}\left|\int_\bbr \frac{1}{\mu}e^{-\frac{K(t-\tau)}{\mu}}\delta(x-y)\left(\frac{K(v-1)^2}{v}+\frac{(K\theta-K)(1-v)}{v}+\frac{\mu u_y(v-1)}{v}-\frac{Ku^2}{2c_v}\right)dy\right|d\tau\notag\\
		&\hspace{0.4cm}+\int_0^{\frac{t-1}{2}}\left|\int_\bbr \sum_{j=1}^3\frac{O(1)e^{-\frac{(x-y+\beta_j(t-\tau))^2}{4C(t-\tau)}}}{t-\tau}\left(\frac{K(v-1)^2}{v}+\frac{(K\theta-K)(1-v)}{v}+\frac{\mu u_y(v-1)}{v}-\frac{Ku^2}{2c_v}\right)dy\right|d\tau\notag\\
		&\hspace{0.4cm}+\int_0^{\frac{t-1}{2}}\left|\int_\bbr O(1)e^{-\sigma_0^*(t-\tau)-\sigma_0|x-y|}\left(\frac{K(v-1)^2}{v}+\frac{(K\theta-K)(1-v)}{v}+\frac{\mu u_y(v-1)}{v}-\frac{Ku^2}{2c_v}\right)dy\right|d\tau\notag\\
		&\hspace{0.4cm}+\int_0^{\frac{t-1}{2}}\left|\int_\bbr \sum_{j=1}^3\frac{O(1)e^{-\frac{(x-y+\beta_j(t-\tau))^2}{4C(t-\tau)}}}{t-\tau}\left((\frac{K\theta}{v}-K)u+\left(\kappa-\frac{\kappa}{v}\right)\theta_y+\left(\frac{\kappa}{c_v}-\frac{\mu}{v}\right)uu_y\right)dy\right|d\tau\notag\\
		&\hspace{0.4cm}+\int_0^{\frac{t-1}{2}}\left|\int_\bbr O(1)e^{-\sigma_0^*(t-\tau)-\sigma_0|x-y|}\left((\frac{K\theta}{v}-K)u+\left(\kappa-\frac{\kappa}{v}\right)\theta_y+\left(\frac{\kappa}{c_v}-\frac{\mu}{v}\right)uu_y\right)dy\right|d\tau\notag\\
		&\leq O(1)\frac{\delta^2}{\sqrt{1+t}}.\label{5-80}
	\end{align}
	Next, for $\mathcal{R}^v_2$, the integrability is same as \eqref{5-80}, while the decay now comes from $\mathcal{G}(\tau)$. More precisely, we apply \eqref{5-65} and \eqref{5-68} again to have,
	\begin{align}
		|\mathcal{R}^v_2|
		&=\int_{\frac{t-1}{2}}^{t-1}\left|\int_\bbr \frac{1}{\mu}e^{-\frac{K(t-\tau)}{\mu}}\delta(x-y)\left(\frac{K(v-1)^2}{v}+\frac{(K\theta-K)(1-v)}{v}+\frac{\mu u_y(v-1)}{v}-\frac{Ku^2}{2c_v}\right)dy\right|d\tau\notag\\
		&\hspace{0.4cm}+\int_{\frac{t-1}{2}}^{t-1}\left|\int_\bbr \sum_{j=1}^3\frac{O(1)e^{-\frac{(x-y+\beta_j(t-\tau))^2}{4C(t-\tau)}}}{t-\tau}\left(\frac{K(v-1)^2}{v}+\frac{(K\theta-K)(1-v)}{v}+\frac{\mu u_y(v-1)}{v}-\frac{Ku^2}{2c_v}\right)dy\right|d\tau\notag\\
		&\hspace{0.4cm}+\int_{\frac{t-1}{2}}^{t-1}\left|\int_\bbr O(1)e^{-\sigma_0^*(t-\tau)-\sigma_0|x-y|}\left(\frac{K(v-1)^2}{v}+\frac{(K\theta-K)(1-v)}{v}+\frac{\mu u_y(v-1)}{v}-\frac{Ku^2}{2c_v}\right)dy\right|d\tau\notag\\
		&\hspace{0.4cm}+\int_{\frac{t-1}{2}}^{t-1}\left|\int_\bbr \sum_{j=1}^3\frac{O(1)e^{-\frac{(x-y+\beta_j(t-\tau))^2}{4C(t-\tau)}}}{t-\tau}\left((\frac{K\theta}{v}-K)u+\left(\kappa-\frac{\kappa}{v}\right)\theta_y+\left(\frac{\kappa}{c_v}-\frac{\mu}{v}\right)uu_y\right)dy\right|d\tau\notag\\
		&\hspace{0.4cm}+\int_{\frac{t-1}{2}}^{t-1}\left|\int_\bbr O(1)e^{-\sigma_0^*(t-\tau)-\sigma_0|x-y|}\left((\frac{K\theta}{v}-K)u+\left(\kappa-\frac{\kappa}{v}\right)\theta_y+\left(\frac{\kappa}{c_v}-\frac{\mu}{v}\right)uu_y\right)dy\right|d\tau\notag\\
		&\leq O(1)\frac{\delta^2}{\sqrt{1+t}}.\label{5-81}
	\end{align}
	Next, for $\mathcal{R}^v_3$, the decay comes from $\mathcal{G}(\tau)$. Moreover, as $\tau$ is close to $t$, we have to use the small time estimates of $\bbg$, and actually we apply \eqref{5-71} and \eqref{5-74} to obtain that, 
	\begin{align}
		|\mathcal{R}^v_3|
		&=\int_{t-1}^t\left|\int_\bbr \frac{1}{\mu}e^{-\frac{K(t-\tau)}{\mu}}\delta(x-y)\left(\frac{K(v-1)^2}{v}+\frac{(K\theta-K)(1-v)}{v}+\frac{\mu u_y(v-1)}{v}-\frac{Ku^2}{2c_v}\right)dy\right|d\tau\notag\\
		&\hspace{0.4cm}+\int_{t-1}^t\left|\int_\bbr \frac{O(1)e^{-\frac{(x-y)^2}{4C(t-\tau)}}}{\sqrt{t-\tau}}\left(\frac{K(v-1)^2}{v}+\frac{(K\theta-K)(1-v)}{v}+\frac{\mu u_y(v-1)}{v}-\frac{Ku^2}{2c_v}\right)dy\right|d\tau\notag\\
		&\hspace{0.4cm}+\int_{t-1}^t\left|\int_\bbr O(1)e^{-\sigma_0|x-y|}\left(\frac{K(v-1)^2}{v}+\frac{(K\theta-K)(1-v)}{v}+\frac{\mu u_y(v-1)}{v}-\frac{Ku^2}{2c_v}\right)dy\right|d\tau\notag\\
		&\hspace{0.4cm}+\int_{t-1}^t\left|\int_\bbr O(1)e^{-\sigma_0|x-y|}\left((\frac{K\theta}{v}-K)u+\left(\kappa-\frac{\kappa}{v}\right)\theta_y+\left(\frac{\kappa}{c_v}-\frac{\mu}{v}\right)uu_y\right)dy\right|d\tau\notag\\
		&\leq O(1)\frac{\delta^2}{\sqrt{t}}\leq O(1)\frac{\delta^2}{\sqrt{1+t}},\label{5-82}
	\end{align}
	where the last inequality is due to the fact $t\geq 1$. Now if $t\leq 1$, the integration of \eqref{5-82} becomes from $0$ to $t$, and we can simply obtain that similar results. Therefore, we combine \eqref{5-76}, \eqref{5-77}, \eqref{5-78}, \eqref{5-79}, \eqref{5-80}, \eqref{5-81}, \eqref{5-82} and above discussion for the case $t\leq 1$ to obtain that,
	\[\sqrt{1+t}\|v(\cdot,t)-1\|_{L^\infty}\leq C(\nu_0)\delta^*+O(1)\delta^2.\]
	Then, we take one more integration with respect to $x$ to obtain the $L^1$ estimates. We omit the details and conclude that,
	\[\|v(\cdot,t)-1\|_{L^1}\leq C(\nu_0)\delta^*+O(1)\delta^2.\]\newline
	
	\noindent $\bullet$ ($\mathbf{\|v(\cdot,t)\|_{BV}}$) Now, we study the BV norm of the specific volume $v(x,t)$. To this end, we need to further investigate the discontinuity. Actually, we can rewrite the representation of $v(x,t)$ as follows,
	\begin{align}
		&v(x,t)-1\label{5-83}\\
		&=e^{-\frac{Kt}{\mu}}(v(x,0)-1)+\int_0^{t}\frac{1}{\mu}e^{-\frac{K(t-\tau)}{\mu}}\left(\frac{K(v-1)^2}{v}+\frac{(K\theta-K)(1-v)}{v}+\frac{\mu u_x(v-1)}{v}-\frac{Ku^2}{2c_v}\right)d\tau\notag\\
		&\hspace{0.4cm}+\int_\bbr\left(\bbg_{11}(x-y,t)-e^{-\frac{Kt}{\mu}}\delta(x-y)\right)(v(y,0)-1)dy\notag\\
		&\hspace{0.4cm}+\int_\bbr \bbg_{12}(x,t;y,0)u(y,0)dy+\int_\bbr \bbg_{13}(x-y,t)\left(E(y,0)-c_v\right)dy\notag\\
		&\hspace{0.4cm}+\int_0^{t}\int_\bbr \left(\partial_y\bbg_{12}(x,t;y,\tau)-\frac{1}{\mu}e^{-\frac{K(t-\tau)}{\mu}}\delta(x-y)\right)\left(\frac{K(v-1)^2}{v}+\frac{(K\theta-K)(1-v)}{v}+\frac{\mu u_y(v-1)}{v}-\frac{Ku^2}{2c_v}\right)dyd\tau\notag\\
		&\hspace{0.4cm}+\int_0^{t}\int_\bbr \partial_y\bbg_{13}(x-y,t-\tau)\left((\frac{K\theta}{v}-K)u+\left(\kappa-\frac{\kappa}{v}\right)\theta_y+\left(\frac{\kappa}{c_v}-\frac{\mu}{v}\right)uu_y\right)dyd\tau.\notag
	\end{align}
	From the representation \eqref{5-83} and properties of Green's function, we immediately see that the discontinuity only appears in the first two terms. Therefore, we can estimates the jumps as follows,
	\begin{align}
		&\sum_{z\in\mathscr{D}}\left|v(z,t)\Big|_{z=z^-}^{z=z^-}\right|\label{5-84}\\
		&\leq e^{-\frac{Kt}{\mu}}\sum_{z\in\mathscr{D}}\left|v(z,0)\Big|_{z=z^-}^{z=z^-}\right|+O(1)\sum_{z\in\mathscr{D}}\int_0^{t}\frac{1}{\mu}e^{-\frac{K(t-\tau)}{\mu}}\left(\|v-1\|_{L^\infty}+\|\theta-1\|_{L^\infty}+\|u_x\|_{L^\infty}\right)\left|v(z,\tau)\Big|_{z=z^-}^{z=z^-}\right|d\tau\notag\\
		&\leq  e^{-\frac{Kt}{\mu}}\delta^*+O(1)\int_0^{t}\frac{1}{\mu}e^{-\frac{K(t-\tau)}{\mu}}\left(\frac{\delta}{\sqrt{1+\tau}}+\frac{\delta}{\sqrt{\tau}}\right)\delta d\tau\notag\\
		&\leq  e^{-\frac{Kt}{\mu}}\delta^*+O(1)\min\left(\delta^2,\frac{\delta^2}{\sqrt{t}}\right).\notag
	\end{align}
	Next, for $x\notin\mathscr{D}$, we can take derivative to \eqref{5-83} and then integrate in the region $\bbr\setminus\mathscr{D}$. More precisely, we apply \eqref{5-66}, \eqref{5-69}, \eqref{5-72} and \eqref{5-75} to obtain that,
	\begin{align}
		&\int_{\bbr\setminus\mathscr{D}}\left|v_x(x,t)\right|\label{5-85}\\
		&\leq e^{-\frac{Kt}{\mu}}\int_{\bbr\setminus\mathscr{D}}\left|v_x(x,0)\right|dx+O(1)\int_{\bbr\setminus\mathscr{D}}\int_0^{t}\frac{1}{\mu}e^{-\frac{K(t-\tau)}{\mu}}\left(|v-1||u_\tau|+|v-1||\theta_x|+|v-1||v_x|\right)d\tau dx\notag\\
		&\hspace{0.4cm}+O(1)\int_{\bbr\setminus\mathscr{D}}\int_0^{t}\frac{1}{\mu}e^{-\frac{K(t-\tau)}{\mu}}\left(|v-1||v_x|+|\theta-1||v_x|+|\theta_x||v-1|+|u||u_x|+|u_x||v_x|\right)d\tau dx\notag\\
		&\hspace{0.4cm}+\int_{\bbr\setminus\mathscr{D}}\int_\bbr\left|\partial_x\bbg_{11}(x-y,t)-e^{-\frac{Kt}{\mu}}\delta'(x-y)\right||v(y,0)-1|dydx\notag\\
		&\hspace{0.4cm}+\int_{\bbr\setminus\mathscr{D}}\left|\int_\bbr \partial_x\bbg_{12}(x-y,t)u(y,0)dy\right|dx+\int_{\bbr\setminus\mathscr{D}}\left|\int_\bbr \partial_x\bbg_{13}(x-y,t)\left(E(y,0)-c_v\right)dy\right|\notag\\
		&\hspace{0.4cm}+\int_{\bbr\setminus\mathscr{D}}\int_0^{t}\int_\bbr \left|-\partial_{xx}\bbg_{12}(x,t;y,\tau)-\frac{1}{\mu}e^{-\frac{K(t-\tau)}{\mu}}\delta'(x-y)\right|\left(\frac{K(v-1)^2}{v}+\frac{|K\theta-K||1-v|}{v}\right)dyd\tau dx\notag\\
		&\hspace{0.4cm}+\int_{\bbr\setminus\mathscr{D}}\int_0^{t}\int_\bbr \left|-\partial_{xx}\bbg_{12}(x,t;y,\tau)-\frac{1}{\mu}e^{-\frac{K(t-\tau)}{\mu}}\delta'(x-y)\right|\left(\frac{\mu |u_y||v-1|}{v}+\frac{Ku^2}{2c_v}\right)dyd\tau dx\notag\\
		&\hspace{0.4cm}+\int_{\bbr\setminus\mathscr{D}}\int_0^{t}\int_\bbr \left|\partial_{xx}\bbg_{13}(x-y,t-\tau)\right|\left(\left|\frac{K\theta}{v}-K\right||u|+\left|\kappa-\frac{\kappa}{v}\right||\theta_y|+\left|\frac{\kappa}{c_v}-\frac{\mu}{v}\right||u||u_y|\right)dyd\tau dx\notag\\
		&\leq C(\nu_0)\delta^*+O(1)\frac{\delta^2}{\sqrt{\nu_0}},\notag
	\end{align}
	 where the term $\sqrt{\nu_0}$ in the last term comes from the estimate of $u_\tau$ in \eqref{u_t-local}. Now, we combine \eqref{5-84} and \eqref{5-85} to obtain that, 
	\[\|v(\cdot,t)\|_{BV}\leq C(\nu_0)\delta^*+O(1)\frac{\delta^2}{\sqrt{\nu_0}}.\]
	
\end{proof}
With the a priori estimates in Lemma \ref{L5.5}, Lemma \ref{L5.6}, Lemma \ref{L5.7}, Lemma \ref{L5.8} and Lemma \ref{L5.9} in hands, we are now ready to prove the global existence of the solution constructed in Proposition \ref{prop:local} for sufficiently small initial data. 

\begin{theorem}\label{T5.1}
	Suppose the initial data $(v_0,u_0,\theta_0)$ of Navier-Stokes equation \eqref{NS2} satisfies
	\bq\label{5-88}
	\|v_0-1\|_{BV}+\|u_0\|_{BV}+\|\theta_0-1\|_{BV}+\|v_0-1\|_{L^1}+\|u_0\|_{L^1}+\|\theta_0-1\|_{L^1}<\delta^*.
	\eq
	for $\delta^*$ sufficiently small. Then the solution constructed in Proposition \ref{prop:local} exists globally in time, and there exists a positive constant $\mathscr{C}$ such that,
	\bq
	\|\sqrt{t+1}(v-1)\|_\infty+\|\sqrt{t+1}u\|_\infty+\|\sqrt{t+1}(\theta-1)\|_\infty+\|\sqrt{t}u_x\|_\infty+\|\sqrt{t}\theta_x\|_\infty<\mathscr{C} \delta^{*}.\label{large-time}
	\eq
\end{theorem}
\begin{proof}
	We let $C_\sharp$, $t_\sharp$ and $\delta$ be the parameters in Proposition \ref{prop:local}, then Propositions \ref{prop:local} and \ref{T3.3} guarantee the existence and uniqueness of weak solution $(v,u,\theta)$ in $[0,t_\sharp)$. Now we let the initial data satisfy the following smallness condition 
	\[\|v_0-1\|_{BV}+\|u_0\|_{BV}+\|\theta_0-1\|_{BV}+\|v_0-1\|_{L^1}+\|u_0\|_{L^1}+\|\theta_0-1\|_{L^1}<\delta^*.\]
	From Lemma \ref{L5.3}, for sufficiently small $\delta^*$, we can define a stopping time $T$ as in \eqref{5-15}, which has the following properties, 
	\begin{align*}
		&T>t_\sharp,\qquad \mathcal{G}(T)\geq\delta,\qquad \mathcal{G}(\tau)<\delta\ \mbox{ for  all $\tau<T$},
	\end{align*}
	where $\mathcal{G}(\tau)$ is defined in \eqref{5-15}. According the the definition of $\mathcal{G}(\tau)$ and Lemma \ref{L5.3}, we immediately conclude that the life span of the solution $(v,u,\theta)$ is larger than $T$. Therefore, in order to prove the global existence of the solution, it suffices to show $T=\infty$ for sufficient small $\delta^*$. We will prove this by contradiction. Suppose not, then we have
	\begin{equation}\label{5-89}
	T<+\infty,\quad \mbox{for  arbitrary postive $\delta^*$ }.
	\end{equation}
	On the other hand, we combine Lemma \ref{L5.5}, Lemma \ref{L5.6}, Lemma \ref{L5.7}, Lemma \ref{L5.8} and Lemma \ref{L5.9} together to find that, the solution at $T$ has the following estimates, 
	\[\mathcal{G}(T)\leq C(\nu_0)\delta^*+O(1)\frac{|\log(\nu_0)|}{\sqrt{\nu_0}}\delta^2+O(1)\sqrt{\nu_0}\delta+\left(C(\nu_0)\delta^*+O(1)\frac{|\log(\nu_0)|}{\sqrt{\nu_0}}\delta^2+O(1)\sqrt{\nu_0}\delta\right)^2,\]
	where $\nu_0$ is some small positive constant such that $\nu_0\leq \frac{t_\sharp}{4}$. Actually, it is observed that when we choose $\delta^*$ and $\delta$ to be smaller, all the other parameters such as $C_\sharp$ and $t_\sharp$ are uniformly bounded, and thus can be viewed as unchanged. Moreover, all the $O(1)$ coefficients in above formula are independent of $\nu_0$, and thus the $O(1)$ coefficients will be uniformly bounded when $\delta$, $\delta^*$ and $\nu_0$ are changed to be smaller. Therefore, we can first choose $\nu_0$ to be sufficiently small such that 
	\[O(1)\sqrt{\nu_0}\delta\leq \frac{\delta}{6}.\]
	Then, we fix the $\nu_0$, and let $\delta$ to be sufficiently small so that 
	\[O(1)\frac{|\log(\nu_0)|}{\sqrt{\nu_0}}\delta^2\leq \frac{\delta}{6}.\]
	Finally, for fixed small $\nu_0$ and $\delta$, we let $\delta^*$ to be sufficiently small, so that we can obtain 
	\begin{equation}\label{eq:4-89-1}
	C(\nu_0)\delta^*\leq \frac{\delta}{6}.
	\end{equation}
	Now, for these well chosen $\nu_0$, $\delta$ and $\delta^*$, we combine all above estimates to obtain that,
	\[\mathcal{G}(T)\leq \frac{\delta}{2}+\frac{\delta^2}{4}\leq\frac{3\delta}{4}<\delta,\]
	which  obviously contradicts to the assumption \eqref{5-89}. Thus, the assumption \eqref{5-89} is false and we conclude that, for sufficiently small positive constant $\delta^*$ such that \eqref{5-88} holds, we can find a small positive constant $\delta$ such that,
	\bq\label{5-90}
	\mathcal{G}(t)<\delta,\quad \mbox{for all  $t>0$}.
	\eq
	From previous discussion and the definition of $\mathcal{G}(t)$ in \eqref{5-15}, \eqref{5-90} immediately implies the global existence and uniqueness of the weak solution. Moreover, the large time behavior \eqref{large-time} directly follows \eqref{5-90}. In view of the relation \eqref{eq:4-89-1} between $\delta$ and $\delta^{*}$,  we finish the proof of the theorem.
	
\end{proof}

\newpage
\appendix

\section{Coefficients of the approximated eigenvalues $\lambda_j^*$}\label{sec:A}
In this part, we list the coefficients obtained in the expansion of $\lambda_j$ at infinity.
\begin{align*}
	\beta_1^*=&\frac{v p_v}{\mu },\\
	A_{1,1}=& -\frac{v^3 \left(\kappa  \theta _e p_v^2+\mu  p p_e p_v\right)}{ \kappa  \mu ^3 \theta _e},\\
	A_{1,2}=&\frac{v^3 \left(\mu ^2 p^2 v^2 p_e^2 p_v+2 \kappa ^2 v^2 \theta _e^2 p_v^3+3 \kappa  \mu 
		p v^2 \theta _e p_e p_v^2+\mu ^2 p v^2 p_e p_v^2-\kappa ^2 \mu ^2 \theta _e^2
		p_v^2+\kappa  \mu ^3 (-p) \theta _e p_e p_v\right)}{ \kappa ^2
		\mu ^5 \theta _e^2},\\
	A_{1,3}=&-\frac{v^3 p_v \left(\mu ^3 p^3 v^4 p_e^3+\mu ^2 p^2 v^2 p_e^2
		\left(3 v^2 p_v \left(2 \kappa  \theta _e+\mu \right)-2 \kappa  \mu ^2 \theta
		_e\right)+\kappa ^3 \theta _e^3 p_v \left(\mu ^4+5 v^4 p_v^2-4 \mu ^2 v^2
		p_v\right)\right)}{ \kappa
		^3 \mu ^7 \theta _e^3}\\
	&-\frac{v^3 p_v \left(\mu  p p_e \left(\kappa ^2 \mu ^4 \theta _e^2+v^4 p_v^2 \left(10 \kappa ^2
		\theta _e^2+4 \kappa  \mu  \theta _e+\mu ^2\right)-2 \kappa  \mu ^2 v^2 \theta _e p_v
		\left(3 \kappa  \theta _e+\mu \right)\right)\right)}{\kappa
		^3 \mu ^7 \theta _e^3}.\\
\end{align*}

\begin{align*}
	\alpha^*_2=&\frac{\mu}{v},\\
	\beta^*_2=&\frac{v \left(\mu  p p_e+\kappa  \theta _e p_v-\mu 
		p_v\right)}{\mu  \left(\mu -\kappa  \theta _e\right)},\\
	A_{2,1}=&\frac{v^3 \left(\mu ^3 p^2 p_e^2-\mu  p p_e p_v \left(\kappa ^2 \theta _e^2-3 \kappa  \mu
		\theta _e+2 \mu ^2\right)+p_v^2 \left(\mu -\kappa  \theta
		_e\right){}^3\right)}{\mu ^3 \left(\mu -\kappa  \theta
		_e\right){}^3},\\
	A_{2,2}=&\frac{v^3 \left(\mu  p p_e+p_v \left(\kappa  \theta _e-\mu
		\right)\right) \left(2 \mu ^4 p^2 v^2 p_e^2+\mu  p p_e \left(\mu -\kappa  \theta
		_e\right) \left(\mu ^3 \left(\mu -\kappa  \theta _e\right)-v^2 p_v \left(\kappa 
		\theta _e-2 \mu \right){}^2\right)\right)}{ \mu ^5 \left(\mu -\kappa 
		\theta _e\right){}^5}\\
	&+\frac{v^3 \left(\mu  p p_e+p_v \left(\kappa  \theta _e-\mu
		\right)\right) \left(p_v \left(2 v^2 p_v-\mu ^2\right) \left(\mu -\kappa
		\theta _e\right){}^4\right)}{ \mu ^5 \left(\mu -\kappa 
		\theta _e\right){}^5},\\
	A_{2,3}=&\frac{v^3 \left(5 \mu ^7 p^4 v^4 p_e^4+\mu ^3 p^3 v^2 p_e^3 \left(\mu -\kappa  \theta
		_e\right) \left(4 \mu ^4 \left(\mu -\kappa  \theta _e\right)+v^2 p_v \left(\kappa ^3
		\theta _e^2-6 \kappa ^2 \mu  \theta _e^2+15 \kappa  \mu ^2 \theta _e-20 \mu
		^3\right)\right)\right)}{ \mu ^7 \left(\mu
		-\kappa  \theta _e\right){}^7}\\
	&+\frac{v^3 \mu ^2 p^2 p_e^2 \left(\mu -\kappa  \theta _e\right){}^2 \left(\mu ^5
		\left(\mu -\kappa  \theta _e\right){}^2+3 v^4 p_v^2 \left(-2 \kappa ^3 \theta _e^3+9
		\kappa ^2 \mu  \theta _e^2-15 \kappa  \mu ^2 \theta _e+10 \mu ^3\right)\right)}{ \mu ^7 \left(\mu
		-\kappa  \theta _e\right){}^7}\\
	&-\frac{v^3 \mu ^2 p^2 p_e^2 \left(\mu -\kappa  \theta _e\right){}^2 2 \mu ^2 v^2
		p_v \left(-\kappa ^3 \theta _e^3+5 \kappa ^2 \mu  \theta _e^2-10 \kappa  \mu ^2 \theta
		_e+6 \mu ^3\right)}{ \mu ^7 \left(\mu
		-\kappa  \theta _e\right){}^7}\\
	&-\frac{v^3 \mu  p p_e p_v \left(\mu -\kappa  \theta _e\right){}^3
		\left(\mu ^4 \left(\mu -\kappa  \theta _e\right){}^2 \left(2 \mu -\kappa  \theta
		_e\right)+v^4 p_v^2 \left(-10 \kappa ^3 \theta _e^3+36 \kappa ^2 \mu  \theta _e^2-45
		\kappa  \mu ^2 \theta _e+20 \mu ^3\right)\right)}{ \mu ^7 \left(\mu
		-\kappa  \theta _e\right){}^7}\\
	&-\frac{v^3 \mu  p p_e p_v \left(\mu -\kappa  \theta _e\right){}^3
		\left(2 \mu ^2 v^2 p_v \left(3 \kappa ^3 \theta
		_e^2-11 \kappa ^2 \mu  \theta _e^2+14 \kappa  \mu ^2 \theta _e-6 \mu
		^3\right)\right)}{ \mu ^7 \left(\mu
		-\kappa  \theta _e\right){}^7}\\
	&+\frac{v^3 p_v^2 \left(\mu ^4+5 v^4 p_v^2-4 \mu ^2 v^2 p_v\right) \left(\mu
		-\kappa  \theta _e\right){}^7}{ \mu ^7 \left(\mu
		-\kappa  \theta _e\right){}^7}.
\end{align*}

\begin{align*}
	\alpha^*_3=&\frac{\kappa\theta_e}{v},\\
	\beta^*_3=&\frac{p v p_e}{\kappa  \theta _e-\mu },\\
	A_{3,1}= &\frac{p v^3 p_e \left(\kappa  p \theta _e p_e+p_v \left(\kappa 
		\theta _e-\mu \right)\right)}{\kappa  \theta _e \left(\kappa 
		\theta _e-\mu \right){}^3},\\
	A_{3,2}= &\frac{p v^3 p_e \left(2 \kappa ^2 p^2 v^2 \theta _e^2 p_e^2+p p_e \left(\mu -\kappa 
		\theta _e\right) \left(\kappa ^2 \theta _e^2 \left(\mu -\kappa  \theta _e\right)+v^2
		p_v \left(\mu -4 \kappa  \theta _e\right)\right)\right)}{ \kappa ^2 \theta _e^2 \left(\kappa  \theta
		_e-\mu \right){}^5}\\
	&+\frac{p v^3 p_e \left(p_v \left(\mu -\kappa  \theta
		_e\right){}^2 \left(\kappa  \theta _e \left(\kappa  \theta _e-\mu \right)+v^2
		p_v\right)\right)}{ \kappa ^2 \theta _e^2 \left(\kappa  \theta
		_e-\mu \right){}^5},\\
	A_{3,3}=&\frac{p v^3 p_e \left(5 \kappa ^3 p^3 v^4 \theta _e^3 p_e^2-p^2 v^2 p_e^2 \left(\mu
		-\kappa  \theta _e\right) \left(4 \kappa ^3 \theta _e^3 \left(\kappa  \theta _e-\mu
		\right)+v^2 p_v \left(15 \kappa ^2 \theta _e^2-6 \kappa  \mu  \theta _e+\mu
		^2\right)\right)\right)}{ \kappa ^3 \theta _e^3 \left(\kappa  \theta _e-\mu \right){}^7}\\
	&+\frac{p v^3 p_e p p_e \left(\mu -\kappa  \theta
		_e\right){}^2 \left(\kappa ^3 \theta _e^3 \left(\mu -\kappa  \theta _e\right){}^2-3
		v^4 p_v^2 \left(\mu -3 \kappa  \theta _e\right)+2 \kappa  v^2 \theta _e p_v \left(4
		\kappa ^2 \theta _e^2-5 \kappa  \mu  \theta _e+\mu ^2\right)\right)}{ \kappa ^3 \theta _e^3 \left(\kappa  \theta _e-\mu \right){}^7}\\
	&+\frac{p v^3 p_e \left(-p_v \left(\mu -\kappa  \theta _e\right){}^3 \left(\kappa  \theta _e
		\left(\kappa  \theta _e-\mu \right)+v^2 p_v\right){}^2\right)}{ \kappa ^3 \theta _e^3 \left(\kappa  \theta _e-\mu \right){}^7}.
\end{align*}

\section{The expression of $M_j^{*,k}$ in Lemma \ref{L4.2}}\label{sec:B}
In this part, we will list the $M_j^{*,k}$ in Lemma \ref{L4.2}.
\bq\label{M_1^*}
\begin{aligned}
	&M_1^{*,0}=\left(
	\begin{array}{ccc}
		1 & 0 & 0 \\
		0 & 0 & 0 \\
		0 & 0 & 0 \\
	\end{array}\right),\quad M_1^{*,1}=\left(
	\begin{array}{ccc}
		0 & \frac{v}{\mu } & 0 \\
		-\frac{v p_v}{\mu } & 0 & 0 \\
		-\frac{u v p_v}{\mu } & 0 & 0 \\
	\end{array}
	\right),\quad M_1^{*,2}=\left(
	\begin{array}{ccc}
		-\frac{v^2 p_v}{\mu ^2} & -\frac{u v^2 p_e}{\kappa  \mu  \theta _e} & \frac{v^2
			p_e}{\kappa  \mu  \theta _e} \\
		0 & \frac{v^2 p_v}{\mu ^2} & 0 \\
		-\frac{p v^2 p_v}{\kappa  \mu  \theta _e} & \frac{u v^2 p_v}{\mu ^2} & 0 \\
	\end{array}
	\right),\\
	&M_1^{*,3}=\left(
	\begin{array}{ccc}
		0 & -\frac{ v^3 \left(p \mu  p_e+2 \kappa  p_v \theta _e\right)}{\kappa  \mu ^3 \theta
			_e} & 0 \\
		\frac{ v^3 p_v \left(p \mu  p_e+2 \kappa  p_v \theta _e\right)}{\kappa  \mu ^3 \theta
			_e} & \frac{ u v^3 p_e p_v}{\kappa  \mu ^2 \theta _e} & -\frac{ v^3 p_e p_v}{\kappa 
			\mu ^2 \theta _e} \\
		\frac{ u v^3 p_v \left(p \mu  p_e+2 \kappa  p_v \theta _e\right)}{\kappa  \mu ^3 \theta
			_e} & -\frac{ v^3 \left(p p_v-u^2 p_e p_v\right)}{\kappa  \mu ^2 \theta _e} &
		-\frac{ u v^3 p_e p_v}{\kappa  \mu ^2 \theta _e} \\
	\end{array}
	\right),\\
	&M_1^{*,4}=\left(
	\begin{array}{ccc}
		\frac{v^4 \left(p p_e p_v \mu ^2+2 p \kappa  p_e p_v \theta _e \mu +3 \kappa ^2 p_v^2
			\theta _e^2\right)}{\kappa ^2 \mu ^4 \theta _e^2} & \frac{u v^4 p_e \left(p \mu 
			p_e+\mu  p_v+2 \kappa  p_v \theta _e\right)}{\kappa ^2 \mu ^3 \theta _e^2} &
		-\frac{v^4 p_e \left(p \mu  p_e+\mu  p_v+2 \kappa  p_v \theta _e\right)}{\kappa ^2 \mu
			^3 \theta _e^2} \\
		0 & \frac{-2 p \mu  p_e p_v v^3-3 \kappa  p_v^2 \theta _e v^4}{\kappa  \mu ^4 \theta _e}
		& 0 \\
		\frac{p v^4 p_v \left(p \mu  p_e+\mu  p_v+2 \kappa  p_v \theta _e\right)}{\kappa ^2 \mu
			^3 \theta _e^2} & \frac{v^4 \left(p u p_e p_v \mu ^2-2 p u \kappa  p_e p_v \theta _e
			\mu -3 u \kappa ^2 p_v^2 \theta _e^2\right)}{\kappa ^2 \mu ^4 \theta _e^2} & -\frac{p
			v^4 p_e p_v}{\kappa ^2 \mu ^2 \theta _e^2}
	\end{array}
	\right).
\end{aligned}
\eq
\bq\label{M_2^*}
\begin{aligned}
	&M_2^{*,0}=\left(
	\begin{array}{ccc}
		0 & 0 & 0 \\
		0 & 1 & 0 \\
		0 & u & 0 \\
	\end{array}
	\right),\quad M_2^{*,1}=\left(
	\begin{array}{ccc}
		0 & -\frac{ v}{\mu } & 0 \\
		\frac{ v p_v}{\mu } & -\frac{u v p_e}{\mu -\kappa  \theta _e} & \frac{v p_e}{\mu
			-\kappa  \theta _e} \\
		\frac{u v p_v}{\mu } & \frac{v \left(p-u^2 p_e\right)}{\mu -\kappa  \theta _e} &
		\frac{ u v p_e}{\mu -\kappa  \theta _e} \\
	\end{array}
	\right),\\
	&M_2^{*,2}=\left(
	\begin{array}{ccc}
		\frac{v^2 p_v}{\mu ^2} & -\frac{u v^2 p_e}{\mu ^2-\kappa  \mu  \theta _e} & \frac{v^2
			p_e}{\mu ^2-\kappa  \mu  \theta _e} \\
		0 & \frac{v^2 \left(p \mu ^2 p_e-p_v \left(\mu -\kappa  \theta _e\right){}^2\right)}{\mu
			^2 \left(\mu -\kappa  \theta _e\right){}^2} & 0 \\
		-\frac{p v^2 p_v}{\mu ^2-\kappa  \mu  \theta _e} & \frac{u v^2 \left(2 p \mu ^2 p_e-p_v
			\left(\mu -\kappa  \theta _e\right){}^2\right)}{\mu ^2 \left(\mu -\kappa  \theta
			_e\right){}^2} & -\frac{p v^2 p_e}{\left(\mu -\kappa  \theta _e\right){}^2} \\
	\end{array}
	\right),\\
	&M_2^{*,3}=\left(
	\begin{array}{ccc}
		0 & \frac{v^3 \left(2 p_v+\frac{p \mu  p_e \left(\kappa  \theta _e-2 \mu
				\right)}{\left(\mu -\kappa  \theta _e\right){}^2}\right)}{\mu ^3} & 0 \\
		-\frac{ v^3 p_v \left(2 p_v+\frac{p \mu  p_e \left(\kappa  \theta _e-2 \mu
				\right)}{\left(\mu -\kappa  \theta _e\right){}^2}\right)}{\mu ^3} & -\frac{ u v^3 p_e
			\left(2 p \mu ^2 p_e-p_v \left(2 \mu ^2-3 \kappa  \theta _e \mu +\kappa ^2 \theta
			_e^2\right)\right)}{\mu ^2 \left(\mu -\kappa  \theta _e\right){}^3} & \frac{ v^3 p_e
			\left(2 p \mu ^2 p_e-p_v \left(2 \mu ^2-3 \kappa  \theta _e \mu +\kappa ^2 \theta
			_e^2\right)\right)}{\mu ^2 \left(\mu -\kappa  \theta _e\right){}^3} \\
		-\frac{ u v^3 p_v \left(2 p_v \left(\mu -\kappa  \theta _e\right){}^2+p \mu  p_e
			\left(\kappa  \theta _e-2 \mu \right)\right)}{\mu ^3 \left(\mu -\kappa  \theta
			_e\right){}^2} & \frac{ v^3 \left(p-u^2 p_e\right) \left(2 p \mu ^2 p_e-p_v \left(2
			\mu ^2-3 \kappa  \theta _e \mu +\kappa ^2 \theta _e^2\right)\right)}{\mu ^2 \left(\mu
			-\kappa  \theta _e\right){}^3} & \frac{ u v^3 p_e \left(2 p \mu ^2 p_e-p_v \left(2
			\mu ^2-3 \kappa  \theta _e \mu +\kappa ^2 \theta _e^2\right)\right)}{\mu ^2 \left(\mu
			-\kappa  \theta _e\right){}^3} \\
	\end{array}
	\right),\\
	&M_2^{*,4}=(\xi_1,\xi_2,\xi_3),\quad \xi_1=\left(
	\begin{array}{ccc}
		\frac{v^4 p_v \left(\frac{p \mu  p_e \left(3 \mu -2 \kappa  \theta _e\right)}{\left(\mu
				-\kappa  \theta _e\right){}^2}-3 p_v\right)}{\mu ^4}  \\
		0 \\
		-\frac{p v^4 p_v \left(p \mu  p_e \left(3 \mu -\kappa  \theta _e\right)+p_v \left(-3 \mu
			^2+5 \kappa  \theta _e \mu -2 \kappa ^2 \theta _e^2\right)\right)}{\mu ^3 \left(\mu
			-\kappa  \theta _e\right){}^3}\\
	\end{array}
	\right),\\
	&(\xi_2,\xi_3)=\left(
	\begin{array}{ccc}
		-\frac{u v^4 p_e \left(p \mu 
			p_e \left(3 \mu -\kappa  \theta _e\right)+p_v \left(-3 \mu ^2+5 \kappa  \theta _e \mu
			-2 \kappa ^2 \theta _e^2\right)\right)}{\mu ^3 \left(\mu -\kappa  \theta
			_e\right){}^3} & \frac{v^4 p_e \left(p \mu  p_e \left(3 \mu -\kappa  \theta
			_e\right)+p_v \left(-3 \mu ^2+5 \kappa  \theta _e \mu -2 \kappa ^2 \theta
			_e^2\right)\right)}{\mu ^3 \left(\mu -\kappa  \theta _e\right){}^3} \\
		\frac{v^4 \left(3 p^2 p_e^2 \mu ^4+2 p p_e p_v \left(-3 \mu ^3+6 \kappa  \theta _e
			\mu ^2-4 \kappa ^2 \theta _e^2 \mu +\kappa ^3 \theta _e^3\right) \mu +3 p_v^2
			\left(\mu -\kappa  \theta _e\right){}^4\right)}{\mu ^4 \left(\mu -\kappa  \theta
			_e\right){}^4} & 0 \\
		\frac{u v^4 \left(6 p^2 p_e^2 \mu ^4+p p_e p_v
			\left(-9 \mu ^3+16 \kappa  \theta _e \mu ^2-9 \kappa ^2 \theta _e^2 \mu +2 \kappa ^3
			\theta _e^3\right) \mu +3 p_v^2 \left(\mu -\kappa  \theta _e\right){}^4\right)}{\mu ^4
			\left(\mu -\kappa  \theta _e\right){}^4} & \frac{p v^4 p_e \left(p_v \left(3 \mu ^2-4
			\kappa  \theta _e \mu +\kappa ^2 \theta _e^2\right)-3 p \mu ^2 p_e\right)}{\mu ^2
			\left(\mu -\kappa  \theta _e\right){}^4} \\
	\end{array}
	\right).
\end{aligned}
\eq
\bq\label{M_3^*}
\begin{aligned}
	&M_3^{*,0}=\left(
	\begin{array}{ccc}
		0 & 0 & 0 \\
		0 & 0 & 0 \\
		0 & -u & 1 \\
	\end{array}
	\right),\quad M_3^{*,1}=\left(
	\begin{array}{ccc}
		0 & 0 & 0 \\
		0 & \frac{u v p_e}{\mu -\kappa  \theta _e} & -\frac{v p_e}{\mu -\kappa  \theta _e} \\
		0 & -\frac{v \left(p-u^2 p_e\right)}{\mu -\kappa  \theta _e} & -\frac{u v p_e}{\mu
			-\kappa  \theta _e} \\
	\end{array}
	\right),\\
	&M_3^{*,2}=\left(
	\begin{array}{ccc}
		0 & \frac{u v^2 p_e}{\kappa  \theta _e \left(\mu -\kappa  \theta _e\right)} & -\frac{v^2
			p_e}{\kappa  \theta _e \left(\mu -\kappa  \theta _e\right)} \\
		0 & -\frac{p v^2 p_e}{\left(\mu -\kappa  \theta _e\right){}^2} & 0 \\
		\frac{p v^2 p_v}{\kappa  \theta _e \left(\mu -\kappa  \theta _e\right)} & -\frac{2 p u
			v^2 p_e}{\left(\mu -\kappa  \theta _e\right){}^2} & \frac{p v^2 p_e}{\left(\mu -\kappa
			\theta _e\right){}^2} \\
	\end{array}
	\right),\\
	&M_3^{*,3}=\left(
	\begin{array}{ccc}
		0 & \frac{p v^3 p_e}{\kappa  \theta _e \left(\mu -\kappa  \theta _e\right){}^2} & 0 \\
		-\frac{p v^3 p_e p_v}{\kappa  \theta _e \left(\mu -\kappa  \theta _e\right){}^2} &
		\frac{u v^3 p_e \left(p_v \left(\mu -\kappa  \theta _e\right)-2 p \kappa  p_e \theta
			_e\right)}{\kappa  \theta _e \left(\kappa  \theta _e-\mu \right){}^3} & \frac{v^3 p_e
			\left(2 p \kappa  p_e \theta _e+p_v \left(\kappa  \theta _e-\mu \right)\right)}{\kappa
			\theta _e \left(\kappa  \theta _e-\mu \right){}^3} \\
		-\frac{p u v^3 p_e p_v}{\kappa  \theta _e \left(\mu -\kappa  \theta _e\right){}^2} &
		\frac{v^3 \left(p-u^2 p_e\right) \left(2 p \kappa  p_e \theta _e+p_v \left(\kappa 
			\theta _e-\mu \right)\right)}{\kappa  \theta _e \left(\kappa  \theta _e-\mu
			\right){}^3} & \frac{u v^3 p_e \left(2 p \kappa  p_e \theta _e+p_v \left(\kappa 
			\theta _e-\mu \right)\right)}{\kappa  \theta _e \left(\kappa  \theta _e-\mu
			\right){}^3} \\
	\end{array}
	\right),\\
	&M_3^{*,4}=\left(
	\begin{array}{ccc}
		-\frac{p v^4 p_e p_v}{\kappa ^2 \theta _e^2 \left(\mu -\kappa  \theta _e\right){}^2} &
		\frac{u v^4 p_e \left(p p_e \left(\mu -3 \kappa  \theta _e\right)+p_v \left(\mu
			-\kappa  \theta _e\right)\right)}{\kappa ^2 \theta _e^2 \left(\kappa  \theta _e-\mu
			\right){}^3} & \frac{v^4 p_e \left(p_v \left(\kappa  \theta _e-\mu \right)-p p_e
			\left(\mu -3 \kappa  \theta _e\right)\right)}{\kappa ^2 \theta _e^2 \left(\kappa 
			\theta _e-\mu \right){}^3} \\
		0 & \frac{p v^4 p_e \left(2 p_v \left(\mu -\kappa  \theta _e\right)-3 p \kappa  p_e
			\theta _e\right)}{\kappa  \theta _e \left(\mu -\kappa  \theta _e\right){}^4} & 0 \\
		\frac{p v^4 p_v \left(p p_e \left(\mu -3 \kappa  \theta _e\right)+p_v \left(\mu -\kappa 
			\theta _e\right)\right)}{\kappa ^2 \theta _e^2 \left(\kappa  \theta _e-\mu
			\right){}^3} & -\frac{p u v^4 p_e \left(6 p \kappa ^2 p_e \theta _e^2+p_v \left(\mu
			^2-6 \kappa  \theta _e \mu +5 \kappa ^2 \theta _e^2\right)\right)}{\kappa ^2 \theta
			_e^2 \left(\mu -\kappa  \theta _e\right){}^4} & \frac{p v^4 p_e \left(3 p \kappa ^2
			p_e \theta _e^2+p_v \left(\mu ^2-4 \kappa  \theta _e \mu +3 \kappa ^2 \theta
			_e^2\right)\right)}{\kappa ^2 \theta _e^2 \left(\mu -\kappa  \theta _e\right){}^4} \\
	\end{array}
	\right).
\end{aligned}
\eq

\section{The expression of $M_j^k$ in \eqref{B53}}\label{sec:C}
The zeroth order coefficient of the expansion of $\hat{\bbm}_j$ are listed below,
\begin{align*}
	&M_1^0=\left(
	\begin{array}{ccc}
		\frac{p p_e}{p p_e-p_v} & -\frac{u p_e}{p p_e-p_v} & \frac{p_e}{p p_e-p_v} \\
		0 & 0 & 0 \\
		-\frac{p p_v}{p p_e-p_v} & \frac{u p_v}{p p_e-p_v} & -\frac{p_v}{p p_e-p_v} \\
	\end{array}
	\right),\\
	&M_2^0=\left(
	\begin{array}{ccc}
		-\frac{p_v}{2 p p_e-2 p_v} & \frac{u p_e-\sqrt{p p_e-p_v}}{2 p p_e-2 p_v} &
		-\frac{p_e}{2 p p_e-2 p_v} \\
		\frac{p_v}{2 \sqrt{p p_e-p_v}} & \frac{1}{2}-\frac{u p_e}{2 \sqrt{p p_e-p_v}} &
		\frac{p_e}{2 \sqrt{p p_e-p_v}} \\
		\frac{\left(p+u \sqrt{p p_e-p_v}\right) p_v}{2 p p_e-2 p_v} & \frac{-p_e \sqrt{p
				p_e-p_v} u^2-p_v u+p \sqrt{p p_e-p_v}}{2 p p_e-2 p_v} & \frac{p_e \left(p+u \sqrt{p
				p_e-p_v}\right)}{2 p p_e-2 p_v} \\
	\end{array}
	\right),\\
	&M_3^0=\left(
	\begin{array}{ccc}
		-\frac{p_v}{2 p p_e-2 p_v} & \frac{u p_e+\sqrt{p p_e-p_v}}{2 p p_e-2 p_v} &
		-\frac{p_e}{2 p p_e-2 p_v} \\
		-\frac{p_v}{2 \sqrt{p p_e-p_v}} & \frac{1}{2} \left(\frac{u p_e}{\sqrt{p
				p_e-p_v}}+1\right) & -\frac{p_e}{2 \sqrt{p p_e-p_v}} \\
		\frac{\left(p-u \sqrt{p p_e-p_v}\right) p_v}{2 p p_e-2 p_v} & -\frac{-p_e \sqrt{p
				p_e-p_v} u^2+p_v u+p \sqrt{p p_e-p_v}}{2 p p_e-2 p_v} & \frac{p_e \left(p-u \sqrt{p
				p_e-p_v}\right)}{2 p p_e-2 p_v} \\
	\end{array}
	\right).
\end{align*}
The first order coefficient of the expansion of $\hat{\bbm}_j$ are listed below,
\begin{align*}
	& M_1^1=\left(
	\begin{array}{ccc}
		0 & \frac{ p \kappa  p_e \theta _e}{v \left(p_v-p p_e\right){}^2} & 0 \\
		-\frac{p \kappa  p_e p_v \theta _e}{v \left(p_v-p p_e\right){}^2} & \frac{u \kappa 
			p_e p_v \theta _e}{v \left(p_v-p p_e\right){}^2} & -\frac{\kappa  p_e p_v \theta
			_e}{v \left(p_v-p p_e\right){}^2} \\
		-\frac{p u \kappa  p_e p_v \theta _e}{v \left(p_v-p p_e\right){}^2} & -\frac{\kappa 
			\left(p-u^2 p_e\right) p_v \theta _e}{v \left(p_v-p p_e\right){}^2} & -\frac{u
			\kappa  p_e p_v \theta _e}{v \left(p_v-p p_e\right){}^2} \\
	\end{array}
	\right),\\
	&M_2^1=(\xi_1,\xi_2,\xi_3),\\
	&\xi_1=\left(
	\begin{array}{ccc}
		\frac{p_v \left(\mu  p_v-p p_e \left(\mu +3 \kappa  \theta _e\right)\right)}{4 v
			\left(p p_e-p_v\right){}^{5/2}}  \\
		\frac{p \kappa  p_e p_v \theta _e}{2 v \left(p_v-p p_e\right){}^2}  \\
		\frac{p p_v \left(p_e \left(p \mu +\kappa  \left(p+2 u \sqrt{p p_e-p_v}\right) \theta
			_e\right)-p_v \left(\mu -2 \kappa  \theta _e\right)\right)}{4 v \left(p
			p_e-p_v\right){}^{5/2}}
	\end{array}
	\right),\\
	&\xi_2=\left(
	\begin{array}{ccc}
		\frac{p_e \left(-2 p \kappa  \sqrt{p p_e-p_v}
			\theta _e-u p_v \left(\mu -2 \kappa  \theta _e\right)+p u p_e \left(\mu +\kappa 
			\theta _e\right)\right)}{4 v \left(p p_e-p_v\right){}^{5/2}}  \\
		-\frac{
			\left(p_e \left(p \sqrt{p p_e-p_v} \mu +\left(2 u \kappa  p_v-p \kappa  \sqrt{p
				p_e-p_v}\right) \theta _e\right)-\mu  \sqrt{p p_e-p_v} p_v\right)}{4 v \left(p_v-p
			p_e\right){}^2} 
		\\
		-\frac{\left(2 p^2 u \left(\mu -\kappa  \theta _e\right)
			p_e^2+u p_v \left(\kappa  \left(5 p+2 u \sqrt{p p_e-p_v}\right) \theta _e-3 p \mu
			\right) p_e+p_v \left(u \mu  p_v-2 p \kappa  \sqrt{p p_e-p_v} \theta
			_e\right)\right)}{4 v \left(p p_e-p_v\right){}^{5/2}}
	\end{array}
	\right),\\
	&\xi_3=\left(
	\begin{array}{ccc}
		-\frac{p_e \left(p
			p_e \left(\mu +\kappa  \theta _e\right)-p_v \left(\mu -2 \kappa  \theta
			_e\right)\right)}{4 v \left(p p_e-p_v\right){}^{5/2}} \\
		\frac{\kappa  p_e p_v \theta _e}{2 v \left(p_v-p p_e\right){}^2}\\
		\frac{p_e \left(p_e
			\left(\mu -\kappa  \theta _e\right) p^2+p_v \left(2 \kappa  \left(2 p+u \sqrt{p
				p_e-p_v}\right) \theta _e-p \mu \right)\right)}{4 v \left(p p_e-p_v\right){}^{5/2}} \\
	\end{array}
	\right),\\
	&M_3^1=(\zeta_1,\zeta_2,\zeta_3),\\
	&\zeta_1=\left(
	\begin{array}{ccc}
		\frac{p_v \left(p p_e \left(\mu +3 \kappa  \theta _e\right)-\mu  p_v\right)}{4 v
			\left(p p_e-p_v\right){}^{5/2}}  \\
		\frac{p \kappa  p_e p_v \theta _e}{2 v \left(p_v-p p_e\right){}^2} \\
		\frac{p_v \left(p p_v \left(\mu -2 \kappa  \theta _e\right)-p p_e \left(p \mu +\kappa 
			\left(p-2 u \sqrt{p p_e-p_v}\right) \theta _e\right)\right)}{4 v \left(p
			p_e-p_v\right){}^{5/2}} 
	\end{array}
	\right),\\
	&\zeta_2=\left(
	\begin{array}{ccc}
		-\frac{p_e \left(2 p \kappa  \sqrt{p p_e-p_v}
			\theta _e-u p_v \left(\mu -2 \kappa  \theta _e\right)+p u p_e \left(\mu +\kappa 
			\theta _e\right)\right)}{4 v \left(p p_e-p_v\right){}^{5/2}} \\
		\frac{\left(p_e
			\left(p \mu  \sqrt{p p_e-p_v}-\kappa  \left(\sqrt{p p_e-p_v} p+2 u p_v\right) \theta
			_e\right)-\mu  \sqrt{p p_e-p_v} p_v\right)}{4 v \left(p_v-p p_e\right){}^2} \\
		\frac{\left(2 p^2 u \left(\mu -\kappa  \theta _e\right)
			p_e^2+u p_v \left(\kappa  \left(5 p-2 u \sqrt{p p_e-p_v}\right) \theta _e-3 p \mu
			\right) p_e+p_v \left(u \mu  p_v+2 p \kappa  \sqrt{p p_e-p_v} \theta
			_e\right)\right)}{4 v \left(p p_e-p_v\right){}^{5/2}} 
	\end{array}
	\right),\\
	&\zeta_3=\left(
	\begin{array}{ccc}
		\frac{p_e \left(p p_e
			\left(\mu +\kappa  \theta _e\right)-p_v \left(\mu -2 \kappa  \theta
			_e\right)\right)}{4 v \left(p p_e-p_v\right){}^{5/2}} \\
		\frac{
			\kappa  p_e p_v \theta _e}{2 v \left(p_v-p p_e\right){}^2} \\
		\frac{p_e \left(p_e
			\left(\kappa  \theta _e-\mu \right) p^2+p_v \left(p \mu +\left(2 u \kappa  \sqrt{p
				p_e-p_v}-4 p \kappa \right) \theta _e\right)\right)}{4 v \left(p
			p_e-p_v\right){}^{5/2}} \\
	\end{array}
	\right).
\end{align*}

\section*{Acknowledgments}

The work of H.T. Wang is supported by National Nature Science Foundation of China under Grant
No. 11901386 and 12031013, the Strategic Priority Research Program of Chinese Academy of Sciences under Grant No. XDA25010403. The work of X.T. Zhang is supported by the National Natural Science Foundation of China (Grant No. 11801194 and 11971188), Hubei Key Laboratory of Engineering Modeling and Scientific Computing.

%
%
%
%

\end{document}